\documentclass[11pt,a4paper]{article}
\usepackage[a4paper,margin=20mm]{geometry}
\usepackage[T1]{fontenc}
\usepackage[utf8]{inputenc}
\usepackage{mathpazo}
\usepackage{anyfontsize}
\usepackage{courier}
\usepackage[expansion=false]{microtype}
\usepackage{amsmath,amssymb,amsthm,amscd,mathtools}
\usepackage{mathrsfs}
\usepackage{esint}
\usepackage{graphicx,bm,float}
\usepackage{enumitem}
\usepackage{xcolor}
\usepackage[numbers,sort&compress]{natbib}
\usepackage{xurl}
\usepackage{hyperref}
\usepackage{tikz-cd}
\hypersetup{
    hypertexnames=false,
	colorlinks,
	citecolor=blue,
	filecolor=blue,
	linkcolor=blue,
	urlcolor=black
}

\allowdisplaybreaks[4]
\numberwithin{equation}{section}

\newtheorem{thm}{Theorem}[section]
\newtheorem{prop}{Proposition}[section]
\newtheorem{lem}{Lemma}[section]
\newtheorem{cor}{Corollary}[section]

\theoremstyle{definition}
\newtheorem{defn}[thm]{Definition}

\newtheorem{rem}{Remark}[section]

\newcommand{\la}{\langle}
\newcommand{\ra}{\rangle}
\newcommand{\lam}{\lambda}
\newcommand{\morse}{\mathsf{morse}\,}
\newcommand{\Morse}{\mathsf{Morse}\,}

\newcommand{\sgn}{\mathrm{sgn}\, }

\newcommand{\cri}{\operatorname{Crit}\, }

\newcommand{\FF}{\mathbb{F}}
\newcommand{\rg}{\mathrm{g}}
\newcommand{\cV}{\mathcal{V}}

\newcommand{\cL}{\mathcal{L}}
\newcommand{\cF}{\mathcal{F}}

\newcommand{\N}{\mathbb{N}}
\newcommand{\Z}{\mathbb{Z}}
\newcommand{\R}{\mathbb{R}}
\newcommand{\B}{\mathbb{B}}
\newcommand{\kap}{\kappa}
\newcommand{\ff}{\mathbf{F}_{p,q}^{V}}
\newcommand{\FpqV}{\mathbf{F}_{p,q}^{V}}
\newcommand{\Fv}[2]{\mathbf{F}_{#1,#2}^{V}}
\newcommand{\ii}{\mathrm{i}}
\newcommand{\intS}{\mathring\Si}
\newcommand{\var}{\varepsilon}
\renewcommand{\d}{\mathrm{d}}
\newcommand{\oH}{\mathring H^1}
\newcommand{\Si}{\Sigma}
\newcommand{\al}{\alpha}
\newcommand{\pa}{\partial}

\newcommand{\si}{\sigma}
\newcommand{\om}{\omega}
\newcommand{\Lam}{\Lambda}
\newcommand{\Ga}{\Gamma}
\newcommand{\Sn}{\mathbb{S}^{n}}

\title{\bfseries Morse theory and topological degree for resonant mean field equations on Riemann surfaces with boundary}

\author{
Mohameden Ahmedou\thanks{\raggedright Mathematisches Institut, Justus-Liebig-Universit\"at Gie\ss en,
Arndtstra\ss e 2, D-35392 Gie\ss en, Germany.
E-mail: \href{mailto:Mohameden.Ahmedou@math.uni-giessen.de}{Mohameden.Ahmedou@math.uni-giessen.de}.},\quad ~
Zhengni Hu\thanks{School of Mathematical Sciences, Shanghai Jiao Tong University,
800 Dongchuan Road, Minhang District, Shanghai 200240, China.
E-mail: \href{mailto:zhengni_hu@sjtu.edu.cn}{zhengni\_hu@sjtu.edu.cn}.},\quad~
Heming Wang\thanks{State Key Laboratory of Mathematical Sciences, Academy of Mathematics and Systems Science,
Chinese Academy of Sciences, Beijing 100190, China.
E-mail: \href{mailto:hmwang@amss.ac.cn}{hmwang@amss.ac.cn}.},
\quad~
Lei Zhang\thanks{Department of Mathematics, University of Florida,
Gainesville, FL 32611, USA.
E-mail: \href{mailto:leizhang@ufl.edu}{leizhang@ufl.edu}.}
}

\date{}

\begin{document}

\maketitle

\begin{abstract}
In this paper, we establish Morse theory for the Neumann mean field equation on a compact
Riemann surface at $\rho=4\pi m$, $m\in\mathbb N$. To address the loss of compactness associated with bubbling, we use Bahri's theory of \emph{critical points at infinity}. 
Assume that the reduced functionals are Morse and that the leading scale coefficient does not vanish at their critical points. We classify the critical points at infinity arising from interior, boundary, and mixed concentration, compute their Morse indices and local relative homology, and establish deformation lemmas at resonance.
We  prove that, for $m\ge2$, every sufficiently negative sublevel is homotopy equivalent to the
boundary-weighted barycenter space of order $m-1$.
The Euler–Poincaré identity gives the Leray–Schauder degree at resonance without assuming nondegeneracy of the solutions. If all solutions are nondegenerate, we further obtain strong Morse inequalities with explicit Betti numbers over \(\mathbb Z_2\).
As consequences, we obtain existence and multiplicity results, as well as criteria for solutions with prescribed Morse index. 
\end{abstract}

\medskip
\noindent\textbf{Keywords.}
Mean field equations; Neumann boundary conditions; critical points at infinity;
Morse theory; resonant Leray--Schauder degree.

\noindent\textbf{2020 Mathematics Subject Classification.}
35J60, 35J25, 35R01.

\tableofcontents

\section{Introduction}
	\subsection{Background and motivation}
	Let $(\Si, g)$ be a compact connected Riemann surface with nonempty smooth boundary $\pa\Si$.
	We consider the following mean field type equation with homogeneous Neumann boundary conditions
	\begin{equation*}
		\label{maineq} \tag{$P_{\rho}$}
		\left\{\begin{aligned}
			&-\Delta_g u =\rho \Big( \frac{V e^u}{\int_{\Si } V e^u \,\d v_g} -\frac{1}{|\Si|_g}\Big)  && \text{  in }\, \intS,\\
			&	\pa_{\nu_g} u=0 && \text{ on }\, \pa\Si ,
		\end{aligned}
		\right.
	\end{equation*}
	where   $\intS=\Si \setminus\pa\Si $ denotes the interior of $\Si $, $\Delta_g$ is the Laplace-Beltrami operator, $\nu_g$ is the unit outward normal to the boundary $\pa\Si $, $\d v_g$ is the volume element in $(\Si , g)$,  $|\Si|_g=\int_{\Si}\, \d v_g$, $V\in C^3(\Si)$ is a positive potential function, and $\rho > 0$ is a parameter.

Since \eqref{maineq} is invariant under addition of constants, we consider solutions in  a subspace of $H^1(\Sigma)$ with average $0$:
\[
 \oH=\Big\{u\in H^1(\Sigma):\int_\Sigma u\,\d v_g=0\Big\},
 \quad \langle u,v\rangle=\int_\Sigma\langle\nabla u,\nabla v\rangle_g\,\d v_g.
\]
We use the Dirichlet norm $\|u\|:=\langle u,u\rangle^{1/2}$.
The solutions of \eqref{maineq} are the critical points of
\[
 J_\rho(u)=\frac12\|u\|^2-\rho\ln\int_\Sigma Ve^u\,\d v_g.
\]
Write $J_\rho^a=\{u\in \oH: J_\rho(u)\le a\}$ and
$\operatorname{Crit}(J_\rho)=\{u\in \oH:DJ_\rho(u)=0\}$. 

For $0<\rho<4\pi$, the   Moser--Trudinger inequality yields
coercivity. At $4\pi$ the functional is bounded below but need not attain
its infimum; for $\rho>4\pi$, $J_\rho$ is unbounded below. If a sequence of solutions with parameters $\rho_n\to\rho_*>0$ blows up, then $\rho_*\in4\pi\mathbb N$, with local masses $8\pi$ at interior blow-up points and $4\pi$ at boundary blow-up points. We call a blow-up configuration of type $(p,q)$ if it consists of $p$ interior points and $q$ boundary points. At $\rho_*=4\pi m$, mass quantization  gives $2p+q=m$.

The non-resonant degree theory originates with Li \cite{L1999} and
Chen--Lin \cite{CL2003,CL2015}. For the Neumann problem \eqref{maineq},
Li--Sun--Yang \cite{LSY2023} proved the compactness alternative and existence in non-resonant case;
Ahmedou--Hu--Wang \cite{AHW2025-1} refined the blow-up analysis and derived
the non-resonant degree formula, and Ahmedou--Bartsch--Hu
\cite{ABH20240829} constructed bubbling solutions.
Bahri's theory of critical points at infinity
\cite{Bbook1989,B1996,BC1988,BC1991} was developed for resonant mean field
equations on planar domains by Ahmedou--Ben Ayed--Lucia \cite{ABL2017},
for singular Dirichlet problems by Ahmedou--Ben Ayed \cite{AB2017},
and on closed surfaces by the same authors \cite{AB2023}.

At a resonant parameter, compactness of the solution set alone does not
determine the topology of the sublevels: descending trajectories may
escape to infinite scale with finite limiting energy.  Bahri's theory describes such behavior through critical points at infinity and their associated Morse indices and local relative homology groups. Deformation lemmas then yield a Morse filtration incorporating both genuine critical points and critical points at infinity. The resulting Morse inequalities relate their numbers in each index to the topology of the sublevels, retaining information that is lost in the alternating sum defining the topological degree. They therefore yield existence and multiplicity criteria that depend on the distribution of Morse indices.

The main contribution of this paper is a Morse theory for the
Neumann problem at resonance. We classify all escaping
trajectories of the constructed descending flow with finite
limiting energy, determine their local Morse pairs, and
establish the global deformation lemmas needed to describe
the topology of the sublevels.
For $m\ge2$, we also identify the homotopy type of the sufficiently negative
sublevels with boundary-weighted barycenters; their homology
\cite{AK2026} records the number of boundary components.
At $m=1$, the sufficiently negative sublevels are empty.
The resulting Euler identity determines the resonant degree even when
genuine solutions are degenerate, through a variational perturbation
in a bounded region. Under genuine nondegeneracy, the same filtration
yields strong Morse inequalities with explicit Betti numbers over $\mathbb Z_2$.
Finally, the classification of bubbling branches for non-resonant parameters close to $4\pi m$ gives criteria for solutions satisfying prescribed Morse index bounds.

Compared with the Dirichlet and closed surface problems
\cite{AB2023,ABL2017}, the Neumann problem allows both
interior and boundary concentration. This introduces additional
difficulties in both the local analysis and the global topology.
Locally, interior and boundary concentration have different
numbers of location variables and different leading orders
in the scale expansion: $\lambda^{-1}$ when boundary bubbles
are present and $\lambda^{-2}\ln\lambda$ for purely interior
concentration. The boundary expansion involves an additional coefficient
depending on the potential and the geodesic curvature.
Its nonvanishing determines the leading relation between
the concentration scale and the deviation of $\rho$ from
resonance. These differences require separate local estimates
and a descending field compatible with all concentration
types. A reduction in the amplitude variables, together with
estimates along the flow, yields a complete descending flow
and the corresponding deformation results at resonance. Globally, the different quantized masses of interior and
boundary bubbles lead to boundary-weighted barycenter spaces rather than
ordinary barycenter spaces
as topological models for sufficiently negative sublevels.
We prove the corresponding homotopy equivalence and use the
homology computations of \cite{AK2026} to obtain explicit
Morse inequalities. 

	\subsection{Main results}\label{sec:main-results}

We first define the reduced energy and scale coefficients. Write
$\mathbb N_0=\{0,1,2,\ldots\}$. Proposition~\ref{prop:CPI-main}
classifies the ends; Theorems~\ref{thm:resonant-case}, \ref{thm:1} and
\ref{thm:2} state the degree formula, prescribed-index existence criteria and Morse relations.

	For any $p,q\in \mathbb{N}_0$ with $2p+q\in\N$, we define the thick diagonal set by
	\[
	\FF_{p,q}(\Si):=\{\xi=(\xi_1,\ldots,\xi_{p+q})\in \intS^p \times (\pa\Si)^{q}: \text{$\xi_i=\xi_j$ for some $ i\neq j$}\},
	\]
	and
	\[
		\Xi_{p,q}:=(\intS^p \times (\pa\Si )^{q})\setminus \FF_{p,q}(\Si).
	\]
    Critical points, gradients and Hessians on $\Xi_{p,q}$ are taken on this mixed configuration manifold; boundary coordinates vary tangentially to $\partial\Sigma$.

Define a free action of
$\mathfrak S_{p,q}=\mathfrak S_p\times\mathfrak S_q$ on $\Xi_{p,q}$.
The reduced quantities below are invariant under this action. All bubbling counts and sums use the unordered
orbits $[\xi]=\{\sigma\xi:\sigma\in\mathfrak S_{p,q}\}$ in
$\Xi_{p,q}/\mathfrak S_{p,q}$. A local representative is again denoted by
$\xi$; its reduced index, the sign of $\cL$, and the index at infinity
$\iota_\infty(\xi)=\iota_\infty([\xi])$ defined in
Proposition~\ref{prop:CPI-main} are independent of the representative.

	In the sequel, we use $\sum_{i}$ to denote $\sum_{i=1}^{p+q}$, and $\sum_{i\neq j}$ to denote $\sum_{i=1}^{p+q}\sum_{\substack{j=1\\j\neq i}}^{p+q}$, for simplicity. A sum $\sum_{2p+q=m}$ always runs over $(p,q)\in\mathbb N_0^2$.

	For any  $\xi:=(\xi_1,\ldots,\xi_{p+q})\in \Xi_{p,q}$,   we define  the function
	\begin{equation}\label{eq:def_F}
		\cF_{\xi}^i(x):= V(x) \exp\Big\{\kap( \xi_i)H^g(x,\xi_i)+ \sum_{j\neq i} \kap( \xi_j) G^g(x,\xi_j)\Big\},\quad x\in\Si \setminus\{\xi_j:j\ne i\},
	\end{equation}
	for each $i=1,\ldots,p+q$, where
\begin{equation}\label{eq:kappa-definition}
 \kap(\xi)=\left\{\begin{aligned}&8\pi&&\text{ if }\,\xi\in\intS,\\&4\pi&&\text{ if }\,\xi\in\pa\Si.\end{aligned}\right.
\end{equation} Here,  $G^g(\cdot,\cdot)$ is the Green's function associated with problem \eqref{maineq}, and $H^g(\cdot,\cdot)$ is its regular part, as introduced in Section \ref{sec:2}.  In particular, $\cF_\xi^i$ is of class $C^3$ in a small neighborhood of $\xi_i$. Furthermore, we define a Kirchhoff-Routh-type functional   $\ff:\Xi_{p,q}\to \R $  by
	\begin{equation}\label{eq:reduce_fun_0}
		\ff(\xi):=\sum_{i} 2\kap( \xi_i)\ln  V(\xi_i) +\sum_{i} \kap^2( \xi_i) R^g(\xi_i) + \sum_{i}\sum_{j\neq i} \kap( \xi_i)\kap( \xi_j) G^g(\xi_i,\xi_j),
	\end{equation}
	where  $R^g(\xi_i):=H^g(\xi_i,\xi_i)$ is the Robin function; see also Section \ref{sec:2}.

	Let    $K_g$ be the Gaussian curvature of the surface $\Si $ with respect to $g$ and $k_g$ be the geodesic curvature of $\pa \Si$.  Define
	\begin{equation}\label{eq:def-cL}
	\cL(\xi):=\left\{\begin{aligned}
		&		\cL_1(\xi) & &\text{ if }\, q\neq 0,\\
		&\cL_2(\xi) & & \text{ if }\, q=0,
	\end{aligned}\right.
	\end{equation}
	where
	\begin{equation*}
		\cL_1(\xi)
		:=-\sum_{i=p+1}^{p+q}
		(\pa_{\nu_g}\ln  V+2k_g)	\sqrt{\cF_{\xi}^i}\Big|_{\xi_i},
	\end{equation*}
	and
	\begin{equation*}
		\cL_2(\xi):= \sum_i
		\kap(\xi_i)(\Delta_g\cF^i_{\xi}
		- 2K_g \cF^i_{\xi})\Big|_{\xi_i}.
	\end{equation*}

Write $\operatorname{Crit}(\ff)=\{\xi\in\Xi_{p,q}:D\ff(\xi)=0\}$ and define
\begin{equation*}
 \cV_{p,q}^-=\{[\xi]\in\Xi_{p,q}/\mathfrak S_{p,q}:
                D\ff(\xi)=0,\ \cL(\xi)<0\}.
\end{equation*}
Proposition~\ref{prop:pseudogradient} identifies $\mathcal V_{p,q}^-$ with the
escaping ends of its descending field.

For a nondegenerate critical point $\xi$ of $\ff$,
$\morse(\ff,\xi)$ denotes the number of negative eigenvalues
of its Hessian on $\Xi_{p,q}$. We impose the following conditions
on $\ff$ and $\cL$:

	\begin{enumerate}[
		label=\textup{(C\arabic*)},
		ref=\textup{(C\arabic*)}
		]
		\item\label{item:C1} \(\ff\) is a Morse function;
		\item\label{item:C2} \(\cL(\xi)\neq0\) for every
		\(\xi\in\cri(\ff)\).
	\end{enumerate}

	\begin{rem}
		\begin{enumerate}
			\item  By \cite[Lemma 4.1]{HB2024},
			\begin{equation}
				\label{eq:grad-F-inf} \lim_{\xi\to \partial \Xi_{p,q}} |\nabla \ff|_g =+\infty.
			\end{equation}
			The condition	\ref{item:C1} together with \eqref{eq:grad-F-inf}  implies that the reduced critical set $\operatorname{Crit}(\ff)$ is finite.
\item Positive $V\in C^3(\Sigma)$ suffices for the local second-order
terms and the differentiated estimates used below. Conditions
\ref{item:C1}--\ref{item:C2} are imposed directly; no genericity assumption
or higher regularity is used. Condition~\ref{item:C2} suffices for
compactness at the fixed resonant parameter, as proved in
Remark~\ref{rem:fixed-resonance-compactness}; condition~\ref{item:C1}
is used to obtain isolated nondegenerate reduced configurations
and their Morse indices.
The $C^3$ assumption is sufficient for the differentiated reduction;
minimal regularity is not claimed.
For related Morse genericity results, see, e.g., 
\cite{BMP2019,HB2024,ABF2023,LLWY2018}.
		\end{enumerate}

	\end{rem}

The first result combines the classification of escaping trajectories
with the local and global deformations needed for Morse theory.
A critical point at infinity is an equivalence class of non-compact
descending orbits of the field constructed below with energy bounded below,
as in
Definition~\ref{def:CPI}.
\begin{prop}\label{prop:CPI-main}
Let $\rho_*=4\pi m$, $m\in\mathbb N$, and let $V\in C^3(\Sigma)$ be positive.
Assume \ref{item:C1}--\ref{item:C2} for all $2p+q=m$.

Then the following statements hold.
\begin{enumerate}
\item  There is a bounded locally Lipschitz descending field for $J_{\rho_*}$
with complete positive flow. 
Every forward trajectory along which the energy is bounded below
is either precompact, with its limit set contained in
$\operatorname{Crit}(J_{\rho_*})$, or converges in bubble parameters
to a critical point at infinity.
\item The critical points at infinity are in bijection with
$\bigcup_{2p+q=m}\mathcal V_{p,q}^-$.
The critical point at infinity corresponding to $[\xi]$ has
level and Morse index as follows:
\begin{align*}
C_\infty(\xi)
&=-\rho_*\bigl(1+\ln(\rho_*/8)\bigr)
  -\frac12\FpqV(\xi),\qquad 
\iota_\infty(\xi)
=3p+2q-1-\morse(\FpqV,\xi).
\end{align*}
\item  If $a<b$ and $[a,b]$ contains neither an  genuine  critical value
nor a critical value at infinity, then $J_{\rho_*}^a$ is a
strong deformation retract of $J_{\rho_*}^b$.
\item  Let $c$ be an isolated critical level, and assume that the
genuine  critical points at level $c\in \R$ are nondegenerate. For every sufficiently small $\varepsilon>0$,
the relative homology of
$(J_{\rho_*}^{c+\varepsilon},J_{\rho_*}^{c-\varepsilon})$
is the direct sum of the local Morse groups of the genuine
critical points and critical points at infinity at level $c$. 
\end{enumerate}
\end{prop}

	Although compactness may fail as $\rho\to 4\pi m$, condition
	\ref{item:C2} yields an a priori estimate at the fixed parameter
	$\rho=4\pi m$ by Remark~\ref{rem:fixed-resonance-compactness}:
	there exists a constant
	$\bar C_m>0$ such that
	\[
	\|u\|<\bar C_m
	\quad
	\text{for every }u\in\operatorname{Crit}(J_{4\pi m}).
	\]
	For $R>0$ and $u_0\in \oH$, set
	\[
	\mathcal B_R^{\oH}(u_0)
	:=
	\{u\in\oH:\|u-u_0\|< R\}, \quad  	\mathcal B_R^{\oH}:= 	\mathcal B_R^{\oH}(0),
	\]
	and let $(-\Delta_g)^{-1}$ denote the inverse Neumann Laplacian on mean-zero functions. Define
	\[
	T_\rho(u)
	:=
	(-\Delta_g)^{-1}
	\Big[
	\rho\Big(
	\frac{Ve^u}{\int_\Sigma Ve^u\,\d v_g}-\frac 1 {|\Sigma|_g}
	\Big)
	\Big],
	\quad u\in\oH.
	\]
	All functional gradients are taken with respect to the Dirichlet
	inner product on $\oH$; in particular,
	$\nabla J_\rho(u)=u-T_\rho(u)$. 
	Since $I-T_{4\pi m}$ has no zeros in
	$\oH \setminus \mathcal B_{\bar C_m}^{\oH}$, the Leray--Schauder degree
	\[
	d_{4\pi m}
	:=
	\deg
	\bigl(
	I-T_{4\pi m},
	\mathcal B_{\bar C_m}^{\oH},
	0
	\bigr)
	\]
	is well defined.

	The following theorem gives our main degree formula at resonance.
The corresponding degree formula in the nonresonant case was
obtained by computing the degree jumps across resonant parameters
in \cite[Theorem~1.1 and Eq.~(4.5)]{AHW2025-1}.
	\begin{thm}\label{thm:resonant-case}
		Let $\rho = 4 \pi m$ with $m\in\N$. Assume that $V$ is a positive $C^{3}$ function on $\Si$ satisfying conditions \ref{item:C1}--\ref{item:C2} for every pair $(p,q)\in\mathbb N_0^2$ with $2p+q=m$. Then the Leray--Schauder degree $d_{4\pi m}$ is well-defined and is given by
\begin{equation}\label{eq:degree-reso}
			d_{4 \pi m} =
			\left\{\begin{aligned}
				&1  + \, \sum_{2p + q = m} \sum_{[\xi^*] \in \cV_{p,q}^- }  (-1)^{p - \morse(\ff,\xi^*)} && \text{ if }\,    m= 1,2, \\
				&\frac{1}{l !}  \prod_{i=1}^{l} (i- \chi(\Si)) \, + \, \sum_{2p + q = m}\sum_{[\xi^*] \in \cV_{p,q}^- }  (-1)^{p - \morse(\ff,\xi^*)}   && \text{ if }\, m \geq 3,
			\end{aligned}\right.
		\end{equation}
		where $l:= \lfloor \frac{m-1}2\rfloor$ denotes the integer part of
		$\frac {m-1} 2$ and  $\chi(\Sigma)$ denotes the Euler characteristic of $\Sigma$.
	\end{thm}
The nonvanishing of this degree yields an existence result  without
any nondegeneracy assumption on the solutions.
\begin{cor}\label{c:Euler-Poincare}
Under the hypotheses of Theorem~\ref{thm:resonant-case}, equation
\eqref{maineq} at $\rho=4\pi m$ admits at least one solution if
\begin{equation}\label{eq:Euler-Poincare}
			\left\{\begin{aligned}
				&\sum_{2p + q = m}   \sum_{[\xi^*] \in \cV_{p,q}^- }  (-1)^{p - \morse(\ff,\xi^*)} \neq  -1 && \text{ if }\,  m = 1,2,  \\
				&	\sum_{2p + q = m}  \sum_{[\xi^*] \in \cV_{p,q}^- }  (-1)^{p - 1 - \morse(\ff,\xi^*)}  \neq  \frac{1}{l !}  \prod_{i=1}^{l} (i- \chi(\Si))   && \text{ if }\,  m \geq 3.
			\end{aligned}\right.
		\end{equation}
\end{cor}
For a compact space $X$ and $l\ge1$, let $B_l(X)$ denote
the space of probability measures supported on at most $l$
points of $X$. Define the boundary-weighted barycenter space by
\begin{equation}\label{eq:weighted-barycenter-space}
B_l^\partial(\Sigma,\partial\Sigma)
      =\bigcup_{2p+q\le l}B_q^p,\quad l\ge1,
    \end{equation}
    where \begin{equation}\label{eq:barycenter-stratum}
    B_q^p=\Big\{\sum_{i}t_i\delta_{x_i}:
      t_i\ge0,\ \sum_i t_i=1,\
      x_1,\ldots,x_p\in\mathring\Sigma,\
      x_{p+1},\ldots,x_{p+q}\in\partial\Sigma\Big\}.
    \end{equation}

    Both barycenter spaces carry the weak topology of measures; zero
    coefficients are omitted, and coincident atoms are combined.
    
Throughout this paper, $b\ge1$ denotes
    the number of connected components of $\partial\Sigma$.

	We introduce the concept of \emph{generalized Morse index}   of a solution $\om$ to equation \eqref{maineq},  denoted by $\Morse(J_{\rho}, \om)$, defined as the sum of  the classical Morse index $\morse(J_{\rho}, \om)$ and the dimension of the kernel of the associated linearized operator
	\[
	\mathcal{T}_\om (\phi):=-\Delta_g \phi-\rho \Big(
	\frac{V e^\om  \phi}{I_\om}
	-
	\frac{V e^\om}{I_\om^2}
	\int_{\Si}V e^\om \phi\,\d v_g
	\Big),
	\quad
	I_\om:=\int_{\Si}V e^\om\,\d v_g.
	\]
	Here \(\mathcal T_\omega\) is the self-adjoint operator in
	\(L^2_0(\Sigma):=\{\phi\in L^2(\Sigma):\int_\Sigma\phi\,\d v_g=0\}\) with
	domain
	\[
	D(\mathcal T_\omega)=\{\phi\in H^2(\Sigma)\cap L^2_0(\Sigma):
	\partial_{\nu_g}\phi=0\text{ on }\partial\Sigma\}.
	\]
	Its quadratic form is the Hessian of $J_\rho$ on $\oH$.
The Morse index $\morse(J_\rho,\om)$ is the number of
negative eigenvalues of $\mathcal T_\om$, counted with
multiplicity. We define the \emph{generalized Morse index} by
    \[
    \Morse(J_\rho,\omega):=\morse(J_\rho,\omega)+\dim\ker\mathcal T_\omega.
    \]

	\begin{thm}\label{thm:1}
		Let \(\rho_*=4\pi m\) with \(m\in\mathbb N\), and let \(V\in C^3(\Sigma)\) be positive. Assume that, for every pair
		\((p,q)\in\mathbb N_0^2\) with \(2p+q=m\), the reduced energy
		\(\FpqV\) has only nondegenerate critical points and
		\(\mathcal L(\xi)\ne0\) at each of them. Then the following statements hold for equation \eqref{maineq} with $\rho=\rho_*$.
		\begin{itemize}
			\item[(1)] If \(\mathcal L(\xi)<0\) at every local minimum
			\(\xi\in\Xi_{0,m}\) of \(\Fv{0}{m}\), then equation \eqref{maineq}
			admits a solution \(\omega\) such that
			\begin{equation}\label{eq:index-window-upper}
				\morse(J_{\rho_*},\omega)\le 2m
				\le \Morse(J_{\rho_*},\omega).
			\end{equation}
			\item[(2)] Let \(m\ge2\). Assume that \(\mathcal L(\xi)>0\) at every
			critical point \(\xi\in\Xi_{0,m}\) of \(\Fv{0}{m}\) with
			\(\morse(\Fv{0}{m},\xi)=1\), and also at every local minimum
			\(\xi\in\Xi_{1,m-2}\) of \(\Fv{1}{m-2}\). Then equation
			\eqref{maineq} admits a solution \(\omega\) such that
			\begin{equation}\label{eq:index-window-lower}
				\morse(J_{\rho_*},\omega)\le 2m-2
				\le \Morse(J_{\rho_*},\omega).
			\end{equation}
		\end{itemize}
If the solution obtained in \textup{(1)} or \textup{(2)}
is nondegenerate, its Morse index is $2m$ or $2m-2$,
respectively.
	\end{thm}

\begin{rem}
For the min--max classes considered here, a generic variational
perturbation yields a nondegenerate critical point of Morse
index $k$. Passing to the limit as the perturbation tends to
zero, spectral semicontinuity gives  $\morse(J_{\rho_*},\omega)\le k\le\Morse(J_{\rho_*},\omega)$. 
Thus the limiting solution has Morse index $k$ if it is
nondegenerate. The values $k=2m$ and $k=2m-2$ arise from
the top-dimensional boundary strata of $B_m^\partial$ and
$B_{m-1}^\partial$, whose dimensions are $2m-1$ and $2m-3$,
respectively, through the corresponding relative min--max
construction.
\end{rem}

Since $\dim\Xi_{p,q}=2p+q$, the index in Proposition~\ref{prop:CPI-main}
satisfies $\iota_\infty([\xi])\ge p+q-1$. The quantity $3p+2q-1$
is the dimension of the barycentric stratum: $2p+q$ location variables
and $p+q-1$ weights.

    We next describe the stronger relations between the numbers of
solutions of each index and the topology of the lower sublevel.
For the remaining statements, let $m\in\mathbb N$ with $m\ge2$. All homology groups
    in these statements have coefficients in $\mathbb Z_2$. For $l\ge1$ and $r\ge0$, write
    \[
    \widetilde\beta_l^r
    :=\dim_{\mathbb Z_2}\widetilde H_r
       (B_l^\partial(\Sigma,\partial\Sigma);\mathbb Z_2),
    \quad
    \widetilde P_l^\partial(t):=\sum_{r\ge0}\widetilde\beta_l^r t^r.
    \]
    Define
    \begin{equation}\label{def:Betti}
    \mathfrak b_{m,0}:=0,\quad
    \mathfrak b_{m,k}:=\widetilde\beta_{m-1}^{k-1}\quad(k\ge1),
    \end{equation}
    and
    \begin{equation}\label{def:P_m}
    P_m(t):=\sum_{k\ge0}\mathfrak b_{m,k}t^k
           =t\widetilde P_{m-1}^\partial(t).
    \end{equation}
    Section~\ref{sec:barycenter-homology} recalls the explicit formulas
    \eqref{eq:topology-P-even}--\eqref{eq:topology-P-odd} and identifies
    $P_m$ as the relative Poincar\'e polynomial of a sufficiently large
    sublevel pair of $J_{4\pi m}$. 

	The following theorem states the strong Morse relation and its consequences.
\begin{thm}\label{thm:2}
		Let \(\rho=4\pi m\) with \(m\geq 2\), and let \(V\in C^3(\Sigma)\) be positive. Assume that \(V\) satisfies
		conditions \ref{item:C1}--\ref{item:C2}  for every pair
		\((p,q)\in\mathbb N_0^2\), \(2p+q=m\).
		Assume furthermore that all critical points of \(J_{4\pi m}\) are non-degenerate, and let
		\[
		\mathfrak M:=\max\Bigl(\{0\}\cup
		\{\morse(J_{4\pi m},u):u\in\operatorname{Crit}(J_{4\pi m})\}\Bigr).
		\]
		For \(k\ge0\), set
		\[
		\nu_k:=\#\{u\in\operatorname{Crit}(J_{4\pi m}):\morse(J_{4\pi m},u)=k\},
		\]
		and
		\[
		\nu_k^{\infty}
		:=
		\sum_{2p+q=m}
		\#\{
		[\xi]\in\cV_{p,q}^{-}:
		\iota_{\infty}([\xi])=k
		\}.
		\]

		Then the following statements hold.
		\begin{itemize}
			\item[(a)] The strong Morse relation is
			\begin{equation*}
				\sum_{k\ge0}(\nu_k+\nu_k^{\infty})t^k
				-
				P_m(t)
				=
				(1+t)Q(t),
			\end{equation*}
			where \(Q(t)\) is a polynomial with nonnegative integer coefficients.
			\item[(b)] Consequently, for every
			\[
			1\leq k\leq \bar N:=\max(2m-1,\mathfrak M),
			\]
			one has the explicit weak Morse inequality
			\begin{equation*}
				\nu_k+\nu_k^{\infty}\geq \mathfrak b_{m,k},
			\end{equation*}
			where \(\mathfrak b_{m,k}\) is defined by \eqref{def:Betti}.
			\item[(c)] The Euler--Poincar\'e identity is
			\[
			\sum_{i=0}^{\mathfrak M}(-1)^i\nu_i
			+
			\sum_{2p+q=m}
			\sum_{[\xi]\in\cV_{p,q}^{-}}
			(-1)^{\iota_{\infty}([\xi])}
			=
			1-\chi(B_{m-1}^{\pa}(\Si,\pa\Si)).
			\]
			In particular,
			\[
			\#\operatorname{Crit}(J_{4\pi m})
			\geq
			\Big|
			1-\chi(B_{m-1}^{\pa}(\Si,\pa\Si))
			-
			\sum_{2p+q=m}
			\sum_{[\xi]\in\cV_{p,q}^{-}}
			(-1)^{\iota_{\infty}([\xi])}
			\Big|.
			\]
		\end{itemize}

	\end{thm}
	\begin{rem}
		By the compactness theorem in the resonant case, together with the nondegeneracy assumption, \(\operatorname{Crit}(J_{4\pi m})\) is finite; hence \(\mathfrak M\) and the numbers \(\nu_k\) are well-defined.
	\end{rem}

\begin{rem}
By \eqref{eq:topology-euler-characteristic},
\[
1-\chi\bigl(B_{m-1}^{\partial}(\Sigma,\partial\Sigma)\bigr)
=\frac{1}{l!}\prod_{j=1}^{l}\bigl(j-\chi(\Sigma)\bigr),
\qquad l=\left\lfloor\frac{m-1}{2}\right\rfloor,
\]
where the empty product is understood to be $1$.
Thus this Euler characteristic depends only on $m$ and
$\chi(\Sigma)$.
\end{rem}

	As an immediate consequence, we obtain the following existence result.
	\begin{cor}\label{cor:1.3}
		Let $ m\geq 2$, and assume that the hypotheses in  Theorem \ref{thm:2} hold.
		Suppose that there exists an integer $ l$  with
		$1\leq l\leq 2m-2$
		such that
		\[
		\nu_l^{\infty}=0
		\quad\text{ and } \quad
		\mathfrak{b}_{m,l}  \neq0.
		\]
		Then  \eqref{maineq} with $ \rho=4\pi m$  admits at least one solution of Morse index  $ l$.  More precisely,
		\[
		\nu_l\ge \mathfrak{b}_{m,l}.
		\]
		In particular, \[
		\mathfrak b_{m,2m-2}
		=
		\widetilde\beta_{m-1}^{2m-3}
		=
		\binom{ b +m-2}{m-1}>0,
		\]
		where \( b \ge1\) is the number of boundary components. If
		there is no critical point at infinity of index \(2m-2\), namely
		$\nu_{2m-2}^{\infty}=0$,
		then equation \eqref{maineq} admits at least one solution whose Morse index is
		\(2m-2\).  Moreover, the number of such solutions is bounded from below by
		\(\binom{ b +m-2}{m-1}\).
	\end{cor}

	\begin{rem}	Since $ \mathfrak{b}_{m,l} =\widetilde{\beta}^{l-1}_{m-1}$  is the coefficient of $ t^{l-1}$  in the explicit polynomial $ \widetilde P_{m-1}^{\pa}(t)$  defined by \eqref{eq:topology-P-even}--\eqref{eq:topology-P-odd}, Corollary \ref{cor:1.3} yields an explicit lower bound for $\nu_l$.
	\end{rem}

	An isolated critical point at infinity also yields a solution when both adjacent infinity indices are absent and the corresponding relative Betti number vanishes.

	\begin{thm}\label{thm:3}
		Let $ \rho=4\pi m$  with $ m\geq2$, and assume the reduced hypotheses of Theorem \ref{thm:resonant-case}. Suppose that there exists $ l_0\in\mathbb N_0$  such that
		\begin{itemize}
			\item[(1)]  for some pair
			$ (p,q)$  with $ 2p+q=m$, there exists a critical point at infinity $ [\xi_0]\in\cV_{p,q}^{-}$  satisfying
			\[
			\iota_{\infty}([\xi_0])=l_0;
			\]
			\item[(2)] for every pair $ (p,q)\in\mathbb N_0^2$  with $ 2p+q=m$, and  every
			$ [\xi]\in\cV_{p,q}^{-}$, one has $\iota_{\infty}([\xi])\notin\{l_0-1,l_0+1\};$
			when $ l_0=0$, the condition at index $ l_0-1$  is void;
			\item[(3)]
			$\mathfrak b_{m,l_0}=0. $
		\end{itemize}
		Then  \eqref{maineq} with $ \rho=4\pi m$  admits at least one solution.
	\end{thm}

	Since the indices of critical points at infinity of type $ (p,q)$  are bounded from below by $ p+q-1$, Theorem \ref{thm:3} yields the following consequence.

	\begin{cor}\label{cor:minimal-infinity-index}
		Let $ \rho=4\pi m$  with $ m\geq2$, and assume the reduced hypotheses of Theorem \ref{thm:resonant-case}.
		Suppose that there exist $ p,q\in\mathbb N_0$  with $ 2p+q=m$  and an orbit $[\xi_0]\in\cV_{p,q}^{-}$
		such that
		\[
		\iota_{\infty}([\xi_0])=p+q-1.
		\]
		Assume moreover that $\mathfrak b_{m,p+q-1}=0$
		and
		\[
		\sum_{2p'+q'=m}
		\#\{
		[\xi]\in\cV_{p',q'}^{-}:
		\iota_{\infty}([\xi])\in\{p+q-2,p+q\}
		\}
		=0.
		\]
		Then  \eqref{maineq} with $ \rho=4\pi m$  admits at least one solution.
	\end{cor}

	More generally, any violation of the Morse inequalities at infinity in Theorem \ref{thm:2} implies the existence of a solution to \eqref{maineq}. In particular, we have the following consequence.
	\begin{cor}\label{cor:1.5}
		Let \(\rho=4\pi m\) with \(m\geq2\), and assume the reduced hypotheses of Theorem \ref{thm:resonant-case}. If one of the following inequalities is not satisfied:
		\begin{itemize}
			\item[(i)]
			$\nu_{2m-2}^{\infty}-\nu_{2m-1}^{\infty}
			\geq
			\mathfrak b_{m,2m-2}$;
			\item[(ii)]
			$\nu_{2m-3}^{\infty}
			-\mathfrak b_{m,2m-3}
			+\mathfrak b_{m,2m-2}
			\geq
			\nu_{2m-2}^{\infty}-\nu_{2m-1}^{\infty}$;
		\end{itemize}
		then equation \eqref{maineq} with \(\rho=4\pi m\) admits at least one solution.
	\end{cor}

\subsection{Organization and notation}\label{sec:proof-strategy}
The present paper is organized as follows. Section~\ref{sec:2} develops modulation and fiber reduction. In
Section~\ref{sec:3}, we study the branches of bubbling solutions
for parameters near the resonant value set $4\pi\N$.
Amplitude preparation in Section~\ref{sec:resonant-descent} connects the
natural tubes to the terminal coordinates. Section~\ref{sec:full-pseudogradient}
constructs the complete flow, local pairs and relative deformation, and
Section~\ref{sec:4} proves the main theorems. Appendix~\ref{app:0} establishes the asymptotic estimates,
while Appendix~\ref{app:topology} proves the homotopy equivalence
between sufficiently negative sublevels and the boundary-weighted
barycenter space.

Throughout, $C>0$ denotes a constant that may vary from line to line.
Its dependence on parameters is indicated when relevant.
We write $\mathbb N=\{1,2,\ldots\}$,
$\mathbb N_0=\mathbb N\cup\{0\}$, $\mathbb R_+=(0,\infty)$,
and
\[
\mathbb R_+^2:=\{(y_1,y_2)\in\mathbb R^2:y_2\ge0\}.
\]
We set $\ii(\xi)=2$ for $\xi\in\mathring\Sigma$ and
$\ii(\xi)=1$ for $\xi\in\partial\Sigma$.
The symbols $\#$, $\lfloor\cdot\rfloor$, $\sqcup$, and $\vee$
denote cardinality, the floor function, disjoint union, and
the wedge sum of pointed spaces, respectively.
We write $\bigvee_N X$ for the wedge sum of $N$ copies of $X$.
For integers $n\ge0$ and $k$, we use the convention
\[
\binom nk=
\left\{\begin{aligned}
  &\dfrac{n!}{k!(n-k)!},&&0\le k\le n,\\
&0,&&\text{otherwise}.  
\end{aligned}\right.
\]
The convex hull of $g_1,\ldots,g_N$ is denoted by
\[
\operatorname{co}\{g_1,\ldots,g_N\}
:=
\Big\{
\sum_{j=1}^N\theta_jg_j:
\theta_j\ge0,\ \sum_{j=1}^N\theta_j=1
\Big\}.
\]

	\section{Preliminaries}\label{sec:2}

For later estimates, the sharp Moser--Trudinger inequality and Jensen's
inequality imply, for $u\in\oH$,
\begin{equation}\label{eq:MT-Jensen-explicit}
 \int_\Si e^{s u}\,\d v_g
 \le C_{\Si,g}\exp\Big(\frac{s^2}{8\pi}\|u\|^2\Big),\quad s\in\R,
 \quad
 \int_\Si Ve^u\,\d v_g\ge(\min_\Si V)|\Si|_g.
\end{equation}
Here $C_{\Si,g}$ is independent of $s$ and $u$; see, e.g.,
\cite{LSY2023,YZ2021} and references therein.

	\subsection{Neighborhood at infinity and lack of compactness}\label{subsec:bubble-neighborhoods}
	Every Riemann surface $(\Si,g)$ is locally conformally flat, and the
local coordinates in which $g$ is conformal to the Euclidean metric are
referred to as isothermal coordinates; see, e.g.,
\cite{C1955,HW1955,V1955,L1957} and the references therein. In this paper,
we choose isothermal coordinates as in
\cite[Section~2]{EF2014} and \cite[Lemma~4]{YZ2021}, which map open neighborhoods in $\Si$ onto open
disks or half-disks in $\R^2$.

	For \(y^*=(y_1^*,y_2^*)\in\R^2\) and \(r>0\), we set
	\begin{gather*}
		\B_r(y^*):=\{y=(y_1, y_2) \in \R^2:(y_1-y_1^*)^2+(y_2-y_2^*)^2<r^2 \}, \\ \B_r:=\B_r(0),\quad  \text { and } \quad
		\R_+^2:=\{y=(y_1,y_2)\in\R^2:y_2\ge0\},\quad
\B_r^{+}:=\B_r\cap\R_+^2.
	\end{gather*}
	We denote by $d_g$ the Riemannian distance on $\Si$ and set
	$B_r(\xi):=\{x\in\Si:d_g(x,\xi)<r\}$; the Euclidean balls are denoted by $\B_r(y^*)$.
	Choose the local charts with
	\[
	y_\xi(\xi)=0,\quad
	g=e^{\hat\varphi_\xi}\langle\cdot,\cdot\rangle_{\R^2},
	\quad
	B^\xi:=y_\xi(U(\xi))
	=
	\left\{\begin{aligned}
	&\B_{2r_\xi}&&\text{ if }\,\xi\in\intS,\\
	&\B_{2r_\xi}^{+}&&\text{ if }\, \xi\in\pa\Si,
		\end{aligned}
	\right.
	\]
	with $\overline{U(\xi)}\Subset\intS$ in the interior case and
	\[
	y_\xi(U(\xi)\cap\pa\Si)=\B_{2r_\xi}^{+}\cap\pa\R_+^2
	\]
	in the boundary case.  In both cases, the conformal factor
	$\hat{\varphi}_{\xi}:B^\xi\to\R$ satisfies
	\[	-\Delta \hat{\varphi}_{\xi}(y)=2 K_g(y_{\xi}^{-1}(y)) e^{\hat{\varphi}_{\xi}(y)} \quad \text { for all }\, y \in B^{\xi},\]
	where $K_g$ is the Gaussian curvature  of $\Si$.
	When $\xi\in \pa \Si$, the conformal factor also satisfies the boundary condition
	\[
	\frac{\pa }{\pa y_2}\hat{\varphi}_{\xi}(y)=-2k_g(y_\xi^{-1}(y))e^{\frac{\hat{\varphi}_{\xi}(y)}{2}} \quad \text { for all }\, y \in \pa B^{\xi}\cap \pa \R_+^{2},
	\]
	where	 $k_g$ denotes the geodesic curvature of the boundary $\pa \Si$ with  respect to  $g$.
	Both $y_{\xi}$ and $\hat{\varphi}_{\xi}$ are chosen locally in the center to depend smoothly on $\xi$  as in \cite{EF2014,F2024}. In addition, $\hat{\varphi}_{\xi}$ satisfies
	\[
	\hat{\varphi}_{\xi}(0)=0 \quad \text{ and } \quad \nabla\hat{\varphi}_{\xi}(0)=\left\{\begin{aligned}
		&(0,0) && \text{ if }\, \xi\in \intS,\\
		&(0,-2k_g(\xi)) && \text{ if }\, \xi\in \pa  \Si.
	\end{aligned}
	\right.
	\]

	Recall that, under the conformal change of metric $\tilde{g}=e^{\varphi} g$,   the Laplace-Beltrami operator transforms according to
	\begin{equation}\label{conformalchange}
		\Delta_{\tilde{g}}=e^{-\varphi} \Delta_g,
	\end{equation}
	and   for any $x \in$ $U(\xi)\cap \pa \Si$, we have
	\[
	(y_{\xi})_*(\nu_g(x))= -\exp\Big( -\frac{\hat{\varphi}_{\xi}(y)}2\Big) \frac {\pa} { \pa y_2 }\Big|_{	y=y_{\xi}(x)},
	\]where $\nu_g(x)$ is the outward unit normal vector at $x$ with respect to  $g$.

	For \(\xi\in\Si\) and \(0<r\leq2r_\xi\), define
	\[	B_r^{\xi}:=B^{\xi} \cap\{y \in \R^2:|y|<r\}, \quad   U_r(\xi):=y_{\xi}^{-1}(B_r^{\xi}). \]

	We next introduce the Neumann Green function associated with
	\eqref{maineq}.  For every \(\xi\in\Si\), let \(G^g(\cdot,\xi)\) be the weak solution of
	\begin{equation}\label{Green}
		\left\{\begin{aligned}
			&-\Delta_g G^g(x,\xi)  =\delta_{\xi} -\frac{1}{|\Si |_g}, &&  x\in\intS ,\\
			&\pa_{ \nu_g } G^g(x,\xi) =0, && x\in \pa \Si, \\
			&\int_{\Si } G^g(x,\xi) \,\d v_g(x) =0.&&
		\end{aligned}
		\right.
	\end{equation}
	Here \(\delta_\xi\) is the Dirac mass on the compact surface with
boundary. In particular, at a boundary pole the meaning of
\eqref{Green} is the weak identity
\begin{equation}\label{eq:Green-weak-boundary}
\int_\Si \langle\nabla G^g(\cdot,\xi),\nabla\phi\rangle_g\,\d v_g
=\phi(\xi)-\frac1{|\Si |_g}\int_\Si \phi\,\d v_g,
\quad \phi\in C^\infty(\Si ).
\end{equation}
The Green function belongs to \(W^{1,s}(\Si )\) for \(s\in (1,2)\) and
has mean zero. The homogeneous Neumann condition holds classically away
from the pole and is understood through \eqref{eq:Green-weak-boundary}
at a boundary pole. The normalization \eqref{eq:Green-weak-boundary}
applies to both interior and boundary poles.  We refer to \cite[Lemma 6]{YZ2021} for the basic properties of the Neumann Green function.

	Let $\vartheta$ be a radial cut-off function in  $C^{\infty}(\R,[0,1])$ such that
	\begin{equation}\label{eq:cut_off}
		\vartheta(s)=\left\{\begin{aligned}
			&1 &&\text{ if }\,|s|\leq 1,\\
			&0 && \text{ if }\, |s|\geq 2 .
		\end{aligned}
		\right.
	\end{equation}
	For each fixed $\xi\in\Si$, define
	\[
	\vartheta_\xi(x):=
	\left\{\begin{aligned}
		&\vartheta \Big(\dfrac{4|y_\xi(x)|}{r_\xi}\Big)
		&& \text{ for }\, x\in U(\xi),\\
		&0
		&& \text{ for }\, x\in\Si \setminus U(\xi).
	\end{aligned}\right.
	\]
Products containing $\vartheta_\xi$ or one of its derivatives
are extended by zero outside $U(\xi)$. Coordinate expressions in
such products are evaluated only in $U(\xi)$.
	We decompose the Green function as
	\[
	G^g(x, \xi)=\frac{4\vartheta_{\xi}(x)}{\kap ( \xi)} \ln \frac{1}{|y_{\xi}(x)|} +H^g(x, \xi),\quad x\in\Si,
	\]
	where \(\kap(\xi)\) is given by \eqref{eq:kappa-definition}.
	The function \(H^g\) is the regular part of the Green function.  It follows
	from \eqref{Green} that \(H^g(\cdot,\xi)\) satisfies
	\begin{equation}\label{Hfunction}
		\left\{\begin{aligned}
			&-\Delta_g H^g(x, \xi)  =- \frac{4}{\kap( \xi)}(\Delta_g \vartheta_{\xi}) \ln  |y_{\xi}|-\frac{8}{\kap( \xi)}\langle\nabla \vartheta_{\xi}, \nabla \ln  |y_{\xi}|\rangle_g-\frac{1}{|\Si |_g} &&\text { in }\, \intS,\\
			&\pa_{\nu_g} H^g(x, \xi)=\frac{4}{\kap( \xi)}(\pa_{\nu_g} \vartheta_{\xi}) \ln  |y_{\xi}|+\frac{4}{\kap( \xi)} \vartheta_{\xi} \pa_{\nu_g} \ln  |y_{\xi}| &&\text { on }\, \pa \Si  , \\
			&\int_{\Si } H^g(x, \xi) \,\d v_g  =\frac{4}{\kap( \xi)} \int_{\Si } \vartheta_{\xi} \ln  |y_{\xi}| \, \d v_g.&&
		\end{aligned}\right.
	\end{equation}
	By elliptic regularity, problem \eqref{Hfunction} admits a unique solution $H^g(\cdot,\xi)\in C^{2,\gamma}(\Si)$ for every $\gamma\in(0,1)$;
	see, for instance, \cite{N2014,ABH20240829,ADN1959}.  The associated
	Robin function is defined by $R^g(\xi):=H^g(\xi,\xi)$.

	We now recall the standard family of entire solutions of the Liouville
	equation.  For every $(\lam, z)\in (0,+\infty)\times  \R^2 $, set
	\[
	\tilde{\delta}_{\lam,z}(y)=\ln  \frac{ 8\lam^2 }{(1 +\lam^2|y-z|^2)^2},\quad y\in \R^2.
	\]
	The functions $\tilde{\delta}_{\lambda,z}$ are precisely the finite-mass solutions of
	\[
	-\Delta u= e^u  \quad \text{  in }\, \R^2,\quad
	\int_{\R^2}e^u\,\d y<\infty.
	\]
	Moreover,  $\int_{\R^2}e^{\tilde{\delta}_{\lam,z}}\, \d y=8\pi $ for every  $(\lam, z)\in (0,+\infty)\times  \R^2 $;   see \cite{CL1991}.

	Using the isothermal coordinate chart \((U(\xi),y_\xi)\), we pull back
	\(\widetilde\delta_{\lam,0}\) to \(\Si\) by setting
	\[
	\delta_{\lam,\xi}(x)=	\left\{\begin{aligned}
		&\tilde{\delta}_{\lam,0}(y_{\xi}(x)) &&\text{ if }\, x\in U(\xi),\\
		&	0 &&  \text{ if }\, x\in \Si\setminus U(\xi).
	\end{aligned}\right.
	\]
	By \eqref{conformalchange},  \(\delta_{\lam,\xi}\) satisfies
	\begin{equation*}
		-\Delta_g  \delta_{\lam,\xi}=e^{ \delta_{\lam,\xi}-\varphi_\xi}\quad \text{ in }\, U(\xi),
	\end{equation*}
	where 	$\varphi_{\xi}=\hat{\varphi}_{\xi}\circ y_{\xi}$. For $f\in L^1(\Si)$, write $\bar{f}:=\frac{1}{|\Si|_g}\int_\Si f \,\d v_g$.
	The projected bubble \(P\delta_{\lam,\xi}\in\oH\) is defined as the unique solution of
	\begin{equation}\label{projectedbubble}
		\left\{\begin{aligned}
			&	-\Delta_g P  \delta_{\lam,\xi}=    \vartheta_{\xi} e^{-\varphi_\xi} e^{\delta_{\lam,\xi}}- \overline{  \vartheta_{\xi} e^{-\varphi_\xi}  e^{\delta_{\lam,\xi}}  } &&  \text{ in }\, \intS,\\
			&	\pa_{ \nu_g } P  \delta_{\lam,\xi}=0 && \text{ on }\, \pa \Si , \\
			&\int_{\Si } P  \delta_{\lam,\xi} \,\d v_g=0. &&
		\end{aligned}\right.
	\end{equation}
	By classical elliptic  theory,  problem \eqref{projectedbubble}  admits a unique solution $ P\delta_{\lam,\xi}$ in $\oH$, and standard regularity theory implies that  $ P\delta_{\lam,\xi}\in C^{\infty}({\Si })$; see, e.g., \cite{ADN1959,YZ2021,N2014}.

	For \(p,q\in\N_0\) satisfying \(2p+q\in\N\), define
	$$
	\psi(\al , \xi, \lam )=\sum_i \al _i P \delta_{\lam _i, \xi_i}.
	$$

	For any $\var>0$, we define a neighborhood of potential critical points at infinity.  Given $\rho>0$, $p,q\in \N_0$ with $2p+q\in\N$, and $\eta>0$, we set
	\begin{align}
		V_\rho^{\eta}(p,q,\var):=\Big\{u\in\oH:
&\,\|\nabla J_\rho(u)\|<\var,\
\exists\,\lam_1,\ldots,\lam_{p+q}>\var^{-1},\nonumber\\
&\,\exists\,(\xi_1,\ldots,\xi_{p+q})\in\Xi_{p,q}\text{ with }\nonumber\\
		&
		\,  d_g(\xi_i, \xi_j)>\eta \text{ for } i \neq j, \,  d_g(\xi_i, \pa  \Si )>\frac \eta 2 \text{ for }1 \leq i \leq p  \nonumber\\
		&\,
		\text{s.t. }	\| u-\psi(\mathbf{1}, \xi, \lam ) \|<\var \Big\},  \label{def:Vm}
	\end{align}
	where $   \mathbf{1}=(1, \ldots, 1)\in \R^{p+q}.$

	\begin{prop}\label{prop:u-k-in-V-eta}
		Let $\{u_k\}$ be a sequence of blow-up solutions to \eqref{maineq} with $\rho=\rho_k$.  Assume that $\rho_k\to\rho_*:=4\pi m$ and that the blow-up set is $(\xi_1^*,\ldots,\xi_{p+q}^*)\in\Xi_{p,q}$, where $p,q\in\N_0$ and $m=2p+q$.  Then there exist $\var_0>0$ and $\eta_0>0$ such that, for every $\var\in(0,\var_0)$, there is $k_0\in\N$ for which
		\[
		u_k\in V_{\rho_k}^{\eta_0}(p,q,\var)
		\quad\text{for all }k\ge k_0.
		\]
	\end{prop}
	\begin{proof}
The blow-up alternative yields
\begin{equation}\label{eq:bridge-basic-masses}
\rho_k\frac{Ve^{u_k}}{\int_\Si  Ve^{u_k}\,\d v_g}\,\d v_g
\rightharpoonup\sum_i\kap _i\delta_{\xi_i^*},\quad
u_k-\ln\Big(\int_\Si  Ve^{u_k}\,\d v_g\Big)
\to -\infty.
\end{equation}
The second limit in \eqref{eq:bridge-basic-masses} is locally uniform
away from the concentration points, and \(\kap _i=\kap (\xi_i^*)\); see
\cite[Section~3.1, Lemma~14 in the preprint]{LSY2023}.
The \(H^1\)-approximation proved in \cite[Proposition~2.1]{AHW2025-1}
yields centers \(\xi_i^k\to\xi_i^*\) of the prescribed types and
scales \(\lam _i^k\to\infty\). We need only compare the projections.
The local estimates in
\cite[Lemmas~2.1--2.2 and Corollary~2.1]{AHW2025-1} apply to coefficients
with a uniform positive lower bound and uniformly bounded derivatives
through order three, as for the present \(C^3\) data.

For this comparison, denote by \(\widehat P\delta_{\lam ,\xi}\) the
mean-zero projection defined using \(\vartheta_\xi e^{\delta_{\lam ,\xi}}\)
in \cite[(2.18)]{AHW2025-1}. Then
\[
-\Delta_g(\widehat P\delta_{\lam ,\xi}
             -P\delta_{\lam ,\xi})
=f_{\lam ,\xi}-\overline{f_{\lam ,\xi}},
\quad
f_{\lam ,\xi}:=\vartheta_\xi
 (1-e^{-\varphi_\xi})e^{\delta_{\lam ,\xi}},
\]
with homogeneous Neumann data and mean zero. Since
\(\varphi_\xi(\xi)=0\), the rescaling \(z=\lam  y_\xi\) yields, for
any fixed \(1<r<2\),
\(\|f_{\lam ,\xi}\|_{L^r(\Si )} \le C\lam ^{1-2/r}.\)
The mean-zero Neumann estimate therefore yields
\begin{equation}\label{eq:ccm-projection-comparison}
\|\widehat P\delta_{\lam ,\xi}-P\delta_{\lam ,\xi}\|_{H^1(\Si )}
\le C\|f_{\lam ,\xi}\|_{L^r(\Si )}
\le C\lam ^{1-2/r}=o(1),
\end{equation}
uniformly for the centers under consideration. Combining
\eqref{eq:ccm-projection-comparison} with the cited approximation yields
\[
\Big\|u_k-\sum_iP\delta_{\lam _i^k,\xi_i^k}\Big\|\to 0.
\]
Finally, \(\nabla J_{\rho_k}(u_k)=0\), the limiting centers are
separated, and the interior centers stay away from the boundary.
Thus all conditions in \eqref{def:Vm} hold. Applying the same argument
to subsequences yields the conclusion for the original sequence.
\end{proof}

	The notation $V^\eta_{\rho}(p,q,\var)$ is fixed by \eqref{def:Vm},
	with $\rho$ close to $\rho_*=4\pi m$.

	Write $\mathbb R_+:=(0,\infty)$. For $\var>0$ and $\eta>0$, define
	\begin{align}
		\widetilde B_{\var}^{\eta}:=\Big\{(\al ,\xi,\lam )\in
		(\mathbb R_+)^{p+q}\times\Xi_{p,q}\times(\mathbb R_+)^{p+q}:&
		|\al _i-1|<\var,\,  \lam _i>\var^{-1} \text{ for } i=1,\ldots,p+q;\notag\\
		& d_g(\xi_i,\xi_j)>\frac\eta 2 \text{ for }i\ne j,\notag\\
		&d_g(\xi_i,\pa \Si )>\frac{\eta}{4} \text{ for }i=1,\ldots,p\Big\}.
		\label{def:ballB-var-eta}
	\end{align}

	The next lemma estimates parameter differences in terms of the Dirichlet distance between separated bubble sums.
	\begin{lem}
		\label{lem:quantitative-bubble-matching}
		Let \(\mathscr D \Subset\Xi_{p,q}\) be a compact subset, and fix
		\(a\in (0,1)\).  There exist \(\Lam _0>1\), \(c_0>0\), and \(C>0\) such
		that the following holds.

		Let
		\[
		\xi,\widehat\xi\in \mathscr{D},\quad
		a\le\al _i,\widehat\al _i\le  a^{-1},\quad
		\lam _i,\widehat\lam _i\ge\Lam _0.
		\]
		If
		$\|
		\psi(\al ,\xi,\lam )
		-\psi(\widehat\al ,\widehat\xi,\widehat\lam )
		\|<c_0$,
		then   there exists \(\si \in\mathfrak S_{p,q}\) such that, after the relabeling
		\[
		(\widehat\al _i,\widehat\xi_i,\widehat\lam _i)
		\mapsto
		(\widehat \al _{\si (i)},
		\widehat \xi_{\si (i)},
		\widehat \lam _{\si (i)}),
		\]
		we have
		\begin{align}
			&\quad \max_i\Big\{
			\min\{\ln^{\frac 1 2}\lam _i,\ln^{\frac 1 2}\widehat\lam _i\}
			|\al _i-\widehat\al _i|
			+|\ln({\lam _i}/{\widehat\lam _i})|
			+\min\{\lam _i,\widehat\lam _i\}
			d_g(\xi_i,\widehat\xi_i)
			\Big\}\nonumber \\
			&\le
			C\|
			\psi(\al ,\xi,\lam )
			-\psi(\widehat\al ,\widehat\xi,\widehat\lam )
			\|.
			\label{eq:quantitative-bubble-matching}
		\end{align}
		The constants are uniform on \(\mathscr D\).
	\end{lem}

	\begin{proof}
We first compare the parameters when the difference of the bubble sums
tends to zero in the Dirichlet norm. Consider two sequences satisfying the
hypotheses, with all scales tending to infinity and with the difference
of the two bubble sums tending to zero in \(\oH\).

\textbf{Step 1.} Pairing the centers.
Pass to limits of the center configurations in the compact set
\(\mathscr D\). Suppose that a limiting center of the first family is not
a limiting center of the second. Rescale a fixed ball about that center by its own
concentration scale. The gradients of all bubbles of the second family,
and of the other bubbles of the first family, converge uniformly to zero
on each fixed rescaled ball. The remaining gradient converges to a
positive multiple of the gradient of
\(U_{1,0}(y):=-2\ln(1+|y|^2).\)
The Dirichlet integral of $U_{1,0}$ on a fixed ball is positive,
contradicting convergence of the difference to zero. The same argument applies on a
half-ball at a boundary center. Thus the limiting centers can be paired.
By separation in \(\mathscr D\), the pairing is unique for large
indices and respects the interior and boundary types.

\textbf{Step 2.} Relative scales and center positions.
Fix a paired center. For this comparison, interchange the two complete
parameter triples if necessary, and set
\[
s_k:=\widehat\lam _i^k/\lam _i^k\ge1,
\quad z_k:=\lam _i^k y_{\xi_i^k}(\widehat\xi_i^k).
\]
The centers already have the same limit. If \(|z_k|\to\infty\), the
gradient of the second bubble in the coordinates of the first converges
to zero on every fixed rescaled ball. Indeed, its distance from that ball
is comparable to \(|z_k|\), and its gradient there is at most
\(C/|z_k|+o(1)\). This again contradicts the nonzero limiting gradient
of the first bubble. Hence \(\{z_k\}\) is bounded.

Suppose next that \(s_k\to\infty\), and pass to a subsequence with
\(z_k\to z\), \(\al _i^k\to\al _*\), and
\(\widehat\al _i^k\to\widehat\al _*\). On compact sets away
from \(z\), the second rescaled gradient converges to
\(-4\widehat\al _*(y-z)/|y-z|^2\). Vanishing of the difference in
Dirichlet norm yields
\[
\al _*\nabla U_{1,0}(y)
=-4\widehat\al _*\frac{y-z}{|y-z|^2},\quad y\ne z.
\]
The left side is bounded near \(z\); the right side is not, since
\(\widehat\al _*\ge a>0\). The contradiction is unchanged on a
half-plane, where \(z\) belongs to the boundary line. Thus
\(\{s_k\}\) is also bounded.

Pass to limits once more. The normalized charts have an orthogonal
linear transition at the limiting center, which does not change a
radial profile. The rescaled gradient identity is therefore
\[
\al _*\nabla U_{1,0}
=\widehat\al _*\nabla U_{s,z},
\quad U_{s,z}(y):=-2\ln(1+s^2|y-z|^2).
\]
The gradients vanish at the same point, so \(z=0\). Comparison as
\(|y|\to\infty\) yields \(\al _* =\widehat\al _*\); evaluation at
one nonzero point then implies \(s=1\). Consequently,
\begin{equation}\label{eq:qualitative-parameter-matching}
\max_i\Big\{
|\ln(\widehat\lam _i^k/\lam _i^k)|
+\min\{\lam _i^k,\widehat\lam _i^k\}
 d_g(\xi_i^k,\widehat\xi_i^k)
\Big\}\to 0.
\end{equation}

The original labeling can be restored, since
\eqref{eq:qualitative-parameter-matching} is symmetric in the parameters.
The normalized scale and center derivatives have uniformly bounded
Dirichlet norms by Lemma~\ref{lem:Gram} in Appendix~\ref{app:0}.
Integration over the short parameter segments in
\eqref{eq:qualitative-parameter-matching} yields
\[
\Big\|\sum_i\widehat\al _i^k
 (P\delta_{\widehat\lam _i^k,\widehat\xi_i^k}
 -P\delta_{\lam _i^k,\xi_i^k})\Big\|\to 0.
\]
Subtract this difference from the original one and apply the amplitude
part of the Gram estimate. We obtain
\[
\max_i\sqrt{\ln\lam _i^k}\,
 |\al _i^k-\widehat\al _i^k|\to 0.
\]

\textbf{Step 3.} The quantitative estimate.
Set
\[
A_i=\sqrt{\ln\lam _i}(\widehat\al _i-\al _i),\quad
 t_i=\ln(\widehat\lam _i/\lam _i),\quad
 v_i=\lam _i y_{\xi_i}(\widehat\xi_i).
\]
Steps 1--2 imply that these differences are uniformly small after
choosing \(\Lam _0\) large and \(c_0\) small. Otherwise one could
choose \(\Lam _0=k\) and \(c_0=1/k\) and contradict those steps.
Let \(\Psi\) denote the bubble sum in these normalized coordinates,
with the first set of parameters at the origin. Lemma~\ref{lem:Gram}, together with
\(a\le\al _i\le a^{-1}\), yields uniform constants \(c,C>0\) such
that
\[
\|D\Psi(0)[h]\|\ge c|h|,\quad
\|D^2\Psi(\zeta)[h,h]\|\le C|h|^2
\]
in a fixed small normalized parameter neighborhood. Taylor's formula
therefore yields
\[
\|\Psi(\zeta)-\Psi(0)\|
\ge c|\zeta|-C|\zeta|^2\ge\frac c2|\zeta|.
\]
Uniform equivalence of chart and Riemannian distances, and of
\(\ln\lam _i\) and \(\ln\widehat\lam _i\) in this neighborhood,
yields \eqref{eq:quantitative-bubble-matching}.
\end{proof}

	Let \(\ii(\xi)=2\) for an interior center and \(\ii(\xi)=1\)
for a boundary center; only tangential center coordinates are admissible
on the boundary.

If \(u\in V_\rho^\eta(p,q,\var)\), the next proposition defines its
	unique bubble parameters.
	\begin{prop}
		\label{prop:unique}
		There exist \(\var_0>0\) and \(C>0\), independent of \(\rho\), such
		that, for every \(\var\in(0,\var_0)\) and every
		\(u\in V_\rho^\eta(p,q,\var)\), the following conclusions hold:

		Let $\theta^0=(\mathbf 1,\xi^0,\lam ^0)$ be a
		representation from   \eqref{def:Vm} such that
		\(\| u-\psi(\theta^0)\|<\var .\)

		After labeling the
		centers by the disjoint balls \(B_{\eta/4}(\xi_i^0)\), define
		\begin{align}
			\mathscr P_u(\theta^0)
			:=\Big\{(\al ,\xi,\lam )\in
			(\mathbb R_+)^{p+q}\times\Xi_{p,q}\times(\mathbb R_+)^{p+q}:\;&
			|\al _i-1|\le4\var,\,
			d_g(\xi_i,\xi_i^0)\le\frac{\eta}{4},\notag\\
			&|\ln (\lam _i/\lam _i^0)|\le1,
			\, i=1,\ldots,p+q,\notag\\
			&\|u-\psi(\al ,\xi,\lam )\|\le2\var
			\Big\}.
			\label{eq:localized-parameter-box}
		\end{align}
		Then
		\begin{equation}
			\min_{(\al ,\xi,\lam )\in\mathscr P_u(\theta^0)}
			\|u-\psi(\al ,\xi,\lam )\|
			\label{eq:min}
		\end{equation}
		has a unique minimizer modulo the action of $\mathfrak S_{p,q}$.  The minimizer is independent of \(\theta^0\) and belongs to
		\(\widetilde B_{2\var}^{\eta}\), where $\widetilde{B}_{2\var}^\eta$ is defined by \eqref{def:ballB-var-eta}.

		Moreover, within this local parameter neighborhood and for a sufficiently small orthogonal remainder, $u$ admits a unique decomposition
		\[
		u=\sum_i
		\al _iP\delta_{\lam _i,\xi_i}+w,
		\]  where  $w\in\oH$ satisfies
		\begin{equation}
			\label{orthogonality}
			\Big\langle w,P\delta_{\lam _i,\xi_i}\Big\rangle
			=
			\Big\langle w,
			\frac{\pa  P\delta_{\lam _i,\xi_i}}{\pa \lam _i}
			\Big\rangle
			=
			\Big\langle w,
			\frac{\pa  P\delta_{\lam _i,\xi_i}}
			{\pa (\xi_i)_j}
			\Big\rangle
			=0
		\end{equation}
		for \(i=1,\ldots,p+q\) and
		\(j=1,\ldots,\ii(\xi_i)\).  Moreover,
		\begin{equation}
			\label{eq:apha_i_estimate}
			\max_i
			\sqrt{\ln\lam _i}\,|\al _i-1|
			\le C\var.
		\end{equation}
		In particular, as \(\var\to0^+\),
		\[
		\max_i|\al _i-1|
		=o\Big(
		\sum_i\frac{1}{\sqrt{\ln\lam _i}}
		\Big)
		=o\Big(
		\frac{1}{\sqrt{|\ln\var|}}
		\Big).
		\]
	\end{prop}

	\begin{proof}
By \eqref{def:Vm}, the localized parameter set
\(\mathscr P_u(\theta^0)\) is nonempty and compact by the separation
conditions in \eqref{def:Vm} and the closed constraints in
\eqref{eq:localized-parameter-box}. Thus \eqref{eq:min} has a minimizer
\(\theta=(\al ,\xi,\lam )\).
Lemma~\ref{lem:quantitative-bubble-matching} yields
\begin{equation}\label{eq:local-parameter-estimate}
\max_i\Big\{
\min\{\sqrt{\ln\lam _i},\sqrt{\ln\lam _i^0}\}\,|\al _i-1|
+|\ln(\lam _i/\lam _i^0)|
+\min\{\lam _i,\lam _i^0\}d_g(\xi_i,\xi_i^0)
\Big\}
\le C\var .
\end{equation}
For sufficiently small \(\var \), the minimizer is therefore an
interior point of the parameter set and belongs to
\(\widetilde B_{2\var }^{\eta}\).

Introduce normalized local coordinates
\[
\al _i=1+\frac{a_i}{\sqrt{\ln\lam _i^0}},
\quad
\lam _i=\lam _i^0e^{s_i},
\quad
\xi_i=y_{\xi_i^0}^{-1}\Big(\frac{z_i}{\lam _i^0}\Big),
\quad
\zeta=(a,s,z).
\]
By \eqref{eq:local-parameter-estimate}, every possible minimizer belongs to a
ball \(|\zeta|\le C\var \). Lemma~\ref{lem:Gram} yields, uniformly in this ball,
\[
c|h|^2\le\|D\Psi(\zeta)[h]\|^2\le C|h|^2,
\quad
\|D^2\Psi(\zeta)[h,h]\|\le C|h|^2,
\]
where \(\Psi(\zeta)=\psi(\al (\zeta),\xi(\zeta),\lam (\zeta))\).
Hence the distance function
\(\mathscr F_u(\zeta) =\frac12\|u-\Psi(\zeta)\|^2\)
satisfies
\[
D^2\mathscr F_u(\zeta)[h,h]\ge \frac c2|h|^2
\]
for \(\var \) small. Thus \(\mathscr F_u\) is strictly convex on
\(\{|\zeta|\le C\var \}\), so the minimizer is unique. Differentiation at the interior minimum
yields \eqref{orthogonality}. The same strict convexity shows that any other
orthogonal representation in the local parameter box, and any minimizer obtained from a different initial representation, are equal after a permutation that preserves the types of the centers.

Finally, \eqref{eq:local-parameter-estimate} yields
$\sqrt{\ln\lam _i}|\al _i-1|\le C\var $,
which is \eqref{eq:apha_i_estimate} and proves the stated \(o\)-estimates.
\end{proof}

	\subsection{Gradient expansions in the neighborhood at infinity}

With \eqref{orthogonality} understood, set
	\[
	\mathcal E_{\lam,\xi}^{p,q}
	:=
	\{w\in\oH:w\text{ satisfies \eqref{orthogonality}}\},
	\quad
	E_{\lam,\xi}^{p,q}(\var)
	:=
	\{w\in\mathcal E_{\lam,\xi}^{p,q}:\|w\|\le C\var\}.
	\]
	In what follows, we also write \(E_{\lam ,\xi}^{p,q}\) for
the full closed subspace \(\mathcal E_{\lam ,\xi}^{p,q}\). A norm restriction is always written explicitly as
\(E_{\lam ,\xi}^{p,q}(\var )\). Unless stated otherwise,
\(u=u_0+w\in V_\rho^\eta(p,q,\var )\) denotes the decomposition
of Proposition~\ref{prop:unique}; in particular, \(\|w\|<\var \).

	Using the decomposition in Proposition~\ref{prop:unique}, we compute
the derivatives of \(J_\rho\) in the amplitude, location, scale, and
orthogonal directions. The expansions yield the equations for the bubble
parameters and the estimates used in the Morse index calculation.

	Throughout this subsection, we fix \(p,q\in\N_0\) with $m:=2p+q\in\N$, $\rho_*:=4\pi m$, and write
	$\rho=(1+\mu)\rho_*$,
	where \(\mu\in\R\) is sufficiently small.  We also fix \(\eta>0\) and assume
	that \(\var>0\) is sufficiently small.  Unless otherwise stated, all
	\(O(\cdot)\)-estimates are uniform in
	\(V_\rho^\eta(p,q,\var)\); \(o(1)\) denotes a quantity tending to zero
	as \(\var\to0^+\).

The unit-area convention used in \cite{AHW2025-1} extends to the
present equation by replacing $g$ by $g/|\Sigma|_g$. The Dirichlet
norm and the mean-zero space are unchanged, the normalized equation
has the same parameter $\rho$, and the two functionals differ by
the additive constant $\rho\ln |\Sigma|_g$. In the calculations below
we retain the area in every mean subtraction and recompute the
projection coefficients from \eqref{projectedbubble}. No parameter
or index is rescaled in this normalization.

The local calculations follow \cite[Section~3.2 and Appendix~C]{AHW2025-1}.
The measure used in integrations by parts with \eqref{projectedbubble} is
\(\vartheta_\xi e^{\delta_{\lam ,\xi}-\varphi_\xi}\,\d v_g\),
which becomes
\(\vartheta_\xi e^{\widetilde\delta_{\lam ,0}}\,\d y\)
in the center chart. The same local measure appears in the proofs of
\cite[Lemmas~A.4--A.5]{AHW2025-1}, although the factor
\(e^{-\varphi_\xi}\) is omitted from the definition in
\cite[(2.18)]{AHW2025-1}. Those local integral calculations are combined
with the exact expansion established in Appendix~\ref{app:0} for
\eqref{projectedbubble}. The comparison
\eqref{eq:ccm-projection-comparison} yields an \(H^1\)-approximation;
the higher-order coefficients and their parameter derivatives are obtained from the exact expansion in Appendix~\ref{app:0}.

For any $\xi_i\in\Si$, write
\[
\d\nu_i:=\vartheta_{\xi_i}
 e^{\delta_{\lam _i,\xi_i}-\varphi_{\xi_i}}\,\d v_g.
\]
For a fixed sufficiently small $\epsilon_0>0$, the weighted estimate is
\begin{equation}\label{eq:A-uniform-weighted-Sobolev}
\int_\Si  |w|^r e^{b|w|}\,\d\nu_i\le C_{r,b}\|w\|^r,
\quad 1\le r<\infty,\quad b\ge0,
\end{equation}
for \(w\in\mathcal E_{\lam ,\xi}^{p,q}\) with
\(\|w\|\le\epsilon_0\), uniformly in the scales. Its proof is given in
Step~1 of the proof of Lemma~\ref{lem:terminal-reduction} in Appendix~\ref{app:0}.

\begin{lem} \label{lem:lamexpansion}
		Let $u=u_0+w \in V_{\rho}^\eta(p,q,\var)$ with $u_0=\sum_{i} \al_i P\delta_{\lam_i,\xi_i}$ and  $w\in E_{\lam,\xi }^{p,q}$. Then, for each $i=1,\ldots,p+q$,  the following expansion holds as  $\var\to  0^+$:
		\begin{align*}
			\Big\langle \nabla J_{\rho }(u),\lam_i\frac{\pa P\delta_{\lam_i,\xi_i}}{\pa \lam_i}\Big\rangle =&\,2\al _i \kap(\xi_i)   (\tau_i-\mu+\mu\tau_i)-\frac \pi 2   \frac{\rho\mathfrak{f}_2(\xi_i)  }{\int_{\Si} V e^u \, \d v_g} \lam_i^{4 \al _i-3}\\
			&\,+\mu\rho_*\frac{\kap(\xi_i)}{|\Si|_g}\frac{\ln\lam_i}{\lam_i^2}- (1+\mu)(1-\tau_i)\sum_{j}\frac{\kap(\xi_i)\kap(\xi_j)}{|\Si|_g}\frac{\ln\lam_j}{\lam_j^2} \\
			&\,+O\Big(\sum_j |\al _j-1|^2+\|w\|^2 \Big)+o\Big(\sum_j\frac{\ln\lam_j}{\lam_j^2}\Big),
		\end{align*}
		where
		\begin{gather}
			\label{def:tau_i}
			\tau_i:= 1-\frac{\rho_*}{8(2 \al _i-1)} \frac{(\mathcal{F}_{\xi}^i \mathrm{g}_{\xi}^i )(\xi_i)}{\int_{\Si} V e^u\, \d v_g} \lam_i^{4 \al _i-2},\\	\label{eq:def_g}
			\rg_{\xi}^i(x) := \exp\Big\{(\al_i-1) \kap( \xi_i)H^g(x,\xi_i)+ \sum_{j\neq i}(\al_j-1) \kap( \xi_j) G^g(x,\xi_j) \Big\},
		\end{gather}
		and \begin{equation}
			\label{def:frak_f2} 	\mathfrak{f}_2(\xi_i):=\left\{
			\begin{aligned}
				&	0&& \text{ if }\,\xi_i\in \intS,\\
				&	-(\pa_{\nu_g}  \ln  (\cF^i_{\xi}\rg^i_{\xi})+2k_g )\cF^i_{\xi}\rg^i_{\xi}|_{\xi_i}&& \text{  if }\,	\xi_i\in \pa\Si.
			\end{aligned}
			\right.
		\end{equation}
		Moreover,
		\[
		|\tau_i|=O\Big(\var+|\mu|+\sum_j |\al _j-1|^2\Big).
		\]
	\end{lem}
\begin{proof}
Write \(\kap _i:=\kap (\xi_i)\),
\(D_i:=\lam _i\pa _{\lam _i}\), and set
\[
M:=\int_\Si  Ve^u\,\d v_g,
\quad
M_i:=\frac{\kap _i(\mathcal F_\xi^i\mathrm g_\xi^i)(\xi_i)}
 {8(2\al _i-1)}\lam _i^{4\al _i-2},
\quad
S_0:=\sum_j\frac{\kap _j}{2|\Si |_g}
             \frac{\ln\lam _j}{\lam _j^2}.
\]
The calculation of the scale pairing is given in
\cite[Proposition~3.2 and its proof in Appendix~C]{AHW2025-1}.
For the present projection, Lemmas~\ref{lem:projbubble_derivative_integral}
and~\ref{lem:projbubble_lamderivative__pointwise} yield
\[
\langle u,D_iP\delta_{\lam _i,\xi_i}\rangle
=2\al _i\kap _i-
\frac{\kap _i\rho_*}{|\Si |_g}\frac{\ln\lam _i}{\lam _i^2}
+o\Big(\sum_j\frac{\ln\lam _j}{\lam _j^2}\Big).
\]
The nonlinear pairing follows from the local factorization in
Lemma~\ref{lem:Ve-u-integral} and the constant and boundary-linear
integrals in the cited proof. For the terms containing \(w\), the exact cancellations are
\[
\int_\Si  w\,\d\nu_i=0,
\quad
\int_\Si \frac{4w}{1+\lam _i^2|y_{\xi_i}|^2}\,\d\nu_i=0,
\quad
\d\nu_i:=\vartheta_{\xi_i}
 e^{\delta_{\lam _i,\xi_i}-\varphi_{\xi_i}}\,\d v_g,
\]
which follow from \eqref{orthogonality} and differentiation of
\eqref{projectedbubble}. Weighted H\"older and the weighted exponential
estimate \eqref{eq:A-uniform-weighted-Sobolev} then show that
\begin{align*}
\frac1M\int_\Si  Ve^u D_iP\delta_{\lam _i,\xi_i}\,\d v_g
=&\,\frac{2\al _iM_i}{M}(1+S_0)
+\frac\pi2\frac{\mathfrak f_2(\xi_i)\lam _i^{4\al _i-3}}M
+h^1_{\lam _i,\xi_i}\\
&\,+O\Big(\sum_j|\al _j-1|^2+\|w\|^2\Big)
+o\Big(\sum_j\frac{\ln\lam _j}{\lam _j^2}\Big),
\end{align*}
 where $h^1_{\lam _i,\xi_i}$ is the scalar defined in
Lemma~\ref{lem:projbubble_lamderivative__pointwise} in
Appendix~\ref{app:0}.
The smooth correction terms contribute
\(O(\sum_j\lam _j^{-2})\). The interaction error
\[
O\Big(\sum_{j\ne i}\frac{\ln\lam _i}{\lam _j^2\lam _i^2}\Big)
\]
is included in the remainder in Lemma~\ref{lem:lamexpansion},
without a comparability assumption on the scales.
The normalization is uniform by the leading mass estimate in
Lemma~\ref{lem:Ve-u-integral}. Finally,
\[
\frac{M_i}{M}=\frac{\kap _i}{\rho_*}(1-\tau_i),
\quad
h^1_{\lam _i,\xi_i}
=-\frac{\kap _i}{|\Si |_g}\frac{\ln\lam _i}{\lam _i^2}
+O(\lam _i^{-2}).
\]
The expansion follows by subtracting \(\rho\) times the nonlinear pairing.
Since the scale direction has bounded Dirichlet norm, the stated estimate
for \(\tau_i\) follows from \eqref{def:Vm}.
\end{proof}

	Lemma~\ref{lem:lamexpansion} implies the following relation between the concentration scales.
\begin{cor}\label{cor:lam_i_samespeed}
Under the assumptions of Lemma~\ref{lem:lamexpansion}, assume in addition
that $\sum_s|\al _s-1|\ln\lam _s=o(1)$. Then, for every
$i,j=1,\ldots,p+q$,
\[
 \frac{\mathcal F_\xi^i(\xi_i)\lam _i^2}
      {\mathcal F_\xi^j(\xi_j)\lam _j^2}=1+o(1).
\]
\end{cor}
\begin{proof}
Set
\[
 M:=\int_\Si  Ve^u\,\d v_g,\quad
 M_i:=\frac{\kap (\xi_i)(\mathcal F_\xi^i\mathrm g_\xi^i)(\xi_i)}
 {8(2\al _i-1)}\lam _i^{4\al _i-2}.
\]
By the mass estimate \eqref{lem:A.8-1} in Appendix~\ref{app:0},
$M=(1+o(1))\sum_iM_i$, so $M_i/M$ is bounded without any
comparison of the scales. The scale directions have uniformly bounded
Dirichlet norm by Lemma~\ref{lem:Gram} in Appendix~\ref{app:0}.
Lemma~\ref{lem:lamexpansion} and $\|\nabla J_\rho(u)\|<\var $ imply
\[
 \tau_i-\mu+\mu\tau_i=o(1),\quad
 (1+\mu)(1-\tau_i)=1+o(1).
\]
Indeed, each boundary term is $O(M_i/(M\lam _i))=o(1)$, and
the remaining error terms tend to zero. Dividing the last identity
for $i$ by the corresponding identity for $j$ and using
\eqref{def:tau_i}, we obtain
\[
 \frac{(\mathcal F_\xi^i\mathrm g_\xi^i)(\xi_i)
              \lam _i^{4\al _i-2}/(2\al _i-1)}
      {(\mathcal F_\xi^j\mathrm g_\xi^j)(\xi_j)
              \lam _j^{4\al _j-2}/(2\al _j-1)}
 =\frac{1-\tau_i}{1-\tau_j}=1+o(1).
\]
Now $\mathrm g_\xi^i(\xi_i)=1+o(1)$ and
$\lam _i^{4(\al _i-1)}=1+o(1)$ by the assumed logarithmic
amplitude estimate. Substitution proves the scale comparison.
\end{proof}

Define the second-order coefficient
	\begin{equation}\label{def:frak_f3}
	\mathfrak f_3(\xi_i):=
\left\{\begin{aligned}
&(\Delta_g(\cF^i_\xi\rg^i_\xi)-2K_g\cF^i_\xi\rg^i_\xi)(\xi_i)
	&&\text{ if }\, \xi_i\in\intS,\\
	&(\Delta_g(\cF^i_\xi\rg^i_\xi)-2K_g\cF^i_\xi\rg^i_\xi
	+4k_g\cF^i_\xi\rg^i_\xi(\pa_{\nu_g}\ln V+k_g))(\xi_i)
	&&\text{ if }\, \xi_i\in\pa\Si.
\end{aligned}\right.
	\end{equation}
	\begin{lem}  \label{lem:sum_tau_i}
		Under the assumptions of Lemma \ref{lem:lamexpansion}, assume  further that $\sum_{i}|\al_i-1|\ln\lam_i=o(1) $. Then
		\begin{align*}
			\sum_{i}\kap(\xi_i)\tau_i=&\,\frac{\pi \rho}{2(1+\mu)}\frac{\sum_{i}\mathfrak{f}_2(\xi_i)\lam_i^{4\al_i-3}}{\int_{\Si} V e^u\, \d v_g}+\frac{\rho}{16(1+\mu)}\frac{\sum_{i}
				\kap ( \xi_i) \mathfrak{f}_3(\xi_i)\ln  \lam_i}{\int_{\Si} V e^u\, \d v_g}\\&\,+\frac{\rho}{1+\mu}\sum_{i}\frac{\kap(\xi_i)}{2|\Si|_g}\frac{\ln\lam_i}{\lam_i^2}+O \Big(\sum_{i}  \Big(
			|\al_i-1|^2 +\frac{1}{\lam_i^2}\Big)+\|w\|^2\Big),
		\end{align*}
		where $\tau_i$, $\mathfrak{f}_2$, and $\mathfrak{f}_3$ are defined in \eqref{def:tau_i}, \eqref{def:frak_f2}, and \eqref{def:frak_f3}, respectively.
	\end{lem}
\begin{proof}
Set $M=\int_\Si Ve^u\,\d v_g$ and
$M_i=\kap(\xi_i)(\cF_\xi^i\rg_\xi^i)(\xi_i)
\lam_i^{4\al_i-2}/[8(2\al_i-1)]$. By \eqref{def:tau_i},
\[
 \sum_i\kap(\xi_i)\tau_i
 =\rho_*\Big(1-\frac{\sum_iM_i}{M}\Big).
\]
Substitute the refined expansion in Lemma~\ref{lem:Ve-u-integral}
and $\rho_*={\rho}/(1+\mu)$. Corollary~\ref{cor:lam_i_samespeed}
permits replacing $\sum_iM_i/M$ by one in the second-order scalar term;
the resulting error is included in the stated remainder.
\end{proof}

	Summing the derivatives in Lemma~\ref{lem:lamexpansion} and using Lemma~\ref{lem:sum_tau_i} yields the derivative of \(J_\rho\) in the direction
	\begin{equation}
		\label{eq:Z-lam} 	Z_\lam
		:=
		\sum_i\frac{\lam_i}{\al_i}
		\frac{\pa P\delta_{\lam_i,\xi_i}}{\pa\lam_i}.
	\end{equation}
	At $\al=\mathbf1$, \eqref{eq:Z-lam} is the common logarithmic scale direction.

	\begin{cor}\label{cor:Projbubble_sumexpansion}
		Under the assumptions of Lemma \ref{lem:sum_tau_i}, the following expansion
		holds as \(\var\to0^+\):	 \begin{align*}
			\la \nabla J_{\rho }(u),  Z_\lam\ra =&\,-\frac{2\mu\rho}{1+\mu}+\frac{\mu \rho}{1+\mu}\sum_{i}\frac{\kap(\xi_i)}{|\Si|_g}\frac{\ln\lam_i}{\lam_i^2}+	\frac{\rho\pi}{2} \frac{\sum_{i}\lam_i^{4\al_i-3}\mathfrak{f}_2(\xi_i)}{\int_{\Si} V e^u\, \d v_g}\\&\,+\frac{\rho}{8}\frac{\sum_{i}
				\kap ( \xi_i) \mathfrak{f}_3(\xi_i)\ln  \lam_i}{\int_{\Si} V e^u\, \d v_g} +O \Big(\sum_{i}
			|\al_i-1|^2 +\|w\|^2\Big)
			+ o\Big (\sum_{i}\frac{\ln \lam_i}{\lam_i^2}\Big).
		\end{align*}
	\end{cor}
	\begin{proof}
		The result follows from Lemmas~\ref{lem:lamexpansion}--\ref{lem:sum_tau_i} and
$\rho_*=\sum_i\kap (\xi_i)$.
	\end{proof}
	The amplitude and location calculations are given in
\cite[Propositions~3.3--3.4 and their proofs in Appendix~C]{AHW2025-1}.
The calculation uses $\nu_i$ and the interaction estimates of
Appendix~\ref{app:0}. In the amplitude formula,
the mixed term containing \(H^g(\xi_j,\xi_i)\) includes the factor
\(\kap (\xi_i)\) from the regular part of \(P\delta_{\lam _i,\xi_i}\).
\begin{lem}\label{lem:amplitude-expansion}
		For each \(i=1,\ldots,p+q\), under the assumptions of
Lemma~\ref{lem:lamexpansion}, one has the following expansion as
\(\var\to0^+\):
		\begin{align*}
			\Big\la\nabla J_{\rho} (u), \frac{P \delta_{\lam_i,\xi_i}  }{\ln \lam_i }\Big\ra=&\, 4\kap (\xi_i)( (\al _i-1)+(\tau_i-\mu+\mu\tau_i)) \\&\, - \Big(2\pi-\frac{(1+2\ln 2)\pi}{\ln\lam_i}\Big) \frac{\rho \mathfrak{f}_2(\xi_i)}{\int_\Si  V e^u  \, \d v_g} \lam_i^{4\al_i-3} \\
			&\,-\sum_{j}  \frac{\rho \pi  \kap (\xi_i)H^g(\xi_j,\xi_i)\mathfrak{f}_2(\xi_j) }{2\int_\Si   V e^u \, \d v_g } \frac{\lam_{j}^{4\al _j-3}}{\ln \lam_i}\\
			&
			\,+O \Big(\sum_{j}  \frac{|\al _j-1|}{\ln \lam_i}
 +\sum_{j}  \frac{ |\tau_j-\mu+\mu\tau_j|}{\ln \lam_i}\\
&\quad+\Big(\frac{1}{\ln \lam_{i}}+ |\al _i-1|\Big)\|w\|+\|w\|^2\Big)\\
			&\, + O\Big(\frac 1 {\ln \lam_i}\sum_j \frac{\ln \lam_j}{\lam_j^2}+ \frac{\ln \lam_i}{\lam_i^2} +\sum_j \frac 1 {\lam_j^{4\al _j-2}}\Big),
		\end{align*}
		where \(\cF_\xi^i\), \(\tau_i\), \(\rg_\xi^i\), and \(\mathfrak f_2\) are
		defined in \eqref{eq:def_F}, \eqref{def:tau_i}, \eqref{eq:def_g}, and
		\eqref{def:frak_f2}, respectively.
	\end{lem}
	\begin{lem}\label{lem:bubble_expansion3} Under the assumptions of Lemma \ref{lem:lamexpansion}, the following
		expansion holds, for every \(i=1,\ldots,p+q\) and
		\(j=1,\ldots,\ii(\xi_i)\), as \(\var\to0^+\):
		\begin{align}
			\Big\la \nabla J_\rho (u), \frac{1}{\lam_i}\frac{\pa P\delta_{\lam_i , \xi_i}}{\pa(\xi_i)_j} \Big\ra =&\, -\frac{(1+\mu)}{2}\frac{1}{\lam_i}\frac{\pa \ff(\xi)}{\pa(\xi_i)_j} +O\Big( \frac{|\mu|}{\lam_i}+\sum_s\frac{|\tau_s|}{\lam_i}+ \frac{\ln \lam_i}{\lam_i^2}\Big)\nonumber\\
			& \, +O\Big(\sum_s |\al _s-1|^2+\sum_s \frac{1}{\lam_s^2}+ \|w\|^2\Big),\label{eq:gradient_pa_xi_i}
		\end{align}
		where $j=1,\ldots,\ii(\xi_i)$ and  $\ff$ is defined in \eqref{eq:reduce_fun_0}.
	\end{lem}

	Subtracting twice the scale expansion in Lemma~\ref{lem:lamexpansion},
divided by $\al _i$, from Lemma~\ref{lem:amplitude-expansion}
yields the following estimate.
\begin{cor}\label{cor:1}
		Under the assumptions of Lemma~\ref{lem:lamexpansion}, the following expansion holds for $i=1,\ldots,p+q$ as $\var\to0^+$:
		\[
		\Big\langle\nabla J_{\rho}(u), \frac{P\delta_{\lam_i,\xi_i}}{\ln \lam_i}-2\frac{\lam_i}{\al _i}  \frac{\pa  P\delta_{\lam_i,\xi_i}}{\pa  \lam_i}\Big\rangle
		=  4\kap(\xi_i)(\al _i-1)+ \epsilon_i(\al ,\xi,\lam,w),
		\]
		where
\[
 \epsilon_i(\alpha,\xi,\lambda,w)=
 \begin{cases}
 O\left(\begin{aligned}
 &\|w\|^2+\frac{\|w\|}{\ln\lambda_i}\\
 &+\sum_j\Big(\frac{|\tau_j-\mu+\mu\tau_j|+|\alpha_j-1|}{\ln\lambda_i}
       +|\alpha_j-1|^2+\frac1{\lambda_j}\Big)
 \end{aligned}\right),&q\ne0,\\[8pt]
 O\left(\begin{aligned}
 &\|w\|^2+\frac{\|w\|}{\ln\lambda_i}\\
 &+\sum_j\Big(\frac{|\tau_j-\mu+\mu\tau_j|+|\alpha_j-1|}{\ln\lambda_i}
       +|\alpha_j-1|^2\\
 &\quad+\frac{\ln\lambda_j}{\lambda_j^2}
       +\frac1{\lambda_j^{4\alpha_j-2}}\Big)
 \end{aligned}\right),&q=0.
 \end{cases}
\]
	\end{cor}

	For exact solutions, Lemma~\ref{lem:lamexpansion} and
Corollary~\ref{cor:1}, together with orthogonal coercivity, yield
the following estimates; compare \cite[Proposition~3.7]{AHW2025-1}.
The proof uses the projection in \eqref{projectedbubble}.
The sharper estimate needed when \(q=0\) is proved separately in
Lemma~\ref{lem:interior-isolating-bootstrap}.
\begin{prop}\label{prop:estimates_rho_k}
Let $u_k$ be a sequence of blowing-up solutions of \eqref{maineq}
with $\rho_k\to\rho_*:=4\pi m$, and suppose its blow-up set is
$\xi^*\in\Xi_{p,q}$, where $2p+q=m$. By Propositions
\ref{prop:u-k-in-V-eta} and \ref{prop:unique}, for all sufficiently
large $k$ there is a unique decomposition, up to permutations,
\[
u_k=\sum_i\al _i^kP\delta_{\lam _i^k,\xi_i^k}+w^k,
\quad w^k\in E^{p,q}_{\lam ^k,\xi^k}.
\]
Suppress the sequence index and set
\[
\Lam :=\min_i\lam _i,\quad
\mu:=\frac{\rho-\rho_*}{\rho_*},\quad
M:=\int_\Si  Ve^u\,\d v_g.
\]
Without assuming \ref{item:C1} or \ref{item:C2}, one has
\[
\|w\|+\sum_i|\al _i-1|
+\sum_i|\tau_i-\mu+\mu\tau_i|=O(\Lam ^{-1}),
\quad \lam _i\asymp\Lam ,
\]
and
\[
M\asymp\Lam ^2,\quad
\frac{\lam _i}{\lam _1}
=\sqrt{\frac{\mathcal F_\xi^1(\xi_1)}
{\mathcal F_\xi^i(\xi_i)}}+o(1),\quad
|\nabla\mathbf F_{p,q}^V(\xi)|
=O(\Lam ^{-1}\ln\Lam ).
\]
In particular, $\xi^*$ is a critical point of
$\mathbf F_{p,q}^V$. Moreover,
\[
\mu=
\left\{\begin{aligned}
&\frac\pi4
\frac{\mathcal L_1(\xi^*)
\sqrt{\mathcal F_{\xi^*}^1(\xi_1^*)}\lam _1}{M}
+o(\lam _1^{-1}) &&\text{ if }\, q\neq 0,\\&
\frac1{16}
\frac{\mathcal L_2(\xi^*)\ln\lam _1}{M}
+o(\lam _1^{-2}\ln\lam _1)&&\text{ if }\, q=0.
\end{aligned}	\right.
\]
Consequently $\mu=O(\Lam ^{-1})$ if $q>0$ and
$\mu=O(\Lam ^{-2}\ln\Lam )$ if $q=0$, and in either case
$\mu\ln\Lam \to0$.
If, in addition, \ref{item:C1} holds, then
\[
\sum_i d_g(\xi_i,\xi_i^*)
=O(\Lam ^{-1}\ln\Lam ).
\]
\end{prop}
\begin{proof}
Consider the decomposition from Proposition~\ref{prop:unique}, with the sequence
index omitted. Set
\[
 \Lam :=\min_i\lam _i,\quad \ell:=\ln\Lam ,\quad
 A:=\sum_i|\al _i-1|,\quad W:=\|w\|,\quad
 T:=\sum_i|\tau_i-\mu+\mu\tau_i|.
\]
Write \(\kap _i:=\kap (\xi_i)\) and
\(P\delta_i:=P\delta_{\lam _i,\xi_i}\). Thus \(A+W=o(1)\) and
\(\Lam \to\infty\). Define the leading local masses by
\[
 M_i:=\frac{\kap _i(\mathcal F_\xi^i\mathrm g_\xi^i)(\xi_i)}
 {8(2\al _i-1)}\lam _i^{4\al _i-2},
\]
and set
\[
 u_0:=\sum_i\al _iP\delta_i,\quad
 M:=\int_\Si  Ve^u\,\d v_g,\quad
 M_0:=\int_\Si  Ve^{u_0}\,\d v_g,\quad
 \d\nu_i:=\vartheta_{\xi_i}
 e^{\delta_{\lam _i,\xi_i}-\varphi_{\xi_i}}\,\d v_g.
\]
We first compute the limiting mass ratios, then prove coercivity. Let $D_i$ be fixed disjoint coordinate neighborhoods of the limiting
centers, small enough that $\vartheta_{\xi_i}=1$ on $D_i$ for all
sufficiently large sequence indices. In the
\(i\)-th rescaled chart,
\(M_i^{-1}Ve^{u_0}\,\d v_g\le CK_i\,\d\nu_i\), where
\(K_i(z):=(1+|z|^2)^{-2(\al _i-1)}.\)
For $|\al _i-1|\le1/8$, the $L^2(\d\nu_i)$ norms of $K_i$ are
uniformly bounded: the full-plane integrable upper bound is a constant
times $(1+|z|^2)^{-3/2}\,\d z$. The same holds on a half-plane.
The weighted estimate \eqref{eq:A-uniform-weighted-Sobolev}, proved in
Appendix~\ref{app:0}, uses only
$\int_\Si  z\,\d\nu_i=0$ for $z\in\mathcal E_{\lam ,\xi}^{p,q}$
and the Dirichlet energy, independently of coercivity of the quadratic
form $Q$ defined in Proposition~\ref{prop:J-u-expansion} in
Subsection~\ref{subsec:finite-dimensional-reduction}. Weighted H\"older and $|e^w-1|\le |w|e^{|w|}$ therefore yield
\[
 \int_{D_i}Ve^{u_0}|e^w-1|\,\d v_g\le CM_iW.
\]
Direct integration of the local profile, by dominated convergence
uniformly as $\al _i\to1$, yields
\[
 \int_{D_i}Ve^{u_0}\,\d v_g=(1+o(1))M_i,\quad
 \int_{D_i}Ve^u\,\d v_g=(1+o(1))M_i.
\]
The integrals on $\Si\setminus\bigcup_iD_i$ are $O(1)$ by \eqref{eq:MT-Jensen-explicit}. Since $\sum_iM_i\to\infty$, we obtain
\begin{equation}\label{eq:coarse-total-mass-balance}
 M=(1+o(1))\sum_iM_i,\quad
 M_0=(1+o(1))\sum_iM_i.
\end{equation}
Quantization for the given exact solutions, already used in
\eqref{eq:bridge-basic-masses}, states
$\rho M^{-1}\int_{D_i}Ve^u\,\d v_g\to\kap _i$. Consequently, by \eqref{eq:coarse-total-mass-balance},
\begin{equation}\label{eq:coarse-quantized-mass-weights}
 \frac{\rho M_i}{M_0}\to \kap _i.
\end{equation}
Equation~\eqref{eq:coarse-quantized-mass-weights} holds without scale
comparability or the condition $A\max_i\ln\lam _i=o(1)$.

The coercivity estimate follows from the mass ratios in \eqref{eq:coarse-quantized-mass-weights}. We have $K_i\to1$ in $L^2(\d\nu_i)$ by the same dominated
convergence, and the smooth local coefficients converge to their center values
in the same weighted norm. Thus
\eqref{eq:coarse-quantized-mass-weights}, weighted H\"older, and
\eqref{eq:A-uniform-weighted-Sobolev} imply, uniformly for
$z\in\mathcal E_{\lam ,\xi}^{p,q}$,
\[
 \frac\rho{M_0}\int_\Si  Ve^{u_0}z^2\,\d v_g
 =\sum_i\int_\Si  z^2\,\d\nu_i+o(\|z\|^2),\quad
 \frac1{M_0}\int_\Si  Ve^{u_0}z\,\d v_g=o(\|z\|).
\]
Weighted H\"older estimates apply to the integrals against $\nu_i$
outside the neighborhoods $D_i$. On $\Si \setminus\bigcup_iD_i$,
\(M_0^{-1}Ve^{u_0}\le CM_0^{-1}\).
Therefore $Q(z,z)=Q_0(z,z)+o(\|z\|^2)$, where
\[Q_0(z,z):=\|z\|^2-\sum_i\int_\Si  z^2\,\d\nu_i.\] The spectral estimate for $Q_0$ in the proof of Lemma~\ref{lem:terminal-reduction}
(Appendix~\ref{app:0}, Step~1) therefore proves uniform coercivity. The
limit \eqref{eq:coarse-quantized-mass-weights}, rather than two-sided
comparability of the mass fractions, is used here.

With $f$ as defined in Proposition~\ref{prop:J-u-expansion}, the Euler
equation in the orthogonal direction implies
\(cW^2\le |f(w)|+CW^3.\)
The cubic error follows from Taylor expansion of the normalized exponential pairing: combine the uniform $L^2(\d\nu_i)$ bound for $K_i$ with
\eqref{eq:A-uniform-weighted-Sobolev} and H\"older's inequality. Equivalently, differentiate
\(M_0^{-1}\int_\Si  e^{sw}Ve^{u_0}\,\d v_g\) for \(s\in [0,1]\); the third
derivative of its logarithm is an integral of a centered cubic power, with absolute value at most \(CW^3\).
No scale comparability or prior smallness of $A\ln\lam _i$ is used.
For large sequence indices, $CW\le c/2$. Moving $CW^3$ to the left
and applying Lemma~\ref{lem:A.7} yields
\begin{equation}\label{eq:coarse-fiber-direct}
 W\le C(A+\Lam ^{-1}).
\end{equation}
Here $\sum_i\ln^{3/2}\lam _i/\lam _i^2=O(\Lam ^{-1})$,
and the remaining region contributes at most $C\Lam ^{-3/2}$
after taking, for example, $\al _i\ge7/8$.

The scale equations in Lemma~\ref{lem:lamexpansion} imply
\(T\le C(A^2+W^2+\Lam ^{-1}).\)
Indeed, $M_i/M$ is uniformly bounded, so every boundary-linear
term is $O(\lam _i^{-1})$; also
$\sum_i\ln\lam _i/\lam _i^2=O(\ell\Lam ^{-2})$.
Corollary~\ref{cor:1}, in either case $q>0$ or $q=0$, yields
\[
 A\le C\Big(W^2+\frac W\ell+\frac{T+A}{\ell}
                 +A^2+\Lam ^{-1}\Big).
\]
For $q=0$ we have used
$\sum_i\lam _i^{-(4\al _i-2)}=O(\Lam ^{-3/2})$.
Substitute \eqref{eq:coarse-fiber-direct} and the estimate for $T$.
Since $A=o(1)$ and $\ell\to\infty$, the coefficients of $A$ on the right tend to zero. Moving these terms to the left yields
\begin{equation}\label{eq:coarse-parameter-smallness}
 A+W+T=O(\Lam ^{-1}).
\end{equation}

The definition \eqref{def:tau_i} and
\eqref{eq:coarse-parameter-smallness} imply
$\tau_i=o(1)$ and
\begin{equation}\label{eq:coarse-scale-ratio-log}
 (4\al _i-2)\ln\lam _i
 -(4\al _j-2)\ln\lam _j=O(1).
\end{equation}
Choose \(j\) with \(\lam _j=\Lam \). Equation~\eqref{eq:coarse-scale-ratio-log}
implies \(\ln\lam _i=O(\ell)\), and hence
\(A\max_i\ln\lam _i=O(\ell\Lam ^{-1})=o(1)\).
Equation~\eqref{eq:coarse-scale-ratio-log} now implies
\(\lam _i\asymp\Lam \). Consequently,
\[
 M:=\int_\Si  Ve^u\,\d v_g\asymp\Lam ^2,
 \quad
 \frac{\lam _i}{\lam _1}
 =\sqrt{\frac{\mathcal F_\xi^1(\xi_1)}{
                  \mathcal F_\xi^i(\xi_i)}}+o(1).
\]

We can now apply the common-scale identity in
Corollary~\ref{cor:Projbubble_sumexpansion}. For an
exact solution, it yields $\mu=O(\Lam ^{-1})$ (and
$\mu=O(\ell\Lam ^{-2})$ when $q=0$).
Consequently $\sum_i|\tau_i|=O(\Lam ^{-1})$.
Multiplying \eqref{eq:gradient_pa_xi_i} in Lemma~\ref{lem:bubble_expansion3}
by $\lam _i$ yields
\(|\nabla\ff(\xi)|=O(\ell\Lam ^{-1}).\)
Thus $\xi^*$ is a critical point. The preceding mass, scale, and gradient
estimates hold independently of \ref{item:C1}--\ref{item:C2}.
For the quantitative location estimate, \ref{item:C1} and the inverse
function theorem for $\nabla\ff$ near $\xi^*$ imply
\(\sum_i d_g(\xi_i,\xi_i^*)=O(\ell\Lam ^{-1}).\)
The quantitative location estimate in the proposition follows.

The parameter asymptotics follow from Corollary~\ref{cor:Projbubble_sumexpansion}, with the direction defined in \eqref{eq:Z-lam}. Neither the quantitative location estimate nor \ref{item:C1}--\ref{item:C2} is needed here. If $q>0$, all quadratic errors are $O(\Lam ^{-2})=o(\Lam ^{-1})$, and
\[
 \sum_i\mathfrak f_2(\xi_i)\lam _i^{4\al _i-3}
 =\mathcal L_1(\xi^*)
   \sqrt{\mathcal F_{\xi^*}^1(\xi_1^*)}\lam _1+o(\Lam ).
\]
Here the Neumann conditions imply
$\pa _{\nu_g}\ln(\mathcal F_\xi^i\mathrm g_\xi^i)
=\pa _{\nu_g}\ln V$ at a boundary center. Hence
\[
 \mu=\frac\pi4\,
 \frac{\mathcal L_1(\xi^*)
       \sqrt{\mathcal F_{\xi^*}^1(\xi_1^*)}\lam _1}{M}
       +o\Big(\frac{1}{\lam _1}\Big).
\]
If $q=0$, the boundary-linear term vanishes. Moreover,
\[
 \sum_i\kap _i\mathfrak f_3(\xi_i)\ln\lam _i
 =\mathcal L_2(\xi^*)\ln\lam _1+o(\ell),
\]
and \eqref{eq:coarse-parameter-smallness} yields
$A^2+W^2=O(\Lam ^{-2})=o(\ell\Lam ^{-2})$. Substituting the
$\mathfrak f_3$ expansion and the quadratic error estimate into
Corollary~\ref{cor:Projbubble_sumexpansion} yields
\[
 \mu=\frac1{16}\,
 \frac{\mathcal L_2(\xi^*)\ln\lam _1}{M}
       +o\Big(\frac{\ln\lam _1}{\lam _1^{2}}\Big).
\]
The parameter estimates have thus been proved for the projection defined in \eqref{projectedbubble}.
\end{proof}

\begin{rem}\label{rem:fixed-resonance-compactness}\upshape
Proposition~\ref{prop:estimates_rho_k} shows that \ref{item:C2}, without
\ref{item:C1}, for all types at
$\rho_*=4\pi m$, implies compactness of the exact solution set at
$\rho=\rho_*$. Indeed, an unbounded sequence has a
limiting reduced critical configuration $\xi^*$ and
$M\asymp\lam _1^2$. Since $\mu=0$, multiplying the parameter
asymptotic by $\lam _1$ when $q>0$, or by
$\lam _1^2/\ln\lam _1$ when $q=0$, yields
$\mathcal L_1(\xi^*)=0$ or $\mathcal L_2(\xi^*)=0$, respectively:
the ratio $\lam _1^2/M$ is bounded above and below by positive
constants, contradicting \ref{item:C2}. For $\|u\|\le R$ and $1<p<\infty$, \eqref{eq:MT-Jensen-explicit} implies
\[
 \Big\|\frac{Ve^u}{\int_\Si Ve^u\,\d v_g}\Big\|_{L^p}
 \le\frac{\max_\Si V}{(\min_\Si V)|\Si|_g}
 C_{\Si,g}^{1/p}e^{pR^2/(8\pi)}.
\]
The mean-zero Neumann estimate first yields a uniform $W^{2,p}$
estimate. Schauder regularity then proves compactness in
$C^{2,\gamma}(\Si)$ for every $\gamma\in(0,1)$.

The compactness argument applies to exact solutions at the fixed
parameter. A Palais--Smale condition and the count of blow-up solutions
associated with degenerate reduced critical points do not follow.
\end{rem}

	\subsection{Finite-dimensional reduction}\label{subsec:finite-dimensional-reduction}
	We now consider the orthogonal remainder \(w\). The first step is to
	expand \(J_\rho\) in this variable and establish uniform coercivity of
	the corresponding quadratic form; compare
\cite[Propositions~3.5--3.6]{AHW2025-1}.

	\begin{prop}\label{prop:J-u-expansion}
		Let $u=u_0+w \in V^\eta_{\rho}(p,q,\var)$ with  $\rho >0$, where $u_0=\sum_{i} \al_i P \delta_{\lam_i,\xi_i}$ and  $w\in E_{\lam,\xi}^{p,q}$. Then
		\begin{equation*}
			J_{\rho }(u)=J_{\rho }(u_0)+f(w)+\frac{1}{2}Q (w,w)+O(\var\|w\|^2),
		\end{equation*}
		where  the linear term  $f(w)$ is given by
		\[
		f(w):=-\rho  \frac{\int_{\Si } V e^{u_0} w \,\d v_g}{\int_{\Si } V e^{u_0}\,\d v_g}
		\]
		and the bilinear form $Q(\cdot, \cdot)$  is defined as
		\[
		Q (w, w'):=\la w, w'\ra -\rho  \frac{\int_{\Si } V e^{u_0} w w' \,\d v_g}{\int_{\Si } V e^{u_0} \,\d v_g}+\rho \Big(\frac{\int_{\Si } V e^{u_0} w  \,\d v_g }{\int_{\Si } V e^{u_0} \,\d v_g }\Big)\Big(\frac{\int_{\Si } V e^{u_0} w' \,\d v_g }{\int_{\Si } V e^{u_0} \,\d v_g }\Big)
		\]
		for any $w,w'\in \mathcal E_{\lam,\xi}^{p,q}$.
	\end{prop}
	The quadratic form is uniformly positive on $\mathcal E_{\lam,\xi}^{p,q}$.
	\begin{prop}\label{prop:Q_0positive}
		Under the same assumptions as in Proposition \ref{prop:J-u-expansion},  there exists  a constant $c_0>0$ such that
		\[
		Q (w, w)= Q _0(w,w) +o(\|w\|^2)\geq c_0 \|w\|^2, \quad \forall\, w \in \mathcal E_{\lam,\xi}^{p,q},
		\]
		where the bilinear form $Q _0(\cdot,\cdot)$ is given by
		\[
		Q _0(w,w'):= \la w, w'\ra  - \sum_{i} \int_{\Si } \vartheta_{\xi_i} e^{\delta_{\lam_i,\xi_i}-\varphi_{\xi_i}} ww' \,\d v_g,\quad \forall\, w,w' \in \mathcal E_{\lam,\xi}^{p,q}.
		\]
		In particular, $Q _0(\cdot,\cdot)$ is symmetric and positive definite.
	\end{prop}
\begin{proof}
With $M_i$ as in the proof of Lemma~\ref{lem:sum_tau_i}, the leading
mass expansion and the scale equations imply
$\rho M_i/\int_\Si Ve^{u_0}\,\d v_g=\kap(\xi_i)+o(1)$.
The weighted estimate \eqref{eq:A-uniform-weighted-Sobolev} and
amplitude orthogonality therefore imply, uniformly for
$z\in\mathcal E_{\lam,\xi}^{p,q}$,
\[
 \frac{\rho\int_\Si Ve^{u_0}z^2\,\d v_g}
      {\int_\Si Ve^{u_0}\,\d v_g}
 =\sum_i\int_\Si z^2\,\d\nu_i+o(\|z\|^2),\quad
 \frac{\int_\Si Ve^{u_0}z\,\d v_g}
      {\int_\Si Ve^{u_0}\,\d v_g}=o(\|z\|).
\]
The estimate for $Q_0$ in Step~1 of the proof of Lemma~\ref{lem:terminal-reduction}
proves the result. That spectral argument does not require comparable
scales or the solution of the finite-dimensional equations.
\end{proof}

	As a consequence, for fixed bubbling parameters the functional can be
	uniquely minimized with respect to the orthogonal remainder.
	\begin{cor}\label{cor:mini_w}
		Let
		$
		\theta=(\al ,\xi,\lam )\in\widetilde B_{2\var}^{\eta},
		\,
		u_0(\theta):=
		\sum_i\al _iP\delta_{\lam _i,\xi_i},
		$
		and assume that \(u_0(\theta)\in V_\rho^\eta(p,q,\var)\).
		For \(\var>0\) sufficiently small, the restricted functional
		\[
		w\mapsto J_\rho(u_0(\theta)+w),
		\quad
		w\in E_{\lam ,\xi}^{p,q},
		\]
		has a unique critical point \(\bar w=\bar w(\theta)\) in a fixed
		neighborhood of the origin, and it is the unique minimizer in that
		neighborhood. Moreover, the map
		$
		\theta\mapsto\bar w(\theta)
		$
		is of class \(C^1\), and
		\begin{equation}\label{eq:barw_error}
			\|\bar w\|\le C\mathscr E,
		\end{equation}
		where
		\begin{equation}
			\label{def:scr-E} \mathscr E:=
			\sum_i\Big(
			|\al _i-1|
			+
			\frac{
				|\nabla_x\mathcal F_\xi^i(\xi_i)|
				+
				|\mathfrak f_2(\xi_i)|
			}{\lam _i}
			+
			\frac{\ln ^{3/2}\lam _i}{\lam _i^2}
			\Big).
		\end{equation}
		Here \(\mathcal F_\xi^i\) and \(\mathfrak f_2\) are defined in
		\eqref{eq:def_F} and \eqref{def:frak_f2}, respectively.
	\end{cor}
	\begin{proof}
Write $E_\theta=\mathcal E_{\lam,\xi}^{p,q}$. Proposition~\ref{prop:J-u-expansion}
and Lemma~\ref{lem:A.7} imply
\begin{equation}\label{eq:fiber-expansion-for-wbar}
 J_\rho(u_0+w)-J_\rho(u_0)=f(w)+\tfrac12Q(w,w)+R(w),
\end{equation}
with
\begin{equation}\label{eq:f-estimate-wbar}
 |f(z)|\le C\mathscr E\|z\|,\quad z\in E_\theta.
\end{equation}
Indeed, for $\Lam=\min_i\lam_i$ the additional normalized terms
in Lemma~\ref{lem:A.7} are at most
$C\Lam^{-1}\sum_j|\al_j-1|+C\Lam^{-1}\sum_j\lam_j^{-2}$,
which are included in \eqref{def:scr-E}. No scale comparison is needed.
Proposition~\ref{prop:Q_0positive} states
\begin{equation}\label{eq:Q-coercivity-wbar}
 Q(z,z)\ge c_0\|z\|^2,\quad z\in E_\theta.
\end{equation}

For $\|w\|\le\epsilon_0$, set
$\d\mu_w=Ve^{u_0+w}\,\d v_g/\int_\Si Ve^{u_0+w}\,\d v_g$.
The weighted estimate \eqref{eq:A-uniform-weighted-Sobolev} implies
$\int_\Si|z|^r\,\d\mu_0\le C_r\|z\|^r$. Jensen's inequality yields
\[
 \frac{\int_\Si Ve^{u_0+w}\,\d v_g}{\int_\Si Ve^{u_0}\,\d v_g}
 =\int_\Si e^w\,\d\mu_0
 \ge\exp\Big(\int_\Si w\,\d\mu_0\Big)
 \ge e^{-C\|w\|}.
\]
The weighted H\"older inequality therefore yields the same integral estimates for $\mu_w$.
To estimate the remainder in \eqref{eq:fiber-expansion-for-wbar}, the third derivative of the logarithmic integral is
\[
 D^3\ln\Big(\int_\Si Ve^{u_0+w}\,\d v_g\Big)[z_1,z_2,z_3]
 =\int_\Si\prod_{j=1}^3\Big(z_j-\int_\Si z_j\,\d\mu_w\Big)\,\d\mu_w,
\]
whose absolute value is at most $C\prod_j\|z_j\|$. Taylor's formula implies
\begin{equation}\label{eq:R-derivative-estimates-wbar}
 |DR(w)[z]|\le C\|w\|^2\|z\|,\quad
 |D^2R(w)[z,z']|\le C\|w\|\|z\|\|z'\|.
\end{equation}
Combining \eqref{eq:Q-coercivity-wbar} and
\eqref{eq:R-derivative-estimates-wbar}, for small $\epsilon_0$ we obtain
\begin{equation}\label{eq:uniform-Hessian}
 D_w^2J_\rho(u_0+w)[z,z]\ge\tfrac12c_0\|z\|^2.
\end{equation}

The projected Euler equation
\begin{equation*}
 DJ_\rho(u_0+\bar w)[z]=0\quad(z\in E_\theta)
\end{equation*}
is equivalently
\begin{equation}\label{eq:projected-Euler}
 f(z)+Q(\bar w,z)+DR(\bar w)[z]=0\quad(z\in E_\theta).
\end{equation}
The inverse of $Q$ and \eqref{eq:R-derivative-estimates-wbar}
define a contraction on a ball of radius $C\mathscr E$.
Strict convexity \eqref{eq:uniform-Hessian} proves uniqueness in the
fixed ball and the minimizing property. Testing
\eqref{eq:projected-Euler} with $\bar w$ and applying
\eqref{eq:f-estimate-wbar} proves \eqref{eq:barw_error}.

For the parameter dependence, fix $\theta^*$ and set $E^*=E_{\theta^*}$.
Lemma~\ref{lem:Gram} implies that the orthogonal projections
$\Pi_\theta^\perp$ depend $C^1$ on the parameters. Locally,
$T_\theta=\Pi_\theta^\perp|_{E^*}:E^*\to E_\theta$ is an isomorphism.
On $E^*$ define
\[
 \langle\Phi(\theta,v),z\rangle
 =DJ_\rho(u_0(\theta)+T_\theta v)[T_\theta z].
\]
At a zero, $D_v\Phi$ is invertible by \eqref{eq:uniform-Hessian}.
The implicit function theorem proves $C^1$ dependence of
$\bar w=T_\theta\bar v$; uniqueness identifies the local maps on overlaps.
\end{proof}

	Next, we introduce local neighborhoods
	that isolate a bubbling configuration near a fixed critical point
	\(\xi^*\) of \(\FpqV\).  Set
	$ \mu=(\rho-\rho_*)/\rho_*$.  If $ q\ne0$, define $ V(\mu,\xi^*)$  as the set
	of all
	\[
	u=\sum_i\al_iP\delta_{\lam_i,\xi_i}+w\in\oH
	\]
	in $ V_{\rho_*}^{\eta}(p,q,C_0|\mu|^{1/2})$, with
	$ w\in E_{\lam,\xi}^{p,q}$, and satisfying
	\begin{equation}
		\label{eq:condi-V-1}
 \begin{gathered}
 \|w\|<C|\mu|,\quad
 \sum_i d_g(\xi_i,\xi_i^*)<C|\mu|\ln\frac1{|\mu|},\\
 \sum_i(|\alpha_i-1|+|\tau_i|)<C|\mu|,\quad
 \frac{|\mu|}{C}<\frac1{\lambda_i}<C|\mu|.
 \end{gathered}
 \end{equation}
	for $i=1,\ldots,p+q.$
	If $ q=0$, define $ V(\mu,\xi^*)$  as the set of all
	\[ u=\sum_i \al _i P\delta_{\lam _i ,\xi_i }+w\in \oH\]
	in $V_{\rho_*}^{\eta}(p,q,C_0|\mu|^{1/3}), $
	\begin{equation}  	\label{eq:condi-V-2}
		\|w\|<C|\mu|^{\frac 3 4}, \quad
		\sum_i d_g(\xi_i,\xi_i^*)<C|\mu|^{\frac3 8},\quad 	\sum_i(|\al_i-1|+|\tau_i|)<C|\mu|,
		\quad
		\frac{|\mu|}{C}<\frac{\ln \lam_i}{\lam_i^2}<C|\mu|,
	\end{equation}
	for $i=1,\ldots,p.$
	Here $C_0,C>0$ are fixed large constants. The scale and amplitude
conditions in these definitions imply
$\sum_i|\al _i-1|\ln\lam _i=o(1)$ as $\mu\to0$.

For \(q=0\), we refine the estimates in
Proposition~\ref{prop:estimates_rho_k} to verify \eqref{eq:condi-V-2}.

\begin{lem}\label{lem:interior-isolating-bootstrap}
Let $p\ge1$, $\rho_*=8\pi p$, and let $u_k$ be a sequence of blowing-up
solutions of \eqref{maineq} with
$\rho_k=(1+\mu_k)\rho_*\to\rho_*$. Suppose that the limiting
configuration $\xi^*\in\Xi_{p,0}$ is a nondegenerate critical point of
\(\Fv{p}{0}\) and that \(\cL_2(\xi^*)\ne0\). For the unique
decomposition from Propositions~\ref{prop:u-k-in-V-eta}
and~\ref{prop:unique}, set
\(\Lam _k:=\min_{1\le i\le p}\lam _i^k, \quad \ell_k:=\ln \Lam _k.\)
Then the scales are comparable, $|\mu_k|\asymp\ell_k\Lam _k^{-2}$,
and
\begin{equation}\label{eq:interior-bootstrap-bounds}
\begin{gathered}
\sum_i|\al _i^k-1|+\sum_i|\tau_i^k|
 =O(|\mu_k|),\quad
\|w^k\|=O(|\mu_k|\sqrt{\ell_k}),\\
\sum_i d_g(\xi_i^k,\xi_i^*)=O(\ell_k\Lam _k^{-1}),\quad
\operatorname{sgn}\mu_k=\operatorname{sgn}\cL_2(\xi^*)
\quad\text{for all sufficiently large }k.
\end{gathered}
\end{equation}
Consequently, for fixed sufficiently large constants $C_0$ and $C$ in
\eqref{eq:condi-V-2}, one has
$u_k\in V(\mu_k,\xi^*)$ for all sufficiently large $k$.
\end{lem}

\begin{proof}
Omit the sequence index and set
\[
 \begin{gathered}
 \Lam:=\min_i\lam_i,\quad \ell:=\ln\Lam,\quad
 A:=\sum_i|\al_i-1|,\quad W:=\|w\|,\\
 T:=\sum_i|\tau_i-\mu+\mu\tau_i|,\quad
 D:=\sum_i d_g(\xi_i,\xi_i^*).
 \end{gathered}
\]
Proposition~\ref{prop:estimates_rho_k} implies
$A+W+T=O(\Lam^{-1})$, $\lam_i\asymp\Lam$, and
$|\nabla\Fv{p}{0}(\xi)|=O(\ell\Lam^{-1})$.
Nondegeneracy of $\xi^*$ implies $D=O(\ell\Lam^{-1})$ and
\begin{equation}\label{eq:interior-bootstrap-log-balance}
 \sum_i|\al_i-1|\ln\lam_i=O(\ell\Lam^{-1})=o(1).
\end{equation}
With $u_0=\sum_i\al_iP\delta_{\lam_i,\xi_i}$, the mass expansion
and the parameter asymptotic in Proposition~\ref{prop:estimates_rho_k}
therefore imply
\[
 \int_\Si Ve^u\,\d v_g\asymp\int_\Si Ve^{u_0}\,\d v_g\asymp\Lam^2,
 \quad |\mu|\asymp\ell\Lam^{-2},\quad
 \sgn\mu=\sgn\cL_2(\xi^*).
\]
The scale equations in Lemma~\ref{lem:lamexpansion}, whose boundary
term vanishes for $q=0$, imply
$T\le C(A^2+W^2+\ell\Lam^{-2})=O(\ell\Lam^{-2})$.
For sufficiently large $\ell$, Corollary~\ref{cor:1},
\eqref{eq:interior-bootstrap-log-balance}, and the coarse estimates
for $A,W$ yield
\begin{equation}\label{eq:interior-bootstrap-amplitude}
 A\le C\bigl(W/\ell+\ell\Lam^{-2}\bigr).
\end{equation}

The weighted Taylor estimate used in
\eqref{eq:coarse-fiber-direct} applies to the exact orthogonal equation:
\[
 0=DJ_\rho(u_0+w)[w]=f(w)+Q(w,w)+O(W^3).
\]
In the admissible interior coordinates,
\[
 \nabla_x\cF_\xi^i(\xi_i)
 =\frac{\cF_\xi^i(\xi_i)}{2\kap(\xi_i)}
       \nabla_{\xi_i}\Fv{p}{0}(\xi),
 \quad |\nabla_x\cF_\xi^i(\xi_i)|\le CD.
\]
Lemma~\ref{lem:A.7} and Proposition~\ref{prop:Q_0positive}
then imply, for sufficiently small $W$,
\begin{equation}\label{eq:interior-bootstrap-fiber}
 W\le C\Big(A+\frac D\Lam+\frac{\ell^{3/2}}{\Lam^2}\Big)
 \le C\Big(A+\frac{\ell^{3/2}}{\Lam^2}\Big).
\end{equation}
Substitution of \eqref{eq:interior-bootstrap-fiber} into
\eqref{eq:interior-bootstrap-amplitude}, followed
by absorption of $CA/\ell$, proves
\[
 A=O(\ell\Lam^{-2}),\quad W=O(\ell^{3/2}\Lam^{-2})
   =O(|\mu|\sqrt\ell).
\]
The identity
$\tau_i=(\mu+(\tau_i-\mu+\mu\tau_i))/(1+\mu)$ implies
$\sum_i|\tau_i|=O(|\mu|)$, completing
\eqref{eq:interior-bootstrap-bounds}.

Finally,
 \begin{gather*}
 W=o(|\mu|^{3/4}),\quad D=o(|\mu|^{3/8}),\quad
 \lam_i|\mu|^{1/3}\to\infty,\\
 \Big\|u-\sum_iP\delta_{\lam_i,\xi_i} \Big\|
 \le CA\sqrt\ell+W=o(|\mu|^{1/3}).
 \end{gather*}
Since $T_\rho(u)=u$ and $T_{\rho_*}=(\rho_*/\rho)T_\rho$,
\begin{equation}\label{eq:interior-bootstrap-gradient}
 \nabla J_{\rho_*}(u)=\frac\mu{1+\mu}u,\quad
 \|\nabla J_{\rho_*}(u)\|=O(|\mu|\sqrt\ell)=o(|\mu|^{1/3}).
\end{equation}
The separation of the centers and the preceding estimates verify
\eqref{def:Vm} and \eqref{eq:condi-V-2}.
\end{proof}

We next specify the neighborhoods in the blow-up alternative of \cite[Theorem~4.1]{AHW2025-1}.
\begin{thm}\label{thm:alter_blowup_sols}
		Let $\rho_*=4\pi m$ with $m\in  \N$,  and let  $V$ be a  positive  $C^{3}$-function on $\Si$    satisfying conditions  \ref{item:C1}--\ref{item:C2} for all $p,q\in \N_0$ such that $m=2p+q$. There exist $\mu_m>0$ and $\bar C_m>0$ such that  any solution $u$ of \eqref{maineq}   with $u\in\oH$ and   $\frac{|\rho-\rho_*|}{\rho_*} \leq  \mu_{m}$ satisfies the following:
		\begin{itemize}
			\item[(1)] If $\rho  =\rho_*$, then $\|u\|< \bar{C}_{m}$;
			\item[(2)] If $\rho=(1+\mu)\rho_*$ with $0<|\mu|\leq\mu_m$, then either
			\begin{itemize}
				\item[i)] $\|u\| < \bar{C}_{m}$ or
				\item[ii)] $\|u\| \geq  \bar{C}_{m}$  and  $u\in V(\mu ,\xi^*)$ for some critical point $\xi^*:=(\xi_1^*,\ldots,\xi^*_{p+q})$ of $\ff$.
			\end{itemize}
		\end{itemize}
	\end{thm}

	\begin{proof}
Remark~\ref{rem:fixed-resonance-compactness} proves \textup{(1)}.
Choose $\bar C_m$ strictly larger than the norms of all solutions at
$\rho_*$ and decrease $\mu_m$ so that the punctured interval contains
no other multiple of $4\pi$.
If \textup{(2)} fails for every such $\mu_m$, there are solutions
$u_k$ at $\rho_k=(1+\mu_k)\rho_*$, with $0<|\mu_k|\to0$ and
$\|u_k\|\ge\bar C_m$, such that
\begin{equation}\label{eq:not-in-bubbling-neighborhood}
 u_k\notin\bigcup_{2p+q=m}\ \bigcup_{\xi^*\in\cri(\FpqV)}
                 V(\mu_k,\xi^*).
\end{equation}
A bounded subsequence converges strongly to a solution at $\rho_*$,
contrary to the choice of $\bar C_m$.
For an unbounded subsequence, Propositions~\ref{prop:u-k-in-V-eta}
and~\ref{prop:estimates_rho_k} identify a type $(p,q)$ and a reduced
critical configuration $\xi^*$, with
\[
 |\mu_k|\asymp
\left\{\begin{aligned}
    &(\lam_1^k)^{-1}&& \text{ if }\,q>0,\\
 &(\lam_1^k)^{-2}\ln\lam_1^k&&\text{ if }\,q=0, 
\end{aligned}\right.
 \quad \sgn\mu_k=\sgn\cL(\xi^*).
\]
For $q=0$, Lemma~\ref{lem:interior-isolating-bootstrap} contradicts
\eqref{eq:not-in-bubbling-neighborhood}.
For $q>0$, Proposition~\ref{prop:estimates_rho_k} implies
\[
 \|w^k\|+\sum_i(|\al_i^k-1|+|\tau_i^k|)=O(|\mu_k|),\quad
 \sum_i d_g(\xi_i^k,\xi_i^*)=O(|\mu_k|\ln(1/|\mu_k|)).
\]
The distance from $u_k$ to its unit-amplitude bubble sum and
$\|\nabla J_{\rho_*}(u_k)\|$ are both
$$O(|\mu_k|\sqrt{\ln(1/|\mu_k|)})=o(|\mu_k|^{1/2});$$
for the gradient, apply \eqref{eq:interior-bootstrap-gradient}.
Also $\lam_i^k|\mu_k|^{1/2}\to\infty$.
Thus \eqref{def:Vm} and \eqref{eq:condi-V-1} hold, again
contradicting \eqref{eq:not-in-bubbling-neighborhood}.
\end{proof}

	\subsection{Homology of boundary-weighted barycenter spaces}
\label{sec:barycenter-homology}

We recall the homological formulas of \cite[Section~7]{AK2026} for the
spaces defined in \eqref{eq:barycenter-stratum}--\eqref{eq:weighted-barycenter-space}.
All homology groups in this subsection have coefficients in
\(\mathbb Z_2\), and all Poincar\'e polynomials have integer coefficients.
The comparison with the sublevels of \(J_\rho\) is proved in
Appendix~\ref{app:topology}.

Let \(\mathfrak g\) be the genus of \(\Si\), and let \(b\) be the number
of its boundary components, as in the introduction. Label these
components \(\Ga _1,\ldots,\Ga _b\), and write \(\sqcup\) for disjoint
union and \(\mathbb{S}^d\) for the \(d\)-dimensional sphere. Set
\[
\pa\Si=\Ga _1\sqcup\cdots\sqcup\Ga _b,
\quad \Ga _a\simeq \mathbb{S}^1\quad(1\le a\le b),
\quad N:=2\mathfrak g+b-1,
\]
In \(\Si/\pa\Si\), the entire boundary is identified with one point.
Let \(\Si _{\mathfrak g}\) be the closed orientable surface of genus
\(\mathfrak g\). We write \(\vee\) for the wedge sum of pointed spaces
and \(\operatorname{Susp}\) for reduced suspension; an empty wedge is
a point. Then
\begin{equation}\label{eq:topology-quotient}
\Si/\pa\Si\simeq\Si _{\mathfrak g}\vee\bigvee_{a=1}^{b-1}\mathbb{S}^1,
\quad
\operatorname{Susp}(\Si/\pa\Si)\simeq \mathbb{S}^3\vee\bigvee_{a=1}^{N}\mathbb{S}^2.
\end{equation}
The notation \(B_n(\Si/\pa\Si)\) is the genuine barycenter
space of this quotient, with the identifications fixed in the introduction.

For \(n\ge1\), denote the reduced Poincar\'e polynomials of
\(B_n(\pa\Si)\) and \(B_n(\Si/\pa\Si)\) by \(R_{b,n}\) and
\(Q^{(2)}_{N,n}\), respectively; the superscript \((2)\) indicates the
coefficient field. Thus
\[
R_{b,n}(t):=\sum_{r\ge0}\dim_{\mathbb Z_2}
 \widetilde H_r(B_n(\pa\Si);\mathbb Z_2)t^r,
 \]
 and
 \[
Q^{(2)}_{N,n}(t):=\sum_{r\ge0}\dim_{\mathbb Z_2}
 \widetilde H_r(B_n(\Si/\pa\Si);\mathbb Z_2)t^r.
 \]
For a formal power series in \(x\), \([x^n]\) denotes the coefficient of \(x^n\). The mod-two calculation in \cite[Lemma~7.2]{AK2026},
applied to \eqref{eq:topology-quotient}, yields
\begin{equation}\label{eq:topology-Q-explicit}
Q^{(2)}_{N,n}(t)
=t^{-1}[x^n]\frac{1}{(1-xt^2)^N}
 \prod_{r\ge0}\frac{1}{1-x^{2^r}t^{2^{r+1}+1}}.
\end{equation}
Only factors with \(2^r\le n\) contribute in
\eqref{eq:topology-Q-explicit}. In particular,
\[
Q^{(2)}_{N,1}(t)=Nt+t^2,
\quad
Q^{(2)}_{N,2}(t)=\binom{N+1}{2}t^3+(N+1)t^4+t^5.
\]
For the boundary, \cite[Proposition~7.5]{AK2026} yields
\begin{equation}\label{eq:topology-R-explicit}
R_{b,n}(t)=\sum_{c=0}^{\min\{b-1,n\}}
 \binom{b-1}{c}\binom{b+n-c-1}{n-c}t^{2n-c-1}.
\end{equation}
Thus \eqref{eq:topology-R-explicit} yields
\(R_{1,n}(t)=t^{2n-1}\) and \(R_{b,1}(t)=(b-1)+bt\).

For nonempty spaces \(X,Y\), let \(X*Y\) denote their topological join.
The symbols $\oplus$ and $\otimes_{\mathbb Z_2}$ denote the direct sum and
tensor product over $\mathbb Z_2$; $\cong$ denotes an isomorphism.
The reduced homology of the join satisfies
\begin{equation}\label{eq:topology-join-homology}
\widetilde H_k(X*Y;\mathbb Z_2)
\cong\bigoplus_{r+s=k-1}
\widetilde H_r(X;\mathbb Z_2)
 \otimes_{\mathbb Z_2}\widetilde H_s(Y;\mathbb Z_2).
\end{equation}
Omitting the coefficient field from the notation, the homology
decomposition in \cite[Theorem~1.2 and Corollary~5.5]{AK2026} yields, for \(s\ge1\),
\begin{equation}\label{eq:topology-even-homology}
\widetilde H_k(B_{2s}^{\pa}(\Si,\pa\Si))
\cong\widetilde H_k(B_{2s}(\pa\Si))
 \oplus\widetilde H_k(B_s(\Si/\pa\Si))\oplus\bigoplus_{i=1}^{s-1}
 \widetilde H_k\Big(B_i(\Si/\pa\Si)*B_{2s-2i}(\pa\Si)\Big).
\end{equation}
For odd weights,
\begin{equation}\label{eq:topology-odd-homology}
\widetilde H_k(B_{2s-1}^{\pa}(\Si,\pa\Si))
\cong\widetilde H_k(B_{2s-1}(\pa\Si))\oplus\bigoplus_{i=1}^{s-1}
 \widetilde H_k\Big(B_i(\Si/\pa\Si)*B_{2s-2i-1}(\pa\Si)\Big).
\end{equation}
The decompositions \eqref{eq:topology-even-homology}--\eqref{eq:topology-odd-homology}
follow from the triangular rank filtration in
\cite[Theorem~1.2 and Corollary~5.5]{AK2026}; only the topological results of that reference are invoked, not its
application to the present equation.

Taking dimensions in
\eqref{eq:topology-even-homology}--\eqref{eq:topology-odd-homology} and
using \eqref{eq:topology-join-homology}, we obtain
\begin{equation}\label{eq:topology-P-even}
\widetilde P_{2s}^{\pa}(t)=R_{b,2s}(t)+Q^{(2)}_{N,s}(t)
 +t\sum_{i=1}^{s-1}Q^{(2)}_{N,i}(t)R_{b,2s-2i}(t),
\end{equation}
and
 \begin{equation}\label{eq:topology-P-odd}
\widetilde P_{2s-1}^{\pa}(t)
=R_{b,2s-1}(t)
 +t\sum_{i=1}^{s-1}Q^{(2)}_{N,i}(t)R_{b,2s-2i-1}(t).
\end{equation}
The sums are empty when \(s=1\). Equations
\eqref{eq:topology-Q-explicit}--\eqref{eq:topology-R-explicit} and
\eqref{eq:topology-P-even}--\eqref{eq:topology-P-odd} determine all the
numbers \(\widetilde\beta_l^r\) introduced in Subsection~\ref{sec:main-results}. For genuine
homology, the degree-zero dimension is \(\widetilde\beta_l^0+1\), and
the positive-degree dimensions are \(\widetilde\beta_l^r\).

The highest-degree term in
\eqref{eq:topology-P-even}--\eqref{eq:topology-P-odd} comes from
\(R_{b,l}\). By \eqref{eq:topology-R-explicit},
\begin{equation}\label{eq:topology-top-betti}
\widetilde\beta_l^{2l-1}=\binom{b+l-1}{l}>0,
\quad \widetilde\beta_l^r=0\quad(r>2l-1).
\end{equation}
Proposition~\ref{prop:negative-sublevels-barycenters} in
Appendix~\ref{app:topology} justifies
\eqref{def:Betti}--\eqref{def:P_m}.

For a polynomial in \(t\), write \([t^{k-1}]\) for its coefficient
of \(t^{k-1}\). By \eqref{def:Betti}--\eqref{def:P_m}, the relative
Betti numbers in the Morse relations are therefore determined by
\[
\mathfrak b_{m,k}=[t^{k-1}]\widetilde P_{m-1}^{\pa}(t),\quad
\forall\,k\ge1,
\quad \mathfrak b_{m,0}=0.
\]
Taking \(l=m-1\) in \eqref{eq:topology-top-betti} yields the coefficient
in Corollary~\ref{cor:1.3}. Reduced homology also accounts for the
contribution of a disconnected boundary: \(B_1^{\pa}(\Si,\pa\Si)=\pa\Si\)
and \(R_{b,1}(t)=(b-1)+bt\) yield
\(\mathfrak b_{2,1}=\widetilde\beta_1^0=b-1\).

\begin{rem}\upshape
For the Euler characteristic, adopt the conventions
\(B_0(X):=\emptyset\) and \(\chi(B_0(X)):=0\).
Then \cite[Remark~7.8]{AK2026} yields
\begin{equation}\label{eq:topology-euler-characteristic}
1-\chi(B_{2s}^{\pa}(\Si,\pa\Si))=1-\chi(B_s(\Si)),\quad
1-\chi(B_{2s-1}^{\pa}(\Si,\pa\Si))=1-\chi(B_{s-1}(\Si)),
\quad s\ge1.
\end{equation}
Together with \eqref{eq:topology-euler-characteristic}, the genuine barycenter formula
below expresses the Euler--Poincar\'e quantity in Theorem~\ref{thm:2}
in terms of \(\chi(\Si)\):
\[
\chi(B_r(\Si))
=1-\frac1{r!}\prod_{j=1}^{r}(j-\chi(\Si)),
\quad r\ge1.
\]
\end{rem}

\section{Bubbling solutions and Morse indices}\label{sec:3}

Fix $p,q\in\mathbb N_0$ with $m:=2p+q\ge1$, set $\rho_*=4\pi m$,
and write $\rho=(1+\mu)\rho_*$ for $|\mu|$ sufficiently small.
We compare the tangent spaces of the bubble manifold and the graph of
the minimizing correction, then apply the parameter-dependent reduction
to compute the Morse indices of nearby blow-up solutions. The branch indices are used only for prescribed-index existence.
The resonant Morse relations and degree in Section~\ref{sec:4} use
the fixed-functional filtration of Section~\ref{sec:full-pseudogradient}.

For $\theta=(\al ,\xi,\lam )$, write
$u_0=\psi_\theta=\sum_i\al _iP\delta_{\lam _i,\xi_i}$,
$\mathscr T_\theta=\operatorname{range}D\psi_\theta$, and
$\kap _i:=\kap (\xi_i)$, which is constant on the chosen stratum.
For the local comparison, the centers range in a fixed compact subset
of $\Xi_{p,q}$, the amplitudes remain in a fixed small neighborhood
of $\mathbf1$, and the scales are sufficiently large.
Let $\bar w_\theta\in E_{\lam ,\xi}^{p,q}$ be a $C^1$ family,
and let $e_a$ be the normalized tangent frame of
Lemma~\ref{lem:Gram} in Appendix~\ref{app:0}. Differentiating
$\langle\bar w,e_a\rangle=0$ in a parameter direction $\dot\theta$ yields
\[
\langle D\bar w[\dot\theta],e_a\rangle
=-\langle\bar w,De_a[\dot\theta]\rangle.
\]
Lemma~\ref{lem:Gram} implies
$\|De_a[\dot\theta]\|\le C\|D\psi[\dot\theta]\|$. Hence
\begin{equation}\label{eq:graph-tangent-comparison}
\|P_{\mathscr T_\theta}D\bar w[\dot\theta]\|
\le C\|\bar w\|\,\|D\psi[\dot\theta]\|,
\quad
P_{\mathscr T_\theta}D(\psi+\bar w)
=D\psi+O(\|\bar w\|)D\psi.
\end{equation}
The map $P_{\mathscr T_\theta}D(\psi+\bar w)$ is invertible when
$\|\bar w\|$ is sufficiently small.
In particular, \eqref{eq:graph-tangent-comparison} applies to the
minimizing corrections in Corollary~\ref{cor:mini_w} and
Lemma~\ref{lem:terminal-reduction} below.

For the index calculation, introduce logarithmic local masses. Set
$n:=p+q$ and $d:=2p+q$. Introduce $t\in\R$ and
$v\in\R^n$ with $\sum_i\kap _i v_i=0$ by
\begin{equation}\label{eq:log-mass-coordinates}
\ln \lam _i=
\frac{2t+v_i-\ln ((\cF_\xi^i\rg_\xi^i)(\xi_i))
 +\ln (2\al _i-1)}{4\al _i-2}.
\end{equation}
Equation~\eqref{eq:log-mass-coordinates} defines an exact change of
variables for $\al _i>1/2$. In particular,
\[
\frac{(\cF_\xi^i\rg_\xi^i)(\xi_i)\lam _i^{4\al _i-2}}
 {2\al _i-1}=e^{2t+v_i}.
\]
At $(\al ,v)=(\mathbf1,0)$, one has
$\lam _i=e^t/\sqrt{\cF_\xi^i(\xi_i)}$. In this section set
\[
h(t):=\left\{\begin{aligned}&e^{-t}&&\text{ if }\, q>0,\\&t e^{-2t}&&\text{ if }\, q=0,\end{aligned}\right.
\quad c(\rho):=-\rho_*-\rho\ln \frac{\rho_*}{8}.
\]
We carry out the reduction for finite \(t\) and estimate the parameter
derivatives uniformly for large \(t\). No smooth extension to \(t=\infty\) is used. Write
$C^3_+(\Si ):=\{V\in C^3(\Si ):V>0\}$.
Appendix~\ref{app:0} proves the following reduction from exact integral estimates.
For $q=0$, only one center derivative of $\cL_2$ is used.
Its amplitude domain is $t|\alpha-\mathbf1|<\delta$.
The passage from the natural radius
$\sqrt t\,|\alpha-\mathbf1|<r$ to that smaller domain is the
sublevel-preserving deformation of
Lemma~\ref{lem:amplitude-collar-retraction}.
\begin{lem}\label{lem:terminal-reduction}
Let \(K\Subset\Xi_{p,q}\) be a compact set contained in admissible
coordinate neighborhoods. There are $\delta>0$, $\mu_0>0$,
and $T>0$ such that, for $t>T$, $|\mu|<\mu_0$,
$\rho=(1+\mu)\rho_*$, and
\[
\xi\in K,\quad t|\al -\mathbf1|+|v|<\delta,
\quad\sum_i\kap _i v_i=0,
\]
there is a unique small minimizing correction
$\bar w=\bar w(\al,v,\xi,t)$ in $\mathcal E_{\lam,\xi}^{p,q}$, with
$\|\bar w\|\le C(|\al-\mathbf1|+e^{-t})$.
Write $j_\rho(\al ,v,\xi,t):=J_\rho(u_0+\bar w)$ with the scales
specified by \eqref{eq:log-mass-coordinates}. For each $(\xi,t)$,  the
$(\al ,v)$ equations have a unique solution in this parameter region,
denoted by $(\al _\rho,v_\rho)$, and their Hessian has $n$ positive
and $n-1$ negative eigenvalues. At
$(\al,v)=(\al_\rho(\xi,t),v_\rho(\xi,t))$, the correction and the
finite-dimensional variables satisfy
\begin{equation}\label{eq:fast-variable-bounds}
|\al _\rho-\mathbf1|\le C(h(t)+e^{-t}/t),\quad
 |v_\rho|\le Ch(t),\quad \|\bar w\|\le Ce^{-t}.
\end{equation}
The remaining function
$K_\rho(\xi,t):=j_\rho(\al _\rho,v_\rho,\xi,t)$ satisfies
\begin{equation}\label{eq:terminal-value}
K_\rho=c(\rho)-\tfrac12\FpqV(\xi)-2(\rho-\rho_*)t
-\left\{\begin{aligned}
     &\dfrac{4\pi\rho}{\rho_*}\cL_1(\xi)e^{-t}&&\text{ if }\, q>0,\\
& \dfrac{\rho}{2\rho_*}\cL_2(\xi)t e^{-2t}&&\text{ if }\,q=0,
\end{aligned}\right.
+R_\rho(\xi,t).
\end{equation}
There is $\om (t)\to0$ such that
\begin{equation}\label{eq:terminal-scale-derivatives}
|\pa _t^j R_\rho(\xi,t)|
\le(\om (t)+C|\mu|)h(t),\quad j=0,1,2.
\end{equation}
In admissible center coordinates,
\begin{equation}\label{eq:terminal-location-derivatives}
|D_\xi K_\rho+\tfrac12D_\xi\FpqV|\le Ch(t),\quad
 |D_tD_\xi K_\rho|\le Ch(t),\quad
\|D_\xi^2K_\rho+\tfrac12D_\xi^2\FpqV\|=o(1).
\end{equation}
The estimates are uniform for $V$ in compact subsets of $C^3_+(\Si )$.
In particular, the estimates are uniform for every sequence
$V_k\to V$ in $C^3$.
\end{lem}

\begin{prop}\label{prop:bubbling-branches}
Let $\xi^*$ be a nondegenerate critical configuration of $\FpqV$ and
suppose $\cL(\xi^*)\ne0$. For sufficiently small $0<|\mu|$, there is
exactly one solution in \(V(\mu,\xi^*)\) when
\(\sgn\mu=\sgn\cL(\xi^*)\), and no such solution when
\(\sgn\mu=-\sgn\cL(\xi^*)\).
If \(\sgn\mu=\sgn\cL(\xi^*)\), this solution is nondegenerate and satisfies
\begin{equation}\label{eq:exact-bubbling-index}
\morse(J_{\rho},u_\rho)
=3p+2q-1-\morse(\FpqV,\xi^*)+\mathbf1_{\{\mu>0\}}.
\end{equation}
Here each unordered critical configuration is counted once. In
particular, for \( [\xi^*]\in\cV_{p,q}^{-}\), the Morse index of the
corresponding blow-up solution for \(\rho<\rho_*\) is
\begin{equation}\label{def:l_inf}
\iota_\infty(\xi^*)=3p+2q-1-\morse(\FpqV,\xi^*),
\end{equation}
and the limiting energy is
\[
C_\infty(\xi^*)=-\rho_*(1+\ln (\rho_*/8))
-\tfrac12\FpqV(\xi^*).
\]
The conclusions hold uniformly under convergent \(C^3\) perturbations
of \(V\) and for the nearby nondegenerate critical configurations obtained
by continuation.
\end{prop}
\begin{proof}
By \eqref{eq:terminal-location-derivatives} and the implicit function
theorem, the center equations have a unique solution
$\xi=\xi_\rho(t)$ near $\xi^*$ for all sufficiently large $t$, with
\(|\xi_\rho(t)-\xi^*|+|\pa _t\xi_\rho(t)|\le Ch(t).\)
The Hessian in the location variables has $d-\morse(\FpqV,\xi^*)$ negative eigenvalues.
Set $k_\rho(t):=K_\rho(\xi_\rho(t),t)$. Differentiating
$D_\xi K_\rho(\xi_\rho(t),t)=0$ and applying the chain rule yields
\begin{equation}\label{eq:branch-chain-rule}
k_\rho'=K_{\rho,t},\quad
k_\rho''=K_{\rho,tt}
-K_{\rho,t\xi}(K_{\rho,\xi\xi})^{-1}K_{\rho,\xi t}.
\end{equation}
By \eqref{eq:terminal-location-derivatives}, the Schur correction in
\eqref{eq:branch-chain-rule} is $O(h(t)^2)=o(h(t))$. Thus
\eqref{eq:terminal-value}--\eqref{eq:terminal-location-derivatives} yield
\begin{equation}\label{eq:scalar-branch-equation}
\begin{aligned}
k_\rho'(t)&=-2\rho_*\mu+
\left\{\begin{aligned}
 &4\pi(1+\mu)\cL_1(\xi^*)e^{-t}&&\text{ if }\, q>0,\\
 &(1+\mu)\cL_2(\xi^*)t e^{-2t}&&\text{ if }\, q=0,
\end{aligned}\right.
+o(h(t))+O(|\mu|h(t)),\\
k_\rho''(t)&=
\left\{\begin{aligned}
 &-4\pi(1+\mu)\cL_1(\xi^*)e^{-t}&&\text{ if }\,q>0,\\
 &-2(1+\mu)\cL_2(\xi^*)t e^{-2t}&&\text{ if }\,q=0,
\end{aligned}\right.
+o(h(t))+O(|\mu|h(t)).
\end{aligned}
\end{equation}
By \eqref{eq:scalar-branch-equation}, choose $T$ large and then $\mu_0$ small. Since $\cL(\xi^*)\ne0$,
$k_\rho'$ is strictly monotone on $(T,\infty)$, with monotonicity sign
$-\sgn\cL(\xi^*)$. Moreover,
$\lim_{t\to\infty}k_\rho'(t)=-2\rho_*\mu$ and
$\sgn k_\rho'(T)=\sgn\cL(\xi^*)$ for $|\mu|$ sufficiently small.
Thus $k_\rho'(t)=0$ has exactly one solution \(t_\rho>T\) precisely when
\(\sgn\mu=\sgn\cL(\xi^*)\). At $t=t_\rho$,
\begin{equation}\label{eq:scale-hessian-at-branch}
h(t_\rho)=
\left\{\begin{aligned}
&\dfrac{\rho_*\mu}{2\pi\cL_1(\xi^*)}(1+o(1))&&\text{ if }\, q>0,\\
&\dfrac{2\rho_*\mu}{\cL_2(\xi^*)}(1+o(1))&&\text{ if }\,q=0,
\end{aligned}\right.\quad
k_\rho''(t_\rho)=
\left\{\begin{aligned}&-2\rho_*\mu+o(|\mu|)&&\text{ if }\,q>0,\\
&-4\rho_*\mu+o(|\mu|)&&\text{ if }\,q=0.
\end{aligned}\right.
\end{equation}
Thus the sign condition in Proposition~\ref{prop:bubbling-branches} is necessary
and sufficient for a bubbling branch near $\xi^*$ when $\cL(\xi^*)\ne0$.
Since $\lam _i\sqrt{\cF_{\xi^*}^i(\xi_i^*)}=e^{t_\rho}(1+o(1))$,
the corresponding scale orders are
\[
\lam _i\asymp
\left\{\begin{aligned}
&|\mu|^{-1}&&\text{ if }\,q>0,\\
&(|\ln|\mu||/|\mu|)^{1/2}&&\text{ if }\,q=0.
\end{aligned}\right.
\]
The orders are exact when $\cL(\xi^*)\ne0$; \eqref{eq:scale-hessian-at-branch}
does not determine the scale orders when $\cL(\xi^*)=0$.
The reduced critical point corresponds to a critical point of the full
functional: the equation projected onto $\mathcal E_{\lam,\xi}^{p,q}$ holds, and the
tangential derivative of the graph of the minimizing correction is
invertible by
\eqref{eq:graph-tangent-comparison}. Conversely, every solution in $V(\mu,\xi^*)$, defined by
\eqref{eq:condi-V-1}--\eqref{eq:condi-V-2}, has this parametrization
by Proposition~\ref{prop:unique} and the estimates in
Proposition~\ref{prop:estimates_rho_k} and
Lemma~\ref{lem:interior-isolating-bootstrap}.
Since $t_\rho\to\infty$, the constructed critical points form a
blowing-up family of exact solutions with limiting configuration
$\xi^*$. Lemma~\ref{lem:interior-isolating-bootstrap} for $q=0$ and the boundary
estimates in Proposition~\ref{prop:estimates_rho_k} for $q>0$ imply
$u_\rho\in V(\mu,\xi^*)$ in both cases. Choose the constants in
\eqref{eq:condi-V-1}--\eqref{eq:condi-V-2} larger than the constants
in Proposition~\ref{prop:estimates_rho_k} and
Lemma~\ref{lem:interior-isolating-bootstrap}.

We next compare the reduced Morse index with the Morse index of the full
Hilbert-space Hessian. At a critical point, let \(\mathscr T\) be the
tangent space generated by the bubble parameters and let
\(\mathscr E=\mathscr T^\perp\). With respect to
$\oH=\mathscr T\oplus\mathscr E$, write
\[
\mathcal H=D^2J_\rho(u_\rho)=
\begin{pmatrix}A&B\\B^*&C\end{pmatrix},\quad C\ge cI>0.
\]
For $z\in\mathscr T$ and $w\in\mathscr E$, completing the square yields
\begin{equation}\label{eq:exact-schur-identity}
\langle\mathcal H(z,w),(z,w)\rangle
=\langle(A-BC^{-1}B^*)z,z\rangle+\langle C(w+C^{-1}B^*z),w+C^{-1}B^*z\rangle.
\end{equation}
Consequently, the negative index and nullity of $\mathcal H$ equal
those of $S:=A-BC^{-1}B^*$. Differentiating the projected Euler
equation at this critical point shows that the Hessian of the reduced
functional is
$M^*SM$, where
$M=P_{\mathscr T}D(u_0+\bar w)$ is invertible. Since
$DJ_\rho(u_\rho)=0$, the chain rule has no term involving a second
derivative of the parametrization.

Apply \eqref{eq:exact-schur-identity} with respect to $\oH=\mathscr T\oplus\mathscr E$,
then complete the square in the $(\al ,v)$ variables and
the location variables. The corresponding Hessians are invertible.
By Lemma~\ref{lem:terminal-reduction}, the first has $n-1$ negative eigenvalues; the
second has $d-\morse(\FpqV,\xi^*)$ negative eigenvalues. Equation
\eqref{eq:scale-hessian-at-branch} shows that the remaining scalar
second derivative is nonzero and is negative exactly when $\mu>0$. This proves
\eqref{eq:exact-bubbling-index}, including nondegeneracy, and then
\eqref{def:l_inf}. Since $\mu t_\rho\to0$, the limiting energy follows
from \eqref{eq:terminal-value}. The estimates in Lemma~\ref{lem:terminal-reduction} are uniform under convergent
$C^3$ perturbations. The implicit function theorem and continuity of
$D^2\FpqV$ and $\cL$ preserve the reduced Morse index and the sign
of $\cL$ on the corresponding critical branches.
\end{proof}

\begin{rem}\upshape
When $\cL(\xi^*)\ne0$, equations
\eqref{eq:scalar-branch-equation}--\eqref{eq:scale-hessian-at-branch}
show that the orders $\lambda^{-1}$ for $q>0$ and
$\lambda^{-2}\ln\lambda$ for $q=0$ are exact, with
$\sgn(\rho-\rho_*)=\sgn\cL(\xi^*)$.
If $\cL(\xi^*)=0$, the first retained scale term vanishes, and
higher-order terms are needed to decide compactness and the branch index.
Thus condition~\ref{item:C2} is a sufficient hypothesis for the present
classification; its necessity under condition~\ref{item:C1} is not asserted.
\end{rem}

\begin{rem}\upshape
The spatial and configuration gradients satisfy the exact identity
\[
\tfrac12\nabla_{\xi_i}\FpqV(\xi)
=\kap _i\nabla_x\ln \cF_\xi^i(x)|_{x=\xi_i}
=\frac{\kap _i}{\cF_\xi^i(\xi_i)}
 \nabla_x\cF_\xi^i(x)|_{x=\xi_i},
\]
with tangential derivatives for boundary centers. On a compact center
set, $\cF_\xi^i(\xi_i)^{-1}$ is bounded above and below by positive
constants. Thus the factor $\cF_\xi^i(\xi_i)^{-1}$ can be included
in comparison constants, but must be kept in an exact identity.

For an interior configuration, \eqref{Green} yields
\(\Delta_g\ln\cF_\xi^i=\Delta_g\ln V+\rho_*/|\Si |_g\)
near \(\xi_i\). At a critical point of \(\Fv{p}{0}\), differentiating
\eqref{eq:reduce_fun_0} yields \(\nabla_x\ln\cF_\xi^i(\xi_i)=0\). Hence
\begin{equation}\label{eq:interior-L2-critical-identity}
\cL_2(\xi)=\sum_i\kap _i\cF_\xi^i(\xi_i)
\Big(\Delta_g\ln V(\xi_i)+\frac{\rho_*}{|\Si |_g}-2K_g(\xi_i)\Big),
\quad \xi\in\operatorname{Crit}(\Fv{p}{0}).
\end{equation}
Indeed, substitute
\(\Delta_g\cF=\cF(\Delta_g\ln\cF+|\nabla\ln\cF|_g^2)\)
into \eqref{eq:def-cL}. Formula~\eqref{eq:interior-L2-critical-identity}
is  sufficient to determine the sign without expanding
\(\Delta_g\cF_\xi^i\) term by term.
\end{rem}

\section{Geometric reduction and amplitude preparation}\label{sec:resonant-descent}

Fix $\rho_*=4\pi m$ and set $J=J_{\rho_*}$. We distinguish
contraction of the orthogonal remainder and amplitude preparation
on the minimizing graph. The reductions below specify the bubble directions
used in Section~\ref{sec:full-pseudogradient}.

\subsection{Contraction of the orthogonal remainder}

The positive orthogonal directions contract to the minimizing graph.
On a geometric modulation chart where Corollary~\ref{cor:mini_w}
applies, write $\theta=(\alpha,\xi,\lambda)$ and choose $r>0$ with
$\|\bar w_\theta\|<r/2$ and, for some $c>0$,
\[
 D_w^2J(\psi_\theta+w)[z,z]\ge c\|z\|^2,
 \quad \|w\|<r,\quad z\in E_\theta.
\]
For $z=w-\bar w_\theta$, the maps
\[
 H_s(\psi_\theta+w)
 =\psi_\theta+\bar w_\theta+(1-s)z,\quad 0\le s\le1,
\]
define a continuous deformation onto the graph and fix the graph
pointwise. Indeed, the segment stays in the radius-$r$ ball, and the
projected Euler equation implies
\[
 \frac{\d}{\d s}J(H_s(\psi_\theta+w))
 =-\int_0^{1-s}D_w^2J(\psi_\theta+\bar w_\theta+az)[z,z]\,\d a
 \le-c(1-s)\|z\|^2.
\]
Thus the deformation preserves every sublevel. Continuity follows from
the local modulation coordinates and Corollary~\ref{cor:mini_w}.

This contraction is defined on the geometric parameter tube. An additional
cutoff condition $\|\nabla J\|<\varepsilon$ need not be preserved by
the segment homotopy. The deformation is asserted on the geometric tube. Also, the natural amplitude estimate
\eqref{eq:apha_i_estimate} cannot be substituted for the hypothesis of
Lemma~\ref{lem:terminal-reduction}: if $\ell\to\infty$ and
$\alpha-1=\ell^{-3/4}$, then
$|\alpha-1|\sqrt\ell\to0$, whereas
$|\alpha-1|\ell\to\infty$. This calculation concerns the two amplitude
estimates; it does not assert that every such parameter sequence is an
orbit of a chosen field.
Lemma~\ref{lem:amplitude-collar-retraction} establishes the required
descending deformation between the natural and terminal amplitude domains.

\subsection{Geometric charts on fixed parameter domains}

Set $n=p+q$, $d=2p+q$, and $\rho_*=\sum_i\kappa_i$.
The logarithmic mass coordinates \eqref{eq:log-mass-coordinates}
are also defined outside the small-gradient neighborhood. We use them on
\begin{equation}\label{eq:resonant-geometric-domain}
 \xi\in K\Subset\Xi_{p,q},\quad t>T,\quad
 \alpha_i=1+b_i/\sqrt t,\quad |b|<r,\quad |v|<r,
 \quad\sum_i\kappa_i v_i=0.
\end{equation}
There is no gradient cutoff in this definition. The scales are the
positive scales specified by \eqref{eq:log-mass-coordinates}.
All statements below are local in the center charts and invariant under
permutations of centers of the same type.

We first establish uniform reduction on a fixed geometric parameter domain.
\begin{prop}
\label{prop:resonant-geometric-reduction}
For sufficiently small fixed $r>0$ and sufficiently large $T$, the
domain \eqref{eq:resonant-geometric-domain} has a unique small fiber
minimizer $\bar w\in E_{\lambda,\xi}^{p,q}$, with
\begin{equation}\label{eq:geometric-fiber-estimate}
 \|\bar w\|\le C(|\alpha-\mathbf1|+e^{-t/4}).
\end{equation}
The fiber Hessian is uniformly positive in a fixed ball about $\bar w$.
The minimizing graph and that ball define actual geometric modulation
charts. Write $j(b,v,\xi,t)=J(\psi+\bar w)$ and
$c_*=c(\rho_*)$. On every smaller parameter domain,
\begin{equation}\label{eq:geometric-fixed-base}
 j(b,v,\xi,t)=F_\infty(b,v,\xi)+R(b,v,\xi,t),\quad
 \|R(\cdot,t)\|_{C^1}\to 0,
\end{equation}
where
\[
 F_\infty(b,v,\xi)=c_*-\frac12\mathbf F_{p,q}^V(\xi)
 +2\sum_i\kappa_i b_i^2
 -\rho_*\ln \Big(\frac{\sum_i\kappa_i e^{v_i}}{\rho_*}\Big).
\]
The retained second derivatives are uniformly bounded. For the $bb$
submatrix, their difference from those of $F_\infty$ is at most
$Cr^2+o(1)$; for every other retained Hessian submatrix the difference is $o(1)$.
The derivatives $j_{tz}$ are bounded, where $z=(b,v,\xi)$ and $t$ is
differentiated with $b$ fixed.

The scale and combined amplitude--scale pairings used in
\eqref{eq:graph-preparation-pairings} hold on this graph, with the
same form of error bound and with $H$ replaced by $e^{-t/4}$.
In particular their validity here does not require
$\|\nabla J\|<\epsilon$ or $t|\alpha-\mathbf1|<r$.
\end{prop}
\begin{proof}
Set $a=\alpha-\mathbf1$, $\zeta=(v,\xi,t)$ and $E(t)=e^{-t/4}$.
The exact change of scales yields, uniformly on the fixed parameter domain,
\begin{equation}\label{eq:geometric-scale-comparison}
 \ln \lambda_i=t+O(\sqrt t),\quad
 D_\zeta^j\ln \lambda_i=O(1),\quad
 D_\alpha^j\ln \lambda_i=O(t^j),\quad j=1,2.
\end{equation}
Here the $\zeta$ derivatives keep $\alpha$ fixed. In particular all
logarithms are comparable, though their exponentials need not be.
Every polynomial in $t$ times a positive-power scale error needed below
is bounded by $CE(t)$ after increasing $T$.

We establish the estimates required on this larger domain.
For a fixed $1<s<3/2$, and $|a_i|$ sufficiently small, the mean-value
theorem yields
\[
 \|(1+|z|^2)^{-2a_i}-1\|_{L^s((1+|z|^2)^{-2}dz)}
 \le C_s|a_i|.
\]
The same estimate holds with any of the finitely many logarithmic
factors produced by the two parameter derivatives, with the
corresponding fixed constant. To see finiteness, dominate the
integrands by
$C(1+|z|^2)^{-2+2s a_0}(1+\ln (1+|z|^2))^N$ and choose
$a_0$ satisfying $2s a_0<1/4$. The radial integrals converge at infinity.
On a half-plane, the integral estimates are unchanged except for fixed factors.
A spatial Taylor error is bounded by $C|z|/\lambda_i$.
Its weighted $L^s$ norm has order $\lambda_i^{-1}$ after choosing $a_0$
still smaller so that $s(1+4a_0)<2$. The cutoff-annulus errors and
the exact projection errors in
\eqref{eq:A1-exact-error-equation} are exponentially small under
\eqref{eq:geometric-scale-comparison}, also after the indicated
parameter derivatives.

Integrate \eqref{eq:A-exact-local-mass-factor} with the preceding integral estimates.
The leading masses are exactly
$\kappa_i e^{2t+v_i}/8$. Their relative weights are
$p_i=\kappa_i e^{v_i}/\sum_j\kappa_j e^{v_j}$.
The normalized density differs, in the weighted dual norms used in
Step 2 of the proof of Lemma~\ref{lem:terminal-reduction}, from the sum of these
weights times the normalized bubble measures by $O(|a|+E)$.
This estimate includes the total-mass normalization: the denominator
is bounded below by a fixed multiple of $e^{2t}$, and differentiating
its reciprocal produces the same weighted estimates.
The weighted Sobolev and exponential estimates in Step 1 of that proof
use orthogonality and rescaling of each individual bubble; they do
not use $t|a|<r$ or a gradient cutoff. They therefore apply here.

At $a=v=0$ the fiber quadratic form has the spectral gap proved in
that Step 1. The changes of density have operator norm
$O(|a|+|v|+E)$ on the fiber by weighted H\"older. The fiber Hessian
thus remains bounded below by a positive constant. Its upper bound
and the bound for its third fiber derivative follow from the same
weighted integral estimates. The fiber linear functional has norm at most
$C(|a|+E)$, since the leading normalized bubble measures annihilate
the fiber. The contraction proof of Corollary~\ref{cor:mini_w} now
applies with those estimates and proves
\eqref{eq:geometric-fiber-estimate}, uniform uniqueness and positivity.
The normalized Gram estimates and the quantitative bubble-matching
argument prove local invertibility and uniqueness of the geometric
modulation chart for a fixed small remainder radius. These arguments
use the geometric parameter and remainder estimates, rather than the
small-gradient condition.

For the differentiated estimates, use the fixed-center pullbacks,
Dirichlet forms and finite-rank projections constructed in Step 2 of
the proof of Lemma~\ref{lem:terminal-reduction}. Their $\zeta$ derivatives through
order two are uniformly bounded by
\eqref{eq:geometric-scale-comparison}. An amplitude derivative costs
at most a factor $Ct$, and two cost at most $Ct^2$.
The preceding weighted kernel estimates, including their logarithmic
factors, yield on the fixed fiber space
\[
 \|D_\zeta^j f\|\le C e,\quad
 \|D_\zeta^j C\|\le C\quad(j\le2),\quad e=|a|+E,
\]
\[
 \|D_\alpha f\|+\|D_\zeta D_\alpha f\|\le C(1+te),\quad
 \|D_\alpha^2 f\|\le C(t+t^2e),\quad
 \|D_\alpha C\|\le Ct,\quad \|D_\alpha^2 C\|\le Ct^2.
\]
Here $f,C$ are the linear and quadratic fiber terms, and their
nonlinear remainders have the uniform cubic estimates from the
normalized exponential integral. The projected inverse has uniformly
bounded norm. Differentiating its equation yields
$\|D_\zeta^j\bar w\|\le Ce$ in those pullbacks for $j\le2$.
The subtracted means cancel because the density errors have total
mass zero. No composition operator on all of $H^1$ is differentiated.

Let $\mathcal E=\mathcal E_{\rho_*}$ be the explicit expression in
\eqref{eq:A-logmass-leading-energy}. Subtract its exact constant-core
integrals from the Dirichlet and mass terms. The spatial-coefficient,
cutoff and projection errors just bounded, and the removal of
the common scalar $S$ by normalization, yield bare remainders
bounded by $CE$ after two $\zeta$ or two amplitude derivatives.
The fiber value correction is quadratic in its linear functional.
The projected equation and the preceding derivative estimates therefore
yield, for $B=j-\mathcal E$,
\begin{equation}\label{eq:geometric-differentiated-errors}
 \begin{split}
 |D_\zeta^jB|&\le C(e^2+E)\quad(j=0,1,2),\\
 |D_\alpha B|+|D_\alpha D_\zeta B|
    &\le Ce(1+te)+CE,\\
 \|D_\alpha^2 B\|&\le C(1+te)^2+CE.
 \end{split}
\end{equation}
The estimates \eqref{eq:geometric-differentiated-errors} are coarser than
the terminal estimates and suffice on a fixed $b$-domain and are not used to determine the
terminal radial coefficient.

Now $\alpha_i^2/(2\alpha_i-1)=1+a_i^2/(1+2a_i)$.
Inserting $a=b/\sqrt t$ into $\mathcal E$ yields its limit exactly as
in \eqref{eq:geometric-fixed-base}. The term involving $2t$ contributes
$2\sum_i\kappa_i b_i^2/(1+2b_i/\sqrt t)$; the remaining smooth terms
converge with their retained derivatives to
$c_*-\tfrac12\mathbf F-\rho_*\ln (\sum_i\kappa_i e^{v_i}/\rho_*)$.
For the correction $B$, \eqref{eq:geometric-differentiated-errors}
yields $B\to0$, $D_\zeta B\to0$ and
\[
 |D_bB|\le C e(1+te)/\sqrt t+CE\to 0,
 \quad \|D_b^2B\|\le C/t+C|b|/\sqrt t+C|b|^2+CE.
\]
Every other retained Hessian submatrix tends to zero. The chain rule
$\partial_t|_b=\partial_t|_\alpha
-\sum_i(a_i/2t)\partial_{\alpha_i}$ also proves that $j_{tz}$ is bounded.
This proves the asserted value and derivative estimates.

For later control of the graph lift one also needs the admissible
center directions $C_{ij}=\lambda_i^{-1}\partial_{(\xi_i)_j}P\delta_i$.
On this geometric graph,
\begin{equation}\label{eq:geometric-normalized-center-bound}
 |DJ(U)[C_{ij}]|\le C(E+AW+W^2)\le C(E+A^2),
 \quad A=\sum_i|a_i|,\quad W=\|\bar w\|.
\end{equation}
The leading integral vanishes by symmetry. On the rescaled core
the leading $C_{ij}$ profile is a fixed linear combination of
$z_k/(1+|z|^2)$, with admissible tangential components at boundary
centers. Its integral against every radial density
$(1+|z|^2)^{-2\alpha_i}$ is zero, also on the half-plane for these
tangential components. Its coefficients and the other bubbles vary
by $O(|z|/\lambda_i)$; the weighted linear integrals above estimate their
normalized integrals by $CE$. The Dirichlet pairing has the same
bound by the exact differentiated projection formula. In the term
linear in $w$, center orthogonality removes the leading odd profile
against $\nu_i$. Replacing its radial weight by the amplitude-modified
one costs $CAW$ by the weighted H\"older estimate; spatial variations
and cutoff terms cost $CEW$. The nonlinear term is bounded by
$CW^2$ using $|e^w-1-w|\le\tfrac12w^2e^{|w|}$ and the uniform
weighted exponential bound. These contributions prove
\eqref{eq:geometric-normalized-center-bound}. No center estimate is
imported from a gradient-cutoff neighborhood.

Finally the amplitude and scale pairings can be evaluated before
assuming the gradient is small. The core radial integrals used in
this calculation satisfy the normalized integral identities
\[
 \mathbb E_\alpha[(1+|z|^2)^{-1}]
 =\frac{2\alpha-1}{2\alpha},\quad
 \mathbb E_\alpha[\ln (1+|z|^2)]=\frac1{2\alpha-1}.
\]
The formulas follow by integrating $r(1+r^2)^{-2\alpha}$ and its logarithmic
derivative. Insert these identities and the exact projected-bubble
expansion into $DJ[P\delta_i/l_i]$ and $DJ[S_i]$.
The scale pairing is $2\kappa_i\tau_i+O(A^2+W^2+E)$.
Subtracting twice this scale pairing from the amplitude pairing cancels
its $4\kappa_i\tau_i$ term. The smooth interaction terms cancel at
$a=0,\tau=0$; their remaining terms are bounded by
$C(A+T)/l$. Terms linear in $w$ are bounded by
$C(W/l+AW)$ using amplitude and scale orthogonality, and quadratic
terms by $CW^2$. The coefficient variation and the cutoff and
projection tails are bounded by the radial integral estimates above and by $CE$.
Thus the combined amplitude--scale pairing is
$4\kappa_i a_i+O(A^2+W^2+(A+T+W)/l+E)$.
The minimizing graph has $W\le C(A+E)$ and small $\tau$ because its
leading weights are $p_i$ and $|v|<r$. These are precisely the
pairings and estimates claimed, with no inference from gradient smallness.
\end{proof}

\subsection{Amplitude preparation}

The natural amplitude radius is of order $(\ln \lambda)^{-1/2}$,
whereas the terminal expansions are used on a radius of order
$(\ln \lambda)^{-1}$. We first remove the intervening amplitude
collar by a descending deformation. The logarithmic mass coordinates
are kept fixed in this step. The coupled field described below
also yields explicit amplitude and relative-mass directions on its
stated domain. The retraction is supplied by the following lemma.

The next deformation reduces the amplitude radius without changing the sublevel pairs.
\begin{lem}
\label{lem:amplitude-collar-retraction}
Fix $K\Subset\Xi_{p,q}$ in admissible center charts and
$\rho_*=4\pi(2p+q)$. There are positive radii
$r_a,r_v,r_z,\delta_0$ such that the following holds.
For every $0<\delta<\delta_0$ there is $T_\delta$ with
these properties.
The geometric tube over
\[
 \xi\in K,\quad t>T_\delta,\quad
 \sqrt t\,|\alpha-\mathbf1|<r_a,\quad |v|<r_v,\quad
 \sum_i\kappa_i v_i=0
\]
has positive fiber coordinates $\mathsf z\in E_\theta$,
$\|\mathsf z\|<r_z$, with
\begin{equation}\label{eq:amplitude-collar-fiber}
 J(u)=j(\alpha,v,\xi,t)+\|\mathsf z\|^2,\quad
 \theta=(\alpha,v,\xi,t).
\end{equation}
Write $\mathcal T$ for this tube, set
$M(u)=t|\alpha-\mathbf1|$, and set
$\mathcal T_\delta=\{u\in\mathcal T:M(u)\le\delta\}$.
There exists a strong deformation retraction
\[
 \mathscr H:[0,1]\times\mathcal T\to \mathcal T
 \quad\text{ onto }\,\mathcal T_\delta
\]
which fixes $(v,\xi,t)$ and does not increase $J$.
The target is contained in the amplitude and relative-mass domain
of Lemma~\ref{lem:terminal-reduction}.
For every $A<B$ the same maps yield a strong deformation of pairs
\begin{equation}\label{eq:amplitude-collar-pair}
 (\mathcal T\cap J^B,\mathcal T\cap J^A)
 \to 
 (\mathcal T_\delta\cap J^B,\mathcal T_\delta\cap J^A).
\end{equation}
The statement remains valid after restricting $(v,\xi,t)$ to any
subset of the stated base domain. No gradient cutoff is imposed.
All constants may be chosen uniformly on the compact center and
potential families of Proposition~\ref{prop:resonant-geometric-reduction}.
\end{lem}
\begin{proof}
All amplitude derivatives in this proof keep $(v,\xi,t)$ fixed.
Write $a=\alpha-\mathbf1$ and use $b=\sqrt t\,a$ for the normalized
amplitude. Proposition~\ref{prop:resonant-geometric-reduction} yields
\[
 D_{bb}^2j=4\operatorname{diag}(\kappa_i)+O(r_a^2)+o(1).
\]
Choose $r_a$ sufficiently small, then increase $T$, to obtain
\begin{equation}\label{eq:amplitude-collar-convexity}
 D_{\alpha\alpha}^2j\ge c t I
 \quad\text{on }\sqrt t\,|a|<r_a
\end{equation}
with a fixed $c>0$. Let $\delta_{\rm ter}$ be an admissible
radius in Lemma~\ref{lem:terminal-reduction}, and choose
$r_v<\delta_{\rm ter}/4$ and
$\delta_0<\delta_{\rm ter}/4$.
At $a=0$, the exact leading energy
\eqref{eq:A-logmass-leading-energy} has
$D_\alpha\mathcal E_{\rho_*}(\mathbf1,v,\xi,t)=0$.
The larger-domain estimate
\eqref{eq:geometric-differentiated-errors}, evaluated at $a=0$,
has $e=E(t)=e^{-t/4}$ and therefore yields
\begin{equation}\label{eq:amplitude-collar-forcing}
 |D_\alpha j(\mathbf1,v,\xi,t)|
 \le CE(t)(1+tE(t))+CE(t)\le C'E(t)=:\eta(t),
 \quad \eta(t)\to 0,
\end{equation}
uniformly for $|v|<r_v$ and $\xi\in K$.
This uses only the geometric estimates. No terminal expansion on
the larger amplitude tube is asserted.

We construct the positive fiber coordinate on that larger tube.
Set $U=\psi_\theta+\bar w_\theta$ and $z=w-\bar w_\theta$.
The uniform fiber convexity and third fiber derivative bound in
Proposition~\ref{prop:resonant-geometric-reduction} yield
\[
 J(U+z)-J(U)=\langle C_\theta(z)z,z\rangle,\quad
 C_\theta(z)=\int_0^1(1-s)D_w^2J(U+sz)\,\d s,
\]
where $c_1I\le C_\theta(z)\le C_1I$ and
$\|D_zC_\theta(z)\|\le C_1$ on a fixed small fiber ball.
Thus $\mathsf z=C_\theta(z)^{1/2}z$ has derivative
$C_\theta(0)^{1/2}+O(\|z\|)$. The inverses of
$C_\theta(0)^{1/2}$ are uniformly bounded; the contraction mapping
argument applied to this derivative estimate yields an inverse
on a fixed ball $\|\mathsf z\|<r_z$.
This proves \eqref{eq:amplitude-collar-fiber}.
For the differentiability used here, the convergent resolvent formula
\[
 C^{1/2}=\frac1\pi\int_0^\infty
 q^{-1/2}C(q+C)^{-1}\,\d q
\]
has derivative in direction $D$ equal to the integral of
$\pi^{-1}q^{1/2}(q+C)^{-1}D(q+C)^{-1}$.
Its norm is bounded by a constant depending only on the positive
lower spectral bound times $\|D\|$; further derivatives follow by
differentiating the resolvents, whose integrals have the same
convergence at zero and infinity.
At every finite scale
the coordinate and its inverse depend smoothly on the raw bubble
parameters, and have the finite differentiability of the mass
coordinates after the change of parameters. The unique modulation
charts of Proposition~\ref{prop:resonant-geometric-reduction} identify
$(\alpha,v,\xi,t,\mathsf z)$ as genuine coordinates on the tube.

In parameter time $s$ define
\begin{equation}\label{eq:amplitude-collar-field}
 \dot\alpha=-a/t,\quad \dot v=\dot\xi=\dot t=0,\quad
 \dot{\mathsf z}=\dot P_\theta\mathsf z-\mathsf z/t,
\end{equation}
where $P_\theta$ is the orthogonal projection onto $E_\theta$.
Differentiating $P_\theta^2=P_\theta$ yields
$P_\theta\dot P_\theta P_\theta=0$. The last equation is therefore
tangent to $P_\theta\mathsf z=\mathsf z$, and
\begin{equation}\label{eq:amplitude-collar-clock}
 a(s)=e^{-s/t}a(0),\quad
 M(s)=e^{-s/t}M(0),\quad
 \|\mathsf z(s)\|^2=e^{-2s/t}\|\mathsf z(0)\|^2.
\end{equation}
In particular, no amplitude, relative-mass, center or fiber face
can cause exit during this preparation.
The scales are defined throughout by \eqref{eq:log-mass-coordinates};
they are not frozen. More explicitly, with
$\ell_i=\ln \lambda_i$ and
$C_i=(\mathcal F_\xi^i\mathrm g_\xi^i)(\xi_i)$, their velocities are
\begin{equation}\label{eq:amplitude-collar-scale-lift}
 \dot\ell_i=
 \frac{-4\ell_i\dot\alpha_i
       -\sum_j\partial_{\alpha_j}\ln  C_i\,\dot\alpha_j
       +2\dot\alpha_i/(2\alpha_i-1)}
      {4\alpha_i-2}.
\end{equation}
This is obtained by differentiating
$(4\alpha_i-2)\ell_i+\ln  C_i-\ln (2\alpha_i-1)=2t+v_i$.
In particular, $|\dot\ell_i|\le C|a|$ and
$\ell_i=t+O(\sqrt t)$ throughout the path, so all scales stay
large after increasing $T$. The graph component of the ambient
field includes $D\bar w_\theta[\dot\theta]$; the fiber equation is
pushed forward by the inverse of the exact positive coordinate.
These definitions include the motion of the orthogonal space.

By \eqref{eq:amplitude-collar-convexity} and
\eqref{eq:amplitude-collar-forcing},
\[
 \langle D_\alpha j(\alpha,v,\xi,t),a\rangle
 \ge ct|a|^2-\eta(t)|a|.
\]
Choose $T_\delta$ with
$\eta(t)\le c\delta/2$ and $\delta/\sqrt t<r_a$ for
$t>T_\delta$. Until $M=\delta$ is reached,
\begin{equation}\label{eq:amplitude-collar-descent}
 -\frac{\d}{\d s}J(u(s))
 =\frac1t\langle D_\alpha j,a\rangle
       +\frac2t\|\mathsf z\|^2
 \ge\frac c2|a|^2+\frac2t\|\mathsf z\|^2.
\end{equation}
Thus the path is descending on the entire collar to be removed.
The pushed-forward field is locally Lipschitz at finite scales.
Division by one plus the ambient norm yields a bounded field with the same oriented paths.

The first-arrival time in parameter time is explicit:
\begin{equation*}
 S(u)=
 \left\{\begin{aligned}
 &t\ln (M(u)/\delta)&& \text{ if }\, M(u)>\delta,\\
 &0&&\text{ if }\, M(u)\le\delta.
\end{aligned}\right.
\end{equation*}
The stopping time is finite and continuous on $\mathcal T$, including the target.
For each starting point the corresponding base segment is compact
and has finite scales. The projections and the inverse fiber
coordinate have bounded derivatives there; parallel transport and
\eqref{eq:amplitude-collar-clock} consequently define the path
throughout $[0,S(u)]$. Its bounded normalization takes this finite
segment to a finite ambient-time segment as well.
If $\eta_\alpha(s,u)$ denotes the unnormalized path, set
\(\mathscr H(\sigma,u)=\eta_\alpha(\sigma S(u),u), \quad 0\le\sigma\le1.\)
Local continuous dependence and the continuous stopping formula prove
joint continuity. The maps fix $\mathcal T_\delta$ pointwise,
and their final image lies there. They preserve the base and
decrease $J$ by \eqref{eq:amplitude-collar-descent}; hence they
preserve each of the two sublevels in \eqref{eq:amplitude-collar-pair}.
The stopping time is locally bounded at every finite-scale point;
no uniform time bound as $t\to\infty$ is needed.
The choices of $r_v,\delta_0$ place the final base in the terminal
domain. Restriction to any subset of $(v,\xi,t)$ is allowed because
those coordinates never move. All choices used uniform estimates on
the stipulated compact families. The formulas commute with
permutations among centers of the same type and therefore descend
to the corresponding unordered configuration charts.
This proves the lemma.
\end{proof}

The deformation is defined on a geometric tube and preserves its
sublevel pairs. Its conclusion does not require preservation of
$\|\nabla J\|<\epsilon$. The smaller amplitude neighborhood is reached
before the terminal Morse coordinates are used. This implements the
amplitude preparation preceding the critical-points-at-infinity
normal form in \cite{ABL2017}.
Since $\ln \lambda_i\asymp t$, its source and target have respectively
the natural square-root logarithmic and the smaller logarithmic
amplitude radii, with adjusted fixed constants.
In particular, for every $k$ and every field $\mathbb F$,
\[
 H_k(\mathcal T\cap J^A,\mathcal T_\delta\cap J^A;\mathbb F)=0.
\]
The removed amplitude collar therefore supplies no additional
relative homology generator.

The explicit coupled bubble field below satisfies an entry-or-exit
alternative on its stated domain;
Lemma~\ref{lem:amplitude-collar-retraction} supplies the relative
deformation across the amplitude collar.

The coupled field satisfies a finite entry-or-exit alternative.
\begin{prop}
\label{prop:resonant-graph-preparation}
Fix a compact separated center set $K\Subset\Xi_{p,q}$. Write
$U(\theta)=\psi_\theta+\bar w_\theta$ on an open parameter domain
supplied by Proposition~\ref{prop:resonant-geometric-reduction}. Set
\[
 a_i=\alpha_i-1,\quad l_i=\ln \lambda_i,\quad
 l=\min_i l_i,\quad A=\sum_i|a_i|,\quad
 T=\sum_i|\tau_i|,\quad H=e^{-t/4},
\]
where $\tau_i$ is evaluated at $U(\theta)$.
There is $\epsilon_0>0$ such that the following holds for
$0<\epsilon<\epsilon_0$. Impose also
\begin{equation*}
 \xi\in K,\quad l>L,\quad
 \max_i|a_i|\sqrt{l_i}<\epsilon,\quad T<\epsilon.
\end{equation*}
For every fixed $\delta>0$, $L$ can be chosen so that the parameter field
\begin{equation}\label{eq:graph-preparation-parameter-field}
 \frac{\d \alpha_i}{\d s}=-\frac{a_i}{l_i},\quad
 \frac{\d l_i}{\d s}=\frac{2a_i-\tau_i}{\alpha_i^2},\quad
 \frac{\d \xi_i}{\d s}=0
\end{equation}
has, on $\max_i|a_i|l_i\ge\delta$, the descent estimate
\begin{equation}\label{eq:graph-preparation-descent}
 -\frac{\d}{\d s}J(U(\theta(s)))\ge c(A^2+T^2).
\end{equation}
Its lift $\mathcal V=DU[\dot\theta]$ can be normalized to the bounded
locally Lipschitz field $\mathcal V/(1+\|\mathcal V\|)$ without
changing its oriented trajectories. Every trajectory either reaches
$\max_i|a_i|l_i<\delta$ in finite time, or leaves the stated open
parameter domain first. On its whole interval in that domain,
\begin{equation}\label{eq:graph-preparation-amplitude-clock}
 \frac{\d}{\d s}(|a_i|l_i)\le-\frac{|a_i|}{2},\quad
 \frac{\d}{\d s}(|a_i|\sqrt{l_i})\le0.
\end{equation}
For sufficiently small $\epsilon,\delta$, entry places the parameters in
the amplitude and relative-mass region of Lemma~\ref{lem:terminal-reduction}.
This statement does not say that every exit from the stated domain is
an exit from all bubbling neighborhoods.
\end{prop}
\begin{proof}
First the logarithms of the scales are comparable on this domain,
although the scales themselves need not be comparable. In fact, set
\[
 L_i=(4\alpha_i-2)l_i+
 \ln ((\mathcal F_\xi^i\mathrm g_\xi^i)(\xi_i))
 -\ln (2\alpha_i-1),\quad M=\int_\Sigma Ve^{U}\,\d v_g.
\]
The exact definition of $\tau_i$ yields
$e^{L_i}=8M(1-\tau_i)/\rho_*$. Thus $L_i-L_j=O(\epsilon)$.
The center coefficients and their amplitude derivatives are bounded on
$K$, and $\alpha_i$ stays close to one. Increasing $L$ therefore yields
\begin{equation}\label{eq:graph-preparation-log-comparison}
 C^{-1}l\le l_i\le Cl.
\end{equation}
The geometric estimate \eqref{eq:geometric-fiber-estimate} yields
$W:=\|\bar w\|\le C(A+H)$; all center derivatives in its right side
are bounded on $K$ and $\ln ^{3/2}\lambda_i/\lambda_i^2\le CH$.

Write $D_iP_i=\lambda_i\partial_{\lambda_i}P\delta_i$ and define
\[
 B_i=\frac{P\delta_i}{l_i}-\frac{2}{\alpha_i}D_iP_i,
 \quad \widehat S_i=\frac{1}{\alpha_i}D_iP_i.
\]
The bubble part of the lifted velocity is exactly
\begin{equation}\label{eq:graph-preparation-bubble-velocity}
 V_0=D\psi[\dot\theta]=-
 \sum_i a_i B_i-\sum_i\tau_i\widehat S_i,\quad
 \|V_0\|\le C(A+T).
\end{equation}
The geometric pairing extension in
Proposition~\ref{prop:resonant-geometric-reduction}, at $\mu=0$, yields
\begin{equation}\label{eq:graph-preparation-pairings}
 DJ(U)[B_i]=4\kappa_i a_i+O(R),\quad
 DJ(U)[\widehat S_i]=2\kappa_i\tau_i+O(R),
\end{equation}
where one common bound is
\[
 R=C\Big(A^2+W^2+\frac{A+T+W}{l}+H\Big).
\]
For the boundary terms, $\lambda_i^{4\alpha_i-3}/M\le C/\lambda_i$
follows directly from $|\tau_i|<\epsilon$. For the other tail terms,
$\alpha_i>3/4$ and \eqref{eq:graph-preparation-log-comparison} yield
$\lambda_i^{-(4\alpha_i-2)}\le H$ and
$\ln \lambda_i/\lambda_i^2\le CH$. These facts justify the common
bound in both the boundary and pure-interior cases. The estimate is used only
for preparation and does not specify the terminal radial
coefficient.

Since $W\le C(A+H)$, pairing
\eqref{eq:graph-preparation-bubble-velocity} with $DJ$ and applying
Young's inequality yields, after choosing $\epsilon$ small and $L$ large,
\begin{equation}\label{eq:graph-preparation-raw-descent}
 -DJ(U)[V_0]\ge c_1(A^2+T^2)-C_1H^2.
\end{equation}
Here $(A+T)A^2+(A+T)^2/l=o(A^2+T^2)$ as
$A+T\to0$ and $l\to\infty$.

The graph lift also contributes to the descent estimate.
The projected Euler equation says $\nabla J(U)\in\mathscr T_\theta$.
Pair the gradient with the normalized frame of
Lemma~\ref{lem:Gram}. The geometric amplitude and scale pairings and
\eqref{eq:geometric-normalized-center-bound} yield
\[
 \|\nabla J(U)\|\le C\sqrt l\,(A+T+H).
\]
For example, amplitude normalization multiplies the pairing with
$P\delta_i/l_i$ by $\sqrt{l_i}$; the normalized center pairing is
$O(H+A^2+W^2)$. The frame Gram matrices and their inverses are uniformly
bounded. Along \eqref{eq:graph-preparation-parameter-field} the centers
are fixed, $|\dot l_i|\le C(A+T)$, and differentiating the normalized
frame yields $\|\dot e_r\|\le C(A+T)$ by
\eqref{eq:normalized-second-derivatives}.
Differentiation of $\langle\bar w,e_r\rangle=0$ now yields
\[
 \|P_{\mathscr T_\theta}D\bar w[\dot\theta]\|
 \le CW(A+T).
\]
Therefore
\[
 |DJ(U)[D\bar w[\dot\theta]]|
 \le C W\sqrt l\,(A+T+H)(A+T).
\]
Since $A\sqrt l\le C\epsilon$ and $H\sqrt l\to0$, this contribution
is at most $\frac12c_1(A^2+T^2)+CH^2$ in
\eqref{eq:graph-preparation-raw-descent}.
Finally $\max_i|a_i|l_i\ge\delta$ implies
$A\ge\delta/(Cl)$, whereas $H$ decays exponentially in $l$. Increasing $L$ depending on
$\delta$ ensures that $CH^2$ is at most a fixed small multiple of $A^2$, which proves
\eqref{eq:graph-preparation-descent}.

The amplitude clock is exact. Away from $a_i=0$,
\[
 \frac{\d}{\d s}(|a_i|l_i)
 =-|a_i|\Big(1-\frac{2a_i-\tau_i}{\alpha_i^2}\Big),\quad
 \frac{\d}{\d s}(|a_i|\sqrt{l_i})
 =\frac{|a_i|}{\sqrt{l_i}}
      \Big(-1+\frac{2a_i-\tau_i}{2\alpha_i^2}\Big).
\]
At $a_i=0$ both quantities remain zero. Smallness proves
\eqref{eq:graph-preparation-amplitude-clock}. We can also arrange
$|\dot l_i|\le1/2$. As long as the orbit stays in the domain,
$l_i(s)\le l_i(0)+s/2$, hence
\[
 |a_i(s)|l_i(s)
 \le |a_i(0)|l_i(0)\,
       \frac{l_i(0)}{l_i(0)+s/2}.
\]
If there is no earlier chart exit, all these quantities eventually
become smaller than $\delta$. No finite parameter-time divergence of
the scales is possible. Finite-scale smoothness of the projected
Euler equation yields $U\in C^2$ locally, so the lifted field is locally
Lipschitz. Its bounded normalization is a positive change of time;
the time integral over each compact parameter segment is finite.
The time change therefore preserves the finite-entry conclusion when there is no prior
exit.

To identify the terminal coordinates, define
$2t=\rho_*^{-1}\sum_i\kappa_i L_i$ and $v_i=L_i-2t$.
Then $\sum_i\kappa_i v_i=0$ and
\[
 v_i=\ln (1-\tau_i)
 -\rho_*^{-1}\sum_j\kappa_j\ln (1-\tau_j),\quad |v|\le C\epsilon.
\]
Equation \eqref{eq:log-mass-coordinates} holds exactly, $t/l_i$ is
bounded above and below, and $t|a|\le C\delta$ after entry.
Choose $C(\epsilon+\delta)$ smaller than the terminal radius in
Lemma~\ref{lem:terminal-reduction}. This proves the last statement.
\end{proof}

The relative-mass boundary of the preparation domain has the following orientation.
Use the coordinates $t,v$ just defined and the parameter time $s$ in
\eqref{eq:graph-preparation-parameter-field}. Set
$|v|_\kappa^2=\sum_i\kappa_i v_i^2$. Throughout the same domain,
\begin{equation}\label{eq:resonant-preparation-mass-face}
 \frac12\frac{\d}{\d s}|v|_\kappa^2
 \ge c_0|v|_\kappa^2
 -C|v|_\kappa\Big(A^2T+A^3+\frac{A}{l}\Big).
\end{equation}
Consequently, for every fixed sufficiently small $r>0$, for sufficiently large $L$
the field is strictly outward on the face $|v|_\kappa=r$, wherever
that face belongs to its stated domain.

Keeping the centers fixed, the amplitude derivatives of
$\ln ((\mathcal F_\xi^i\mathrm g_\xi^i)(\xi_i))$ are bounded. Thus
exact differentiation of $L_i$ yields
\[
 \dot L_i=-2\tau_i+E_i,\quad
 |E_i|\le C\Big(A^2T+A^3+\frac{A}{l}\Big).
\]
Indeed the coefficients of the two principal terms are
\[
 \frac{(4\alpha_i-2)(2a_i-\tau_i)}{\alpha_i^2}-4a_i
 =-2\tau_i+\frac{2a_i^2\tau_i-4a_i^3}{\alpha_i^2};
\]
the derivative of $-\ln (2\alpha_i-1)$ and of the smooth coefficient
is $O(A/l)$. Set $c=\rho_*e^{2t}/(8M)$, so that
$\tau_i=1-ce^{v_i}$. The preceding weighted mean identity shows
$\ln  c=\rho_*^{-1}\sum_i\kappa_i\ln (1-\tau_i)$, hence $c$ is
bounded above and below by positive constants. Subtracting the weighted
mean of $\dot L_i$ yields
\[
 \dot v_i=2c\Big(e^{v_i}
       -\rho_*^{-1}\sum_j\kappa_j e^{v_j}\Big)
       +E_i-\rho_*^{-1}\sum_j\kappa_jE_j.
\]
Pair with $\kappa_i v_i$ and use $\sum_i\kappa_i v_i=0$.
For $|v_i|\le C\epsilon$, the mean-value theorem yields
$v_i(e^{v_i}-1)\ge e^{-C\epsilon}v_i^2$.
This proves \eqref{eq:resonant-preparation-mass-face}.
On the fixed face, $A\le C\epsilon/\sqrt l$ and $T<\epsilon$ imply
uniform convergence of the error to zero as $L\to\infty$. For $p+q=1$ the
relative-mass space is zero-dimensional and there is no such face.

The preparation field also extends to the positive orthogonal directions.
Over the graph domain of Proposition~\ref{prop:resonant-graph-preparation}
there are uniform $r,r_+>0$ and fiber coordinates
\[
 y=\Phi_\theta(w-\bar w_\theta)\in E_\theta,\quad \|y\|<r_+,
\]
whose inverses take values in $\|w-\bar w_\theta\|<r$, such that
\begin{equation}\label{eq:resonant-exact-positive-fiber}
 J(\psi_\theta+w)=J(U(\theta))+\|y\|^2.
\end{equation}
They define a geometric tube in the Hilbert space by unique modulation.
On this tube the preparation field has a bounded locally Lipschitz
extension whose oriented trajectories satisfy, in parameter time,
\begin{equation}\label{eq:resonant-preparation-fiber-decay}
 -\frac{\d}{\d s}J(u(s))\ge c(A^2+T^2)+2\|y\|^2,
 \quad \frac{\d}{\d s}\|y\|^2=-2\|y\|^2,
\end{equation}
as long as $\max_i|a_i|l_i\ge\delta$ and the base remains in the
stated graph domain. The amplitude clocks and finite entry or exit
alternative of Proposition~\ref{prop:resonant-graph-preparation} are
unchanged. Here $\tau_i$ in the base field is evaluated at $U(\theta)$,
not at $u=\psi_\theta+w$.

Use the positive fiber coordinate constructed in the proof of
Lemma~\ref{lem:amplitude-collar-retraction}, with its radius reduced
if necessary on the present geometric domain. Denote that coordinate
by $y$ and its fiber map by $\Phi_\theta$.
The exact identity \eqref{eq:amplitude-collar-fiber} is precisely
\eqref{eq:resonant-exact-positive-fiber}, and its uniform inverse
and unique-modulation properties yield the asserted tube.

Let $P_\theta$ denote the orthogonal projection onto $E_\theta$.
Use the base equation \eqref{eq:graph-preparation-parameter-field}
and the following explicit transport in the total positive bundle:
\(\dot y=\dot P_\theta y-y, \quad \dot P_\theta=DP_\theta[\dot\theta].\)
Differentiating $P_\theta^2=P_\theta$ yields
$P_\theta\dot P_\theta P_\theta=0$.
Consequently the transport is tangent to the constraint
$P_\theta y=y$, and $\langle y,\dot P_\theta y\rangle=0$.
This proves the second identity in
\eqref{eq:resonant-preparation-fiber-decay}. The first follows by
exact differentiation of \eqref{eq:resonant-exact-positive-fiber}
and \eqref{eq:graph-preparation-descent}; there is no unestimated
moving-fiber error.

Push this bundle field forward by the inverse fiber coordinate and
modulation map, then divide its ambient value by one plus its norm.
The result is bounded and locally Lipschitz. The fiber norm decreases,
so it cannot cause exit through $\|y\|=r_+$. The base clocks are
unchanged in parameter time. On each finite base segment contained
in its domain the projection coefficients and coordinate derivatives
are bounded; the transported fiber stays bounded as well. The time
change therefore takes finite parameter segments to finite ambient
time segments. This proves the asserted entry or exit alternative.

The natural bubbling neighborhoods are contained in the geometric tubes.
\begin{cor}
\label{cor:resonant-geometric-coverage}
Fix $K_0\Subset\operatorname{int}K\Subset\Xi_{p,q}$ and fixed
geometric retained and positive-fiber radii as above. There is
$\epsilon_1>0$ such that every
$u\in V_{\rho_*}^{\eta}(p,q,\epsilon)$, $0<\epsilon<\epsilon_1$,
whose unique centers belong to $K_0$, belongs to the geometric tube
over $K$. In its coordinates,
\[
 |b|+|v|+\|y\|\le C\epsilon,
 \quad t\ge c|\ln \epsilon|.
\]
Fiber contraction keeps $(b,v,\xi,t)$ fixed and stays in that tube.
This statement does not require its contracted points to satisfy the
original gradient cutoff.
For every fixed $0<\delta<\delta_0$ in
Lemma~\ref{lem:amplitude-collar-retraction}, $\epsilon_1$ may be
chosen so that its deformation applies to every such point and
prepares it for the smaller terminal amplitude domain.
\end{cor}
\begin{proof}
Proposition~\ref{prop:unique} yields $\|w\|\le\epsilon$ and
$|a_i|\sqrt{l_i}\le C\epsilon$, where $l_i=\ln \lambda_i$.
Lemma~\ref{lem:lamexpansion}, applied to the original point $u$,
yields $|\tau_i(u)|\le C\epsilon$. Define $L_i$ from its parameters
as in the preparation proof. The identity
$e^{L_i}=8M(u)(1-\tau_i(u))/\rho_*$ implies
\[
 v_i=L_i-\rho_*^{-1}\sum_j\kappa_jL_j
 =\ln (1-\tau_i(u))-
   \rho_*^{-1}\sum_j\kappa_j\ln (1-\tau_j(u)).
\]
Thus $|v|\le C\epsilon$. The smooth center coefficients are bounded
on $K$, and $|a_i|\sqrt{l_i}\le C\epsilon$ implies
$L_i=2l_i+O(\epsilon\sqrt{l_i}+1)$. The preceding differences
$L_i-L_j=O(\epsilon)$ first yield comparability of all $l_i$, and
then comparability of $t=(2\rho_*)^{-1}\sum\kappa_iL_i$ and each
$l_i$. Consequently $b_i=a_i\sqrt t$ satisfies $|b|\le C\epsilon$.
The scale lower estimates in Proposition~\ref{prop:unique} yield
$t\ge c|\ln \epsilon|$. These choices place the parameters strictly
inside the fixed geometric domain when $\epsilon_1$ is small.

On those parameters, Proposition~\ref{prop:resonant-geometric-reduction}
yields $\|\bar w\|\le C(|a|+e^{-t/4})=o(1)$. Both $w$ and
$\bar w$ therefore lie in its fixed fiber ball. Strong fiber
convexity and $P_{E_\theta}\nabla J(\psi+\bar w)=0$ yield
\[
 c\|w-\bar w\|^2
 \le \langle P_{E_\theta}\nabla J(u),w-\bar w\rangle
 \le\epsilon\|w-\bar w\|.
\]
Thus $\|w-\bar w\|\le C\epsilon$, without assuming that
$\psi$ or $\psi+\bar w$ satisfies a gradient cutoff.
The uniform upper bound of the integral fiber Hessian then yields
$\|y\|\le C\epsilon$.
The exact fiber identity and the contraction
$y\mapsto(1-s)y$ prove the fiber-contraction statement.
For the preparation statement, choose the retained radii inside
those of Lemma~\ref{lem:amplitude-collar-retraction}.
Decrease $\epsilon_1$ until $C\epsilon_1$ is below each of these
radii and $c|\ln \epsilon_1|>T_\delta$.
Lemma~\ref{lem:amplitude-collar-retraction} then supplies the finite stopped deformation on the
geometric tube, with all $(v,\xi,t)$ fixed.
\end{proof}

The deformations in Corollary~\ref{cor:resonant-geometric-coverage}
are defined on the geometric tube and need not preserve its small-gradient
subset. The complete field in Section~\ref{sec:full-pseudogradient}
supplies the global orbit alternative.

The distinction between geometric boundaries and the gradient cutoff is useful here.
\begin{rem}
Fix $r>0$. The small-gradient set
$V_{\rho_*}^{\eta}(p,q,\epsilon)$ cannot contain the
whole fixed face $|v|_\kappa=r$ at arbitrary depth. In fact,
$|v|_\kappa=r$ implies $\max_i|\tau_i|\ge c r$ by the exact logarithmic
mass identity. Since $\max_i|a_i|=O(\epsilon/\sqrt l)$, some index has
$|a_i+\tau_i|\ge c r/2$ for large $l$. Combining
\eqref{eq:graph-preparation-pairings} yields
\[
 DJ(U)\Big[\frac{P\delta_i}{\sqrt{l_i}}\Big]
 =4\kappa_i\sqrt{l_i}(a_i+\tau_i)+O(\sqrt l R).
\]
Here $\sqrt l R\to0$ on the stated domain for this fixed face.
The normalized amplitude direction has bounded norm. Consequently
$\|\nabla J(U)\|\ge c_r\sqrt l$ for large $l$, contrary to
$U\in V_{\rho_*}^{\eta}(p,q,\epsilon)$.
Therefore the geometric domain
\eqref{eq:resonant-geometric-domain} omits the gradient cutoff.
Proposition~\ref{prop:resonant-geometric-reduction} verifies its
first-derivative and fiber estimates on that larger domain. Fixed guards
in that domain are therefore distinct from boundaries of the
small-gradient localization set.
\end{rem}

\section[Pseudogradient and deformation at infinity]{Pseudogradient near infinity and global deformation}
\label{sec:full-pseudogradient}

We work on the full Hilbert space at fixed resonance. The field
\eqref{eq:CPI-full-field} continues trajectories across bubble-chart exits;
Proposition~\ref{prop:pseudogradient} proves its orbit alternative.
Lemma~\ref{lem:amplitude-collar-retraction} supplies the relative
comparison of natural tubes and terminal neighborhoods.

\subsection{Bubble directions and the graph lift}
Write $a_i=\alpha_i-1$, $l_i=\ln \lambda_i$ and
$S_i=\lambda_i\partial_{\lambda_i}P\delta_{\lambda_i,\xi_i}$.
The amplitude and relative-mass directions on the bubble manifold are
\[
 A_i=\frac{P\delta_{\lambda_i,\xi_i}}{l_i}
                 -\frac{2}{\alpha_i}S_i,\quad
 V_{\rm am}=-\sum_i a_iA_i-\sum_i\frac{\tau_i}{\alpha_i}S_i .
\]
Thus $\alpha_i^{-1}S_i$ is the normalized direction
$\widehat S_i$ used in the proof of
Proposition~\ref{prop:resonant-graph-preparation}.
Their parameter expression is
\[
 \alpha_i'=-a_i/l_i,\quad
 (\ln \lambda_i)'=(2a_i-\tau_i)/\alpha_i^2,\quad \xi_i'=0.
\]
Thus these are the directions in
Proposition~\ref{prop:resonant-graph-preparation}, with its stated
domain and descent estimate. The correction is lifted as
$D\bar w[\alpha',\xi',\lambda']$; the derivative of the moving
orthogonal space is included in \eqref{eq:graph-tangent-comparison}.

For a center direction, write $\nabla_i^\top$ for the full surface
gradient at an interior center and the tangential gradient at a
boundary center. The corresponding normalized center and common-scale
directions are
\[
 C_i=\lambda_i^{-1}D_{\xi_i}P\delta_{\lambda_i,\xi_i},
 \quad
 V_{\rm sc}=-\operatorname{sgn}\mathcal L(\xi^*)
                    \sum_i\alpha_i^{-1}S_i .
\]
Center vectors are paired with $\nabla_i^\top\ln \mathcal F_\xi^i$.
Only tangential components are used on the boundary. The
common-scale direction is used near a reduced critical configuration;
its sign is determined by the differentiated scale expansion.
The scale function $h$ is defined in Section~\ref{sec:3}.
Section~\ref{sec:resonant-descent} supplies the estimates before entry;
the following construction concerns the end neighborhoods and their ambient complement.
\subsection{Finite-scale normal forms and the local descending field}

In the rest of this section the parameter is fixed at $\rho_*=4\pi m$,
and $J=J_{\rho_*}$. Conditions \ref{item:C1}--\ref{item:C2} are imposed
for every type $2p+q=m$. The separation of the critical configurations
in \cite[Lemma~4.1]{HB2024}, together with \ref{item:C1}, implies that their
unordered set is finite. All center neighborhoods below have compact
closure in their separated configuration stratum. A boundary center
has only its tangential coordinate. We use $Q(u)=\|u\|^2/2$.

The local construction follows \cite[Proposition~8]{ABL2017} and
\cite[Propositions~5.2--5.3 and Lemma~5.4]{AB2023}.
The height argument below is proved for this Neumann functional;
\cite{AKN2026} serves as an organizational comparison.

We construct the coordinates near each candidate end.
\begin{lem}\label{lem:CPI-coordinates}
Let $\xi^*$ be a nondegenerate critical configuration of $\FpqV$,
$n=p+q$, $d=2p+q$, and $k=\morse(\FpqV,\xi^*)$. Set
\[
 c_*=-\rho_*(1+\ln(\rho_*/8))-\tfrac12\FpqV(\xi^*),
 \quad \epsilon_*=-\sgn\cL(\xi^*).
\]
For sufficiently large finite $t$, there are coordinates
$(X_-,X_+,B,y,b,\mathsf z)$ in which
\begin{equation}\label{eq:CPI-exact-form}
 J=c_*-|X_-|^2+|X_+|^2+|B|^2-|y|^2
       +\epsilon_*b^2+\|\mathsf z\|^2,
 \quad b>0.
\end{equation}
Here $\dim X_-=d-k$, $\dim X_+=k$, $\dim B=n$, and
$\dim y=n-1$. The fiber is a Hilbert space with its positive norm.
The coordinates are smooth at every finite scale and are defined on
coordinate neighborhoods satisfying
\begin{equation}\label{eq:CPI-weighted-domain}
 |X_-|+|X_+|+|y|<r,\quad
 \sqrt{1+t}\,|B|<r,\quad \|\mathsf z\|<r,\quad t>T.
\end{equation}
The scalar $b=b(t)$ is strictly decreasing, and
\begin{equation}\label{eq:CPI-scale-comparison}
 b^2\asymp h(t),\quad
 -\frac{b}{b'(t)}\to 
  \left\{\begin{aligned}
  &2&&\text{ if }\,q>0,\\&1&& \text{ if }\,q=0.
\end{aligned}\right.
\end{equation}
No differentiable extension of these coordinates to $b=0$ is asserted.
If $\theta=(\alpha,v,\xi,t)$ denotes the original finite parameters,
then the finite coordinate vector
\begin{equation}\label{eq:CPI-finite-vector}
 Z_f=(X_-,-X_+,-B,y,-\epsilon_*b)
\end{equation}
satisfies
\begin{equation}\label{eq:CPI-parameter-speed}
 |Z_ft|+|Z_f\xi|+|Z_fv|\le C,
 \quad |\alpha-\mathbf1|+|Z_f\alpha|\le C/(1+t).
\end{equation}
Moreover, for $U=\psi(\alpha,\xi,\lambda)$,
\begin{equation}\label{eq:CPI-norm-scalar}
 c(1+t)\le1+Q(U)\le C(1+t),\quad |Z_f Q(U)|\le C.
\end{equation}
\end{lem}

\begin{proof}
We first ensure the regularity needed for a locally Lipschitz lifted
field. The functional $J$ and the projected bubbles are
smooth in their Hilbert and finite bubble variables at every finite
scale. In contrast, the point evaluation of a $C^3$ potential in
\eqref{eq:log-mass-coordinates} need not be smooth. We therefore use a
smooth finite-scale parametrization for this construction only.
Extend $V$ as a positive $C^3$ function to a neighborhood of the
double of the surface, and smooth it in fixed coordinate charts
at radius $\delta(t)=e^{-8t}$, with a partition of unity. Such an
extension is obtained in a boundary chart by a linear combination
of four reflected values, with coefficients satisfying
$\sum_{j=1}^4c_j(-j)^k=1$ for $k=0,1,2,3$; this matches the first
three boundary jets. Let $V_t$ denote the resulting smoothing.
Convolution and differentiation of its kernel yield
\[
 \|V_t-V\|_{C^2}+\sum_{j=1}^3\|\partial_t^jV_t\|_{C^2}
       \le Ce^{-8t},\quad \|V_t\|_{C^3}\le C.
\]
For the first estimate subtract the value of each derivative of
order at most two at the center of convolution; its increment is
bounded by $C\delta(t)$. A differentiated convolution kernel has
zero integral, giving the same estimate for each listed $t$ derivative.
The smoothing is positive for large $t$ and smooth in $(t,\xi)$.
In this construction replace \eqref{eq:log-mass-coordinates} by
\begin{equation}\label{eq:CPI-smooth-mass-coordinates}
 \ln\lambda_i=
 \frac{2t+v_i-\ln((\mathcal F^{i,V_t}_\xi
                         \mathrm g^i_\xi)(\xi_i))
                  +\ln(2\alpha_i-1)}{4\alpha_i-2},
 \quad \sum_i\kappa_i v_i=0.
\end{equation}
The functional still uses the original $V$; only its parametrization
has changed. Relative to the original mass coordinates the changes
in $t$ and $v$ are $O(e^{-8t})$, including two center derivatives
and the preceding scale derivatives. Indeed the logarithmic mass
increments are $\ln(V/V_t)(\xi_i)$; their weighted average yields
twice the change of $t$, and subtracting that average yields the
change of $v_i$. The derivative in the common scale is $1+O(e^{-8t})$,
so the change is invertible for large $t$.
This change of $(v,\xi,t)$ is independent of $\alpha$:
the increments $\ln (V/V_t)(\xi_i)$ and their weighted mean
contain no amplitude variable. Hence the fixed-mass preparation
in Lemma~\ref{lem:amplitude-collar-retraction} fixes the smoothed
base coordinates as well. After increasing its threshold, its target
lies strictly inside the terminal domain in either choice of coordinates.
We distinguish the approximation estimates from the additional
third-order estimates used for the coordinate changes. For the value,
first two common-scale derivatives, first center and mixed
common-scale/center derivatives, and second center derivatives in
\eqref{eq:terminal-value}--\eqref{eq:terminal-location-derivatives},
the change of parametrization contributes at most
$C(1+t)^3e^{-8t}=o(h(t))$. To obtain these comparisons,
differentiate the exact integral formulas before estimating their
remainders. Each difference contains a mass-coordinate increment
or one of its derivatives through two center orders; these have
size $O(e^{-8t})$. The remaining differentiated factors have the
polynomial estimates in $t$ from Appendix~\ref{app:0}. The inverse
fiber Hessian is uniformly bounded, as are the inverses of the
normalized fast Hessians. Differentiating their defining equations
therefore preserves these orders when the fiber and fast variables
are eliminated. Thus Lemma~\ref{lem:terminal-reduction} applies in
the smoothed coordinates with the same leading coefficients.

For three center derivatives only uniform estimates are required below.
They follow directly from $\|V_t\|_{C^3}\le C$ and the fixed-coordinate
integral estimates, rather than from an exponential $C^3$ approximation
rate. More precisely, if $\omega_3(r)$ is a common modulus of continuity
of the third derivatives of the extended $V$ in the finitely many
charts, convolution yields
\[
 \|V_t-V\|_{C^3}\le C\bigl(\omega_3(Ce^{-8t})+e^{-8t}\bigr)=o(1).
\]
Indeed, at third order subtract the third derivative at the center
inside the convolution; all terms where a derivative falls on a
fixed chart coefficient or partition of unity are controlled by the
lower-order estimate. No rate for $\omega_3$ is needed. On a compact subset of
$C^3_+(\Si)$ the modulus and constants may be chosen uniformly.
The bounded third-order estimates are proved directly below. All coordinate
changes can now be chosen smooth at finite scale for the stated
$C^3$ potential.

First minimize the
fiber as in Lemma~\ref{lem:terminal-reduction}. For $z=w-\bar w$,
Taylor's formula on the fixed orthogonal fiber yields
\[
 J(U+\bar w+z)-J(U+\bar w)
 =\langle A_\theta(z)z,z\rangle,
 \quad
 A_\theta(z)=\int_0^1(1-s)D_w^2J(U+\bar w+sz)\,\d s.
\]
Uniform coercivity and the cubic remainder estimate yield
$cI\le A_\theta(z)\le CI$. Thus
$\mathsf z=A_\theta(z)^{1/2}z$ is a local diffeomorphism, uniformly in its
fiber radius: its derivative is $A_\theta(0)^{1/2}+O(\|z\|)$.
This proves the exact positive fiber term and
$\|z\|\asymp\|\mathsf z\|$. The construction is made on the original
orthogonal fibers, not by differentiating composition by a moving
surface diffeomorphism on $H^1$.

For clarity, record the derivative information used for the remaining
changes. In the variables of \eqref{eq:log-mass-coordinates}, set
$a=\alpha-\mathbf1$ and $\zeta=(v,\xi,t)$. The integral formulas in
Appendix~\ref{app:0} yield, on $t|a|+|v|<\delta$,
\begin{equation}\label{eq:CPI-splitting-estimates}
\begin{gathered}
 j_{\alpha\alpha}=4t\operatorname{diag}(\kappa_i)+O(1+t^2h(t)),
 \quad
 j_\alpha(\mathbf1,v,\xi,t)=O(th(t)+e^{-t}),\\
 D_\zeta j_\alpha(\mathbf1,v,\xi,t)=O(th(t)+e^{-t}),
 \quad j_{vv}|_{\sum\kappa_i v_i=0}\le-cI,
\end{gathered}
\end{equation}
after taking the amplitude Schur complement in the last inequality.
The normalized finite-variable derivatives needed to differentiate
these splittings once are bounded. Here is an explicit way to check
the additional differentiation, rather than differentiating an
undifferentiated error term. Differentiate
\eqref{eq:A-logmass-leading-energy} and the exact local mass factors
before integration. A logarithmic scale derivative replaces $K_\lambda$
by $K_\lambda q_j(\lambda^2|y|^2)$ with bounded $q_j$.
An amplitude derivative introduces $O(t)$ and
$\ln(1+\lambda^2|y|^2)$. Through three derivatives their integrals
are bounded by \eqref{eq:A-logarithmic-kernel-bounds}, with $k=3$
proved by the same radial substitution. After rescaling amplitude
increments by $\sqrt t$, the third derivatives of the leading
amplitude term are $O(t^{-1/2})$, and the error terms are bounded by
$C(1+t^{3/2}h(t))$. Center derivatives of total order at most three
fall on the smooth coefficients in the fixed coordinate integrals;
their estimates require only $V\in C^3$. For the fiber value, differentiate
its projected equation on the fixed pulled-back forms used in
\eqref{eq:A-pulledback-correction-bounds}; each differentiated equation
has the same uniformly invertible fiber Hessian. The chain rule at
its zero eliminates the highest derivative of the minimizing
correction. More explicitly, write the pulled-back functional as
$F(\theta,\eta)$, its minimizing correction as $\eta(\theta)$,
and $C=F_{\eta\eta}$. At that correction
$C\eta_a=-F_{\eta a}$, and its reduced value satisfies
\[
\begin{aligned}
j_{abc}={}&F_{abc}
 +F_{ab\eta}\eta_c+F_{ac\eta}\eta_b+F_{bc\eta}\eta_a\\
&+F_{a\eta\eta}[\eta_b,\eta_c]
 +F_{b\eta\eta}[\eta_a,\eta_c]
 +F_{c\eta\eta}[\eta_a,\eta_b]
 +F_{\eta\eta\eta}[\eta_a,\eta_b,\eta_c].
\end{aligned}
\]
Thus no third derivative of $\eta$ is needed. For three center
derivatives of the coefficients, including the smoothed mass
parametrization, use the uniform $C^3$ bound without a spatial Taylor
remainder. The corresponding third derivatives of the residual
have bounded norm; neither an $O(e^{-t})$ residual bound nor an
$O(e^{-8t})$ comparison with the original third derivatives is used.
They are multiplied by the small correction in the scalar value. Amplitude derivatives introduce at most the
powers of $t$ just recorded; the correction and its first two
derivatives have the estimates in
\eqref{eq:A-pulledback-correction-bounds} and the amplitude estimates
preceding it. The terms in this identity remain bounded on $t|a|<\delta$
after division by $\sqrt t$ for each amplitude index.
This proves the stated third-derivative estimates. Applying the same
identity to the finite amplitude and relative-scale eliminations
yields bounded $D_tD_\xi^2K$ and $D_\xi^3K$.
Only scalar derivatives are estimated on the pulled-back forms.

For each $(v,\xi,t)$ solve $j_\alpha=0$ first, obtaining
$\alpha^\sharp(v,\xi,t)$. The first three estimates in \eqref{eq:CPI-splitting-estimates}
and the inverse amplitude Hessian yield, after moving the
$o(1)|a|$ term to the left-hand side,
\begin{equation}\label{eq:CPI-amplitude-graph}
 |\alpha^\sharp-\mathbf1|+|D_\zeta\alpha^\sharp|
 \le C(h(t)+e^{-t}/t).
\end{equation}
The segment integral of $j_{\alpha\alpha}$ about this graph is
positive; its positive square root defines $B$, exactly as in the
fiber construction. Thus
$j=j(\alpha^\sharp,v,\xi,t)+|B|^2$ and
$|B|\asymp\sqrt t\,|\alpha-\alpha^\sharp|$.
Next minimize the negative of $j(\alpha^\sharp,v,\xi,t)$ in the
constrained relative-scale space. Its Hessian is uniformly positive.
Its square-root construction yields
\[
 j=K_{\rho_*}(\xi,t)+|B|^2-|y|^2.
\]
Both constructions keep $(\xi,t)$ fixed. Their derivatives and inverse
derivatives in the normalized variables are bounded by
\eqref{eq:CPI-splitting-estimates}. In the amplitude inverse its
$\zeta$-derivatives are bounded by
$C(h(t)+e^{-t}/t+|B|/\sqrt t)$.

The center equation has a unique zero $\xi(t)$ near $\xi^*$, with
$|\xi(t)-\xi^*|+|\xi'(t)|\le Ch(t)$, by
\eqref{eq:terminal-location-derivatives}. Set $k(t)=K_{\rho_*}(\xi(t),t)$.
Taylor's integral formula about $\xi(t)$ writes
$K_{\rho_*}-k(t)$ as a quadratic form with coefficient matrix
converging to $-D^2\FpqV(\xi^*)/4$.
Choose a fixed basis diagonalizing this limiting matrix. Successive
completion of squares, with all pivots bounded away from zero,
defines a smooth congruence to $\operatorname{diag}(-I,I)$.
The resulting map from the center displacement to $(X_-,X_+)$
has derivative uniformly close to its invertible derivative at zero,
after shrinking the center neighborhood. It is a local diffeomorphism,
and its $t$-derivative is bounded by the derivative estimates above.
There are $d-k$ negative and $k$ positive center directions.

The chain rule at $\xi(t)$ and Lemma~\ref{lem:terminal-reduction} yield
\[
 k(t)-c_*=-a_*\cL(\xi^*)h(t)+o(h(t)),\quad
 k'(t)=-a_*\cL(\xi^*)h'(t)+o(h(t)),
\]
where $a_*=4\pi$ if $q>0$ and $a_*=1/2$ if $q=0$.
The second derivative has the corresponding estimate, since its
center Schur term is $O(h(t)^2)$. Define
$b(t)=\sqrt{\epsilon_*(k(t)-c_*)}$. Since
$h'/h\to-1$ or $-2$, respectively, this proves
\eqref{eq:CPI-scale-comparison} and the exact form.
The amplitude radius in \eqref{eq:CPI-weighted-domain} is essential:
it ensures $t|\alpha-\mathbf1|<\delta$ throughout the construction.
The inverse changes, \eqref{eq:CPI-amplitude-graph}, and
$|Z_fb|=b$ yield \eqref{eq:CPI-parameter-speed}.

Finally, the Dirichlet part of \eqref{eq:A-logmass-leading-energy}
and the differentiated projection estimates yield
\[
 Q(U)=2t\sum_i\kappa_i\frac{\alpha_i^2}{2\alpha_i-1}
           +Q_0(\alpha,v,\xi)+o(1),
\]
with bounded $Q_0$ and bounded derivatives in the directions under
consideration. The remainder has bounded derivative along $Z_f$ by
the same exact kernel estimates. Since
$\alpha^2/(2\alpha-1)=1+(\alpha-1)^2/(2\alpha-1)$,
\eqref{eq:CPI-parameter-speed} proves both statements in
\eqref{eq:CPI-norm-scalar}.
\end{proof}

The exact coordinates determine the local descending field.
\begin{prop}\label{prop:CPI-local-pseudogradient}
At every candidate end there is a bounded locally Lipschitz descending
field with the following properties. A forward orbit trapped in a
sufficiently small end neighborhood escapes to infinite scale if and only if
$\mathcal L(\xi^*)<0$ and its negative coordinates vanish. Each such
configuration has a nonempty local stable set. For
$\mathcal L(\xi^*)>0$ every orbit leaves the neighborhood.
\end{prop}

\begin{proof}
Use the finite vector \eqref{eq:CPI-finite-vector} and fiber contraction
$\nabla_{Z_f}\mathsf z=-\mathsf z$ in an isometric trivialization of the fiber bundle.
Equivalently, one may use the metric connection induced by its
orthogonal projections. Its parallel transport preserves the fiber
norm. These are smooth operations at finite scale; no differentiability
of a translation operator on $H^1$ is required. Denote the lifted
field by $Z$ and set
\[
 W_* =\Theta Z,\quad
 \Theta=(1+\|Z\|+|Z_f|+\|\mathsf z\|)^{-1}.
\]
The locally Lipschitz field $W_*$ satisfies $\|W_*\|\le1$ and
\begin{equation}\label{eq:CPI-descent}
 -DJ[W_*]=2\Theta
 (|X_-|^2+|X_+|^2+|B|^2+|y|^2+b^2+\|\mathsf z\|^2)>0.
\end{equation}
In the coordinate time $\sigma$, defined by $\d\sigma/\d s=\Theta$,
the finite equations are precisely \eqref{eq:CPI-finite-vector},
and $\|\mathsf z(\sigma)\|=e^{-\sigma}\|\mathsf z(0)\|$.
For $\epsilon_*=1$, trapping implies $X_-=y=0$; all other
coordinates then tend to zero and $b\to0$. This supplies actual
orbits, including those with $B=\mathsf z=X_+=0$. If $\epsilon_*=-1$,
$b$ increases and the finite-scale face is reached unless an unstable
face is reached earlier. Infinite scale cannot be reached in finite
Hilbert time because $|\dot b|\le b$. On a trapped orbit coordinate
time is unbounded: a finite limit of coordinate time implies a positive
limit of $b$ and finite limits of all coordinates and their fiber
transports. The normalization remains positive at the limiting point,
so the coordinate flow extends further. Thus the asserted limits are limits of the complete flow.

Along a trapped orbit $t=O(1+\sigma)$ and
$|B|+\|\mathsf z\|+|X_+|=O(e^{-\sigma})$.
The original parameters consequently satisfy
$\alpha\to\mathbf1$, $v\to0$, $\xi\to\xi^*$,
$\lambda_i\to\infty$, and $w\to0$.
The normalized frame estimates, the projected fiber equation, and
the exponential-integral expansion yield $\|\nabla J\|\to0$:
the finite pairings tend to zero by the gradient expansions, while
the fiber component is bounded by $C(\|\mathsf z\|+e^{-t})$; the possible
amplitude-normalization factors grow only polynomially in $t$ and
have product $o(1)$ with $e^{-\sigma}$. The energy tends to $c_*$.
\end{proof}

\subsection{The field in the original bubble parameters}

We express the local field in the original bubble parameters. This also identifies
the combined amplitude--scale direction and the separate scale and tangential-center directions
used in the preparation construction. In this subsection $b$ is the
positive radial coordinate in \eqref{eq:CPI-exact-form}; the vector
$B$ comprises the positive amplitude coordinates. Write $P_i=
P\delta_{\lambda_i,\xi_i}$ and $l_i=\ln \lambda_i$.
Let $\mathcal N$ denote the finite-dimensional coordinate map
constructed by the three integral quadratic forms in
Lemma~\ref{lem:CPI-coordinates}. Its input is
$\theta=(\alpha,v,\xi,t)$ and its output is
$(X_-,X_+,B,y,b)$. Define its parameter velocity by
\begin{equation}\label{eq:explicit-inverse-coordinate-field}
 \dot\theta=(D\mathcal N(\theta))^{-1}
           (X_-,-X_+,-B,y,-\epsilon_*b).
\end{equation}
The inverse in this formula is a finite matrix inverse at a positive
finite scale. The construction of $\mathcal N$ fixes this matrix;
there is no unspecified choice of a descent direction.

Set
\[
 A_i(\alpha,\xi,t)=
 \ln \bigl((\mathcal F_\xi^{i,V_t}\mathrm g_\xi^i)(\xi_i)\bigr).
\]
The dependence on $t$ here is solely the smoothing used in
\eqref{eq:CPI-smooth-mass-coordinates}. At finite scale the following
formulas determine every original parameter velocity:
\begin{equation}\label{eq:explicit-original-scale-velocity}
\begin{split}
 \dot l_i=\frac{1}{4\alpha_i-2}\Big[
 2\dot t+\dot v_i-D_\alpha A_i[\dot\alpha]
 -D_\xi A_i[\dot\xi]-(\partial_t A_i)\dot t
 +\Big(\frac{2}{2\alpha_i-1}-4l_i\Big)\dot\alpha_i
 \Big],\quad
 \dot\lambda_i=\lambda_i\dot l_i.
\end{split}
\end{equation}
For a boundary center $\dot\xi_i$ is tangential to $\partial\Sigma$.
The combined amplitude--scale and normalized scale directions are
\begin{equation*}
 \mathscr A_i=\frac{P_i}{l_i}
       -\frac{2}{\alpha_i}\lambda_i\partial_{\lambda_i}P_i,
 \quad
 \mathscr S_i=\frac{1}{\alpha_i}
                    \lambda_i\partial_{\lambda_i}P_i,
 \quad
 \mathscr C_{ik}=\frac{1}{\lambda_i}\partial_{(\xi_i)_k}P_i.
\end{equation*}
In the last formula $k$ ranges over two interior coordinates or
one boundary coordinate. Define
\begin{equation}\label{eq:explicit-bubble-velocity}
 Z_{\rm bub}=\sum_i l_i\dot\alpha_i\mathscr A_i
 +\sum_i(\alpha_i^2\dot l_i+2l_i\dot\alpha_i)\mathscr S_i
 +\sum_{i,k}\alpha_i\lambda_i\dot{(\xi_i)_k}\mathscr C_{ik}.
\end{equation}
Thus all three types of bubble directions occur with specified
coefficients. For the coupled preparation field the parameter equations
\eqref{eq:graph-preparation-parameter-field} instead yield the particularly
simple identity
\[
 Z_{\rm bub}=-\sum_i(\alpha_i-1)\mathscr A_i
                         -\sum_i\tau_i\mathscr S_i.
\]
The descent and the amplitude clock for this latter field were proved
in Proposition~\ref{prop:resonant-graph-preparation}.
For the fixed-mass preparation whose stopped flow retracts the
amplitude collar, use instead
\eqref{eq:amplitude-collar-field} and
\eqref{eq:amplitude-collar-scale-lift}.
They determine the same bubble-direction coefficients in
\eqref{eq:explicit-bubble-velocity}. Its positive-fiber velocity is
$DP_\theta[\dot\theta]\mathsf z-\mathsf z/t$;
Lemma~\ref{lem:amplitude-collar-retraction} proves its full ambient
descent and the continuous finite stopping rule.

To specify the normal part of the local end field, let
$P_\theta$ be the orthogonal projection onto $E_\theta$ and
$\Phi_\theta$ the exact positive-fiber map from the proof of
Lemma~\ref{lem:CPI-coordinates}. For
$\mathsf z=\Phi_\theta(w-\bar w_\theta)$, set
\begin{equation}\label{eq:explicit-normal-velocity}
 \dot{\mathsf z}=DP_\theta[\dot\theta]\mathsf z-\mathsf z,
 \quad
 Z=Z_{\rm bub}+D\bar w_\theta[\dot\theta]
       +D\bigl[(\theta,\mathsf z)\mapsto
                    \Phi_\theta^{-1}(\mathsf z)\bigr]
                     [\dot\theta,\dot{\mathsf z}].
\end{equation}
The final derivative is taken on the total normal bundle.
Equations \eqref{eq:explicit-inverse-coordinate-field}--
\eqref{eq:explicit-normal-velocity}, followed by the positive
normalization in Proposition~\ref{prop:CPI-local-pseudogradient}, define $W_*$
in the original Hilbert space.

We verify the exact descent of the field in the original parameters.
On every local end neighborhood, these formulas yield the field of
Proposition~\ref{prop:CPI-local-pseudogradient}. It is bounded and locally
Lipschitz, and its descent is exactly \eqref{eq:CPI-descent}.
For a retained end, in coordinate time,
\[
 \dot X_-=X_-,\quad \dot y=y,\quad
 \dot X_+=-X_+,\quad \dot B=-B,\quad
 \dot b=-b,
 \quad \frac{\d}{\d\sigma}\|\mathsf z\|^2=-2\|\mathsf z\|^2.
\]
For an end with $\mathcal L>0$, the equation for $b$ is
$\dot b=b$. No estimate discards the derivative of $\bar w$
or the derivative of the normal projection.

Differentiate
$(4\alpha_i-2)l_i=2t+v_i-A_i+\ln (2\alpha_i-1)$.
This yields \eqref{eq:explicit-original-scale-velocity}, including
the term $-(\partial_tA_i)\dot t$. Expanding the first two sums
of \eqref{eq:explicit-bubble-velocity} cancels the terms
$2l_i\dot\alpha_i\lambda_i\partial_{\lambda_i}P_i/\alpha_i$,
leaving
\[
 Z_{\rm bub}=\sum_i\dot\alpha_i P_i
 +\sum_i\alpha_i\dot l_i\lambda_i\partial_{\lambda_i}P_i
 +\sum_{i,k}\alpha_i\dot{(\xi_i)_k}\partial_{(\xi_i)_k}P_i
 =D\psi_\theta[\dot\theta].
\]
Since $P_\theta(DP_\theta)P_\theta=0$, the first equation in
\eqref{eq:explicit-normal-velocity} preserves
$P_\theta\mathsf z=\mathsf z$ and yields the asserted norm identity.
The second equation is consequently the derivative of
$u=\psi_\theta+\bar w_\theta+\Phi_\theta^{-1}(\mathsf z)$.
By \eqref{eq:explicit-inverse-coordinate-field} all finite-coordinate
equations are precisely the preceding formulas. Differentiating the exact
identity \eqref{eq:CPI-exact-form} yields
\[
 -DJ[Z]=2\bigl(|X_-|^2+|X_+|^2+|B|^2+|y|^2
                             +b^2+\|\mathsf z\|^2\bigr).
\]
Multiplication by the positive normalization proves
\eqref{eq:CPI-descent}. Smoothness of the coordinate maps and
orthogonal normal projections at finite scale proves local Lipschitz
regularity; the chosen normalization keeps the Hilbert norm at most one.
The sign and scale orders are therefore $b^2\asymp e^{-t}$
in the boundary case and $b^2\asymp te^{-2t}$ in the pure-interior
case, as already established in Lemma~\ref{lem:CPI-coordinates}.

\subsection{Auxiliary height and exhaustion of its exact critical points}

Only the ends with $\mathcal L<0$ will be retained in the global
construction. On one such end, set
\[
 \xi_-=(X_-,y),\quad \eta_0=(X_+,\mathsf z),\quad
 H=|\xi_-|^2,\quad G=\|\eta_0\|^2+(1+t)|B|^2+b^2.
\]
In coordinate time the local field satisfies
\begin{equation*}
 H'=2H,\quad G'\le-G,\quad J'=-2D,
 \quad D=H+\|\eta_0\|^2+|B|^2+b^2,
\end{equation*}
after increasing $T$: the derivative of $(1+t)|B|^2$ is
$(-2+t'/(1+t))(1+t)|B|^2$, and $t'$ is bounded.
Choose nested coordinate neighborhoods
\begin{equation}\label{eq:CPI-blocks}
 \mathscr B_i^j=\{H_i\le R_j^2,\ G_i\le r_j^2,\ b_i>0\},
 \quad j=0,-,+,
\end{equation}
with strictly increasing radii and positive coordinate margins.
Choose the outer radii sufficiently small that the finite-coordinate
neighborhoods have closure in the domains of Lemma~\ref{lem:CPI-coordinates}.
At the same time increase their scale thresholds until $b<r_+$
implies $t>T$. All patches are disjoint after passing to unordered
configurations. Unique modulation and the separated center
neighborhoods ensure this also as subsets of $\oH$.

We next construct an auxiliary family adapted to these end neighborhoods.
\begin{lem}\label{lem:CPI-height}
The neighborhoods can be chosen with a smooth height $Q_*\ge1$, comparable
to $1+\|u\|^2$, constant on a ball containing $\operatorname{Crit}(J)$,
such that the following properties hold. Write
$J_s^*=sJ+(1-s)Q_*$.
\begin{enumerate}[label=\textup{(\roman*)}]
\item For some $\varepsilon>0$, every bounded sequence with
$s_n\in[1-2\varepsilon,1]$ and $\nabla J_{s_n}^*(u_n)\to0$
has a strongly convergent subsequence.
\item For the same interval of parameters,
\[
 \operatorname{Crit}(J_s^*)\subset
 \operatorname{Crit}(J)\cup\bigcup_i\mathscr B_i^0.
\]
\item On each collar $\mathscr B_i^+\setminus\mathscr B_i^0$,
\begin{equation}\label{eq:CPI-local-growth}
 D\ln  Q_*[W_*]\le C(-DJ[W_*]).
\end{equation}
The normalized local descent has a positive lower bound on bounded
closed portions of that collar.
\end{enumerate}
No nondegeneracy of the genuine critical points is required.
\end{lem}

\begin{proof}
The exact resonant critical set is compact by
Remark~\ref{rem:fixed-resonance-compactness} and
Lemma~\ref{lem:B-compact-gradient}. Choose $R$ containing the critical set, and
choose a smooth nondecreasing $f$ with $f'=0$ on $[0,2R^2]$,
$f'=1$ on $[4R^2,\infty)$, $0\le f'\le1$, and $f(q)\ge1+q$.
The last condition holds after increasing the additive constant.
Take the end thresholds beyond this ball. On an end set
\[
 q_i=Q(U)+\|\mathsf z\|^2,\quad
 \Delta_i=q_i-Q(u)=\|\mathsf z\|^2-Q(w).
\]
The orthogonality $\langle U,w\rangle=0$ is used here.
Choose smooth cutoffs $\chi_i$ equal to one on neighborhoods of
$\mathscr B_i^+$, with disjoint larger supports in the coordinate
patches, and define
\begin{equation*}
 Q_*=f(Q)+C_*+\sum_i\chi_i\Delta_i.
\end{equation*}
Cutoffs in finite variables use $t^2|\alpha-\mathbf1|^2$, $|v|^2$,
and center coordinates, together with a large-scale cutoff.
The fiber cutoff uses $\|w\|^2$ and equals one on a fixed smaller
ball. The radii can be chosen so all neighborhoods and their collars are
inside the region where $\chi_i=1$.
Coercivity yields a uniform estimate for $\Delta_i$; choose
$C_*>1+\sum_i\|\Delta_i\|_\infty$. Then
$c(1+\|u\|^2)\le Q_*\le C(1+\|u\|^2)$, and on the neighborhoods
$Q_*=q_i+$ a constant.

We first verify compactness on bounded sets for this actual family.
The exact fiber identity yields
\begin{equation}\label{eq:CPI-height-compact-identity}
 \Delta_i=Q(U)-j_{\rho_*}(\theta)
             -\rho_*\ln\int_\Sigma Ve^u\,\d v_g.
\end{equation}
On bounded parts of a separated end chart the scales are bounded
by \eqref{eq:CPI-norm-scalar}. Gradients of the finite parameter
functions lie in finite spans of bubbles and their parameter
derivatives, with bounded coefficients obtained by differentiating
the modulation equations. The normalized Gram inverse is uniformly
bounded on those bounded parts. These spans form relatively compact
families. The last term in
\eqref{eq:CPI-height-compact-identity} has compact gradient by
Lemma~\ref{lem:B-compact-gradient}. Thus $\nabla\Delta_i$ is compact
on bounded sets. The fiber-radius cutoff can contribute a multiple
of $u$, because $Q(w)=Q(u)-Q(U)$. Its coefficient is uniformly
bounded; every other cutoff-gradient term is compact on bounded sets.
Consequently,
\[
 \nabla Q_*(u)=(f'(Q(u))+a(u))u+K_0(u),\quad |a(u)|\le C_0,
\]
where $K_0$ maps bounded sets to relatively compact sets. Since
$\nabla J=u-T_{\rho_*}(u)$, we have
\[
 \nabla J_s^*=[s+(1-s)(f'(Q)+a)]u+K_s(u),
\]
with $K_s$ collectively compact on bounded sets. Take
$0<\varepsilon<1/4$ and $2\varepsilon(1+C_0)<1/2$. Extracting the scalar coefficient and
the compact term in this identity proves (i).

We next prove that every unbounded sequence of exact critical points
of $J_s^*$ with $s\uparrow1$ enters every fixed inner neighborhood belonging
to its limiting end. Outside the height cutoffs, $f'=1$ at large
norm and the equation is exactly $\nabla J_{s\rho_*}=0$.
Theorem~\ref{thm:alter_blowup_sols} and
Proposition~\ref{prop:bubbling-branches} place such a sequence in a balanced
chart with $\mathcal L<0$ and all its end coordinates tending to zero.
The cutoff transition regions need a separate verification;
the PDE theorem is not applied to a modified equation there.

Set $c=1-s$, $a=\alpha-\mathbf1$, $e=|a|+e^{-t}$ and
$\zeta=(\xi,t,v)$. All constants below are uniform for
$0\le c\le c_0$ on the fixed cutoff supports, after choosing
$c_0>0$ small and then the scale threshold large.
The fiber radial equation and uniform coercivity first yield
$\|w\|\le Ce$ at a fiber critical point. To see this also in the
fiber cutoff collar, pair the equation with $w$: the $J$ term is
at least $c_{\rm f}\|w\|^2-Ce\|w\|$ for a uniform $c_{\rm f}>0$,
whereas the height perturbation is bounded by
$Cc\|w\|^2+Cc e\|w\|$ there. Choose $c$ so small that $Cc\le c_{\rm f}/2$.
The critical equation then implies
$(c_{\rm f}/2)\|w\|^2\le Ce\|w\|$, hence $\|w\|\le Ce$.
Thus an unbounded critical sequence eventually lies where the
fiber cutoff is identically one. Its fiber equation is
\[
 (s+c\chi_i)D_wJ+c(1-\chi_i)\langle w,\cdot\rangle=0.
\]
Its Hessian is uniformly positive, and its unique solution $w_s$
satisfies $\|w_s-\bar w\|\le Cce$.
Evaluation at that solution yields the reduced scalar function
\begin{equation}\label{eq:CPI-modified-reduction}
 j_s^*=s j_{\rho_*}+c\{Q(U)+(1-\chi_i)S\}+E_s,
 \quad S=Q(\bar w),\quad |E_s|\le Cc^2e^2.
\end{equation}
The harmless additive constant in the height has been omitted.

All derivatives needed in this equation can be estimated as scalar
contractions on the fixed forms of Appendix~\ref{app:0}. In particular,
\begin{equation}\label{eq:CPI-height-error-derivatives}
\begin{gathered}
 |D_\zeta^jS|\le Ce^2\quad(0\le j\le2),\quad
 |D_\alpha S|+|D_\zeta D_\alpha S|\le Ce(1+te),\quad
 \|D_\alpha^2S\|\le C(1+te)^2,\\
 |D_\zeta^jE_s|\le Cc^2e^2\quad(0\le j\le2),\quad
 |D_\alpha E_s|\le Cc^2e(1+te),\quad
 \|D_\alpha^2E_s\|\le Cc^2(1+te)^2.
\end{gathered}
\end{equation}
For $S$, these follow by differentiating
$a_\xi(T_\theta\bar\eta,T_\theta\bar\eta)/2$ and applying
\eqref{eq:A-pulledback-correction-bounds}; the amplitude estimates
in Step 2 of that proof supply the second bound.
For the second amplitude derivative, the fixed-space Euler equation yields
$\|D_\alpha\bar\eta\|\le C(1+te)$ and
$\|D_\alpha^2\bar\eta\|\le C(t+t^2e)$; the product rule for
$S=a_\xi(T_\theta\bar\eta,T_\theta\bar\eta)/2$ yields the stated estimate.
For $E_s$, Taylor's integral formula in $w_s-\bar w$ and the
differentiated fiber equation yield the same estimates with the factor
$c^2$. Finite cutoff derivatives cost $Ct$ per amplitude derivative,
$C$ per $(v,\xi)$ derivative and $C/t$ per scale derivative on their
supports. Thus multiplication by $1-\chi_i$ preserves these orders.

To compare the full reductions, the minimizing fiber for $J_{s\rho_*}$
satisfies $sD_wJ+c\langle w,\cdot\rangle=0$. Coercivity and the same
differentiation yield $\|\bar w_{s\rho_*}-\bar w\|\le Cce$ and
\[
 j_{s\rho_*}=s j_{\rho_*}+cQ(U)+cS+\widehat E_s,
 \quad j_s^*=j_{s\rho_*}-c\chi_iS+\widetilde E_s,
\]
where $\widehat E_s$ and $\widetilde E_s=E_s-\widehat E_s$ satisfy
the estimates for $E_s$ in \eqref{eq:CPI-height-error-derivatives}.
In variables $(\sqrt t\,a,v)$, the perturbation of the fast Hessian
has amplitude submatrix $O(c(1+te)^2/t)$, mixed submatrix
$O(ce(1+te)/\sqrt t)$ and mass submatrix $O(ce^2)$.
Since $t|a|<\delta$, all tend to zero uniformly for $c\le c_0$.
The uniform invertibility used in Step~4 of Lemma~\ref{lem:terminal-reduction}
therefore persists through the cutoff transitions.
For the fast equations, divide each amplitude equation by $t$ as
in Step~4 of Lemma~\ref{lem:terminal-reduction}.
The additional amplitude error is bounded by
\[
 Cc\,e(1+te)/t=Cc(e/t+e^2).
\]
On $t|a|<\delta$,
$e/t+e^2\le(1+2\delta+2te^{-t})|a|/t+e^{-t}/t+2e^{-2t}$.
Choose $T$ so the coefficient of $|a|$, including the original
$o(1)|a|$, is below $1/2$ uniformly for $c\le c_0$.
The amplitude equation then yields $|a|\le C(h+e^{-t}/t)$.
The relative-mass equations have the additional error $O(ce^2)$;
their uniform negative Hessian yields $|v|\le Ch$.
Coercivity yields $\|w_s\|\le Ce\le Ce^{-t}$.
These estimates are obtained before any comparison between $c$ and $h$.
At a fast critical point the scale derivative is the partial derivative,
so the additional terms in \eqref{eq:CPI-modified-reduction} contribute
$O(ce^{-2t})$. Indeed, \eqref{eq:CPI-height-error-derivatives} and
$|\partial_t\chi_i|\le C/t$ apply uniformly through the transition region, and
\[
 \frac{ce^{-2t}}{h(t)}=
\left\{\begin{aligned}
    &ce^{-t}&&\text{ if }\, q>0,\\ &c/t&& \text{ if }\,q=0. 
\end{aligned}\right.
\]
Together with the original $O(ch)$ parameter error, this yields
$o(h)+O(ch)$ uniformly for $c\le c_0$.
The center equation tends to $D\FpqV=0$, so its limit is the unique
critical configuration in the chosen patch. The common-scale equation is
\begin{equation*}
 0=2\rho_*c-a_*\mathcal L(\xi^*)h'(t)+o(h(t))+O(ch(t)).
\end{equation*}
Since $h'/h$ has a strictly negative limit and $\mathcal L<0$
on these patches, this implies $0\le c\le Ch(t)$.
All weighted amplitude, relative-scale, center, and fiber coordinates
therefore tend to zero, and $b\to0$. The critical sequence enters
$\mathscr B_i^0$ as asserted.

If (ii) fails for every smaller $\varepsilon$, a sequence of exact
auxiliary critical points with $s\to1$ outside the stated union
is either unbounded, contradicting the preceding exhaustion,
or bounded. In the latter case (i) yields a limit in
$\operatorname{Crit}(J)$, inside the ball where $Q_*$ is constant.
There $\nabla J_s^*=s\nabla J$, so the sequence itself eventually
belongs to $\operatorname{Crit}(J)$, again a contradiction.

Finally, on an end collar
$Q_*\asymp1+t+\|\mathsf z\|^2$, and
\eqref{eq:CPI-norm-scalar} yields
$|Dq_i[W_*]|\le C\Theta(1+\|\mathsf z\|^2)$.
Outside the inner neighborhood, $H+G\ge d_i>0$, hence
$D\ge c d_i/(1+t)$. This and \eqref{eq:CPI-descent} imply (iii).
On a bounded part of the collar, $t$ is bounded, the Hilbert-space
coordinate maps have bounded derivatives, and $\Theta$ has a
positive lower bound. The same inequality proves the last statement.
\end{proof}

\subsection[Complete resonant flow and orbit classification]{A complete resonant flow and its critical points at infinity}

The local fields extend to a complete descending field on the whole space.
\begin{prop}\label{prop:pseudogradient}
There is a bounded locally Lipschitz field $W$ on $\oH$, zero exactly
on $\operatorname{Crit}(J)$ and strictly descending elsewhere, which
agrees with $W_*$ near each retained end. Its positive flow is complete.
In particular it is defined on the entire natural neighborhood
$V_{\rho_*}^{\eta}(p,q,\varepsilon_0)$ of every type $2p+q=m$,
for every fixed admissible $\eta>0$ and sufficiently small
$\varepsilon_0>0$, including points outside the balanced end neighborhoods.
Every positive orbit whose energy is bounded below is either
precompact, with its limit set contained in $\operatorname{Crit}(J)$,
or eventually stays in one retained end and has the asymptotics of
Proposition~\ref{prop:CPI-local-pseudogradient}.
\end{prop}

\begin{proof}
Choose the neighborhoods in \eqref{eq:CPI-blocks} and numbers $e_i>0$ with
\begin{equation}\label{eq:CPI-no-return-radii}
 R_-^2>2e_i+r_+^2,\quad R_0^2+r_0^2<e_i/2.
\end{equation}
First choose the outer radii, then $e_i$, and finally the inner
radii. Thus an unstable exit from $\mathscr B_i^-$ is below $c_i-2e_i$,
whereas every point of $\mathscr B_i^0$ has energy in
$(c_i-e_i/2,c_i+e_i/2)$.

Take $Q_*$ and $\varepsilon$ from Lemma~\ref{lem:CPI-height}, set
$K_*=Q_*-J$, $X=\nabla J$, and $Y=\nabla K_*$.
Let $\omega$ be smooth, zero on $[0,\varepsilon]$, positive above
$\varepsilon$, and one on $[2\varepsilon,\infty)$. For $Y\ne0$ set
\begin{equation*}
 Z_0=-\omega(\|X\|/\|Y\|)X-\|Y\|X-\|X\|Y,
 \quad W_0=Z_0/(1+\|Z_0\|).
\end{equation*}
At $Y=0$, $X\ne0$, use its continuous extension $Z_0=-X$.
Near $\operatorname{Crit}(J)$ the height is constant and $Y=-X$;
there the formula is exactly $Z_0=-X$, including its zero extension.
Thus the field is locally Lipschitz everywhere. This is the angle
construction underlying \cite[Proposition~1.1]{L2007}; we verify
its required properties directly.

Cauchy--Schwarz yields
\begin{equation}\label{eq:CPI-angle-identities}
\begin{aligned}
 -DJ[Z_0]&=\omega\|X\|^2+
       \|X\|(\|X\|\|Y\|+\langle X,Y\rangle)\ge0,\\
 DK_*[Z_0]&\le\omega\|X\|\|Y\|
       \le\varepsilon^{-1}(-DJ[Z_0]).
\end{aligned}
\end{equation}
Equality in the first line implies $X=-cY$ for
$0\le c\le\varepsilon$, hence $DJ_{1-c}^*=0$.
By Lemma~\ref{lem:CPI-height}, these zeros lie in the genuine
critical set or in the inner end neighborhoods. Since $Q_*=J+K_*\ge1$,
\begin{equation}\label{eq:CPI-ambient-growth}
 D\ln  Q_*[W_0]\le(\varepsilon^{-1}-1)(-DJ[W_0]).
\end{equation}

Choose disjoint cutoffs $\zeta_i$ supported in $\mathscr B_i^+$
and equal to one near $\mathscr B_i^-$. Define
\begin{equation}\label{eq:CPI-full-field}
 W=(1-\sum_i\zeta_i)W_0+\sum_i\zeta_i W_i .
\end{equation}
At a point where a local term is present, let
$D_i=|X_-|^2+|X_+|^2+|B|^2+|y|^2+b^2+\|\mathsf z\|^2$,
and write $D_0=-DJ[Z_0]$. The exact descent identity on the
whole natural neighborhood, and indeed on $\oH$, is
\begin{equation*}
 -DJ[W]=(1-\sum_i\zeta_i)\frac{D_0}{1+\|Z_0\|}
                 +2\sum_i\zeta_i\Theta_iD_i .
\end{equation*}
Terms supported on an end chart are extended by zero. Their supports
have positive margins in all finite parameter and fiber faces.
A strongly convergent sequence in such a support cannot have
$t\to\infty$, by \eqref{eq:CPI-norm-scalar} and orthogonality.
At bounded scale its parameters have a convergent subsequence, and
its fiber remainders then converge strongly. The support therefore
has no accumulation on the exterior of the containing chart.
Consequently these zero extensions are locally Lipschitz.

Here $W_i$ is the local field. The field has norm at most one. It is strictly descending
off the genuine critical set, because every extra zero of $W_0$
is in a region where the local field has full weight. Bounded speed
and local existence imply completeness on the Hilbert space: an
orbit with a finite maximal time is Cauchy and extends
from its limit.

Linearity in the vector field, \eqref{eq:CPI-local-growth}, and
\eqref{eq:CPI-ambient-growth} yield the single estimate
\begin{equation}\label{eq:CPI-global-growth}
 D\ln  Q_*[W]\le C(-DJ[W])
 \quad\text{outside }\bigcup_i\mathscr B_i^0.
\end{equation}
No derivative of a gluing cutoff enters this inequality.
On every bounded set outside the inner neighborhoods and at positive
distance from $\operatorname{Crit}(J)$, $-DJ[W]$ has a positive
lower bound. To verify this for the ambient field, suppose its
descent tends to zero along a bounded sequence. The gradients
$X,Y$ are bounded there. If $X\to0$, bounded-sequence compactness
yields a forbidden genuine critical point. Otherwise the two
nonnegative terms in \eqref{eq:CPI-angle-identities} imply
\[
 X_n+c_nY_n\to0,\quad
 c_n=\|X_n\|/\|Y_n\|\le\varepsilon+o(1).
\]
After extraction this is a bounded approximate critical sequence
for $J_{1-c_n}^*$. Lemma~\ref{lem:CPI-height}(i)--(ii) yields a strong
limit in the genuine critical set or an inner neighborhood, contrary to
the coordinate margins wherever $W_0$ has nonzero weight.
For the local part the last statement of that lemma yields the
lower bound. Their convex combination preserves it.

After entering $\mathscr B_i^0$, an orbit follows the local field
until it leaves $\mathscr B_i^-$. If $\xi_-=0$ it remains there
and tends to the end. Otherwise $|\xi_-|$ grows exponentially in
coordinate time and it exits the unstable face. The exit occurs
in finite Hilbert time: before the finite coordinate exit time
all coordinates have finite-scale limits, where the normalization
is bounded below. By \eqref{eq:CPI-no-return-radii} the exit energy
excludes any return to $\mathscr B_i^0$. There are only finitely
many ends, so a noncaptured orbit eventually avoids every inner neighborhood.

If its energy tends to a finite number $c$, integration of
\eqref{eq:CPI-global-growth} on this tail yields
\begin{equation}\label{eq:CPI-orbit-bound}
 Q_*(u(s))\le Q_*(u(s_0))
       \exp\{C(J(u(s_0))-c)\},\quad s\ge s_0.
\end{equation}
Thus the tail is bounded. For every $\delta>0$, infinitely many
visits at distance at least $\delta$ from the genuine critical
set entail an infinite energy decrease: bounded speed requires time at
least $\delta/2$ to reach its $\delta/2$ neighborhood, and the
preceding lower descent bound applies during each such passage.
Therefore $\operatorname{dist}(u(s),\operatorname{Crit}(J))\to0$.
If that set is empty, the positive slope on the bounded tail yields
an immediate contradiction. Otherwise its compactness and compactness
of finite orbit segments imply precompactness of the orbit.
Its limit set is contained in the critical set. This does not
assert convergence to a single degenerate critical point.
\end{proof}

\begin{defn}\label{def:CPI}
A critical point at infinity of the field $W$ is an equivalence class
of its nonprecompact forward orbits with energy bounded below.
Two such orbits are equivalent when their limiting unordered bubble
configurations agree. Proposition~\ref{prop:pseudogradient} proves the existence of
these limits; the definition does not assume an exhaustion property.
\end{defn}

The preceding flow identifies the escaping ends and their indices.
\begin{thm}\label{thm:CPI-classification}
The critical points at infinity in Definition~\ref{def:CPI}
are in bijection with
$\bigcup_{2p+q=m}\mathcal V_{p,q}^-$.
The point $(\xi^*)_\infty$ has level
\[
 C_\infty(\xi^*)=-\rho_*(1+\ln(\rho_*/8))
                   -\tfrac12\FpqV(\xi^*)
\]
and index
\begin{equation*}
 \iota_\infty(\xi^*)=3p+2q-1-\morse(\FpqV,\xi^*).
\end{equation*}
\end{thm}

\begin{proof}
Proposition~\ref{prop:pseudogradient} places every nonprecompact orbit
of bounded energy in a retained end, and
Proposition~\ref{prop:CPI-local-pseudogradient} constructs a nonempty stable
family at every retained end. This proves both directions of the
bijection. The exact form \eqref{eq:CPI-exact-form} determines the level.
The negative directions are $d-k$ center directions and $n-1$
relative-scale directions. The amplitude and fiber directions
are positive, and $\epsilon_*=1$ corresponds to a positive common-scale direction. Their count is $(d-k)+(n-1)=3p+2q-1-k$.
\end{proof}

Near an escaping end $(\xi^*)_\infty$, Lemma~\ref{lem:CPI-coordinates}
yields the exact finite-scale form
\begin{equation}\label{eq:morse-normal-form}
 J=C_\infty(\xi^*)-|\xi_-|^2+\|\eta_0\|^2+|B|^2+b^2,
 \quad b>0,
\end{equation}
with $\dim\xi_-=\iota_\infty(\xi^*)$ and with the weighted
amplitude domain \eqref{eq:CPI-weighted-domain}.
In particular $b^2\asymp\lambda_1^{-1}$ for $q>0$ and
$b^2\asymp\lambda_1^{-2}\ln\lambda_1$ for $q=0$.

Apply Lemma~\ref{lem:CPI-coordinates} with $\mathcal L<0$,
combine $X_-$ and $y$ into $\xi_-$, and combine $X_+$ and $\mathsf z$
into $\eta_0$. The balanced mass coordinates yield
$\lambda_1\asymp e^t$; hence \eqref{eq:CPI-scale-comparison}
establishes the asserted orders. All changes are made at $b>0$.

We now determine the local relative homology of an escaping end.
\begin{prop}
\label{prop:complete-cell-infinity}
Let $c=C_\infty(\xi^*)$, $\iota=\iota_\infty(\xi^*)$, and choose
an end neighborhood disjoint from the other objects at level $c$.
For sufficiently small $e>0$ its local pair of sublevels has
relative homology, with coefficients in any field $\mathbb F$,
\begin{equation*}
 H_r(J^{c+e},J^{c-e};\mathbb F)_{\xi^*}
 \cong\left\{\begin{aligned}
 &\mathbb F&& \text{ if }\, r=\iota,\\&0&& \text{ if }\, r\ne\iota.
 \end{aligned}\right.
\end{equation*}
The local relative pair has the homotopy type of
$(D^\iota,S^{\iota-1})$, with the second member empty for $\iota=0$.
The same local contribution is obtained if its amplitude domain is
enlarged to a fixed natural radius
$\sqrt t\,|\alpha-\mathbf1|<r_a$ in a geometric tube over the
same sufficiently small $(v,\xi,t)$ base domain.
\end{prop}

\begin{proof}
Use \eqref{eq:morse-normal-form} on
\[
 \mathscr B=\{|\xi_-|\le R,
    \ \|\eta_0\|^2+(1+t(b))|B|^2+b^2\le r^2,\ b>0\},
 \quad R^2>2e+r^2.
\]
The unstable boundary lies below $c-e$.
Contract $(\eta_0,B)$ to zero at fixed $(\xi_-,b)$; both the
weighted radius and $J$ decrease, so this preserves the pair.
It remains to analyze
\[
 A=\{(\xi,b):|\xi|\le R,\ 0<b\le r,\ b^2-|\xi|^2\le e\},
 \quad A_- =\{(\xi,b)\in A:b^2-|\xi|^2\le-e\}.
\]
Fix $0<b_0<\min(r,\sqrt e/2)$. For $s\in [0,1]$, set
\[
 b_s=(1-s)b+sb_0,\quad
 a_s=\max\Big(1,\sqrt{\frac{e+b_s^2}{e+b^2}}\Big),\quad
 \xi_s=\left\{\begin{aligned}
 &\min\{a_s,R/|\xi|\}\xi&&\text{ if }\,\xi\ne0,\\&0&& \text{ if }\,\xi=0.
 \end{aligned}\right.
\]
If $b_s\le b$, the energy cannot increase. If $b_s>b$, then
$b_s\le b_0<\sqrt e$, so the upper constraint is automatic.
For a point of $A_-$, radial expansion preserves
$|\xi_s|^2\ge e+b_s^2$; truncation at $R$ preserves it because
$R^2>e+r^2$. This is therefore a deformation of pairs onto
\[
 \Big(D_R^\iota,\{\xi:\sqrt{e+b_0^2}\le|\xi|\le R\}\Big)
\]
at the finite scale $b=b_0$, fixing that slice.
The annulus is a collar of the sphere, so this pair has the
homotopy type of $(D^\iota,S^{\iota-1})$ and the stated homology.
For $\iota=0$ the lower member is empty and the upper member is
contractible. No point at $b=0$ has been adjoined. In particular,
the positive center factor has been contracted together with the
positive amplitude and fiber factors.

We verify explicitly the comparison with the larger amplitude tube.
In the construction of Lemma~\ref{lem:CPI-coordinates},
$(X_-,X_+,y,b)$ depend only on $(v,\xi,t)$:
the amplitude minimization keeps that base fixed, the relative-mass
splitting is performed after amplitude elimination, and the center
and common-scale changes use only $(\xi,t)$.
This is true in both mass parametrizations, by the
amplitude-independent change of base proved there.
Thus, with $\xi_-=(X_-,y)$, choose the base domain
\[
 |\xi_-|\le R,\quad |X_+|^2+b^2\le r_0^2,\quad b>0,
\]
where $r_0$ is smaller than $b(T')$ for a common threshold $T'$
large enough for both parametrizations and for the preparation.
Take over this base the geometric tube
$\sqrt t\,|\alpha-\mathbf1|<r_a$, $\|\mathsf z\|<r_z$.
Lemma~\ref{lem:amplitude-collar-retraction}, restricted to this
base, retracts each of its local sublevel pairs onto
$t|\alpha-\mathbf1|\le\delta$ without changing the base.
Choose $\delta$ inside the terminal radius.
By \eqref{eq:CPI-amplitude-graph}, increasing the scale threshold
places $\alpha^\sharp(v,\xi,t)$ in
$t|\alpha^\sharp-\mathbf1|<\delta/2$ uniformly on this base.

First contract the positive fiber coordinate to zero.
Then, at fixed $(v,\xi,t)$, contract $\alpha$ along the segment
to $\alpha^\sharp$. The latter segment remains in
$t|\alpha-\mathbf1|\le\delta$ by convexity of this ball.
Uniform amplitude convexity implies that its energy is nonincreasing:
the derivative along a ray out of its minimum is nonnegative.
These two contractions preserve both sublevels and fix the
amplitude-minimizing graph with zero positive fiber.
On that graph the energy is
$c-|\xi_-|^2+|X_+|^2+b^2$.
Contracting $X_+$ to zero preserves the chosen base domain and
again decreases the energy. The remaining local pair is exactly
$(A,A_-)$ above, with $r_0$ in place of $r$.
Choose $R^2>2e+r_0^2$ and apply its finite-$b_0$ deformation.
This proves the asserted comparison of local groups and pair
homotopy types. In particular, no full terminal normal form on
the larger natural amplitude radius has been assumed.
\end{proof}

\subsection{Relative deformation at the resonant parameter}

The local deformations combine into the following global result.
\begin{thm}
\label{thm:global-deformation-infinity}
Let $a<b$ be levels which are neither genuine critical values nor critical values at infinity.
If $[a,b]$ contains no genuine critical value and no critical value
at infinity, $J^a$ is a strong deformation retract of $J^b$.
Across an isolated critical level whose genuine critical points are
finite in number and nondegenerate, the upper sublevel is, up to homotopy
relative to the lower sublevel, obtained by attaching one cell of the
corresponding index for each genuine critical point and each critical
point at infinity at that level. Its relative homology is the direct sum of their local Morse groups.
In particular, if the interval contains only $(\xi^*)_\infty$,
then
\begin{equation*}
 H_r(J^b,J^a;\mathbb F)
 \cong\left\{\begin{aligned}
     &\mathbb F&&\text{ if }\,r=\iota_\infty(\xi^*),\\
                      &0&&\text{ if }\, r\ne\iota_\infty(\xi^*).
 \end{aligned}\right.
\end{equation*}
\end{thm}

\begin{proof}
The end neighborhoods in this proof lie in the smaller terminal domain.
Corollary~\ref{cor:resonant-geometric-coverage} first places the
localized natural bubble neighborhoods in geometric tubes.
Lemma~\ref{lem:amplitude-collar-retraction} yields the relative
deformation from their natural amplitude radii to the terminal
amplitude domain. Proposition~\ref{prop:complete-cell-infinity}
therefore computes the contribution for the larger tubes as well.
This is a comparison of local sublevel pairs.
The ambient relative deformation below uses the terminal neighborhoods
and the complete field of Proposition~\ref{prop:pseudogradient};
it does not require extending the local preparation by the identity
across an outer chart face.

Choose a buffered window $[c-2e,c+2e]$ containing only the critical
level in question. Take disjoint inner and outer genuine Morse
neighborhoods and end neighborhoods at that level. Choose the inner neighborhoods of
all other ends sufficiently small that their energy ranges miss
the buffered window. genuine neighborhoods lie in the ball on which
$Q_*$ is constant. At an end choose the radii in
\eqref{eq:CPI-no-return-radii}; the stable boundary $G=r_-^2$
is strictly incoming, and every unstable exit $H=R_-^2$ is below
$c-2e$. The same statements hold in an genuine Morse neighborhood, with
no $b$ or weighted amplitude factor.

Use the global field of Proposition~\ref{prop:pseudogradient},
replacing it inside the genuine neighborhoods by the bounded local Morse
fields and gluing in their collars. The height is constant there,
so \eqref{eq:CPI-global-growth} remains valid outside the inner
neighborhoods. Before first contact with
\begin{equation}\label{eq:CPI-relative-target}
 J^{c-e}\cup
 \bigl(J^{c+e}\cap\textstyle\bigcup\mathscr B^-\bigr),
\end{equation}
an orbit has energy between $c-e$ and $c+e$. Integration of the
growth estimate yields a norm estimate depending on its initial point,
as in \eqref{eq:CPI-orbit-bound}. The bounded-set descent argument
in the proof of that proposition yields a positive lower slope outside
the target, since all possible auxiliary zeros are inside the
inner neighborhoods. Thus first contact occurs in finite time.

This first-contact time is continuous, including extension by zero
on the target. At a level face the derivative of $J-(c-e)$ is
strictly negative; at a stable boundary face the derivative of its
defining function is strictly negative. The exit faces lie strictly
below the target level. These signs hold on collars. Local
continuous dependence and transverse crossing prove continuity of
contact times; at a corner take the minimum of the finitely many
face contact times. Stopping the flow at $sT(u)$ therefore yields
a strong deformation onto \eqref{eq:CPI-relative-target}.

The genuine local pairs are Morse pairs. For an end the local
pair is that of Proposition~\ref{prop:complete-cell-infinity}.
The exit faces are already in the lower sublevel; slightly enlarging
the neighborhoods yields open collars with the same pair deformations.
Excision therefore identifies the relative homology of the target
and the lower sublevel with the direct sum of the local relative
groups. This also treats several objects at the same value.

The cell statement uses the pair homotopy, not only the homology
calculation. At an end take the negative disk $|\xi_-|\le R$ at
$b=b_0$, with the positive coordinates zero. Its boundary lies in
the lower sublevel. The deformation in Proposition
\ref{prop:complete-cell-infinity} identifies its inclusion as a homotopy
equivalence onto the upper local neighborhood, and identifies its boundary
inclusion as a homotopy equivalence onto the lower local part (contract
the final annulus onto its outer sphere). The level and neighborhood
faces have collars by the strict signs already proved. Their
inclusions are therefore cofibrations. The homotopy extension
property for these collars identifies the pushout with the lower
sublevel with the attachment of this negative disk. Carrying this
out in disjoint neighborhoods proves the asserted simultaneous cell
attachments, relative to the lower sublevel.

On an interval with no critical level the same argument has no
target neighborhood in that interval. For completeness, choose the
inner neighborhoods at all other ends small enough that their energy
miss $[a,b]$. Genuine critical values miss a buffered enlargement
of this interval, by compactness of the genuine critical set.
The height estimate then holds until the lower level is reached,
and the bounded-set slope is positive throughout. The continuous
stopped flow defines the stated strong deformation retraction.
Applying these deformations to the intervals on either side of
$[c-e,c+e]$ proves the statement for general $a,b$.
One height and one glued field can be used for a finite collection
of windows, by choosing all the neighborhoods first.
\end{proof}

The auxiliary-family compactness and the common height estimate
exclude escaping trajectories outside the retained end neighborhoods. Thus
the same complete field applies to initial data outside the balanced
charts and to trajectories passing between coordinate neighborhoods.

\begin{proof}[Proof of Proposition~\ref{prop:CPI-main}]
The proof follows from Proposition~\ref{prop:pseudogradient},
Theorem~\ref{thm:CPI-classification},
Proposition~\ref{prop:complete-cell-infinity}, and
Theorem~\ref{thm:global-deformation-infinity}.
\end{proof}

\section{Proof of the main theorems}\label{sec:4}

All degree and Morse computations below concern $J_{4\pi m}$.
The nearby branches of Proposition~\ref{prop:bubbling-branches} are
used separately for the prescribed-index existence theorem.
Appendix~\ref{app:topology} determines the topology of the sufficiently high and low sublevels.

\subsection{A variational perturbation in a bounded region}

Let $\rho_*=4\pi m$ be fixed. The following lemma removes the
nondegeneracy assumption on the genuine solutions from the degree
calculation. Compactness in its proof concerns bounded subsets of
$\oH$; it does not assert a Palais--Smale condition at infinity.

\begin{lem}\label{lem:bounded-Morse-perturbation}
Assume \ref{item:C1}--\ref{item:C2}, and write $J=J_{\rho_*}$.
There are arbitrarily small $C^1$ variational perturbations $J_\epsilon$
on a fixed bounded ball such that:
\begin{enumerate}[label=\textup{(\roman*)}]
\item $J_\epsilon=J$ outside that ball, all its genuine critical
points are nondegenerate, and their set is finite;
\item $\nabla J_\epsilon=a_\epsilon(u)u-K_\epsilon(u)$,
where $a_\epsilon\ge1/2$ and $K_\epsilon$ maps bounded sets into
relatively compact sets;
\item for a ball $\mathcal B_R^{\oH}$,
\[
 d_{\rho_*}=\deg(a_\epsilon^{-1}\nabla J_\epsilon,
          \mathcal B_R^{\oH},0)
 =\sum_{u\in\operatorname{Crit}(J_\epsilon)}
                  (-1)^{\morse(J_\epsilon,u)};
\]
\item the critical points at infinity, their local pairs, and
sufficiently high and low sublevels are unchanged. The
fixed-resonance filtration of Theorem~\ref{thm:global-deformation-infinity}
applies to $J_\epsilon$ with those same end neighborhoods.
\end{enumerate}
\end{lem}
\begin{proof}
Set $Q(u)=\|u\|^2/2$ and
$L(u)=\rho_*\ln \int_\Sigma Ve^u\,\d v_g$, so
$J=Q-L$ and $\nabla L=T_{\rho_*}=:K$.
Lemma~\ref{lem:B-compact-gradient} and resonant compactness show that
the genuine critical set is compact. Choose
$R_0<R_1<R_2$ with that set in $\mathcal B_{R_0}^{\oH}$.
There is $\gamma>0$ such that
\begin{equation}\label{eq:bounded-annular-gap}
 \|\nabla J(u)\|\ge\gamma,\quad R_0\le\|u\|\le R_2.
\end{equation}
Indeed, a contrary sequence is bounded and approximate-critical;
Lemma~\ref{lem:B-compact-gradient} yields a strongly convergent
subsequence to a forbidden critical point in this closed annulus.

Let $P_N$ be the orthogonal projections onto the first $N$ members
of a smooth Neumann eigenfunction basis of $\oH$, and set $E_N=P_N\oH$.
On $\overline{\mathcal B}_{R_2}^{\oH}$ one has
\begin{equation}\label{eq:uniform-variational-approximation}
 \sup\bigl(|L(u)-L(P_Nu)|
       +\|K(u)-P_NK(P_Nu)\|\bigr)\to 0.
\end{equation}
Here is a direct proof of the uniformity. A bounded sequence has a
weakly convergent subsequence, $u_N\rightharpoonup u$. Since
$P_N\to I$ strongly, also $P_Nu_N\rightharpoonup u$.
Compact Sobolev embedding implies convergence in measure for both
sequences. The exponential bound \eqref{eq:B-normalized-density-estimate}
and uniform integrability imply $e^{u_N},e^{P_Nu_N}\to e^u$ in
every finite $L^p$. Thus both logarithmic integrals converge to
$L(u)$, while both $K(u_N)$ and $K(P_Nu_N)$ converge strongly to
$K(u)$. Finally $P_NK(P_Nu_N)\to K(u)$. A sequence contradicting
\eqref{eq:uniform-variational-approximation} contradicts these
conclusions.

Choose a smooth cutoff $\vartheta_0(Q)$, equal to one for $\|u\|\le R_1$
and zero for $\|u\|$ in a neighborhood of $R_2$ and beyond.
For $v\in E_N$ define
\[
 H_{N,v}(u)=L(u)-L(P_Nu)-\langle v,P_Nu\rangle,
 \quad J_{N,v}=J+\vartheta_0(Q)H_{N,v}.
\]
First take $N$ large and then $\|v\|$ small. By \eqref{eq:uniform-variational-approximation},
$\|J_{N,v}-J\|_{C^1(\mathcal B_{R_2}^{\oH})}$
can be chosen arbitrarily small.
In particular, the gradient difference has norm less than $\gamma/2$.
There are no critical points in the annulus in
\eqref{eq:bounded-annular-gap}, or outside $R_2$.
Inside $R_0$ the functional is
\[
 Q(u)-L(P_Nu)-\langle v,P_Nu\rangle.
\]
Its complementary critical equation is $(I-P_N)u=0$; its equation
on $E_N$ is $\nabla g_N(u)=v$, where $g_N=Q|_{E_N}-L|_{E_N}$.
This finite-dimensional function is smooth: the basis functions are
smooth and bounded, and differentiation of its exponential integral
on compact subsets is justified by domination. Sard's theorem yields
arbitrarily small regular values $v$. Every resulting critical point
is nondegenerate, with complementary Hessian $I$. All zeros lie in
the closed finite-dimensional ball of radius $R_0$, and no zero is
on its boundary. A compact set of isolated nondegenerate zeros is
finite. Fix such $N,v$ and denote this functional by $J_\epsilon$.

Direct differentiation, including the norm cutoff, yields
\begin{equation}\label{eq:normalized-perturbation-gradient}
\begin{split}
 \nabla J_\epsilon&=a_\epsilon(u)u-K_\epsilon(u),\quad
 a_\epsilon=1+\vartheta_0'(Q)H_{N,v},\\
 K_\epsilon&=(1-\vartheta_0(Q))K(u)
            +\vartheta_0(Q)P_NK(P_Nu)+\vartheta_0(Q)v.
\end{split}
\end{equation}
Smallness ensures $a_\epsilon\ge1/2$. Scalar multiplication by
bounded real coefficients preserves relative compactness of a
bounded-set image, so $K_\epsilon$ and $a_\epsilon^{-1}K_\epsilon$
are compact on bounded sets. In particular the cutoff term
$\vartheta_0'(Q)H_{N,v}u$ has not been treated as a compact map.

For $0\le\tau\le1$, use $J+\tau\vartheta_0(Q)H_{N,v}$.
The same positive scalar representation holds and its normalized
gradient equals $\nabla J$ near the sphere $\|u\|=R_2$.
The stability property of degree therefore identifies the degrees
on this ball. At a zero, the derivative of the normalized gradient
is $a_\epsilon(u)^{-1}D^2J_\epsilon(u)$; the derivative of the scalar
factor multiplies the vanishing gradient and contributes nothing.
Its sign is $(-1)^{\morse(J_\epsilon,u)}$. Excision over the finite
critical set proves (iii).

For completeness, the deformation argument is unchanged at infinity.
In constructing Lemma~\ref{lem:CPI-height}, choose its constant-height
ball to contain $\overline{\mathcal B}_{R_2}^{\oH}$ and place all end
supports outside a larger ball. Inside the former ball the auxiliary
gradient is exactly $s\nabla J_\epsilon$; outside the support of the
perturbation it is the original auxiliary gradient. Bounded
approximate-critical compactness follows from
\eqref{eq:normalized-perturbation-gradient}, whose scalar coefficient
is uniformly positive, also for the small auxiliary interval.
The unbounded exact auxiliary sequences are unchanged. Thus the
proofs of the common height estimate, orbit exhaustion, no-return,
and the global deformation theorem apply with precisely the old
end neighborhoods and their local groups. Near the new genuine critical
points use their genuine Morse charts and patch descending fields
in the bounded annulus; the positive lower bound
\eqref{eq:bounded-annular-gap} permits this patching.
The amplitude preparation in
Lemma~\ref{lem:amplitude-collar-retraction} is also unchanged:
increase $T_\delta$ until its entire tube lies outside the
perturbation ball. This is possible because
$\ln \lambda_i=t+O(\sqrt t)$ and the Dirichlet bubble estimate
yield $\|u\|^2\ge ct$ there. Hence $J_\epsilon=J$ throughout
that deformation, including its stopping target.

Finally $J$ and $J_\epsilon$ are bounded above and below on
$\mathcal B_{R_2}^{\oH}$ by Moser--Trudinger and Jensen's inequality.
For $D$ larger than the constants in these estimates, their sublevels at $D$ and $-D$
coincide as sets. This proves (iv) and the lemma.
\end{proof}

\subsection{Degree from the resonant Morse--Euler identity}

\begin{lem}\label{lem:first-wall-endpoints}
For $m=1$, for sufficiently large $D$, one has $J_{4\pi}^{-D}=\varnothing$
and $J_{4\pi}^{D}\simeq\oH$.
\end{lem}
\begin{proof}
The sharp Neumann Moser--Trudinger inequality yields a lower bound
for $J_{4\pi}$, so $J_{4\pi}^{-D}=\varnothing$ for sufficiently large $D$.
The upper-sublevel argument in
Proposition~\ref{prop:B-resonant-comparison} holds for every $m\ge1$:
Lemma~\ref{lem:B-exact-energy} establishes an upper bound for $J_{4\pi m}$ on all exact
solutions $u=T_r(u)$, $0\le r\le4\pi m$.
For $m=1$, the only possible unbounded sequences in this interval
concentrate at $4\pi$ and have the finite limiting energy in
\eqref{eq:B-endpoint-energy}; a limit at zero is excluded in that
lemma. Hence the upper bound is uniform. For $D$ above it, any zero
of a convex combination of $\nabla J_{4\pi}(u)$ and $u$ is a
forbidden exact solution in $\{J_{4\pi}\ge D\}$.
On bounded parts, compactness of $T_r$ turns absence of such zeros
into positive separation. The one-functional part of
Lemma~\ref{lem:B-bounded-descent} retracts $\oH$ onto $J_{4\pi}^{D}$.
\end{proof}

\begin{proof}[Proof of Theorem~\ref{thm:resonant-case}]
First suppose all genuine critical points are nondegenerate.
Fix $\rho_*=4\pi m$ and write
\[
 S_\infty=\sum_{2p+q=m}\sum_{[\xi]\in\mathcal V_{p,q}^-}
                          (-1)^{\iota_\infty([\xi])}.
\]
For $m\ge2$, the fixed-functional Morse relation of
Theorem~\ref{thm:2}, evaluated at $t=-1$, yields
\begin{equation*}
 \sum_{u\in\operatorname{Crit}(J_{\rho_*})}
       (-1)^{\morse(J_{\rho_*},u)}+S_\infty
 =1-\chi(B_{m-1}^\partial(\Sigma,\partial\Sigma)).
\end{equation*}
For $m=1$, the identical finite-filtration argument uses the sublevel
pair in Lemma~\ref{lem:first-wall-endpoints}, whose Euler
characteristic is $1$. It yields the same identity with right-hand
side $1$. This also accounts for critical groups in degree zero.

At each genuine critical point the Hessian is identity minus a
compact self-adjoint operator. The local degree of $I-T_{\rho_*}$
is therefore its genuine Morse sign. The genuine critical set is
compact and finite; additivity yields
\begin{equation}\label{eq:fixed-resonance-degree}
 d_{\rho_*}=\left\{\begin{aligned}
    &1-S_\infty&&\text{ if }\, m=1,\\
 &1-\chi(B_{m-1}^\partial(\Sigma,\partial\Sigma))-S_\infty&&\text{ if }\, m\ge2.  
 \end{aligned}\right.
\end{equation}
If genuine solutions are degenerate, apply
Lemma~\ref{lem:bounded-Morse-perturbation}. The pair of high and low sublevels and
all infinite local groups are unchanged. Apply the same filtration
and Euler calculation to $J_\epsilon$, then identify its genuine
signed count with $d_{\rho_*}$ by that lemma. Thus
\eqref{eq:fixed-resonance-degree} holds under exactly the stated
reduced assumptions.

Finally, $\iota_\infty=3p+2q-1-\morse(\FpqV,\xi)$, so
$-(-1)^{\iota_\infty}=(-1)^{p-\morse(\FpqV,\xi)}$.
For $m\ge2$, \eqref{eq:topology-euler-characteristic} yields
\[
 1-\chi(B_{m-1}^\partial(\Sigma,\partial\Sigma))
 =\frac1{l!}\prod_{j=1}^l(j-\chi(\Sigma)),
 \quad l=\lfloor(m-1)/2\rfloor,
\]
with the empty product equal to one. Substitution proves
\eqref{eq:degree-reso}, including $m=1,2$.
\end{proof}

\begin{proof}[Proof of Corollary~\ref{c:Euler-Poincare}]
By \eqref{eq:degree-reso}, condition \eqref{eq:Euler-Poincare} is exactly
\(d_{4\pi m}\ne0\). The existence property of the Leray--Schauder degree
then yields a solution of \eqref{maineq} at \(\rho=4\pi m\).
\end{proof}

		\subsection{Existence with prescribed Morse index}
		\begin{proof}[Proof of Theorem \ref{thm:1}]
			Set \(\rho_*=4\pi m\).  The same argument for choosing generic perturbations of \(V\) applies in both parts. Given a sequence \(\rho_k\to\rho_*\) approaching from one side,
			choose positive potentials \(V_k\to V\) in \(C^3(\Si )\) such that
			\(J_{\rho_k,V_k}\) is Morse. The perturbations follow from the following transversality argument.
			Fix \(r>2\) and set
			\[
			\mathcal X=\{u\in W^{2,r}(\Si ):\pa _{\nu_g}u=0,
			\ \textstyle\int_\Si  u\,\d v_g=0\},\quad
			\mathcal Y=\{f\in L^r(\Si ):\textstyle\int_\Si  f\,\d v_g=0\}.
			\]
			On the open cone \(\mathcal P=C^3_+(\Si )\) consider the map
			\[
			\mathscr E_k:\mathcal P\times\mathcal X\to \mathcal Y,
			\quad
			\mathscr E_k(V,u)=-\Delta_gu-\rho_k
			\Big(\frac{Ve^u}{\int_\Si  Ve^u\,\d v_g}-\frac1{|\Si |_g}\Big).
			\]
			The map is \(C^1\) on the stated \(C^3\) potential space, since \(W^{2,r}\hookrightarrow C^0\) for \(r>2\), and \(D_u\mathscr E_k\) is an elliptic Neumann Fredholm
			operator of index zero.  If \(\psi\) represents a cokernel vector for the
			full differential at a zero, elliptic regularity applies to \(\psi\).  Its
			pairing with every potential variation \(\dot V\in C^3(\Si )\) is
			\[
			-\frac{\rho_k}{I}\int_\Si
			e^u(\psi-c_\psi)\dot V\,\d v_g,
			\quad
			I=\int_\Si  Ve^u\,\d v_g,
			\quad
			c_\psi=I^{-1}\int_\Si \psi Ve^u\,\d v_g.
			\]
			Since the pairing vanishes for every \(\dot V\), one has
\(\psi=c_\psi\) pointwise.  A representative of the cokernel on the space of mean-zero functions may be chosen with zero mean, hence
			\(\psi=0\).

The universal differential is therefore onto and split,
since \(D_u\mathscr E_k\) has closed range and finite-dimensional
kernel and cokernel. Thus \(\mathcal Z_k:=\mathscr E_k^{-1}(0)\)
is a \(C^1\) Banach submanifold, and the projection
\(\pi_k:\mathcal Z_k\to\mathcal P,\quad (V,u)\mapsto V,\)
is Fredholm of index zero. Its regular values are exactly the potentials
for which every zero is nondegenerate. The spaces \(C^3(\Si )\)
and \(\mathcal X\) are separable, so \(\mathcal Z_k\) has a countable
base. Smale's theorem \cite[Theorem~1.3, p.~862]{Smale1965}
therefore applies with \(C^1\) regularity. We may choose regular
potentials \(V_k\) with \(\|V_k-V\|_{C^3}<1/k\).
By nonresonant compactness and nondegeneracy, the critical set of
\(J_{\rho_k,V_k}\) is finite.

			We next verify that the reduced critical points, their indices, and the signs of the scale coefficients persist under these perturbations.  The separation estimate in \cite[Lemma~4.1]{HB2024}, with constants
			chosen uniformly for \(V'\) in a sufficiently small \(C^3\)-neighborhood of
			\(V\), shows that every critical point of \(\mathbf F_{p,q}^{V'}\) belongs to a
			fixed compact subset \(K_{p,q}\Subset\Xi_{p,q}\).  On that set
			\(\mathbf F_{p,q}^{V'}\to\mathbf F_{p,q}^{V}\) in \(C^2\) and
			\(\mathcal L_{V'}\to\mathcal L_V\) uniformly.

Choose disjoint neighborhoods
			of the finitely many nondegenerate critical orbits of
			\(\mathbf F_{p,q}^{V}\).  The gradient has a positive lower bound on the
			part of \(K_{p,q}\) outside these neighborhoods; the implicit-function
			theorem yields exactly one nearby critical orbit in each neighborhood,
			with the same Morse index. The uniform lower bound excludes all other critical orbits, and the
			separation lemma excludes approach to the collision set \(\mathbb{F}_{p,q}(\Si )\) or to the boundary of the interior stratum.  Since \(\mathcal L_V\ne0\) on the finite critical
			set, all strict signs also persist for \(V_k\) when \(k\) is large.

			We first prove (1). Choose $\rho_k>\rho_*$ with
$\rho_k\downarrow\rho_*$. For large $k$,
$\rho_k\in(4\pi m,4\pi(m+1))$.
Corollary~\ref{prop:B-nonresonant-Morse}, applied to the Morse
functional $J_{\rho_k,V_k}$ with $l=m$, yields a critical point $u_k$
of Morse index $2m$. The sublevel depth in
Corollary~\ref{prop:B-nonresonant-Morse} may depend on $k$.

Suppose that \(\{u_k\}\) blows up. By the compactness alternative and the
finite-dimensional reduction, there are a type \((p,q)\), \(2p+q=m\), and a critical
orbit \([\xi^*]\) of \(\FpqV\), with \(\cL(\xi^*)>0\) by the blow-up sign condition.
Proposition~\ref{prop:bubbling-branches}, applied uniformly to the nearby reduced
critical points associated with \(V_k\), implies
			\begin{equation}\label{eq:supercritical-bubbling-index}
				\morse(J_{\rho_k,V_k},u_k)
				=3p+2q-\morse(\FpqV,\xi^*)
			\end{equation}
			for large \(k\); the implicit function theorem identifies these nearby
			reduced critical points with the corresponding critical point for \(V\),
			and their Morse indices are unchanged.  Using \eqref{eq:supercritical-bubbling-index}, since the left-hand side
			is \(2m\) and \(2p+q=m\),
			\(2m=2m-p-\morse(\FpqV,\xi^*).\)
			Thus \(p=0\) and \(\morse(\Fv{0}{m},\xi^*)=0\).  The limit configuration is a
			local minimum of \(\Fv{0}{m}\), where the hypothesis of (1) yields
			\(\cL(\xi^*)<0\), a contradiction.  Therefore \(\{u_k\}\) is compact.

			After passing to a subsequence, \(u_k\to\om \) in \(C^{2,\gamma}(\Si )\) for every \(\gamma\in(0,1)\),
			and \(\om \) solves \eqref{maineq} for \((\rho_*,V)\).  The linearized operators converge in the norm-resolvent sense.  Since every \(u_k\) has Morse index
			\(2m\), spectral semicontinuity yields
			\[
			\morse(J_{\rho_*},\om )\le2m
			\le\morse(J_{\rho_*},\om )+\dim\ker\mathcal T_\om ,
			\]
			which is \eqref{eq:index-window-upper}.

			For (2), choose $\rho_k<\rho_*$ with $\rho_k\uparrow\rho_*$ and
the same type of generic positive potentials $V_k\to V$.
For large $k$, $\rho_k\in(4\pi(m-1),4\pi m)$.
Corollary~\ref{prop:B-nonresonant-Morse}, now with $l=m-1\ge1$,
yields a critical point $u_k$ of Morse index $2m-2$.

If \(\{u_k\}\) blows up, its limiting critical orbit \([\xi^*]\) satisfies
			\(\cL(\xi^*)<0\).  For \(\rho_k<\rho_*\), Proposition~\ref{prop:bubbling-branches} yields
			\begin{equation}\label{eq:subcritical-bubbling-index}
				\morse(J_{\rho_k,V_k},u_k)
				=3p+2q-1-\morse(\FpqV,\xi^*).
			\end{equation}
			Using \eqref{eq:subcritical-bubbling-index}, \(2p+q=m\), and \(\morse(J_{\rho_k,V_k},u_k)=2m-2\), we obtain
			\(\morse(\FpqV,\xi^*)=1-p.\)
			Hence either
			$p=0$, $q=m$, $\morse(\Fv{0}{m},\xi^*)=1$,
			or
			$p=1$, $q=m-2$, $\morse(\Fv{1}{m-2},\xi^*)=0$.
			In both cases the sign assumptions in (2) yield \(\cL(\xi^*)>0\), contradicting
			the subcritical blow-up sign.  Thus \(\{u_k\}\) is compact.  Passing to a limit
			and applying the same spectral argument yields
			\[
			\morse(J_{\rho_*},\om )\le2m-2
			\le\morse(J_{\rho_*},\om )+\dim\ker\mathcal T_\om ,
			\]
			which is \eqref{eq:index-window-lower}. If the limiting solution is
			nondegenerate, then \(\ker\mathcal T_\om =\{0\}\), and the corresponding
			two inequalities are equalities.
		\end{proof}

		\subsection{Morse inequalities and applications}

		\begin{proof}[Proof of Theorem~\ref{thm:2}]
Set $\rho_*=4\pi m$. The genuine critical set is compact and, by nondegeneracy, finite. The reduced critical set is finite
as well. Choose $D$ larger than the absolute values of all genuine
and infinite critical levels and large enough for
Proposition~\ref{prop:negative-sublevels-barycenters}.
Theorem~\ref{thm:global-deformation-infinity} yields a finite filtration
of $(J_{\rho_*}^{D},J_{\rho_*}^{-D})$ by its critical levels.
genuine points contribute their local Morse groups; Proposition
\ref{prop:complete-cell-infinity} contributes one generator in degree
$\iota_\infty$ for each end. Objects at the same value are retained
in disjoint neighborhoods and contribute simultaneously.
By Proposition~\ref{prop:negative-sublevels-barycenters}, the relative
Poincar\'e polynomial of this pair is $P_m(t)$: the upper sublevel
is contractible and the lower one has the homotopy type of
$B_{m-1}^\partial(\Sigma,\partial\Sigma)$. The long exact sequence
identifies the relative Betti numbers with \eqref{def:Betti}.

This filtration has at least one critical level. Otherwise
Theorem~\ref{thm:global-deformation-infinity} implies vanishing relative
homology, whereas \eqref{eq:topology-top-betti} yields
$\dim H_{2m-2}(J_{\rho_*}^{D},J_{\rho_*}^{-D};\mathbb Z_2)
=\binom{b+m-2}{m-1}>0$.

Write the filtration as $M_0\subset\cdots\subset M_s$,
where $M_0=J_{\rho_*}^{-D}$ and $M_s=J_{\rho_*}^{D}$. Let
\[
\begin{gathered}
b_{i,j}:=\dim H_i(M_j,M_0;\Z_2),\quad
\theta_{i,j}:=\dim H_i(M_j,M_{j-1};\Z_2),\\
r_{i,j}:=\dim\ker[H_i(M_{j-1},M_0;\Z_2)\to H_i(M_j,M_0;\Z_2)].
\end{gathered}
\]
The long exact sequence of $(M_j,M_{j-1},M_0)$ yields
\begin{equation}\label{eq:filtration-dimension-identity}
\theta_{i,j}=b_{i,j}-b_{i,j-1}+r_{i,j}+r_{i-1,j},
\quad r_{-1,j}=0.
\end{equation} Here $b_{i,0}=0$, $b_{i,s}=\mathfrak b_{m,i}$, and
the local genuine and infinite Morse groups yield
$\sum_{j=1}^s\theta_{i,j}=\nu_i+\nu_i^\infty$.
Set $q_i:=\sum_{j=1}^sr_{i,j}\ge0$ and $q_{-1}:=0$. Summing
\eqref{eq:filtration-dimension-identity} over $j=1,\ldots,s$ and using
$b_{i,0}=0$ and $b_{i,s}=\mathfrak b_{m,i}$ yields
\begin{equation}\label{eq:morse-proof-identity}
\nu_i+\nu_i^\infty-\mathfrak b_{m,i}=q_i+q_{i-1},\quad i\ge0.
\end{equation}
The relative groups of the finite Morse filtration vanish above its
largest critical-point index.
Consequently $Q(t):=\sum_{i\ge0}q_it^i$ is a polynomial, and
\eqref{eq:morse-proof-identity} is equivalent to
\begin{equation}\label{eq:morse-strong}
\sum_{i\ge0}(\nu_i+\nu_i^\infty)t^i-P_m(t)=(1+t)Q(t).
\end{equation}
Identity~\eqref{eq:morse-strong} proves \textup{(a)}, and the
nonnegativity in \eqref{eq:morse-proof-identity} proves \textup{(b)},
also in degree zero.
Since $\iota_\infty\le3p+2q-1\le2m-1$, the index range in the
statement is valid. Finally, evaluate \eqref{eq:morse-strong} at $t=-1$.
Using \eqref{def:P_m} yields
\[
\sum_{i\ge0}(-1)^i(\nu_i+\nu_i^\infty)
=P_m(-1)=1-\chi(B_{m-1}^\pa (\Si ,\pa \Si )),
\]
which is (c). The Euler characteristic is given by
\eqref{eq:topology-euler-characteristic}. The multiplicity estimate follows
from $|\sum_{i\ge0}(-1)^i\nu_i|\le\sum_{i\ge0}\nu_i$.
\end{proof}

		\begin{proof}[Proof of Corollary~\ref{cor:1.3}]
By Theorem~\ref{thm:2}\textup{(b)} and \(\nu_l^\infty=0\),
\(\nu_l\ge\mathfrak b_{m,l}>0\). The stated value of
\(\mathfrak b_{m,2m-2}\), and hence the multiplicity estimate, follows from
\eqref{eq:topology-top-betti} applied to \(B_{m-1}^{\pa }\).
\end{proof}

		\begin{proof}[Proof of Theorem~\ref{thm:3}]
Suppose, to the contrary, that no solution exists. Then
$\operatorname{Crit}(J_{4\pi m})=\varnothing$, so the nondegeneracy
hypothesis in Theorem~\ref{thm:2} is satisfied. The remaining hypotheses
of Theorem~\ref{thm:2} are part of the present assumptions.
Thus $\nu_i=0$ for every $i$, and \eqref{eq:morse-proof-identity} yields
\[
\nu_i^\infty=\mathfrak b_{m,i}+q_i+q_{i-1},\quad
q_i\ge0,\quad q_{-1}=0.
\]
Nonnegativity and $\nu_{l_0+1}^\infty=0$ imply $q_{l_0}=0$.
If $l_0\ge1$, the same argument with $\nu_{l_0-1}^\infty=0$ implies
$q_{l_0-1}=0$; for $l_0=0$, the convention is $q_{-1}=0$.
The degree-$l_0$ identity, together with $\mathfrak b_{m,l_0}=0$, then yields
$\nu_{l_0}^\infty=0$, contrary to the existence of $\xi_0$.
Thus a solution exists.
\end{proof}

		\begin{proof}[Proof of Corollary \ref{cor:minimal-infinity-index}]
			Apply Theorem \ref{thm:3} with $ l_0=p+q-1$.  The hypothesis $ \iota_\infty(\xi_0)=p+q-1$  yields the required critical point at infinity of index $ l_0$, the vanishing of $ \mathfrak b_{m,p+q-1}$  is precisely condition (3) of Theorem \ref{thm:3}, and the stated absence of indices $ p+q-2$  and $ p+q$  is condition (2).  Hence, Theorem \ref{thm:3} establishes the existence of a genuine solution.
		\end{proof}

\begin{proof}[Proof of Corollary~\ref{cor:1.5}]
Suppose that \eqref{maineq} has no solution at \(\rho=4\pi m\).
The genuine critical set is empty, so the nondegeneracy assumption on
genuine solutions in Theorem~\ref{thm:2} is satisfied. The remaining
hypotheses of Theorem~\ref{thm:2} are assumed. Hence
\[
\sum_{i\ge0}\nu_i^\infty t^i-P_m(t)=(1+t)\sum_{i\ge0}q_it^i,
\quad q_i\ge0.
\]
Set \(\widetilde N:=2m-1\). By \eqref{def:l_inf},
\(\nu_i^\infty=0\) for \(i>\widetilde N\), and
\eqref{def:P_m} and \eqref{eq:topology-top-betti} yield
\(\deg P_m\le\widetilde N-1\). Comparing the four coefficients
of degrees \(\widetilde N+1,\widetilde N,\widetilde N-1,
\widetilde N-2\) yields
\[
\begin{aligned}
q_{\widetilde N}&=0,\\
q_{\widetilde N-1}&=\nu_{\widetilde N}^\infty,\\
q_{\widetilde N-2}
 &=\nu_{\widetilde N-1}^\infty
   -\mathfrak b_{m,\widetilde N-1}-\nu_{\widetilde N}^\infty\ge0,\\
q_{\widetilde N-3}
 &=\nu_{\widetilde N-2}^\infty-\mathfrak b_{m,\widetilde N-2}
   +\mathfrak b_{m,\widetilde N-1}
   -\nu_{\widetilde N-1}^\infty+\nu_{\widetilde N}^\infty\ge0.
\end{aligned}
\]
The last two inequalities are  \textup{(i)} and \textup{(ii)}.
Failure of either inequality contradicts the supposition that
\eqref{maineq} has no solution.
\end{proof}

		\appendix

		\section{Basic estimates}\label{app:0}
The notation is as in Section~\ref{sec:2}. The estimates below are local
and uniform on fixed compact subsets of the corresponding configuration
stratum. We include the projection identities and parameter estimates
needed for the differentiated reduction; the other local estimates are
stated with their references.

On each fixed compact configuration neighborhood, choose a common
$r_0>0$ and restrict the admissible charts so that $r_\xi=4r_0$.
Throughout this appendix the cutoff is as in \eqref{eq:cut_off},
as in \eqref{projectedbubble} and \eqref{Hfunction}:
\(\vartheta_\xi=\vartheta(|y_\xi|/r_0).\)
Its derivatives are supported in a fixed annulus away from the center.
Set $r:=|y_\xi|$ and
\[
 f_\xi:=\frac{\Delta_g\vartheta_\xi}{r^2}
       +2\langle\nabla\vartheta_\xi,\nabla r^{-2}\rangle_g.
\]
Define $F_\xi$ by
		\begin{equation}
			\label{Fxi}
			\left\{\begin{aligned}
				&-\Delta_g F_{\xi}= f_{\xi}-\frac1{|\Si |_g}\int_\Si  f_\xi\,\d v_g&& \text{  in }\,\Si, \\
				&\pa_{\nu_g}F_{\xi}=0 && \text{ on }\,\pa\Si, \\
				&\int_{\Si } F_{\xi} \,\d v_g=0,&&
			\end{aligned}\right.
		\end{equation}
		The right-hand side of \eqref{Fxi} has zero mean, so the problem has a
unique mean-zero solution; see, for instance, \cite{N2014,EF2014}.

We record an exact expansion, including its parameter derivatives.
Set $L:=\lam \pa _\lam $ and
\[
 \widehat h_{\lam ,\xi}
 :=\frac2{|\Si |_g}\int_\Si
       \vartheta_\xi\ln(1+\lam ^{-2}|y_\xi|^{-2})\,\d v_g.
\]
Define $\mathcal R_{\lam ,\xi}$ by the exact identity
\begin{equation*}
 P\delta_{\lam ,\xi}
 =\vartheta_\xi\Big(\delta_{\lam ,\xi}
                       +\ln\frac{\lam ^2}{8}\Big)
                  +\kap (\xi)H^g(\cdot,\xi)
 +\widehat h_{\lam ,\xi}
       -\frac{2F_\xi}{\lam ^2}+\mathcal R_{\lam ,\xi}.
\end{equation*}
The identity
\[
 \vartheta_\xi\Big(\delta_{\lam ,\xi}+\ln\frac{\lam ^2}{8}\Big)
 +\kap (\xi)H^g(\cdot,\xi)
 =\kap (\xi)G^g(\cdot,\xi)
 -2\vartheta_\xi\ln(1+\lam ^{-2}|y_\xi|^{-2})
\]
and the mean-zero normalizations of $G^g$ and $F_\xi$ yield
$\int_\Si \mathcal R_{\lam ,\xi}\,\d v_g=0$.
For $s:=\lam ^{-2}r^{-2}$, define the annular function
\[
 B_{\lam ,\xi}
 :=2(\Delta_g\vartheta_\xi)[\ln(1+s)-s]
     +4\langle\nabla\vartheta_\xi,\nabla[\ln(1+s)-s]\rangle_g.
\]
Both terms in \(B_{\lam ,\xi}\) are extended by zero outside the cutoff annulus.
Direct differentiation of \eqref{projectedbubble} and
\eqref{Hfunction}, after the concentrated terms have been combined exactly, yields the equation
\begin{equation}\label{eq:A1-exact-error-equation}
 \left\{\begin{aligned}
 -\Delta_g\mathcal R_{\lam ,\xi}
   &=-B_{\lam ,\xi}
      +\frac1{|\Si |_g}\int_\Si  B_{\lam ,\xi}\,\d v_g,\\
 \pa _{\nu_g}\mathcal R_{\lam ,\xi}&=0,\quad
 \int_\Si \mathcal R_{\lam ,\xi}\,\d v_g=0.
 \end{aligned}\right.
\end{equation}
All center derivatives below are in fixed admissible center coordinates,
with the interior or boundary stratum kept fixed. On the annulus $r$
is bounded away from zero; therefore, for every fixed finite set of
orders $j,\beta,k$ needed below, with $\beta$ a multi-index in the
admissible center coordinates,
\begin{equation}\label{eq:A-annular-derivative-bounds}
 \|L^j\pa _\xi^\beta B_{\lam ,\xi}\|_{C^k(\Si )}
 \le C_{j,\beta,k}\lam ^{-4}.
\end{equation}
Differentiate \eqref{eq:A1-exact-error-equation} before applying the
mean-zero Neumann inverse. Its coefficients are independent of the
center parameters. Apply \eqref{eq:A-annular-derivative-bounds} with one additional
spatial derivative and the mean-zero Neumann Schauder estimate to obtain
$\|L^j\pa _\xi^\beta\mathcal R_{\lam ,\xi}\|_{C^{k+2}(\Si )}\le C\lam ^{-4}$. In particular, for $j\le3$, $|\beta|\le2$,
\begin{equation}\label{eq:A-joint-remainder-bounds}
 \|L^j\pa _\xi^\beta\mathcal R_{\lam ,\xi}\|_{C^2(\Si )}
 \le C\lam ^{-4},\quad
 |L^j\pa _\xi^\beta(\widehat h_{\lam ,\xi}-h_{\lam ,\xi})|
 \le C\lam ^{-3}\ln\lam .
\end{equation}
Here \(h_{\lam ,\xi}\) is the scalar defined in
Lemma~\ref{lem:projbubble_pointwise} below in this appendix. To verify the second estimate,
write \(r=|y|\) in the fixed center chart. Integration yields
\[
\int_{B_{2r_0}^{\xi}}\vartheta(r/r_0)
 \ln\Big(1+\frac1{\lam ^2r^2}\Big)\,\d y
=\frac{\kap (\xi)}{4\lam ^2}\ln\lam
 +\frac{\kap (\xi)}{8\lam ^2}-\frac1{\lam ^2}\int_{B_{2r_0}^{\xi}}
 \frac{\vartheta'(r/r_0)}{r_0}\frac{\ln r}{r}\,\d y
 +O(\lam ^{-4}).
\]
The term containing \(e^{\hat\varphi_\xi}-1\) is expanded once in
\(\lam ^{-2}\); its error is \(O(\lam ^{-3})\), also after the
listed derivatives. Since \(\d y=e^{-\varphi_\xi}\,\d v_g\), the
annular integral in \(h_{\lam ,\xi}\) contains
\(e^{-\varphi_\xi}\). Multiplication by \(2/|\Si |_g\) proves the
scalar estimate in \eqref{eq:A-joint-remainder-bounds}. The chain rule
also yields \eqref{eq:A-joint-remainder-bounds} after the pullbacks in
Lemma~\ref{lem:terminal-reduction}, and hence the stated \(C^1\) remainder.

On $U_{r_0}(\xi_i)$, all cutoffs except $\vartheta_{\xi_i}$ vanish,
and $\vartheta_{\xi_i}=1$. Write $x=y_{\xi_i}^{-1}(y)$ and set
\[
 \mathcal A_i(y):=(\mathcal F_\xi^i\mathrm g_\xi^i)(x)
                                  e^{\hat\varphi_{\xi_i}(y)},\quad
 S:=\sum_j\al _j\widehat h_{\lam _j,\xi_j},\quad
 q_i(y):=\sum_j\al _j
   [-2\lam _j^{-2}F_{\xi_j}(x)
                         +\mathcal R_{\lam _j,\xi_j}(x)].
\]
The exact local mass factorization is
\begin{equation}\label{eq:A-exact-local-mass-factor}
 Ve^{u_0}\,\d v_g
 =\mathcal A_i(y)\frac{\lam _i^{4\al _i}}
                 {(1+\lam _i^2|y|^2)^{2\al _i}}
                  e^{S+q_i(y)}\,\d y.
\end{equation}
The correction terms in \eqref{eq:A-exact-local-mass-factor} are differentiated
in the center, spatial, and logarithmic scale variables before being estimated. In the logarithmic
mass coordinates \eqref{eq:log-mass-coordinates}, after the moving-chart
identification, $D_{(\xi,t,v)}^j q_i=O(e^{-2t})$ and
$D_{(\xi,t,v)}^jS=O(te^{-2t})$ for $j\le2$; the scalar $S$ cancels exactly from normalized
mass integrals. Amplitude derivatives introduce at most the factors
$Ct$ and $Ct^2$ for first and second derivatives, respectively.

The following expansion is recalled from Lemma A.1 in \cite{AHW2025-1}.
\begin{lem}\label{lem:projbubble_pointwise}
			As $\lam  \to + \infty$, the following expansion  holds in $C^1(\Si )$:
			\begin{align*}
				P  \delta_{\lam,\xi}=&\,\vartheta_{\xi} \Big( \delta_{\lam,\xi}+\ln  \frac{\lam^2}{8} \Big)+\kap( \xi) H^g(\cdot, \xi)+h_{\lam,\xi}-\frac{2F_{\xi}}{\lam^2}  +O\Big(\frac{\ln  \lam }{\lam ^3}\Big)\nonumber\\=&\,2\vartheta_{\xi}\ln \Big( \frac{\lam^2}{1+\lam^2|y_{\xi}|^2}\Big)+ \kap( \xi) H^g(\cdot, \xi) +\frac{\kap( \xi)}{2|\Si |_g}\frac{\ln  \lam }{\lam^2}+O\Big(\frac{1}{\lam^2}\Big) ,
			\end{align*}
			where $F_{\xi}$ is given by \eqref{Fxi} and
			\[
h_{\lam ,\xi}:=
\frac{\kap (\xi)}{2|\Si |_g}\frac{\ln\lam }{\lam ^2}
+\frac{2}{|\Si |_g\lam ^2}\Big(
\frac{\kap (\xi)}8
+\int_\Si \vartheta_\xi
 \frac{1-e^{-\varphi_\xi}}{|y_\xi|^2}\,\d v_g
-\int_\Si \frac1{r_0}\vartheta'\Big(\frac{|y_\xi|}{r_0}\Big)
 \frac{\ln|y_\xi|}{|y_\xi|}e^{-\varphi_\xi}\,\d v_g\Big).
\]
			In particular,  the following  expansion holds in $C_{loc}^1(\Si \setminus\{\xi\})$ as $\lam  \to + \infty$:
			\begin{align*}
				P  \delta_{\lam,\xi}=&\,\kap( \xi)G^g(\cdot, \xi)+h_{\lam,\xi}- \frac {2\vartheta_{\xi}} {|y_{\xi}|^2}\frac {1} {\lam^2}-\frac{2F_{\xi}}{\lam^2}+O\Big(\frac{\ln  \lam}{\lam ^3} \Big)\nonumber\\=&\,\kap( \xi)G^g(\cdot, \xi)+\frac{\kap( \xi)}{2|\Si |_g}\frac{\ln  \lam }{\lam^2}+O\Big(\frac{1}{\lam^2}\Big).
			\end{align*}
		\end{lem}

		The following expansion is recalled from Lemma A.2 in \cite{AHW2025-1}.
\begin{lem} \label{lem:projbubble_xi_pointwise}  As $\lam  \to + \infty$, the following expansion  holds in $C(\Si )$:
			\begin{align*}
				\frac{\pa 	P\delta_{\lam,\xi}}{\pa\xi_j}=&\, \vartheta_{\xi} \frac{\pa  \delta_{\lam,\xi}}{\pa \xi_j}   +\kap( \xi) \frac{\pa H^g(\cdot,\xi)}{\pa \xi_j} -4\ln  (|y_{\xi}|)\frac{\pa \vartheta_{\xi}}{\pa \xi_j } +O \Big(\frac{\ln \lam}{\lam^2}\Big)\\=&\,-\frac{4 \vartheta_{\xi}\lam^2 \langle y_\xi,\pa _{\xi_j}y_\xi\rangle_{\mathbb R^2}}{1+\lam^2|y_{\xi}|^2}    +\kap( \xi) \frac{\pa H^g(\cdot,\xi)}{\pa \xi_j} -4\ln  (|y_{\xi}|)\frac{\pa \vartheta_{\xi}}{\pa \xi_j } +O \Big(\frac{\ln \lam}{\lam^2}\Big),
			\end{align*}
			where \(j=1,\ldots,\ii(\xi)\).  In particular, the following  expansion holds in $C_{loc}(\Si \setminus\{\xi\})$ as $\lam  \to + \infty$:
			\[
			\frac{\pa 	P\delta_{\lam,\xi}}{\pa\xi_j}=\kap( \xi) \frac{\pa G^g(\cdot,\xi)}{\pa \xi_j}  +O\Big(\frac{\ln \lam}{\lam^2}\Big),
			\]
			where \(j=1,\ldots,\ii(\xi)\).
		\end{lem}

		\begin{lem} \label{lem:projbubble_lamderivative__pointwise} As $\lam  \to + \infty$, the following expansions  hold in $C(\Si )$:
			\begin{align*}
				\lam  \frac{\pa P\delta_{\lam,\xi}}{\pa \lam} =&\vartheta_{\xi}\Big(\lam  \frac{\pa  \delta_{\lam,\xi}}{\pa \lam} +2\Big)+h^1_{\lam,\xi} +\frac{4F_{\xi}}{\lam^{2}}+o\Big(\frac{1}{\lam^{2}}\Big)\notag\\=&\frac{4\vartheta_{\xi}}{1+\lam^2|y_{\xi}|^2}-\frac{\kap ( \xi)}{|\Si |_g}\frac{\ln \lam}{\lam^2}+O \Big(\frac{1}{\lam^2}\Big)
			\end{align*}and
			\begin{align*}
				\lam^2	\frac{\pa^2 P\delta_{\lam,\xi}}{\pa\lam^2}  =&\vartheta_{\xi} \Big(\lam^2	\frac{\pa^2 \delta_{\lam,\xi}}{\pa\lam^2}-2\Big)+h^2_{\lam,\xi}-\frac{12F_{\xi}}{\lam^2}+o\Big(\frac{1}{\lam^2}\Big) \\
				=&-\frac{4 \vartheta_{\xi} (1+3\lam^2|y_{\xi}|^2)}{(1+\lam^2|y_{\xi}|^2)^2}+h^2_{\lam,\xi}-\frac{12F_{\xi}}{\lam^2}+o\Big(\frac{1}{\lam^2}\Big),
			\end{align*}
where the scalar terms are the derivatives of
\(h_{\lam ,\xi}\) in Lemma~\ref{lem:projbubble_pointwise}:
\[
h^1_{\lam ,\xi}:=\lam \pa _\lam  h_{\lam ,\xi},
\quad
h^2_{\lam ,\xi}:=\lam ^2\pa ^2_{\lam \lam }h_{\lam ,\xi}.
\]
\end{lem}

\begin{lem}\label{lem:projbubble_derivative_integral}
			Let $\xi,\zeta\in\Si$ with $\xi\ne\zeta$, and let $\lam $ and $\lam _2$ be the two concentration scales. Then the following expansions hold as $\lam,\lam _2\to+\infty$:
			\begin{align*}
				\Big\langle P\delta_{\lam,\xi}, \frac{1}{\lam}\frac{\pa P\delta_{\lam,\xi}}{\pa \xi_j} \Big\rangle=& \,\frac{\kap^2( \xi)}{2\lam} \frac{\pa R^g(\xi) }{\pa \xi_j}+ O\Big(\frac{\ln \lam }{\lam ^{3} }\Big), \\
				\Big\langle P\delta_{\lam,\xi},\frac{1}{\lam _2} \frac{\pa P\delta_{\lam _2,\zeta}}{\pa \zeta_l} \Big\rangle =&\, \frac{\kap( \xi)\kap( \zeta)}{\lam _2}	\frac{\pa G^g(\xi,\zeta)}{\pa \zeta_l}	+O\Big(\frac{1}{\lam _2\lam ^2}+\frac{\ln\lam _2}{\lam _2^3}\Big), \\
				\Big\langle P\delta_{\lam,\xi}, \lam \frac{\pa P\delta_{\lam,\xi}}{\pa\lam} \Big\rangle=&\,2\kap( \xi)-\frac{\kap^2( \xi)}{|\Si |_g} \frac{\ln \lam }{\lam^2} + O\Big(\frac{1}{\lam^2}\Big), \\
				\Big\langle P\delta_{\lam,\xi}, \lam _2\frac{\pa P\delta_{\lam _2,\zeta}}{\pa \lam _2 } \Big\rangle=& \,-\frac{\kap( \xi)\kap( \zeta)}{|\Si |_g} \frac{\ln \lam _2}{\lam _2^2}+ O\Big(\frac{1}{\lam _2^2}+\frac{\ln\lam _2}{\lam ^2\lam _2^2}\Big),
			\end{align*}
			where \(j=1,\ldots,\ii(\xi)\) and \(l=1,\ldots,\ii(\zeta)\).
		\end{lem}
		\begin{proof}
The self-interaction estimates follow from
\cite[Lemma~A.4]{AHW2025-1}. The mixed estimates are obtained by
integration by parts using \eqref{projectedbubble} together with
Lemmas~\ref{lem:projbubble_xi_pointwise} and
\ref{lem:projbubble_lamderivative__pointwise}; the separation of the
centers proves the stated uniform error estimates without any restriction on
the relative scales.
\end{proof}

The following expansion is recalled from Lemma A.5 in \cite{AHW2025-1}.
\begin{lem}\label{lem:projbubble_integral_interaction}
			Under the assumptions of Lemma \ref{lem:projbubble_derivative_integral}, the expansions hold as $\lam,\lam _2 \to +\infty$:
			\[
			\langle P\delta_{\lam,\xi}, P\delta_{\lam,\xi}\rangle = 4\kap( \xi) \ln \lam +\kap^2( \xi) R^g(\xi)-2\kap( \xi) +\frac{\kap^2( \xi)}{|\Si |_g}\frac{\ln \lam }{\lam^2} + O\Big(\frac{1}{\lam^2}\Big)
			\]
			and
			\[
			\langle P\delta_{\lam,\xi}, P\delta_{\lam _2,\zeta}\rangle =\kap( \xi)\kap( \zeta) G^g(\xi,\zeta) +\frac{\kap( \xi)\kap( \zeta)}{2|\Si |_g}\frac{\ln \lam }{\lam^2}+\frac{\kap( \xi)\kap( \zeta)}{2|\Si |_g}\frac{\ln \lam _2}{\lam _2^2} + O\Big(\frac{1}{\lam^2}+\frac{1}{\lam _2^2}\Big).
			\]
		\end{lem}

		\begin{lem}\label{lem:Gram}
On a compact separated admissible coordinate set, write
$P_i:=P\delta_{\lam_i,\xi_i}$, $d_i:=\ii(\xi_i)$ and $\kap_i:=\kap(\xi_i)$, and set
\[
 A_i:=\frac{P_i}{\sqrt{\ln\lam_i}},\quad
 S_i:=\lam_i\pa_{\lam_i}P_i,\quad
 T_{ij}:=\lam_i^{-1}\pa_{(\xi_i)_j}P_i\quad(1\le j\le d_i).
\]
Let $(e_r)$ enumerate the normalized vectors. For sufficiently large
$\min_i\lam_i$, there are constants $c,C>0$, independent of the relative scales, such that
\begin{equation}\label{eq:normalized-frame-bounds}
 c\sum_r b_r^2\le\Big\|\sum_r b_re_r\Big\|^2
                 \le C\sum_r b_r^2\quad(b_r\in\mathbb R).
\end{equation}
Moreover, for every $i$ and admissible $j,k$,
\begin{equation}\label{eq:normalized-second-derivatives}
 \|\lam_i^2\pa_{\lam_i\lam_i}^2P_i\|
 +\|\pa_{\lam_i(\xi_i)_j}^2P_i\|
 +\|\lam_i^{-2}\pa_{(\xi_i)_j(\xi_i)_k}^2P_i\|\le C.
\end{equation}
\end{lem}
\begin{proof}
The amplitude pairings follow from Lemma~\ref{lem:projbubble_integral_interaction}.
For the other same-center pairings, set $z=\lam_i y_{\xi_i}$ and
$v_{ij}:=(\d y_{\xi_i})_{\xi_i}\pa_{(\xi_i)_j}$.
Differentiating the exact projected-bubble expansion and using
\eqref{eq:A-joint-remainder-bounds}, the normalized scale and center
profiles converge in Dirichlet norm to
\[
 s(z)=\frac4{1+|z|^2},\quad
 t_{ij}(z)=\frac{4z\cdot v_{ij}}{1+|z|^2}.
\]
The variable chart coefficients contribute
$O(\lam_i^{-1}\sqrt{\ln\lam_i})$; the smooth and cutoff errors also tend to zero.
Let $D_i=\mathbb R^2$ for an interior center and $D_i=\mathbb R^2_+$
for a boundary center. Direct integration yields
\[
 \int_{D_i}|\nabla s|^2\,\d z=\frac{4\kap_i}{3},\quad
 \int_{D_i}\nabla s\cdot\nabla t_{ij}\,\d z=0,\quad
 \int_{D_i}\nabla t_{ij}\cdot\nabla t_{ik}\,\d z
 =\frac{4\kap_i}{3}\,v_{ij}\cdot v_{ik}.
\]
For tangential vectors the half-plane integrals are half the plane integrals.
The matrix with entries
$v_{ij}\cdot v_{ik}=g_{\xi_i}(\pa_{(\xi_i)_j},\pa_{(\xi_i)_k})$
is uniformly positive definite. The remaining same-center pairings
are $o(1)$ by Lemma~\ref{lem:projbubble_derivative_integral} after amplitude normalization.
Every normalized vector has bounded norm, and its energy outside a fixed
neighborhood of its center tends to zero. By localization,
every different-center pairing is $o(1)$, independently of the relative scales.
Together with $\|A_i\|^2=4\kap_i+o(1)$, this proves \eqref{eq:normalized-frame-bounds}.

For \eqref{eq:normalized-second-derivatives}, differentiate the same exact expansion before
rescaling. The leading profiles are linear combinations of
\[
 \frac{1+3|z|^2}{(1+|z|^2)^2},\quad
 \frac{z_j}{(1+|z|^2)^2},\quad
 \frac1{1+|z|^2},\quad
 \frac{z_jz_k}{(1+|z|^2)^2},
\]
whose gradients are square integrable. Uniform estimates for the charts and
\eqref{eq:A-joint-remainder-bounds} yield the required estimates for the remaining terms.
\end{proof}

The following expansion is recalled from Lemma A.9 in \cite{AHW2025-1}.
\begin{lem}\label{lem:A.7}
			For    $u_0= \sum_{i}\al_i P\delta_{\lam_i,\xi_i}$ and
			$w\in E_{\lam,\xi }^{p,q}$, we have
\begin{align*}
\int_\Si  Ve^{u_0}w\,\d v_g
=&\,O\Big(\sum_i\lam _i^{4\al _i-2}\Big[
 |\al _i-1|+
 \frac{|(\nabla_x\mathcal F_\xi^i)(\xi_i)|
             +|\mathfrak f_2(\xi_i)|}{\lam _i}\\
&\,+\frac1{\lam _i}\sum_{j\ne i}
       (|\al _j-1|+\lam _j^{-2})
 +\frac{\ln^{3/2}\lam _i}{\lam _i^2}\Big]\|w\|\Big)
 +O(\|w\|),
\end{align*}
			where $\cF_{\xi}^i$ and $\mathfrak{f}_2$ are defined in  \eqref{eq:def_F}  and \eqref{def:frak_f2}, respectively. Here $\nabla_x\cF_\xi^i$ denotes the spatial gradient.
		\end{lem}

		The following expansion is recalled from Lemma A.10 in \cite{AHW2025-1}.
\begin{lem}\label{lem:Ve-u-integral}
			Let $u= u_0+w\in V_{\rho}^\eta(p,q,\var)$ with $\rho>0$, where
			$u_0= \sum_{i}\al_i P\delta_{\lam_i,\xi_i}$ and $w\in E_{\lam,\xi }^{p,q}$.
			Then
			\begin{align}
				\int_{\Si } V e^u \,\d v_g =&\,\sum_{i} \frac{\kap( \xi_i)(\cF^i_{\xi} \rg^i_{\xi})(\xi_i)}{8(2\al_i-1)}\lam_i^{4\al_i-2} + \frac{\pi}{2}\sum_{i} \mathfrak{f}_2(\xi_i)\lam_i^{4\al_i-3} \nonumber\\
				&\,+
				O\Big(\Big(\sum_i\lam _i^{4\al _i-2}\Big)
\Big[\sum_j\frac{\ln\lam _j}{\lam _j^2}
+\sum_j|\al _j-1|^2+\|w\|^2\Big]+1\Big),\label{lem:A.8-1}
			\end{align}
			where $\mathfrak{f}_2$ is defined in \eqref{def:frak_f2}.	Furthermore, if $\sum_{i}|\al_i-1|\ln \lam_i=o(1)$, then \eqref{lem:A.8-1} can be refined as follows:
			\begin{align*}
				\int_{\Si } V e^u \,\d v_g =&\,\sum_{i} \frac{\kap( \xi_i)(\cF^i_{\xi} \rg^i_{\xi})(\xi_i)}{8(2\al_i-1)}\lam_i^{4\al_i-2} \Big(1+\sum_j \frac{\kap(\xi_j)}{2|\Si|_g}\frac{\ln \lam_j}{\lam_j^2}\Big)\nonumber \\
				&\,+\frac{\pi}{2} \sum_{i} \mathfrak{f}_2(\xi_i)\lam_i^{4\al_i-3}+\sum_{i}
				\frac{\kap( \xi_i)}{{16}} \mathfrak{f}_3(\xi_i)\ln \lam_i\nonumber \\
				&\,+O\Big(\sum_{i} \lam_i^{2}(
				|\al_i-1|^2 +\|w\|^2)+1\Big),
			\end{align*}
			where $\mathfrak f_3$ is defined in \eqref{def:frak_f3}.
		\end{lem}

\begin{proof}
Integrate \eqref{eq:A-exact-local-mass-factor} on disjoint center
charts. Lemma~\ref{lem:A.7} and
\eqref{eq:A-uniform-weighted-Sobolev} estimate the terms containing
$w$. The integral outside the charts is $O(1)$; the common factor
$e^S$ multiplies the entire leading mass. This proves
\eqref{lem:A.8-1}; compare \cite[Lemma~A.10]{AHW2025-1}.

For the refined formula, Corollary~\ref{cor:lam_i_samespeed}
implies $\lam_i\asymp\Lam:=\min_j\lam_j$.
Write $a_i=\al_i-1$. Uniformly when
$\max_i|a_i|\ln\lam_i=o(1)$, radial integration yields
\begin{equation}\label{eq:A-uniform-radial-moments}
\begin{aligned}
 \int_{B_{r_0}^{\xi_i}}
 \frac{\lam_i^{4\al_i}|y|^2}{(1+\lam_i^2|y|^2)^{2\al_i}}\,\d y
 &=\frac{\kap_i}{4}\ln\lam_i
     +O(1+|a_i|\ln^2\lam_i),\\
 \int_{B_{r_0}^{\xi_i}}
 \frac{\lam_i^{4\al_i}|y|^3}{(1+\lam_i^2|y|^2)^{2\al_i}}\,\d y
 &=O(1),\quad \kap_i:=\kap(\xi_i).
\end{aligned}
\end{equation}
Indeed, the first radial integral is
\[
 \frac{\kap_i}{4}\lam_i^{4a_i}
 \int_0^{\lam_i r_0}\frac{r^3}{(1+r^2)^{2+2a_i}}\,\d r.
\]
Its value at $a_i=0$ is $(\kap_i/4)\ln\lam_i+O(1)$;
its first $a_i$-derivative is $O(\ln^2\lam_i)$ in the stated
region. The cubic radial integral in \eqref{eq:A-uniform-radial-moments}
is uniformly integrable. At a boundary center, the integral with the normal coordinate is
\[
 \int_{\B_{r_0}^{+}}
 \frac{\lam_i^{4\al_i}y_2}{(1+\lam_i^2|y|^2)^{2\al_i}}\,\d y
 =\frac\pi2\lam_i^{4\al_i-3}
       +O(1+|a_i|\lam_i).
\]
The last estimate follows by $z=\lam_i y$ and the differentiability
near $\al_i=1$ of the convergent integral with the normal coordinate.
Taylor expansion of the smooth factor in
\eqref{eq:A-exact-local-mass-factor}, followed by
\(|a_i|\ln^2\lam_i+|a_i|\lam_i \le C(1+\lam_i^2a_i^2),\)
proves the stated second-order terms. The smooth projection terms
contribute $O(1)$ after scale comparison. The linear and quadratic
terms in $w$ are at most
$C\sum_i\lam_i^2(a_i^2+\|w\|^2)+C$, by
Lemma~\ref{lem:A.7}. Expanding $e^S$ retains the common scalar term
and completes the refined formula.
\end{proof}

\begin{proof}[Proof of Lemma~\ref{lem:terminal-reduction}]
All center configurations remain in a fixed compact subset of
\(\Xi_{p,q}\). Recall $n=p+q$ and $d=2p+q$. Set
\(a:=\al -\mathbf1\), \(a_i=\al _i-1\), with \(t\), \(v\), and \(h(t)\) as in
\eqref{eq:log-mass-coordinates}.

\textbf{Step 1.} Coercivity on the orthogonal subspace.
From \eqref{eq:log-mass-coordinates},
\begin{equation}\label{eq:A-logmass-scales}
\ln\lam _i=t+O(1),\quad
D_{(t,v,\xi)}\ln\lam _i=O(1),\quad
D_\al \ln\lam _i=O(t),
\end{equation}
and the corresponding second derivatives satisfy analogous estimates.
For the measure \(\nu_i\) defined in Section~\ref{sec:2}, amplitude
orthogonality yields \(\int_\Si  w\,\d\nu_i=0\).
In rescaled coordinates, \(\nu_i\) is a fixed multiple of spherical
area times a cutoff. Extend across the missing cap by inversion and
harmonic extension of the trace; the trace estimate on the adjacent
annulus is scale invariant. At a boundary center first reflect evenly.
The extension preserves constants, and its Dirichlet energy is at most
a fixed multiple of the original energy. Since the cutoff measure has mass bounded
below and is bounded above by a fixed multiple of spherical area,
the condition \(\int_\Si  w\,\d\nu_i=0\) and the spherical Poincar\'e
inequality estimate the spherical mean of the extension by its
Dirichlet norm. The spherical Moser--Trudinger inequality implies
\[
 \int_\Si e^{b|w|}\,\d\nu_i
 \le C\exp(Cb^2\|w\|^2),\quad b\ge0,
 \quad \|w\|_{L^r(\nu_i)}\le C\sqrt r\,\|w\|\quad(r\ge2).
\]
H\"older's inequality now yields
\[
 \int_\Si |w|^r e^{b|w|}\,\d\nu_i
 \le\Big(\int_\Si |w|^{2r}\,\d\nu_i\Big)^{1/2}
      \Big(\int_\Si e^{2b|w|}\,\d\nu_i\Big)^{1/2}
 \le C_{r,b}\|w\|^r
\]
for $\|w\|\le\epsilon_0$, proving
\eqref{eq:A-uniform-weighted-Sobolev}.
For arbitrary $w\in H^1(\Si)$, the corresponding $L^r$ estimate is
\[
 \|w\|_{L^r(\nu_i)}\le C\sqrt r\Big(
 \|\nabla w\|_{L^2(\Si)}+
\Big|\int_\Si w\,\d\si_i\Big|\Big),\quad
 \si_i:=\nu_i/\nu_i(\Si),\quad r\ge2.
\]

We next prove uniform coercivity. For a sequence with \(\|w_n\|=1\),
the extensions above are bounded in spherical \(H^1\) and converge,
after passage to a subsequence, strongly in \(L^2\). The orthogonality
conditions pass to the limit: the cutoff derivative terms vanish on the
annulus, and the amplitude, scale, and center conditions remove the
constant and first spherical harmonics. At the boundary, even reflection
also excludes the normal odd harmonic. The next eigenvalue is \(6\).
In stereographic coordinates,
$8(1+|z|^2)^{-2}\,\d z=2\,\d A_{\mathbb S^2}$.
For each limiting profile $W_i$, the spectral inequality is therefore
\[
 2\int_{\mathbb S^2}W_i^2\,\d A
 \le\frac13\int_{\mathbb S^2}|\nabla W_i|^2\,\d A.
\]
For a boundary profile, divide both integrals of the even extension
by two. To compare the limiting energy, first restrict to fixed
rescaled disks (or half-disks), which correspond to disjoint subsets
of $\Si$. Weak lower semicontinuity and then increasing the rescaled
radii show that the sum of the limiting energy is at most
$\lim_n\|w_n\|^2=1$. Consequently,
\[
 \limsup_n\sum_i\int_\Si w_n^2\,\d\nu_i\le\frac13.
\]
Strong \(L^2\) convergence, or the uniform \(L^r\) estimate with \(r>2\),
excludes loss of weighted mass at the cap. A contradiction argument thus
yields, for any fixed \(c<2/3\) and all sufficiently large scales,
\[
\|w\|^2-\sum_i\int_\Si  w^2\,\d\nu_i\ge c\|w\|^2.
\]
The exact local factorization
\eqref{eq:A-exact-local-mass-factor}, \eqref{eq:A-uniform-weighted-Sobolev},
and \eqref{eq:A-logmass-scales} show that the Hessian of \(J_\rho\) on the
subspace $\mathcal E_{\lam,\xi}^{p,q}$ differs from the preceding quadratic form by
\(O(|\mu|+|v|+|a|+e^{-t})\|w\|^2.\)
Thus, after decreasing \(\delta,\mu_0\) and increasing \(T\), it is
uniformly positive. Positivity persists in a fixed small ball of
$\mathcal E_{\lam,\xi}^{p,q}$.

\textbf{Step 2.} The minimizing correction.
Fix a reference configuration \(\xi^0\) and disjoint coordinate
neighborhoods with smaller concentric disks or half-disks.
A cutoff applied to the differences between the local coordinate maps yields
a smooth family of boundary-preserving diffeomorphisms
\(\Phi_\xi:\Si\to\Si\), equal to the identity outside the larger neighborhoods,
such that
\[
 \Phi_\xi=y_{\xi_i}^{-1}\circ y_{\xi_i^0}
 \quad\text{on each smaller neighborhood},\quad
 \Phi_{\xi^0}=\operatorname{id}.
\]
For \(\xi\) near \(\xi^0\), all required derivatives are bounded
independently of the scales.

We first work in \(\mathcal H:=H^1(\Si )/\mathbb R\) and return to mean-zero representatives at the end. The pulled-back Dirichlet form is
\[
 a_\xi(v,z):=
 \int_\Si
 \langle\nabla v,\nabla z\rangle_{\Phi_\xi^*g}
 \,\d v_{\Phi_\xi^*g}.
\]
The forms \(a_\xi\) are uniformly elliptic and depend \(C^2\)-smoothly on
\(\xi\) in the operator norm from \(\mathcal H\) to \(\mathcal H^*\).
On each smaller disk or half-disk they equal the Euclidean Dirichlet form in the fixed chart, by conformal invariance. On the remaining region,
the operator estimates follow by differentiating the smooth uniformly
bounded coefficients. Only the coefficients of
the forms are differentiated; no composition operator on \(H^1\) is differentiated.

Let \(e_r(\theta)\) be the normalized amplitude, scale, and admissible
center frame of Lemma~\ref{lem:Gram}, and set
\(\widehat e_r(\theta):=[e_r(\theta)\circ\Phi_\xi]\in\mathcal H\).
For \(\zeta=(\xi,t,v)\), the exact projected-bubble formulas yield
\begin{equation}\label{eq:A-pulledback-frame-bounds}
 \|D_{\zeta}^j\widehat e_r(\theta)\|_{\mathcal H}\le C,
 \quad j=0,1,2.
\end{equation}
Here the center motion has been removed before differentiation. For
example, on the fixed smaller disk or half-disk, the pulled-back normalized center profile has the form
\[
 -\frac{4z\cdot A_{ij}(\xi_i,z/\lam _i)}{1+|z|^2}
 +O_{\mathcal H}(\lam _i^{-1}),
 \quad
 A_{ij}(\xi_i,y):=
 (\pa _{(\xi_i)_j}y_{\xi_i})(y_{\xi_i}^{-1}(y)),
 \quad z=\lam _i y.
\]
The coefficients \(A_{ij}\) are smooth, and all derivatives of \(\ln\lam _i\) with respect to \(\zeta\) through order two are bounded.
The scale profiles are handled in the same way. The spatially constant part of the amplitude profile represents zero in \(\mathcal H\), and
the annular and global smooth terms satisfy
\eqref{eq:A-pulledback-frame-bounds}.

Define
\[
 G_{rs}(\theta):=a_\xi(\widehat e_r,\widehat e_s),
 \quad
 \mathcal P_\theta z
 :=z-\sum_{r,s}\widehat e_r
       (G^{-1})_{rs}a_\xi(z,\widehat e_s).
\]
The inverse $G^{-1}$ is uniformly bounded by Lemma~\ref{lem:Gram}.
Differentiating the formula for $\mathcal P_\theta$ and using
\eqref{eq:A-pulledback-frame-bounds}, we obtain
\[
\|D_{\zeta}^j\mathcal P_\theta\|_{\mathcal L(\mathcal H)}\le C,
 \quad j=0,1,2.
\]
Amplitude differentiation uses instead
\(D_\al \ln\lam _i=O(t)\). The same frame and Gram formulas yield
\[
 \|D_\al \mathcal P_\theta\|
 +\|D_{\zeta}D_\al \mathcal P_\theta\|\le Ct,
 \quad
 \|D_\al ^2\mathcal P_\theta\|\le Ct^2,
\]
The same estimates hold for the local isomorphism $T_\theta$
below and its inverse.
Let \(\mathcal E^0\) be the image of \(\mathcal P_\theta\) at a fixed
reference parameter. The image at any nearby parameter is the pullback
of $\mathcal E_{\lam,\xi}^{p,q}$. After shrinking the
parameter neighborhood,
\(T_\theta:=\mathcal P_\theta|_{\mathcal E^0} :\mathcal E^0\to \operatorname{im}\mathcal P_\theta\)
is an isomorphism. The maps \(T_\theta\), their inverses, and their first two
\(\zeta\)-derivatives have uniformly bounded operator norms. For
\(\eta\in\mathcal E^0\), let \(z\) be the representative of
\(T_\theta\eta\) with zero mean for the reference metric, and set
\[
 w=z\circ\Phi_\xi^{-1}
   -\frac1{|\Si |_g}\int_\Si  z\,\d v_{\Phi_\xi^*g}.
\]
Then \([w\circ\Phi_\xi]=T_\theta\eta\). The subtracted mean and its
first two parameter derivatives are bounded linear functionals of \(z\).
We include the subtracted mean in the logarithmic integral and differentiate the
resulting functional on \(\mathcal E^0\).

On the fixed space \(\mathcal E^0\), write the functional increment as
\[
 J_\rho(u_0+w)-J_\rho(u_0)
 =f_\theta(\eta)+\tfrac12 C_\theta[\eta,\eta]+N_\theta(\eta),
 \quad C_\theta\ge cI,
\]
where $\theta=(\al ,v,\xi,t)$ and $w$ is reconstructed from $\eta$ as above.
We estimate the parameter derivatives of \(f_\theta\), \(C_\theta\), and \(N_\theta\) on \(\mathcal E^0\).
Set $e:=|a|+e^{-t}$, $\zeta=(\xi,t,v)$, and set
\[
 m_\theta:=\frac{Ve^{u_0}\,\d v_g}{\int_\Si  Ve^{u_0}\,\d v_g},
 \quad p_i:=\frac{\kap _i e^{v_i}}{\sum_j\kap _j e^{v_j}},
 \quad R_\theta:=m_\theta-\sum_i p_i\si _i.
\]
The signed measure $R_\theta$ has total mass zero. If $z_\theta$ represents $T_\theta\eta$ and has zero mean with
respect to the reference metric, then
amplitude orthogonality and the mean-zero normalization yield the identity
\[
 f_\theta(\eta)
 =-\rho\int_\Si  z_\theta\,\d(\Phi_\xi^*R_\theta).
\]
The subtracted mean cancels by the zero total mass of $R_\theta$.

On a fixed smaller disk or half-disk, apply \eqref{eq:A-exact-local-mass-factor}. In the
coordinates \(z=\lam _i y\), the error in freezing the smooth
coefficient at the center consists of the exponent change, the variation
of \(\mathcal A_i\), and the factor \(e^{S+q_i}\). Uniformly for \(j\le2\),
\begin{equation}\label{eq:A-local-coefficient-bounds}
 \begin{split}
 &(1+|z|^2)^{-2a_i}-1
       =O_{L^r((1+|z|^2)^{-2}\d z)}(|a_i|),\\
 &D_{\zeta}^j[\mathcal A_i(y)-\mathcal A_i(0)]=O(|y|),\\
 &D_{\zeta}^j q_i=O(e^{-2t}),\\
 &\int_{\R^2}|y|\,|D_{\zeta}^j K_{\lam _i}(y)|\,\d y=O(e^{-t}),
 \quad K_\lam (y)=\frac{\lam ^2}{(1+\lam ^2|y|^2)^2}.
 \end{split}
\end{equation}
Here any one sufficiently small fixed $r>1$ may be used in the first
line; amplitudes are restricted to their fixed small neighborhood of
$\mathbf1$. All $\zeta$-derivatives of $\ln\lam _i$ through order two
are bounded. The scalar $S$ cancels when the masses on $U_{r_0}(\xi_i)$ are normalized;
one may also include $S$ in the error since
$D_{\zeta}^jS=O(te^{-2t})=O(e^{-t})$. The normalized integral outside the union of the smaller coordinate neighborhoods, and its \(\zeta\)-derivatives through order two, are $O(e^{-2t})$ in the dual norms of the fixed Sobolev spaces. The same estimates hold after differentiating the
cutoffs, whose derivatives are supported away from the centers.
The first two derivatives of \(p_i\) are bounded. Normalization by the
total mass, using \eqref{eq:A-exact-local-mass-factor}, preserves the
orders in \eqref{eq:A-local-coefficient-bounds} for
\(\Phi_\xi^*R_\theta\).

For \(D_{\zeta}^2q_i\), differentiate
\eqref{eq:A1-exact-error-equation} before applying the Neumann inverse.
Equation~\eqref{eq:A-annular-derivative-bounds} and Schauder estimates
yield the joint estimates for the center and spatial derivatives of
\(\mathcal R_{\lam ,\xi}\circ\Phi_\xi\).
The same argument applies to \(F_\xi\) by
\eqref{Fxi}; the remaining terms in \(q_i\) are evaluated
away from their poles.

We estimate the pairing with \(\Phi_\xi^*R_\theta\) in weighted norms,
not its norm on the full \(H^1\) space before pullback. The estimates for
\(T_\theta\) yield \(\|D_{\zeta}^jz_\theta\|_{H^1/\mathbb R}\le C\|\eta\|\)
for \(j\le2\). Moreover,
\[
 \int_\Si  z_\theta\,\d(\Phi_\xi^*\si _i)
 =\frac1{|\Si |_g}\int_\Si  z_\theta\,\d v_{\Phi_\xi^*g}.
\]
Differentiate the preceding mean identity once and twice with respect to \(\zeta\). On each fixed smaller disk or half-disk, every derivative of the normalized kernel is the kernel times a bounded rational function of $\lam _i y$; cutoff errors are bounded on the
fixed annulus. The generalized weighted estimate from Step 1,
applied to $z_\theta,D_{\zeta}z_\theta,D_{\zeta}^2z_\theta$, yields
\[
 \|D_{\zeta}^j z_\theta\|_{L^r(\Phi_\xi^*\nu_i)}
       \le C_r\|\eta\|,\quad j\le2.
\]
Differentiating the pairing with \(\Phi_\xi^*R_\theta\) twice and using
the weighted estimates yields \(\|D_{\zeta}^j f_\theta\|\le Ce\) for \(j\le2\).

The quadratic term is the pulled-back Dirichlet form minus $\rho$
times covariance with respect to $m_\theta$. The form has bounded
first two derivatives by construction. The same kernel and test-function
estimates imply $\|D_{\zeta}^jC_\theta\|\le C$, $j\le2$. For higher derivatives in \(\eta\), differentiate the exact normalized
exponential integral, as in the proof of
\eqref{eq:R-derivative-estimates-wbar}. By H\"older's inequality and
\eqref{eq:A-uniform-weighted-Sobolev}, the estimates remain uniform after up to
two \(\zeta\)-derivatives. Taylor's integral remainder then yields
\(\|D_\eta N\|\le C\|\eta\|^2\) and
\(\|D_\eta^2N\|\le C\|\eta\|\), also after up to two \(\zeta\)-derivatives.

Differentiating the radial factor in an amplitude variable introduces
$\ln(1+|z|^2)$, each fixed power of which is integrable in the same weighted norms; an amplitude derivative of a scale introduces at
most a factor $Ct$. Together with the amplitude estimates for
$T_\theta$, differentiation yields
\[
\begin{aligned}
 \|D_\al f_\theta\|+\|D_{\zeta}D_\al f_\theta\|
       &\le C(1+te),
 &\quad \|D_\al^2f_\theta\|&\le C(t+t^2e),\\
 \|D_\al C_\theta\|+\|D_{\zeta}D_\al C_\theta\|
       &\le Ct,
 &\quad \|D_\al^2C_\theta\|&\le Ct^2.
\end{aligned}
\]
with the corresponding Taylor remainder estimates. Thus the projected
Euler equation and its first two derivatives in the fixed space yield
\begin{equation}\label{eq:A-pulledback-correction-bounds}
 \|\bar\eta\|+\|D_{\zeta}\bar\eta\|+\|D_{\zeta}^2\bar\eta\|\le Ce,
 \quad T_\theta\bar\eta=[\bar w\circ\Phi_\xi].
\end{equation}
Uniform norm equivalence yields \(\|\bar w\|\le Ce\).
Uniqueness implies that the local solutions agree on the intersections of their
parameter domains.

By the preceding estimates and \(C_\theta\ge cI\), the projected equation admits a contraction formulation. At its solution the \(\eta\)-derivative
vanishes, so the chain rule and
\eqref{eq:A-pulledback-correction-bounds} yield
\begin{equation}\label{eq:A-fiber-value-derivatives}
 |D_{\zeta}^j(j_\rho-J_\rho(u_0))|
 \le C(|a|+e^{-t})^2,\quad j=0,1,2,
\end{equation}
and
\begin{equation}\label{eq:A-fiber-fast-derivatives}
\begin{gathered}
 |D_\al (j_\rho-J_\rho(u_0))|
 +|D_\al  D_{(\xi,t,v)}(j_\rho-J_\rho(u_0))|
 \le C(|a|+e^{-t})(1+t|a|+te^{-t}),\\
 \|D_\al ^2(j_\rho-J_\rho(u_0))\|
 \le C(1+t|a|+te^{-t})^2.
\end{gathered}
\end{equation}

\textbf{Step 3.} Expansion before solving the \((\al ,v)\)-equations.
Substitution of \eqref{eq:log-mass-coordinates} into the Dirichlet and mass
terms yields
\begin{equation}\label{eq:A-logmass-leading-energy}
\begin{aligned}
\mathcal E_\rho(\al ,v,\xi,t)
={}&
\sum_i
\frac{\kap _i\al _i^2}{2\al _i-1}
\Big[
2t+v_i-\ln\Big((\cF_\xi^i\rg_\xi^i)(\xi_i)\Big)
+\ln(2\al _i-1)
\Big]\\
&-\sum_i\kap _i\al _i^2
+\frac12\sum_i\al _i^2\kap _i^2R^g(\xi_i)
+\frac12\sum_i\sum_{j\ne i}
\al _i\al _j\kap _i\kap _jG^g(\xi_i,\xi_j)\\
&-\rho\Big[
2t+\ln\Big(\frac18\sum_i\kap _i e^{v_i}\Big)
\Big].
\end{aligned}
\end{equation}
At \((\al ,v)=(\mathbf1,0)\), set
\[
P_\rho(\xi,t)
:=
-\left\{\begin{aligned}
&\dfrac{4\pi\rho}{\rho_*}\cL_1(\xi)e^{-t}&&\text{ if }\,q>0,\\
&\dfrac{\rho}{2\rho_*}\cL_2(\xi)t e^{-2t}&&\text{ if }\,q=0.
\end{aligned}\right.
\]
The local integral expansions in Lemma~\ref{lem:Ve-u-integral}, together
with the projected-bubble estimates above, yield
\begin{equation}\label{eq:A-bare-differentiated-remainder}
\begin{gathered}
|
\pa _t^j
[J_\rho(u_0)-\mathcal E_\rho-P_\rho](\mathbf1,0,\xi,t)
|
\le(\om (t)+C|\mu|)h(t),
\quad j=0,1,2,\\
|D_\xi(J_\rho(u_0)-\mathcal E_\rho)|
+
|D_tD_\xi(J_\rho(u_0)-\mathcal E_\rho)|
\le Ch(t),\\
\|D_\xi^2(J_\rho(u_0)-\mathcal E_\rho)\|=o(1),
\end{gathered}
\end{equation}
and
\begin{equation}\label{eq:A-bare-fast-derivatives}
\begin{gathered}
|D_\al \pa _t^j(J_\rho(u_0)-\mathcal E_\rho)|
\le Cth(t),\quad
|D_v\pa _t^j(J_\rho(u_0)-\mathcal E_\rho)|
\le Ch(t),\\
\|D_{\al \al }(J_\rho(u_0)-\mathcal E_\rho)\|
\le Ct^2h(t),\quad
\|D_{\al  v}(J_\rho(u_0)-\mathcal E_\rho)\|
\le Cth(t),\\
\|D_{vv}(J_\rho(u_0)-\mathcal E_\rho)\|
\le Ch(t),
\end{gathered}
\end{equation}
for \(j\le2\), with one center derivative allowed in the first two estimates.

For each fixed $j\le3$, direct differentiation shows
\[
 L^jK_\lam (y)=K_\lam (y)q_j(\lam ^2|y|^2),
 \quad \sup_{s\ge0}|q_j(s)|<\infty.
\]
For \(k=0,1,2\), radial integration after \(z=\lam  y\) yields
\begin{equation}\label{eq:A-logarithmic-kernel-bounds}
\begin{aligned}
 \int_{|y|<R}|y|\,|L^jK_\lam (y)|
       [1+\ln(1+\lam ^2|y|^2)]^k\,\d y
      &\le C\lam ^{-1},\\
 \int_{|y|<R}|y|^2\,|L^jK_\lam (y)|
       [1+\ln(1+\lam ^2|y|^2)]^k\,\d y
      &\le C\lam ^{-2}(\ln\lam )^{k+1},\\
 \int_{|y|<R}|y|^3\,|L^jK_\lam (y)|
       [1+\ln(1+\lam ^2|y|^2)]^k\,\d y
      &\le C\lam ^{-2}(\ln\lam )^k.
\end{aligned}
\end{equation}
The estimates in \eqref{eq:A-logarithmic-kernel-bounds} also hold on half-disks.
On $t|a|+|v|<\delta$, the factor $(1+|z|^2)^{-2a_i}$ is uniformly
bounded for $|z|<C\lam_i$, by \eqref{eq:A-logmass-scales}.
Each amplitude derivative introduces either $\ln(1+|z|^2)$ or a factor
$O(t)$; two derivatives therefore introduce at most a factor $Ct^2$.
Together with the exact leading masses $\kap_i e^{2t+v_i}/8$,
the kernel estimates yield \eqref{eq:A-bare-fast-derivatives}
throughout this region, including the mixed scale derivatives.

For center derivatives, work in the fixed local coordinates from Step 2.
The coefficients \(\mathcal A_i\) in the fixed coordinates have uniformly bounded
joint derivatives through total order three. One center derivative may
therefore be Taylor-expanded to second spatial order before
integration, yielding estimates of order $Ch(t)$. For two center derivatives
only convergence to the center value is needed: uniformly on the
compact configuration set,
\[
 \Big|\int_{B_{2r_0}^{\xi_i}}\vartheta(|y|/r_0)K_{\lam _i}(y)
   [D_\xi^2\mathcal A_i(y)-D_\xi^2\mathcal A_i(0)]\,\d y\Big|
 \le C\lam _i^{-1}=o(1).
\]
The derivatives of the logarithms of the scales are bounded, so the same estimates for the rational kernels apply. The exact scalar mean
and the annular remainder are differentiated from their integral
and Neumann-equation formulas, respectively. Thus the estimate for
\(D_\xi^2(J_\rho(u_0)-\mathcal E_\rho)\) uses $V\in C^3$
without differentiating $\mathcal L_2$ twice.

\textbf{Step 4.} The \((\al ,v)\)-equations.
From \eqref{eq:A-logmass-leading-energy},
\begin{equation}\label{eq:A-leading-fast-Hessian}
\begin{gathered}
\mathcal E_\rho(\mathbf1,0,\xi,t)
=
c(\rho)-\frac12\FpqV(\xi)-2(\rho-\rho_*)t,\\
D_\al \mathcal E_\rho(\mathbf1,v,\xi,t)=0,
\quad
D_v\mathcal E_\rho(\mathbf1,0,\xi,t)=0,\\
D_{\al \al }\mathcal E_\rho
=
4t\,\operatorname{diag}(\kap _i)+O(1),
\quad
D_{\al  v}\mathcal E_\rho(\mathbf1,0,\xi,t)=0,\\
D_{vv}\mathcal E_\rho(\mathbf1,0,\xi,t)[z,z]
=
-\frac{\rho}{\rho_*}\sum_i\kap _i z_i^2,
\quad
\sum_i\kap _i z_i=0.
\end{gathered}
\end{equation}
Set $b_i=\sqrt t\,a_i$ and
$E_v=\{v\in\mathbb R^n:\sum_i\kap _i v_i=0\}$, equipped with its
Euclidean norm. Write $z=(b,v)$ and
\[
 \Omega_t:=\{(b,v)\in\mathbb R^n\times E_v:
                 \sqrt t\,|b|+|v|<\delta\},\quad
 G_{\xi,t,\rho}(z):=D_zj_\rho.
\]
The set $\Omega_t$ is convex. Let $M_\rho$ be the self-adjoint operator defined by the quadratic form
\[
 \langle M_\rho(b,v),(b,v)\rangle
 =4\sum_i\kap _i b_i^2
   -\frac{\rho}{\rho_*}\sum_i\kap _i v_i^2,
 \quad v\in E_v.
\]
The inverse $M_\rho^{-1}$ is uniformly bounded for $|\mu|<\mu_0$.
Equations \eqref{eq:A-leading-fast-Hessian},
\eqref{eq:A-fiber-fast-derivatives}, and
\eqref{eq:A-bare-fast-derivatives} imply, after decreasing $\delta$
and increasing $T$, that
\[
 \sup_{z\in\Omega_t}
 \|M_\rho^{-1}(D_zG_{\xi,t,\rho}(z)-M_\rho)\|\le\tfrac12,
 \quad
 |M_\rho^{-1}G_{\xi,t,\rho}(0)|
 \le C\Big(\sqrt t\,h(t)+\frac{e^{-t}}{\sqrt t}\Big).
\]
Set
\[
 r_t:=2C\Big(\sqrt t\,h(t)+\frac{e^{-t}}{\sqrt t}\Big).
\]
Since $(\sqrt t+1)r_t\to0$, the closed Euclidean ball
$\overline B_{r_t}(0)$ is contained in $\Omega_t$ for sufficiently
large $t$. The map
\(z\mapsto z-M_\rho^{-1}G_{\xi,t,\rho}(z)\)
is a contraction of \(\overline B_{r_t}(0)\) into itself.
The map therefore has a fixed point. If $z,z'\in\Omega_t$ are any two
zeros of $G_{\xi,t,\rho}$, the fundamental theorem of calculus on the
segment joining \(z\) and \(z'\) yields
$|z-z'|\le\tfrac12|z-z'|$, hence $z=z'$.
Thus the zero is unique in \(\Omega_t\). The uniform Hessian estimate also preserves the numbers of positive and negative eigenvalues of $M_\rho$, namely $n$ and $n-1$.

The sharper componentwise estimates follow from the exact identities
\[
 D_\al \mathcal E_\rho(\mathbf1,v,\xi,t)=0,\quad
 D_{\al v}\mathcal E_\rho(\mathbf1,v,\xi,t)=0
\]
from \eqref{eq:A-logmass-leading-energy}. The amplitude equations, divided
by $t$, together with \eqref{eq:A-fiber-fast-derivatives} and
\eqref{eq:A-bare-fast-derivatives}, yield
\[
 |a|\le C(h(t)+\frac{e^{-t}}t)
        +o(1)|a|.
\]
For large $t$, moving the $o(1)|a|$ term to the left proves the first
estimate in \eqref{eq:fast-variable-bounds}. The relative-scale equations imply
\[
 |v|\le C(h(t)+|a|^2+e^{-2t})+o(1)|v|
       \le Ch(t)+o(1)|v|.
\]
For large $t$, move the $o(1)|v|$ term to the left. Then
\eqref{eq:A-pulledback-correction-bounds} yields
$\|\bar w\|\le C(|a|+e^{-t})\le Ce^{-t}$.
All choices are uniform on the compact center and potential families
in the lemma.

\textbf{Step 5.} The remaining variables.
By \eqref{eq:fast-variable-bounds} and
\eqref{eq:A-fiber-value-derivatives}, at \((\al _\rho,v_\rho)\),
\[
t|a|^2+|v|^2=o(h(t)),
\quad
j_\rho-J_\rho(u_0)=O(e^{-2t})=o(h(t)).
\]
Hence \eqref{eq:A-bare-differentiated-remainder} yields
\eqref{eq:terminal-value}. For the second derivatives, with
\(y=(\al ,v)\), differentiating $D_yj_\rho=0$ on the solution graph
$y=(\al _\rho,v_\rho)$ with respect to \(\xi\) and \(t\) yields
\begin{equation}\label{eq:A-reduced-chain-rule}
K_{\rho,ab}
=
j_{\rho,ab}
-
j_{\rho,ay}(j_{\rho,yy})^{-1}j_{\rho,yb},
\quad
a,b\in\{t,\xi_1,\ldots,\xi_d\}.
\end{equation}
At the critical point of the \((\al ,v)\)-equations,
\eqref{eq:A-logmass-leading-energy} and \eqref{eq:A-fiber-fast-derivatives} and
\eqref{eq:A-bare-fast-derivatives} imply
\[
|j_{\rho,t\al }|+|j_{\rho,\xi\al }|
 \le C(th(t)+e^{-t}),\quad
|j_{\rho,tv}|+|j_{\rho,\xi v}|\le Ch(t).
\]
Here \(D_\al \mathcal E_\rho(\mathbf1,v,\xi,t)=0\), while the
\(v\)-dependent part of \(\mathcal E_\rho(\mathbf1,v,\xi,t)\) is
independent of both \(\xi\) and \(t\). Consequently the preceding mixed
derivatives have no constant leading term. Step 4 yields a uniform inverse after rescaling amplitude
variations by \(\sqrt t\). Therefore, for a scale
or center index \(b\),
\begin{equation}\label{eq:A-Schur-bound}
|j_{\rho,ty}(j_{\rho,yy})^{-1}j_{\rho,yb}|
\le C\Big(\frac{(th(t)+e^{-t})^2}{t}+h(t)^2\Big)
\le C\Big(th(t)^2+\frac{e^{-2t}}t+h(t)^2\Big)=o(h(t)).
\end{equation}
The inverse in \eqref{eq:A-Schur-bound} is defined on the full
$(\al,v)$ space, including the mixed entries.
Substituting \eqref{eq:A-Schur-bound} into
\eqref{eq:A-reduced-chain-rule} proves
\eqref{eq:terminal-scale-derivatives}. Equations
\eqref{eq:A-reduced-chain-rule} and
\eqref{eq:A-bare-differentiated-remainder} yield
\eqref{eq:terminal-location-derivatives}. All estimates are uniform for
\(V\) in compact subsets of \(C^3_+(\Si )\), hence also for every
sequence \(V_k\to V\) in \(C^3\).
\end{proof}

\section{Topology of the very negative sublevels}\label{app:topology}

We prove the sublevel--barycenter comparison used in
Section~\ref{sec:4}, with the notation
\eqref{eq:barycenter-stratum}--\eqref{eq:weighted-barycenter-space}.

We keep the standing assumptions on $(\Si ,g)$ and $V$, and all sublevels
are closed. We also reuse the compact operator $T_\rho$ defined in the
introduction. With the Hilbert structure of $\oH$ fixed there,
\[
 \nabla J_\rho(u)=u-T_\rho(u).
\]
Thus $u=T_\rho(u)$ is exactly the equation at parameter $\rho$.
For a function on the surface, define the normalized probability measure
\begin{equation*}
 \mathfrak m_v:=\frac{Ve^v\,\d v_g}{\int_\Si  Ve^v\,\d v_g},
 \quad v\in H^1(\Si ).
\end{equation*}
The measure $\mathfrak m_v$ is invariant under addition of constants.
For a set $\mathcal U$ of probability measures,
$\mathfrak m^{-1}(\mathcal U)$ denotes its preimage
under $v\mapsto\mathfrak m_v$.

The following comparison is independent of \ref{item:C1}--\ref{item:C2}.
\begin{prop}\label{prop:negative-sublevels-barycenters}
Let $(\Si ,g)$ satisfy the standing geometric assumptions, let
$V\in C^3(\Si )$ be positive, and let $\rho_*=4\pi m$ with $m\ge2$.
For all sufficiently large $L$, there are homotopy equivalences
\begin{equation}\label{eq:B-main-topology}
 J_{\rho_*}^{-L}\simeq B_{m-1}^\pa (\Si ,\pa \Si ),\quad
 J_{\rho_*}^{L}\simeq\oH.
\end{equation}
The upper sublevel is contractible, and, for every field $\mathbb F$,
\begin{equation}\label{eq:B-main-relative}
 H_k(J_{\rho_*}^{L},J_{\rho_*}^{-L};\mathbb F)
 \cong\widetilde H_{k-1}(B_{m-1}^\pa (\Si ,\pa \Si );\mathbb F),\quad k\ge1.
\end{equation}
The relative group in degree zero vanishes. The homotopy types are
unchanged when $L$ is increased further.
\end{prop}

The comparison of resonant and nonresonant sublevels uses the energy
estimates in Lemma~\ref{lem:B-exact-energy}. Their proof combines
\cite[Section~3.1, Lemma~14 in the preprint]{LSY2023} with the estimates in
Proposition~\ref{prop:estimates_rho_k} valid independently of
\ref{item:C1}--\ref{item:C2}. For the nonresonant sublevels, the argument starts from the
concentration and test-function construction of
\cite{B2019} and \cite{LSY2023}, followed by local
contractions and relative extensions over finite polyhedra; compare
\cite{M2008}.

\begin{lem}\label{lem:B-compact-gradient}
For every $R>0$, the family $T_r:\oH\to\oH$, $0\le r\le R$, is
locally Lipschitz in $u$ and maps bounded sets into relatively compact sets,
uniformly in $r$. If $r_n$ is bounded, $u_n$ is bounded in $\oH$, and
$u_n-T_{r_n}(u_n)\to0$ in $\oH$, then, after passing to a subsequence,
$r_n\to\bar r$ and $u_n\to u$ strongly in $\oH$, with
$u=T_{\bar r}(u)$.
\end{lem}
\begin{proof}
Let $B\subset\oH$ and $R:=\sup_{u\in B}\|u\|<\infty$.
Applying \eqref{eq:MT-Jensen-explicit} to $s=q$ and $s=-q$, we obtain
\begin{equation}\label{eq:B-normalized-density-estimate}
\begin{gathered}
 \sup_{u\in B}\int_\Si e^{q|u|}\,\d v_g
 \le 2C_{\Si,g}\exp\Big(\frac{q^2R^2}{8\pi}\Big),\quad q>0,\\
 \inf_{u\in B}\int_\Si Ve^u\,\d v_g\ge(\min_\Si V)|\Si|_g,\\
 \sup_{u\in B}\Big\|\frac{Ve^u}{\int_\Si Ve^u\,\d v_g}\Big\|_{L^p}
 \le\frac{(\max_\Si V)C_{\Si,g}^{1/p}}
          {(\min_\Si V)|\Si|_g}
          \exp\Big(\frac{pR^2}{8\pi}\Big),\quad p\ge1.
\end{gathered}
\end{equation}
If $u_n\rightharpoonup u$ in $\oH$ and $u_n\to u$ almost everywhere,
\eqref{eq:B-normalized-density-estimate}, applied with an exponent
larger than $p$, implies
\[
\frac{Ve^{u_n}}{\int_\Si  Ve^{u_n}\,\d v_g}
\to \frac{Ve^u}{\int_\Si  Ve^u\,\d v_g}
\quad\text{ in }\,L^p(\Si ),\quad p<\infty.
\]
The mean-zero Neumann estimate in \(W^{2,p}\), \(p>1\), and compact
embedding into \(H^1\) prove compactness, uniformly for bounded \(r\).
H\"older's inequality and
\(|e^v-e^u|\le|v-u|(e^v+e^u)\) prove local Lipschitz continuity.
Finally, take a subsequence with \(r_n\to\bar r\) and
\(T_{r_n}(u_n)\to u\) in \(\oH\). The assumed difference tends to
zero, so \(u_n\to u=T_{\bar r}(u)\) strongly.
\end{proof}

\begin{lem}\label{lem:B-exact-energy}
Suppose $u_n=T_{\rho_n}(u_n)$, $0\le\rho_n\le\rho_{\max}<\infty$,
and $\|u_n\|\to\infty$. After passing to a subsequence,
$\rho_n\to\bar\rho\in4\pi\mathbb N$ and there is a reduced critical
configuration $\xi^*\in\Xi_{p,q}$ with
$\bar\rho=4\pi(2p+q)$. Set
\[
 h_n:=\ln\int_\Si  Ve^{u_n}\,\d v_g,
 \quad
 C_\infty:=-\bar\rho\Big(1+\ln\frac{\bar\rho}{8}\Big)
              -\frac12\mathbf F_{p,q}^V(\xi^*).
\]
Then $h_n\to\infty$, $(\rho_n-\bar\rho)h_n\to0$, and
\begin{equation}\label{eq:B-energy-height}
 \frac12\|u_n\|^2=\rho_nh_n+C_\infty+o(1),
 \quad
 J_s(u_n)=(\rho_n-s)h_n+C_\infty+o(1)
\end{equation}
for every fixed $s\ge0$. In particular,
$J_s(u_n)\to-\infty$ if $\bar\rho<s$. At $\bar\rho=\rho_*$,
\begin{equation}\label{eq:B-endpoint-energy}
 J_{\rho_*}(u_n)\to
 -\rho_*\Big(1+\ln\frac{\rho_*}{8}\Big)
 -\frac12\mathbf F_{p,q}^V(\xi^*).
\end{equation}
The conclusions of the lemma hold independently of
\ref{item:C1} and \ref{item:C2}.
\end{lem}
\begin{proof}
First pass to a subsequence with $\rho_n\to\bar\rho$.
If $\bar\rho=0$, the Green representation and Jensen's inequality yield
uniform integrability of $e^{p u_n}$ for some fixed $p>1$: the
logarithmic Green singularity is integrable after exponentiation with
coefficient $p\rho_n$ sufficiently small. Jensen's inequality yields
$\int_\Si  Ve^{u_n}\,\d v_g\ge(\min V)|\Si |_g$, so the
normalized right-hand sides are bounded in $L^p$. Neumann elliptic
regularity then yields a uniform estimate for $\|u_n\|$, a contradiction. Thus
$\bar\rho>0$.

The basic Neumann blow-up alternative and quantization
\cite[Section~3.1, Lemma~14 in the preprint]{LSY2023}, as used in
\eqref{eq:bridge-basic-masses}, now yield finitely many distinct limiting
points, with masses $\kap _i=8\pi$ in the interior and $4\pi$ on the
boundary, and
\(\bar\rho=\sum_i\kap _i=4\pi(2p+q).\)
Propositions~\ref{prop:u-k-in-V-eta} and~\ref{prop:unique} yield the
unique decomposition
\[
 u_n=\sum_i\al _{n,i}P\delta_{\lam _{n,i},\xi_{n,i}}+w_n,
 \quad w_n\in\mathcal E_{\lam _n,\xi_n}^{p,q},
 \quad \xi_n\to\xi^*.
\]
Set $\Lam _n:=\min_i\lam _{n,i}$, $\ell_n:=\ln\Lam _n$, and
$A_n:=\sum_i|\al _{n,i}-1|$, $W_n:=\|w_n\|$.
The estimates of Proposition~\ref{prop:estimates_rho_k}, valid without
\ref{item:C1}--\ref{item:C2}, imply
\begin{equation}\label{eq:B-unconditional-parameter-bounds}
 \lam _{n,i}\asymp\Lam _n,\quad
 A_n+W_n=O(\Lam _n^{-1}),\quad
 A_n\ell_n=o(1),\quad
 (\rho_n-\bar\rho)\ell_n\to 0.
\end{equation}
Proposition~\ref{prop:estimates_rho_k} also yields $\tau_{n,i}=o(1)$ and
$\nabla\mathbf F_{p,q}^V(\xi_n)\to0$, so the limiting configuration
$\xi^*$ is critical. No nondegeneracy assumption is needed.

Write
\[
 F_{n,i}:=\mathcal F_{\xi_n}^i((\xi_n)_i),\quad
 B_n:=F_{n,1}\lam _{n,1}^2.
\]
The definition of \(\tau_{n,i}\) and
\eqref{eq:B-unconditional-parameter-bounds} imply
\(F_{n,i}\lam _{n,i}^2=B_n(1+o(1)).\)
Here \(\tau_{n,i}=o(1)\),
\(\mathrm g_{\xi_n}^i((\xi_n)_i)=1+o(1)\), and
\(\lam _{n,i}^{4(\al _{n,i}-1)}=1+o(1)\) by
\eqref{eq:B-unconditional-parameter-bounds}. Hence
Lemma~\ref{lem:Ve-u-integral} yields
\begin{equation}\label{eq:B-unconditional-logmass}
 \int_\Si  Ve^{u_n}\,\d v_g
 =\sum_i\frac{\kap _i}{8}F_{n,i}\lam _{n,i}^2(1+o(1))
 =\frac{\bar\rho}{8}B_n(1+o(1)).
\end{equation}
Consequently,
\[
 h_n=\ln B_n+\ln(\bar\rho/8)+o(1)
      =2\ell_n+O(1)\to \infty,
 \quad
 (\rho_n-\bar\rho)h_n\to 0.
\]

Define the interaction expression
\[
 I_n:=\sum_i\kap _i^2R^g((\xi_n)_i)
       +\sum_i\sum_{j\ne i}\kap _i\kap _j
                        G^g((\xi_n)_i,(\xi_n)_j).
\]
By orthogonality and \eqref{eq:B-unconditional-parameter-bounds},
setting \(\al _{n,i}=1\) changes the quadratic energy by
\(O(A_n\ell_n+W_n^2)=o(1)\). Thus
Lemma~\ref{lem:projbubble_integral_interaction}, the mass balance, and
\eqref{eq:def_F} yield
\[
\begin{aligned}
\frac12\|u_n\|^2
 &=2\sum_i\kap _i\ln\lam _{n,i}-\bar\rho+\frac12 I_n+o(1),\\
2\sum_i\kap _i\ln\lam _{n,i}
 &=\bar\rho\ln B_n-\sum_i\kap _i\ln F_{n,i}+o(1),\\
\sum_i\kap _i\ln F_{n,i}
 &=\sum_i\kap _i\ln V((\xi_n)_i)+I_n.
\end{aligned}
\]
Substitution of \eqref{eq:B-unconditional-logmass} now yields
\[
\frac12\|u_n\|^2
=\bar\rho h_n-\bar\rho\Big(1+\ln\frac{\bar\rho}{8}\Big)
-\frac12\mathbf F_{p,q}^V(\xi_n)+o(1)
=\rho_nh_n+C_\infty+o(1),
\]
since \((\rho_n-\bar\rho)h_n\to0\). This proves
\eqref{eq:B-energy-height} and \eqref{eq:B-endpoint-energy}.
\end{proof}

The following deformation lemma keeps trajectories bounded.
\begin{lem}\label{lem:B-bounded-descent}
Let $\mathscr X$ be a separable Hilbert space, and let $F_0,F_1\in C^1(\mathscr X)$ have locally
Lipschitz gradients, and $F_1\ge F_0$. Fix $\ell\in\mathbb R$, and set
\(U_j=\{F_j\le\ell\},\quad S=\{F_0\le\ell\le F_1\}.\)
The convex hull is taken in $\mathscr X$, with the convention in
Subsection~\ref{sec:proof-strategy}.
Suppose that, for each finite $R$, there exists $d_R>0$ such that
\begin{equation}\label{eq:B-two-separation}
 \operatorname{dist}(0,
 \operatorname{co}\{\nabla F_0(u),\nabla F_1(u)\})\ge d_R, \quad \forall\, u\in S,\ \|u\|\le R.
\end{equation}
Assume also the following property for some $R_0$: for each
$R\ge R_0$, there is $e_R>0$ such that
\begin{equation}\label{eq:B-three-separation}
 \operatorname{dist}(0,
 \operatorname{co}\{\nabla F_0(u),\nabla F_1(u),u\})\ge e_R,
\quad \forall\, u\in S,\ R_0\le\|u\|\le R.
\end{equation}
Then $U_1$ is a strong deformation retract of $U_0$.
The same conclusion holds for a single functional. Let
\(F\in C^1(\mathscr X)\)
have a locally Lipschitz gradient. If $\|\nabla F\|$ has a positive lower
bound on every bounded part of $\{F\ge\ell\}$, and, outside a fixed ball,
\begin{equation}\label{eq:B-one-function-separation}
\operatorname{dist}(0,\operatorname{co}\{\nabla F(u),u\})\ge e_R>0,
\quad \forall\, F(u)\ge\ell,\ R_0\le\|u\|\le R,
\end{equation}
then $\{F\le\ell\}$ is a strong deformation retract of $\mathscr X$.
\end{lem}
\begin{proof}
For \(G=\{g_1,\ldots,g_k\}\), let \(p_G\) be the projection of
\(0\) onto \(\operatorname{co}G\). If \(p_G\ne0\), then
\[
\langle g_j,p_G\rangle\ge\|p_G\|^2,
\quad
\Big\langle g_j,-\frac{p_G}{\|p_G\|}\Big\rangle\le-\|p_G\|.
\]
Equations \eqref{eq:B-two-separation}--\eqref{eq:B-three-separation},
local continuity, and a locally finite Lipschitz partition of unity
yield a locally Lipschitz field \(W\), defined near \(S\), with
\(\|W\|\le1\). On each bounded part of \(S\),
\[
\langle\nabla F_j,W\rangle\le-c_R\quad(j=0,1),
\quad
\langle u,W(u)\rangle\le-c_R\quad\text{if }\|u\|\ge R_0+2.
\]
Before the trajectory \(\dot\eta=W(\eta)\) reaches \(U_1\), it remains
in \(S\): \(F_0\) decreases and \(F_1>\ell\). Hence
\begin{equation}\label{eq:B-stopped-flow-bounds}
\|\eta(t,u)\|\le\max\{\|u\|,R_0+2\},
\quad
\frac{\d}{\d t}F_1(\eta(t,u))\le-c<0.
\end{equation}
The constant \(c\) is positive on this bounded trajectory. Bounded
speed excludes a finite maximal existence time: at any finite terminal time, the curve is
Cauchy, and local existence at its limit extends it. Equation
\eqref{eq:B-stopped-flow-bounds} therefore yields a finite hitting time
\(T(u)\). Continuous dependence and strict crossing of \(F_1=\ell\)
imply continuity of \(T\) outside \(U_1\). Locally near \(U_1\),
\(T(u)\le C(F_1(u)-\ell)_+,\)
so it extends continuously by zero on \(U_1\). Stopping \(\eta\) at
this time defines the strong deformation retraction
\(\mathcal H(t,u)=\eta(tT(u),u),\quad 0\le t\le1.\)
For one functional, apply the same construction to \(\nabla F\) inside
the fixed ball and to \(\{\nabla F,u\}\) outside it. Equation
\eqref{eq:B-one-function-separation} yields the radial estimate, and the
same hitting-time argument retracts \(\mathscr X\) onto \(\{F\le\ell\}\).
\end{proof}

We next compare the resonant and nonresonant deep sublevels by a deformation argument; see also \cite{L2007}. Fix \(\rho_*=4\pi m\).
\begin{prop}\label{prop:B-resonant-comparison}
Fix $\rho_-\in(4\pi(m-1),\rho_*)$ and set $c_*:=\rho_*/\rho_->1$.
For all sufficiently large $L$, the inclusion
\begin{equation}\label{eq:B-reinforced-set}
 J_{\rho_-}^{-L/c_*}=\{c_*J_{\rho_-}\le-L\}
 \hookrightarrow J_{\rho_*}^{-L}
\end{equation}
is a strong deformation retract. Moreover, $J_{\rho_*}^{L}$ is a strong
deformation retract of $\oH$.
If $L'\ge L$ are sufficiently large, $J_{\rho_*}^{-L'}$ is a strong
deformation retract of $J_{\rho_*}^{-L}$.
\end{prop}
\begin{proof}
Lemma~\ref{lem:B-exact-energy} and
Lemma~\ref{lem:B-compact-gradient} yield constants \(C_\pm>0\) such that
\begin{equation}\label{eq:B-lower-exact-bound}
J_{\rho_*}(u)>-C_-
\quad\text{if }\\
u=T_r(u),\quad r\in[\rho_-,\rho_*],
\end{equation}
and
\begin{equation}\label{eq:B-upper-exact-bound}
J_{\rho_*}(u)<C_+
\quad\text{if }\\
u=T_r(u),\quad 0\le r\le\rho_*.
\end{equation}
For bounded sequences of exact solutions, apply
Lemma~\ref{lem:B-compact-gradient}. An unbounded sequence with
\(r\in[\rho_-,\rho_*]\) can concentrate only at \(\rho_*\), where
\eqref{eq:B-endpoint-energy} yields a finite energy limit. For
\(r\in[0,\rho_*]\), any smaller resonant limit yields
\(J_{\rho_*}\to-\infty\) by \eqref{eq:B-energy-height}. These
alternatives prove the lower and upper estimates, respectively, without
assuming finiteness of the reduced critical set.

Take \(L>C_-+1\) and set
\[
F_0:=J_{\rho_*},
\quad
F_1:=c_*J_{\rho_-},
\quad
S_L:=\{F_0\le-L\le F_1\}.
\]
Since
\[
F_1(u)-F_0(u)
=
\frac{c_*-1}{2}\|u\|^2\ge0,
\]
the inclusion \eqref{eq:B-reinforced-set} follows. Moreover,
\begin{equation*}
\{
u\in S_L:
u=T_r(u)\text{ for some }r\in[0,\rho_*]
\}
\end{equation*}
is bounded by \eqref{eq:B-energy-height} and
\eqref{eq:B-lower-exact-bound}. Choose \(R_0\) larger than its norm.

For \(0\le t\le1\),
\[
(1-t)\nabla F_0+t\nabla F_1
=
[1+t(c_*-1)]u-T_{\rho_*}(u),
\]
so any zero satisfies
\[
u=T_r(u),
\quad
r=\frac{\rho_*}{1+t(c_*-1)}\in[\rho_-,\rho_*],
\]
contrary to \eqref{eq:B-lower-exact-bound}. Likewise,
\[
a\nabla F_0+b\nabla F_1+du=0,
\quad
a,b,d\ge0,\quad a+b+d=1,
\]
implies
\[
u=T_r(u),
\quad
r=\frac{\rho_*(a+b)}{a+c_*b+d}\in[0,\rho_*],
\]
which is impossible on \(S_L\cap\{\|u\|\ge R_0\}\).
We verify the positive separation constants required by
Lemma~\ref{lem:B-bounded-descent}. If a sequence in a fixed bounded
part of $S_L$ violates the two-vector separation, there exist
$t_n\in[0,1]$ such that
\[
 [1+t_n(c_*-1)]u_n-T_{\rho_*}(u_n)\to 0.
\]
After extraction $t_n\to t$, and compactness of $T_{\rho_*}$ on
bounded sets implies strong convergence of $u_n$. Its limit belongs to
the closed set $S_L$ and is a forbidden exact solution. For the
three-vector separation the same argument applies to
\[
 (a_n+c_*b_n+d_n)u_n-(a_n+b_n)T_{\rho_*}(u_n)\to 0,
 \quad a_n+b_n+d_n=1.
\]
The coefficient of $u_n$ is at least one, and the three coefficients
have a convergent subsequence. Thus strong convergence again yields a
forbidden exact solution in the closed bounded annulus. These
contradictions establish both separation estimates, and
Lemma~\ref{lem:B-bounded-descent} proves \eqref{eq:B-reinforced-set}.

For the upper sublevel, take \(L>C_++1\) and work on the closed set
$\{u\in\oH:J_{\rho_*}(u)\ge L\}$. The gradient identity yields
\begin{equation*}
\begin{aligned}
 \operatorname{co}\{\nabla J_{\rho_*}(u),u\}
 &=\{u-\theta T_{\rho_*}(u):0\le\theta\le1\},\\
 0\in\operatorname{co}\{\nabla J_{\rho_*}(u),u\}
 &\quad\Longleftrightarrow\quad
 u=T_{\theta\rho_*}(u)\text{ for some }\theta\in[0,1].
\end{aligned}
\end{equation*}
Thus the convex-hull condition on this set contradicts
\eqref{eq:B-upper-exact-bound}. On any bounded part of this closed
set, failure of positive separation yields $\theta_n\in[0,1]$
and bounded $u_n$ with
\(u_n-\theta_nT_{\rho_*}(u_n)\to 0.\)
After extraction, $\theta_n\to\theta$, and
Lemma~\ref{lem:B-compact-gradient} yields strong convergence to
$u=T_{\theta\rho_*}(u)$ with $J_{\rho_*}(u)\ge L$, the same
contradiction. This proves the separation required by the
one-functional part of Lemma~\ref{lem:B-bounded-descent}, including
the gradient lower bound since $\nabla J_{\rho_*}(u)$ belongs to
its convex hull with $u$. That lemma therefore shows that
\(J_{\rho_*}^{L}\) is a strong deformation retract of \(\oH\).
Finally, for \(L'\ge L\), apply Lemma~\ref{lem:B-bounded-descent} to
\[
F_0=J_{\rho_*},
\quad
F_1=J_{\rho_*}+L'-L,
\quad
\ell=-L,
\]
which shows that \(J_{\rho_*}^{-L'}\) is a strong deformation retract
of \(J_{\rho_*}^{-L}\).
\end{proof}

For the nonresonant comparison, fix
\(l\ge1\), \(\rho\in(4\pi l,4\pi(l+1))\), and choose
\begin{equation}\label{eq:B-coefficient-range}
 1<a_{\max}<\frac{\rho}{4\pi l}.
\end{equation}
For $v\in H^1(\Sigma)$, write
$\overline v:=|\Sigma|_g^{-1}\int_\Sigma v\,\d v_g$ for its integral mean.
For $1\le a\le a_{\max}$, introduce the single auxiliary functional
\begin{equation}\label{eq:B-invariant-family}
 \mathscr J_a(v):=\frac a2\int_\Si |\nabla v|^2\,\d v_g
 +\rho\overline v-\rho\ln\Big(\int_\Si  Ve^v\,\d v_g\Big),
 \quad v\in H^1(\Si ).
\end{equation}
The functional \eqref{eq:B-invariant-family} is invariant under addition of constants, and
$\mathscr J_a(u)=aJ_{\rho/a}(u)$ for $u\in\oH$. Define
\begin{equation}\label{eq:B-reinforced-family}
 X_{a,D}:=\{u\in\oH:\mathscr J_a(u)\le-D\}.
\end{equation}
The parameter $a$ is used only in this appendix and imposes no additional
assumption on $V$ or on $g$.

\begin{lem}\label{lem:B-reinforcement}
For the sets in \eqref{eq:B-reinforced-family}, there is $D_0$ such that, whenever
$1\le a\le b\le a_{\max}$ and $D'\ge D\ge D_0$, $X_{b,D'}$ is a strong deformation retract of $X_{a,D}$.
The upper sublevel $J_\rho^D$ is contractible for sufficiently large $D$.
\end{lem}
\begin{proof}
By \eqref{eq:B-invariant-family},
\[
\mathscr J_a(u)=aJ_{\rho/a}(u),
\quad
\frac{\rho}{a}\in
\Big[\frac{\rho}{a_{\max}},\rho\Big]
\Subset(4\pi l,4\pi(l+1)).
\]
Nonresonant compactness and Lemma~\ref{lem:B-compact-gradient} imply
that
\[
 \mathcal K:=\{u\in\oH:u=T_r(u)\text{ for some }
 r\in[\rho/a_{\max},\rho]\}
\]
is compact. If $\mathcal K$ is empty, take any $D_0>0$;
otherwise choose $D_0>0$ so large that
\begin{equation}\label{eq:B-cross-energy-threshold}
 -D_0<\inf\{\mathscr J_s(u):u\in\mathcal K,
                                  \ 1\le s\le a_{\max}\}.
\end{equation}
For \(\mathcal K=\varnothing\), the infimum is interpreted as
\(+\infty\); otherwise it is finite by compactness and continuity.
A zero of
$\operatorname{co}\{\nabla\mathscr J_a(u),\nabla\mathscr J_b(u)\}$
satisfies $u=T_{\rho/c}(u)$ for some $c\in[a,b]$, hence belongs to
$\mathcal K$. Equation~\eqref{eq:B-cross-energy-threshold} then yields
\(\mathscr J_a(u)>-D_0\ge-D\), so no such zero can occur on the set under
consideration.
If \(u\) is included in the convex hull, a zero has the form
\(u=T_r(u)\) with
\(0\le r\le\rho/a\). On
\(\{\mathscr J_a\le-D\le\mathscr J_b\}\), such zeros are bounded:
otherwise \eqref{eq:B-energy-height} yields
\[
\bar r\le4\pi l<\frac{\rho}{b}.
\]
Consequently,
\(\mathscr J_b(u_n)\to-\infty,\)
which is a contradiction. Lemmas~\ref{lem:B-compact-gradient} and
\ref{lem:B-bounded-descent} therefore show that \(X_{b,D}\) is a strong
deformation retract of \(X_{a,D}\). Applying the same argument on
\(-D'\le\mathscr J_b\le-D\) shows that \(X_{b,D'}\) is a strong
deformation retract of \(X_{b,D}\),
and consequently also of \(X_{a,D}\). The upper-sublevel statement is
obtained in the same way from
\(\operatorname{co}\{\nabla J_\rho(u),u\}\), exactly as in
Proposition~\ref{prop:B-resonant-comparison}.
\end{proof}

\begin{rem}\upshape
The strict inequality in \eqref{eq:B-coefficient-range} cannot
in general be replaced by equality. For $l=1$ and
$a_{\max}=\rho/(4\pi)$, the functional
$\mathscr J_{a_{\max}}=a_{\max}J_{4\pi}$ is bounded below by the
sharp Moser--Trudinger inequality. Thus $X_{a_{\max},D}$ is empty for
large $D$. In contrast, unboundedness of $J_\rho$ below for
$\rho>4\pi$ implies $X_{1,D}\ne\varnothing$. The retraction in
Lemma~\ref{lem:B-reinforcement} therefore fails at equality.
\end{rem}

Fix \(\rho\in(4\pi l,4\pi(l+1))\) and assume
\eqref{eq:B-coefficient-range}. For probability measures on \(\Si \), set
\[
d_{\rm BL}(\nu_1,\nu_2):=
\sup\Big\{
\Big|\int_\Si  f\,\d(\nu_1-\nu_2)\Big|:
\|f\|_{L^\infty}+\operatorname{Lip}(f)\le1
\Big\},
\]
where $\operatorname{Lip}(f)$ denotes the Lipschitz seminorm of $f$.
For a set $\mathcal A$ of probability measures, define
\(d_{\rm BL}(\nu,\mathcal A):=\inf_{\si \in\mathcal A} d_{\rm BL}(\nu,\si ).\)
Since $\Si $ is compact, $d_{\rm BL}$ induces the weak topology on probability measures.

\begin{lem}\label{lem:B-concentration}
Let $l\ge1$ and $\rho\in(4\pi l,4\pi(l+1))$ be fixed. For every
$\delta>0$, there exists $D_\delta>0$ such that, whenever $D\ge D_\delta$,
\[
d_{\rm BL}(
\mathfrak m_u,
B_l^\pa (\Si ,\pa \Si )
)<\delta
\quad
\text{for every }u\in J_\rho^{-D}.
\]
Consequently there is a continuous map
$\psi_D:J_\rho^{-D}\to B_l^\pa (\Si ,\pa \Si )$ for all sufficiently large $D$.
\end{lem}
\begin{proof}
Fix \(\delta>0\). The set of probability measures
\[
 \mathcal K_\delta:=
 \{\nu:d_{\rm BL}(\nu,B_l^\pa (\Si ,\pa \Si ))\ge\delta\}
\]
is compact. For each \(\nu\in\mathcal K_\delta\), its support contains
finitely many points of total weight at least \(l+1\), where interior
points have weight two and boundary points have weight one; otherwise
\(\nu\) itself belongs to \(B_l^\pa \), a contradiction.
Choose small disk or half-disk neighborhoods \(\Omega_i\) of these
points, with pairwise disjoint larger chart neighborhoods, and keep
the interior charts away from \(\pa \Si \).
Each \(\nu(\Omega_i)\) is positive. In a sufficiently small weak
neighborhood of \(\nu\), every measure \(\widetilde\nu\) satisfies
\(\widetilde\nu(\Omega_i)>\nu(\Omega_i)/2\) for every $i$.
A finite subcover of \(\mathcal K_\delta\) therefore yields
\(\gamma>0\) and a finite collection of such chart families, so that
every \(\mathfrak m_u\in\mathcal K_\delta\) satisfies, for one family,
\[
 \mathfrak m_u(\Omega_i)\ge\gamma,
 \quad s_i\in\{1,2\},\quad \sum_i s_i\ge l+1.
\]
All chart separations, interior distances from the boundary, and
estimates for the cutoffs can consequently be chosen uniformly. If
\(\mathcal K_\delta\) is empty the conclusion is immediate.
The improved Moser--Trudinger inequality on disks and half-disks then yields,
for every \(\var >0\),
\[
\ln\Big(\int_\Si  Ve^u\,\d v_g\Big)
\le
\frac{1+\var }{8\pi\sum_i s_i}\|u\|^2+C_\var
\le
\frac{1+\var }{8\pi(l+1)}\|u\|^2+C_\var ;
\]
Here is the mixed disk--half-disk estimate with its weights.
Choose cutoffs $\zeta_i=1$ on $\Omega_i$ with disjoint larger supports.
Split $u=v+w$ into its first $N$ Neumann eigenfunctions and their
orthogonal complement. Then
$\|v\|_\infty\le C_N\|u\|$ and
$\|w\|_2^2\le\lambda_{N+1}^{-1}\|u\|^2$.
The mass lower bound and the upper and lower estimates for $V$ imply
\[
 \ln \int_\Sigma Ve^u\,\d v_g
 \le C+\|v\|_\infty+\ln \int_{\Omega_i}e^w\,\d v_g.
\]
Apply the disk Dirichlet Moser--Trudinger inequality to $\zeta_iw$.
For a half-disk reflect it evenly across the flat edge first; its
Dirichlet energy doubles. With $s_i=2$ for a disk and $s_i=1$ for
a half-disk, multiplication by $s_i$ and summation yield
\[
 \Big(\sum_i s_i\Big)\ln \int_\Sigma Ve^u\,\d v_g
 \le C+\Big(\sum_i s_i\Big)\|v\|_\infty
    +\frac{1+\eta}{8\pi}\|\nabla w\|_2^2
    +C_\eta\|w\|_2^2.
\]
Given the desired $\var>0$, first choose $\eta$ small, then $N$
large, and finally apply Young's inequality to $C_N\|u\|$.
This proves the mixed estimate with coefficient
$(1+\var)/(8\pi\sum_i s_i)$. The constants are uniform for the
finite chart families fixed above; compare
\cite[Section~3.2]{LSY2023} for the cutoff argument.
Choose
\[
\frac12-\frac{\rho(1+\var )}{8\pi(l+1)}>0.
\]
Then
\[
J_\rho(u)
\ge
\Big(
\frac12-\frac{\rho(1+\var )}{8\pi(l+1)}
\Big)\|u\|^2-C_\var ,
\]
Hence the normalized measures of functions in sufficiently negative sublevels belong to the \(\delta\)-neighborhood
of \(B_l^\pa (\Si ,\pa \Si )\).

To construct the continuous map to the barycenter space, triangulate
\((\Si ,\pa \Si )\) as a finite geometric simplicial pair in
\(\mathbb R^d\), and express the next integral in these coordinates.
For a multi-index $\beta\in\mathbb N_0^d$, write
$|\beta|:=\sum_{j=1}^d\beta_j$ and
$x^\beta:=\prod_{j=1}^dx_j^{\beta_j}$. Consider the map
\[
\si
\mapsto
\Big(
\int_\Si  x^\beta\,\d\si (x)
\Big)_{|\beta|\le2l-1}.
\]
This map is injective on \(B_l^\pa (\Si ,\pa \Si )\). Indeed, if two
such measures differ at a support point \(a\), the Lagrange-type polynomial
\[
\prod_{b\ne a}
\frac{\langle x-b,a-b\rangle}{|a-b|^2},
\quad
\deg\le2l-1,
\]
separates the corresponding coefficient. The parameter domains for
finitely many allowed support types are semialgebraic, and the coordinates of the integral map
are polynomial in the weights and support coordinates.
The image is semialgebraic by the Tarski--Seidenberg theorem.
Compactness follows from compactness of $B_l^\pa$; triangulation
identifies the image with a compact polyhedron
\cite[Section~9.2]{BCR1998}; compare \cite[Lemma~3.1]{B2019},
\cite[Lemma~4.10]{N2008}, and \cite{AK2026}. The integral map extends continuously to all probability measures
on $\Si $. Composing this extension with a neighborhood retraction
of the compact semialgebraic image and the inverse of the integral
embedding defines a continuous retraction
\(r_l: \mathcal U_l \to B_l^\pa (\Si ,\pa \Si )\)
from a sufficiently small weak neighborhood \(\mathcal U_l\) of the
barycenter space. For \(D\) large, define
\(\psi_D(u):=r_l(\mathfrak m_u).\)
Continuity follows from Lemma~\ref{lem:B-compact-gradient}.
\end{proof}

\begin{lem}\label{lem:B-test-functions}
For each $l\ge1$, there exist continuous maps
$\phi_\lam :B_l^\pa (\Si ,\pa \Si )\to\oH$, $\lam \ge1$, such that
\begin{equation}\label{eq:B-test-bounds}
 d_{\rm BL}(\mathfrak m_{\phi_\lam (\si )},\si )\le C\lam ^{-1},
 \quad
 J_\rho(\phi_\lam (\si ))
 \le-(2\rho-8\pi l)\ln\lam +C_\rho
\end{equation}
uniformly in $\si \in B_l^\pa (\Si ,\pa \Si )$. In particular, when
$\rho\in (4\pi l,4\pi(l+1))$ and $D$ is sufficiently large, a choice
$\phi_D:=\phi_{\lam (D)}:B_l^\pa (\Si ,\pa \Si )\to J_\rho^{-D}$ satisfies
\begin{equation}\label{eq:B-one-sided-homotopy}
 \psi_D\phi_D\simeq\operatorname{Id}_{B_l^\pa (\Si ,\pa \Si )}.
\end{equation}
\end{lem}
\begin{proof}
Extend \(g\) smoothly to a closed surface and denote the restricted
distance by \(d\). Set
\[
q_{z,\lam }(x)
:=
\frac{\lam ^2}{(1+\lam ^2d(x,z)^2)^2},
\quad
M_{z,\lam }
:=
\int_\Si  Vq_{z,\lam }\,\d v_g.
\]
Uniform disk and half-disk estimates yield
\begin{equation}\label{eq:B-kernel-bounds}
0<c\le M_{z,\lam }\le C,
\quad
\int_\Si  d(x,z)Vq_{z,\lam }\,\d v_g
\le C\lam ^{-1}.
\end{equation}
For
\(\si =\sum_{i=1}^N\gamma_i\delta_{z_i}\), define
\begin{equation*}
S_{\lam ,\si }
:=
\sum_{i=1}^N\gamma_i
\frac{q_{z_i,\lam }}{M_{z_i,\lam }},
\quad
\phi_\lam (\si )
:=
\ln S_{\lam ,\si }
-\overline{\ln S_{\lam ,\si }}.
\end{equation*}
Then
\[
\int_\Si  VS_{\lam ,\si }\,\d v_g=1,
\quad
d_{\rm BL}
(\mathfrak m_{\phi_\lam (\si )},\si )
\le C\lam ^{-1}
\]
by \eqref{eq:B-kernel-bounds}. For fixed \(\lam \),
\(z\mapsto q_{z,\lam }\) is continuous in \(C^0\cap H^1\).
Indeed, the distances converge uniformly, their gradients converge
almost everywhere outside the limiting cut locus and pole, and the
kernel gradients are uniformly bounded. Dominated convergence yields
the \(H^1\) continuity, jointly for \(\lam \) in compact intervals.
Since \(M_{z,\lam }>0\) is continuous, the same holds for
\(q_{z,\lam }/M_{z,\lam }\).

A continuous map from a compact space to a Banach space can be uniformly
approximated by finite sums of fixed vectors with continuous scalar
coefficients, using a finite partition of unity. Its integral therefore
depends continuously in norm on a weakly convergent probability measure.
Apply this fact in \(C^0\cap H^1\) to obtain continuity of
\(\si \mapsto S_{\lam ,\si }\). For fixed \(\lam \), these
functions have a common positive lower bound; hence
\(\nabla\ln S=S^{-1}\nabla S\) proves continuity of \(\phi_\lam \)
in \(H^1\), including collisions, vanishing coefficients, and changes
of support type.

At a.e. \(x\), set
\[
\om _i(x)
:=
\frac{\gamma_iq_{z_i,\lam }(x)/M_{z_i,\lam }}
{S_{\lam ,\si }(x)}.
\]
Then
\[
\om _i\ge0,\quad
\sum_i\om _i=1,\quad
\nabla\ln S_{\lam ,\si }
=
\sum_i\om _i\nabla\ln q_{z_i,\lam }.
\]
Consequently,
\[
|\nabla\ln S_{\lam ,\si }|^2
\le
\sum_i\om _i|\nabla\ln q_{z_i,\lam }|^2.
\]
For an interior, respectively boundary, center,
\[
\int_\Si |\nabla\ln q_{z,\lam }|^2\,\d v_g
\le
\left\{\begin{aligned}
&32\pi\ln\lam +C&&\text{ if }\, z\in\mathring\Si ,\\
&16\pi\ln\lam +C&&\text{ if }\, z\in\pa \Si .
\end{aligned}\right.
\]
Thus, for \(2p+q\le l\),
\begin{equation}\label{eq:B-test-dirichlet}
\frac12\|\phi_\lam (\si )\|^2
\le
8\pi(2p+q)\ln\lam +C
\le
8\pi l\ln\lam +C.
\end{equation}
Moreover,
\[
-2\ln\lam -C
\le
\overline{\ln S_{\lam ,\si }}
\le
-2\ln\lam +C,
\]
because
\[
\ln S_{\lam ,\si }(x)
\le
-2\ln\lam +C
+4\sum_i\ln^+\frac1{d(x,z_i)}
\]
and the logarithmic singularities are uniformly integrable. Hence
\[
\ln\Big(
\int_\Si  Ve^{\phi_\lam (\si )}\,\d v_g
\Big)
=
2\ln\lam +O(1),
\]
and \eqref{eq:B-test-dirichlet} yields \eqref{eq:B-test-bounds}.

By compactness of \(B_l^\pa \), choose \(\eta>0\) so that its
\(\eta\)-neighborhood in the probability measures is contained in
\(\mathcal U_l\). Choose \(\lam (D)\) large enough that
\(\phi_D:=\phi_{\lam (D)}\) takes values in \(J_\rho^{-D}\) and
\(C/\lam (D)<\eta\). Then, for all \(\si \in B_l^\pa \)
and \(t\in[0,1]\),
\[
 d_{\rm BL}((1-t)\mathfrak m_{\phi_D(\si )}+t\si ,\si )
 \le C/\lam (D)<\eta.
\]
Thus the whole segment is contained in the domain of \(r_l\), and
\[
H(t,\si )
:=
r_l(
(1-t)\mathfrak m_{\phi_D(\si )}+t\si ),
\quad
t\in[0,1],
\]
proves \eqref{eq:B-one-sided-homotopy}.
\end{proof}

The same computation, uniformly for
\(1\le a\le a_{\max}\), yields
\begin{equation}\label{eq:B-reinforced-test}
 \mathscr J_a(\phi_\lam (\si ))
 \le-(2\rho-8\pi al)\ln\lam +C.
\end{equation}
Thus, for sufficiently large scale, the image of the same map is contained in any prescribed $X_{a,D}$.
Moreover, if $r_l$ is the neighborhood retraction used above, then
\begin{equation}\label{eq:B-projection-closeness}
 \sup_{u\in X_{a,D}}d_{\rm BL}(\mathfrak m_u,r_l(\mathfrak m_u))\to 0
 \quad\text{ as }\, D\to \infty,
\end{equation}
uniformly for $ a \in [1, a_{\max}]$. Indeed,
$X_{a,D}\subset J_\rho^{-D}$, and continuity of $r_l$ near the compact
set $B_l^\pa (\Si ,\pa \Si )$, together with $r_l|_{B_l^\pa (\Si ,\pa \Si )}=\operatorname{Id}$,
yields \eqref{eq:B-projection-closeness} from Lemma~\ref{lem:B-concentration}.

The reverse homotopy follows from the next local contraction.
\begin{lem}\label{lem:B-local-contraction}
Fix $1\le a<b\le a_{\max}$, $D>0$, $\si \in B_l^\pa (\Si ,\pa \Si )$, and a
weak neighborhood $\mathcal U$ of $\si $ in the probability measures.
There are a weak neighborhood $\mathcal V$ of $\si $ and $D'\ge D$
such that the inclusion
\begin{equation}\label{eq:B-local-inclusion}
 X_{b,D'}\cap\mathfrak m^{-1}(\mathcal V)
 \to  X_{a,D}\cap\mathfrak m^{-1}(\mathcal U)
\end{equation}
is homotopic to a constant map through maps with values in the set on the right.
\end{lem}
\begin{proof}
Write
\[
\si =\sum_{i=1}^{N_\si }\gamma_i\delta_{z_i},
\quad
\gamma_i>0,
\quad
\sum_i\gamma_i=1,
\quad
2p+q\le l.
\]
Here $p$ and $q$ count the interior and boundary support points, so
$N_\si =p+q$. Choose pairwise disjoint coordinate neighborhoods
$\Omega_i:=U_{R_i}(z_i)$ and $\Omega_i':=U_{R_i'}(z_i)$, where
$R_i'\in (0,R_i)$. Their coordinate images are the disks or half-disks
$B_{R_i}^{z_i}$ and $B_{R_i'}^{z_i}$ defined in
Subsection~\ref{subsec:bubble-neighborhoods}. Set
\[
W_i(y):=
V(y_{z_i}^{-1}(y))e^{\hat\varphi_{z_i}(y)},
\quad y\in B_{R_i}^{z_i}.
\]
After decreasing $R_i$,
\begin{equation}\label{eq:B-density-oscillation}
e^{-\var _{\rm loc}}
\le
\frac{W_i(y)}{W_i(0)}
\le
e^{\var _{\rm loc}},
\quad\forall\,y\in B_{R_i}^{z_i}.
\end{equation}
Choose \(\mathcal V\Subset\mathcal U\) so that
\[
\nu(\Omega_i')
>
(1-\var _{\rm loc})\gamma_i
\quad
\text{for every }\nu\in\mathcal V
\text{ and every }i.
\]

\textbf{Step 1.} Truncation and localization.
For \(u\in X_{b,D'}\cap\mathfrak m^{-1}(\mathcal V)\), set
\(e:=\frac12\|u\|^2, \quad T:=M\sqrt{1+2e}.\)
Since \(\mathscr J_b(u)\le-D'\),
\begin{equation}\label{eq:B-source-mass}
\ln\Big(\int_\Si  Ve^u\,\d v_g\Big)
\ge
\frac{be+D'}{\rho},
\quad
e^T
\Big(\int_\Si  Ve^u\,\d v_g\Big)^{-1}
\le
C_Me^{-D'/\rho}.
\end{equation}
Set \(I:=\int_\Si  Ve^u\,\d v_g\) and \(K:=e^{-T}I\).
Since
\[
 \inf_{e\ge0}\Big\{\frac{be}{\rho}
       -M\sqrt{1+2e}\Big\}>-\infty,
\]
we have \(K\ge c_Me^{D'/\rho}\). Constants in the local construction
may depend on the fixed configuration \(\si \), the neighborhoods $\Omega_i$, and
parameters already chosen, but not on \(u\) or the homotopy parameters.
Set \(u^+=\max\{u,0\}\), \(u^-=\max\{-u,0\}\),
\(u_T=(u-T)^+\), and \(v=\eta_0u_T\), where
\(\eta_0=\sum_i\eta_i\), \(\eta_i\in [0,1]\),
\(\eta_i=1\) on \(\Omega_i'\), and
\(\operatorname{supp}\eta_i\subset\Omega_i\). The first three
segments, each parametrized by \(s\in [0,1]\), are
\[
u^+-(1-s)u^-,\quad (u-sT)^+,\quad (1-s+s\eta_0)u_T.
\]
They join \(u\) to \(v\) and are strongly continuous in \(H^1\) by
the Sobolev chain rule and approximation.
For the truncation,
\begin{equation*}
Ve^{(u-tT)^+}
=
e^{-tT}Ve^u+R_t,
\quad
0\le R_t\le V.
\end{equation*}
Also
\[
\|u_T\|_{L^2}^2
\le
T^{-2}\int_\Si |u|^4\,\d v_g
\le
CM^{-2}e.
\]
For
\(v_t:=(1-t+t\eta_0)u_T\),
\begin{equation}\label{eq:B-localization-energy}
\frac12\int_\Si |\nabla v_t|^2\,\d v_g
\le
(1+\var )e+C_\var  M^{-2}e.
\end{equation}
Choose \(\var \) and then \(M\) so that
\[
a(\var +C_\var  M^{-2})
<
\frac{b-a}{4}.
\]
Equations
\eqref{eq:B-source-mass}--\eqref{eq:B-localization-energy} yield,
uniformly on the three segments,
\begin{equation*}
\mathscr J_a
\le
\mathscr J_b(u)
-\frac12(b-a)e
+C_M\sqrt{1+e}+C.
\end{equation*}
For \(D'\) large the right-hand side is at most \(-D\).
The mass estimate is also uniform on the three segments.
For nonnegative $f,g\in L^1(\Si ,\d v_g)$ with positive integrals,
\begin{equation}\label{eq:B-normalized-L1-comparison}
 \Big\|\frac f{\int_\Si  f\,\d v_g}
             -\frac g{\int_\Si  g\,\d v_g}\Big\|_{L^1(\Si )}
 \le\frac{2\|f-g\|_{L^1(\Si )}}{\int_\Si  f\,\d v_g}.
\end{equation}
During the first segment, the increase in \(Ve^u\) is at most \(V\).
During truncation, compare with \(e^{-tT}Ve^u\), whose integral is at
least \(K\); the difference is again between zero and \(V\).
During localization, the exponential term equals \(Ve^{u_T}\) on
\(\bigcup_i\Omega_i'\) and only decreases elsewhere. Its mass on the
region outside this union is at most \(\var _{\rm loc}K+C\),
and \(\int_\Si  Ve^{u_T}\,\d v_g\ge K\).
By \eqref{eq:B-normalized-L1-comparison}, every function $w$ on
these three segments satisfies
\[
 d_{\rm BL}(\mathfrak m_w,\mathfrak m_u)
 \le C\var _{\rm loc}+C/K
 \le C\var _{\rm loc}+C_Me^{-D'/\rho}.
\]
Set
\[
M_i^0:=\int_{\Omega_i}Ve^v\,\d v_g
=\int_{B_{R_i}^{z_i}}W_i(y)e^{v(y_{z_i}^{-1}(y))}\,\d y.
\]
Since
\(v=u_T\) on \(\Omega_i'\) and \(0\le v\le u_T\),
\[
 K\mathfrak m_u(\Omega_i')\le M_i^0
 \le K\mathfrak m_u(\Omega_i)+C.
\]
The lower estimates on $\mathfrak m_u(\Omega_j')$, $j\ne i$, also yield
\(\mathfrak m_u(\Omega_i)\le\gamma_i+
\var _{\rm loc}(1-\gamma_i)\). Consequently
\(M_i^0=K\gamma_i+O(\var _{\rm loc}K+1).\)

\textbf{Step 2.} Dilation in local coordinates.
Write $v_i(y):=v(y_{z_i}^{-1}(y))$ for $y\in B_{R_i}^{z_i}$.
For $s\ge0$, define on $B_{R_i}^{z_i}$
\begin{equation}\label{eq:B-dilation}
(\mathcal D_sv_i)(y)
=
\left\{\begin{aligned}
&v_i(e^sy)+4s,
&&|y|\le R_ie^{-s},\\
&4\ln(R_i/|y|),
&&R_ie^{-s}\le|y|\le R_i.
\end{aligned}\right.
\end{equation}
By construction, $v_i$ has zero trace on the circular arc
$|y|=R_i$; the trace on the flat edge of a half-disk need not vanish.
At $|y|=R_ie^{-s}$ both traces in \eqref{eq:B-dilation} equal $4s$,
so the piecewise definition is in $H^1$. The map
$(s,v_i)\mapsto\mathcal D_s v_i$ is strongly continuous in $H^1$:
for a smooth function vanishing on the circular arc, continuity follows
directly from \eqref{eq:B-dilation}; approximation and the scale invariance of
the two-dimensional Dirichlet integral extend continuity to every $v_i$. The change of variables also proves continuity in $L^2$ on bounded
$s$ intervals. Thus the dilation defines a continuous homotopy.
Pull back each $\mathcal D_sv_i$ by $y_{z_i}$ and extend by zero
on $\Si \setminus\bigcup_i\Omega_i$; denote the resulting function
by $\mathcal D_sv$. Then
\begin{equation*}
\frac12\int_\Si |\nabla(\mathcal D_sv)|_g^2\,\d v_g
=
\frac12\int_\Si |\nabla v|_g^2\,\d v_g
+8\pi(2p+q)s.
\end{equation*}
Changing variables on $B_{R_ie^{-s}}^{z_i}$ yields
\[
 \int_{B_{R_ie^{-s}}^{z_i}}W_i(y)
       e^{v_i(e^sy)+4s}\,\d y
 =e^{2s}\int_{B_{R_i}^{z_i}}W_i(e^{-s}y)e^{v_i(y)}\,\d y,
\]
The domain is a disk for $z_i\in\intS$ and a half-disk for
$z_i\in\pa \Si $.
By \eqref{eq:B-density-oscillation}, the integral equals
\(e^{2s}M_i^0(1+O(\var _{\rm loc}))\).
The annular integral is at most $Ce^{2s}$, and the integral on
$\Si \setminus\bigcup_i\Omega_i$ is bounded. Thus, uniformly for
$s\ge0$,
\[
 \mathfrak m_{\mathcal D_sv}(\Omega_i)
 =\gamma_i+O(\var _{\rm loc})+O(K^{-1}).
\]
The estimates for $\mathfrak m_{\mathcal D_sv}(\Omega_i)$ and the
small diameters of $\Omega_i$ yield an estimate for the bounded-Lipschitz distance
throughout the dilation. No total-variation estimate is used on this
segment.
Direct integration also yields
\[
\overline{\mathcal D_sv}
\le C+Ce^{-2s}\sqrt e,
\quad
\ln\Big(\int_\Si  Ve^{\mathcal D_sv}\,\d v_g\Big)
\ge
\ln\Big(\int_\Si  Ve^u\,\d v_g\Big)-T+2s-C.
\]
Therefore
\begin{equation*}
\mathscr J_a(\mathcal D_sv)
\le
\mathscr J_b(u)
-\frac12(b-a)e
+C_M\sqrt{1+e}+C
-(2\rho-8\pi a(2p+q))s.
\end{equation*}
The coefficient is positive by \eqref{eq:B-coefficient-range}; choose
\begin{equation*}
s=s(u):=s_0+C_{\rm dil}(1+e).
\end{equation*}

\textbf{Step 3.} Interpolation of the local profiles.
Let \(h_{i,s}=1\) on \(|y|\le R_ie^{-s}\),
\(h_{i,s}=\ln(R_i/|y|)/s\) on the annulus, and \(h_{i,s}=0\)
outside $B_{R_i}^{z_i}$; the same notation denotes its pullback to
$\Omega_i$, extended by zero on $\Si \setminus\Omega_i$. With
\(I_i:=\int_{B_{R_i}^{z_i}}W_i e^{\mathcal D_sv_i}\,\d y,\)
choose \(\beta_i(t)\ge0\) by
\begin{equation}\label{eq:B-prescribed-cell-mass}
\int_{B_{R_i}^{z_i}}
W_i
e^{\mathcal D_s((1-t)v_i)+\beta_i(t)h_{i,s}}
\,\d y
=
I_i.
\end{equation}
The left side of \eqref{eq:B-prescribed-cell-mass} is continuous in
\(\beta\), with derivative
\[
\int_{B_{R_i}^{z_i}}W_i h_{i,s}
 e^{\mathcal D_s((1-t)v_i)+\beta h_{i,s}}\,\d y>0.
\]
At \(\beta=0\), the left side of \eqref{eq:B-prescribed-cell-mass} is
at most \(I_i\) and tends to infinity as \(\beta\to\infty\).
Thus \(\beta_i(t)\ge0\) exists uniquely.
For \(s\ge1\),
integration over the central disk and the surrounding annulus yields
\[
I_i\le Ce^{2s}\Big(1+\int_{B_{R_i}^{z_i}}e^{v_i}\,\d y\Big),
\quad
\int_{B_{R_i}^{z_i}}
W_i e^{\mathcal D_s((1-t)v_i)+\beta h_{i,s}}\,\d y
\ge ce^{2s+\beta}|B_{R_i}^{z_i}|.
\]
The zero trace on the circular arc permits the disk Moser--Trudinger
inequality (after even reflection for a half-disk):
\[
 \ln\Big(1+\int_{B_{R_i}^{z_i}}e^{v_i}\,\d y\Big)
 \le C+\frac1{8\pi}\int_{B_{R_i}^{z_i}}|\nabla v_i|^2\,\d y
 \le C(1+e).
\] The upper bound for $I_i$ and the lower bound $ce^{2s+\beta}|B_{R_i}^{z_i}|$
therefore imply
\begin{equation}\label{eq:B-core-offset-bound}
0\le\beta_i(t)\le C(1+e).
\end{equation}
The integrals in \eqref{eq:B-prescribed-cell-mass} are continuous on bounded \(H^1\) sets
by Moser--Trudinger.
Dilation continuity, \eqref{eq:B-core-offset-bound}, and the strictly
increasing scalar inverse imply joint continuity of \(\beta_i(t)\). The functions
\(h_{i,s}\) are strongly continuous for \(s>0\).
For the resulting function \(\widetilde v_t\),
\begin{equation*}
\frac12\int_\Si |\nabla\widetilde v_t|^2\,\d v_g
=
(1-t)^2\frac12\int_\Si |\nabla v|^2\,\d v_g
+
\sum_i
\Big(
2\kap (z_i)s
+\kap (z_i)\beta_i(t)
+\frac{\kap (z_i)\beta_i(t)^2}{8s}
\Big).
\end{equation*}
Furthermore,
\[
\overline{\widetilde v_t}
\le
C+Ce^{-2s}\sqrt e
+C(1+e)(e^{-2s}+s^{-1}),
\quad
\ln\int_\Si  Ve^{\widetilde v_t}\,\d v_g
\ge2s-C,
\]
hence
\begin{equation}\label{eq:B-core-interpolation-bound}
\mathscr J_a(\widetilde v_t)
\le
-(2\rho-8\pi a(2p+q))s
+C(1+e)
+\frac{C(1+e)^2}{s}
+C.
\end{equation}
By \eqref{eq:B-coefficient-range},
\(c_0:=2\rho-8\pi a_{\max}l>0\). With
\(s=s_0+C_{\rm dil}(1+e)\), \eqref{eq:B-core-interpolation-bound} is at most
\[
 -c_0s_0+
 \Big(-c_0C_{\rm dil}+C+\frac C{C_{\rm dil}}\Big)(1+e)+C.
\]
Choose \(C_{\rm dil}\) to obtain a negative coefficient of \(1+e\),
and then choose \(s_0\) satisfying \(-c_0s_0+C\le-D\).
Thus $\mathscr J_a(\widetilde v_t)\le-D$ uniformly on the domain of the homotopy.

\textbf{Step 4.} A fixed reference function.
Choose \(A_i\ge0\) and \(C_\si >0\) so that
\[
\frac{\kap (z_i)}4R_i^2W_i(0)e^{A_i}
=
C_\si \gamma_i,
\]
and define
\[
I_i^{\rm can}(s)
:=
\int_{B_{R_i}^{z_i}}
W_i e^{\mathcal D_s(0)+A_ih_{i,s}}\,\d y.
\]
Then
\begin{equation*}
I_i^{\rm can}(s)
=
C_\si \gamma_i e^{2s}(1+o(1)),
\quad
s\to\infty.
\end{equation*}
Indeed, the constant-density core and annular integrals equal
\[
 \frac{\kap(z_i)}4R_i^2W_i(0)
 \Big[\frac{e^{2s+A_i}}2+
       \frac{e^{2s+A_i}-1}{2+A_i/s}\Big].
\]
After division by $e^{2s}$, continuity of $W_i$ changes the limit
by zero. The two terms each contribute one half of the required mass.
Interpolate the integrals over $\Omega_i$ by
\(I_i(t):= I_i^{\,1-t} [I_i^{\rm can}(s)]^t, \quad 0\le t\le1,\)
and choose the unique coefficient of $h_{i,s}$ for which the integral equals $I_i(t)$.
Since \(v_i,A_i\ge0\), both masses are at least
\(\int_{B_{R_i}^{z_i}}W_i e^{\mathcal D_s(0)}\,\d y\). Their interpolated value
is at least this common lower bound and remains between the two masses. The strictly increasing
scalar inverse from Step 3 therefore defines a continuous coefficient in
\([0,\max\{\beta_i(1),A_i\}]\). Consequently,
\eqref{eq:B-core-offset-bound} and \eqref{eq:B-core-interpolation-bound}
hold throughout this segment.

To estimate the normalized integrals over $\Omega_i$ during the interpolation,
let \(\mathbf p=(p_i)\) and \(\mathbf q=(q_i)\) be the normalized initial
and final mass vectors. For sufficiently small errors,
\(p_i,q_i\ge\gamma_i/2\). On this set of pairs of probability vectors,
the map
\[
 (\mathbf p,\mathbf q,t)\mapsto
 \frac{p_i^{1-t}q_i^t}{\sum_jp_j^{1-t}q_j^t},
 \quad 0\le t\le1,
\]
is uniformly Lipschitz in \(\mathbf p,\mathbf q\) and equals \(\gamma_i\) when
\(\mathbf p=\mathbf q=(\gamma_i)_i\). Its Lipschitz constant may depend on the fixed
positive weights \(\gamma_i\). Hence, by taking the initial mass
error sufficiently small and \(s_0\) sufficiently large,
\[
\sup_{0\le t\le1}
\Big|
\frac{I_i^{\,1-t}[I_i^{\rm can}(s)]^t}
{\sum_j I_j^{\,1-t}[I_j^{\rm can}(s)]^t}
-\gamma_i
\Big|
\le
C\var _{\rm loc}+o_{D'}(1)+o_{s_0}(1).
\]
The normalized mass on $\Si \setminus\bigcup_i\Omega_i$ is
$O(e^{-2s})$, and the neighborhoods $\Omega_i$ may be chosen with
arbitrarily small diameter. Hence the whole interpolating path remains
in \(\mathcal U\) in the bounded-Lipschitz topology.
Finally,
\[
\mathscr J_a
\Big(
\mathcal D_s(0)+\sum_iA_ih_{i,s}
\Big)
\le
-(2\rho-8\pi a(2p+q))s+C.
\]
Then decrease \(s\) to \(s_0\). The resulting final function is
independent of \(u\).
We finish by making the neighborhood and parameter choices explicit.
Choose \(r_{\mathcal U}>0\) such that the bounded-Lipschitz ball of
radius \(r_{\mathcal U}\) about \(\si \) is contained in
\(\mathcal U\). For any probability measure \(\nu\),
\[
 d_{\rm BL}(\nu,\si )
 \le \max_i\operatorname{diam}_g(\Omega_i)
 +\sum_i|\nu(\Omega_i)-\gamma_i|
 +\nu\Big(\Si \setminus\bigcup_i\Omega_i\Big).
\]
Indeed, subtract \(f(z_i)\) from a bounded-Lipschitz test function
on each $\Omega_i$ and estimate the integral on
$\Si \setminus\bigcup_i\Omega_i$ directly.
For every function on the seven path segments, the preceding estimates yield
an estimate for the bounded-Lipschitz distance of its normalized measure
from $\si $:
\[
 \max_i\operatorname{diam}_g(\Omega_i)
 +C_\gamma\var _{\rm loc}+C_{\rm loc}K^{-1}
 +\om _\si (s_0),
 \quad \om _\si (s_0)\to 0.
\]
The coefficient \(C_\gamma\) depends only on the fixed weights;
\(C_{\rm loc}\) is fixed after the neighborhoods $\Omega_i$ and truncation parameters
are chosen. The term \(\om _\si (s_0)\to0\) includes the reference
mass error and the mass on $\Si \setminus\bigcup_i\Omega_i$,
uniformly for $s\ge s_0$.
Choose $\var _{\rm loc}$ and the neighborhoods $\Omega_i$ so that
\eqref{eq:B-density-oscillation} holds and the first two terms are less
than \(r_{\mathcal U}/2\). Then choose the cutoffs, \(\var ,M\)
for \eqref{eq:B-localization-energy}, \(A_i,C_\si \), and
\(C_{\rm dil},s_0\), in this order. Increase \(s_0\) so that
\(\om _\si (s_0)<r_{\mathcal U}/4\), preserving the energy estimates.
Finally choose \(D'\ge D\) large enough for Steps 1--2 and for
\(C_{\rm loc}K^{-1}<r_{\mathcal U}/4\), using
\eqref{eq:B-source-mass}. All choices are independent of \(u\).

Apply \(\Pi w=w-\overline w\) and concatenate the seven segments on
fixed subintervals of \([0,1]\). The successive paths agree at their joining times, and the
continuity proved in Steps 1--4 yields a jointly continuous \(H^1\)
homotopy. The projection preserves \(\mathscr J_a\) and \(\mathfrak m\),
fixes the initial mean-zero function, and yields a final reference
function independent of \(u\). By the energy and measure estimates,
the entire homotopy takes values in
\(X_{a,D}\cap\mathfrak m^{-1}(\mathcal U)\), proving
\eqref{eq:B-local-inclusion}.
\end{proof}

Fix $1\le a<b\le a_{\max}$, $D>0$ and $\var >0$.
For each $\tau\in B_l^\pa (\Si ,\pa \Si )$, apply
Lemma~\ref{lem:B-local-contraction} so that the homotopy takes values in
\(\{\nu:d_{\rm BL}(\nu,\tau)<\var /3\}\).
Choose \(0<r_\tau<\var /6\) so that the neighborhood on which the
local contraction is defined contains
\(\{\nu:d_{\rm BL}(\nu,\tau)<2r_\tau\}\).
Choose finitely many $\tau_i$ so that the balls
$\{\si :d_{\rm BL}(\si ,\tau_i)<r_{\tau_i}\}$ cover
$B_l^\pa (\Si ,\pa \Si )$, take
$0<\delta<\min_i r_{\tau_i}$, and let $D'$ dominate all their
thresholds. Given $\si $ in this compact barycenter space, choose
$i$ with $d_{\rm BL}(\si ,\tau_i)<r_{\tau_i}$.
The triangle inequality shows that the ball of radius \(\delta\) about
\(\si \) is contained in the domain of the local contraction centered at
\(\tau_i\), and the contraction has image in the ball of radius \(\var \) about \(\si \).
The constants can therefore be chosen uniformly in \(\si \), without
a continuous choice of null-homotopies. Thus the inclusion
\begin{equation}\label{eq:B-uniform-local-contraction}
X_{b,D'}\cap
\{
d_{\rm BL}(\mathfrak m_u,\si )<\delta
\}
\to
X_{a,D}\cap
\{
d_{\rm BL}(\mathfrak m_u,\si )<\var
\}
\end{equation}
is null-homotopic through maps with values in the set on the right, uniformly in
\(\si \in B_l^\pa (\Si ,\pa \Si )\).

The finite-family argument below combines the preceding local contractions
with extension over simplices; compare the closed-surface low-sublevel
homotopy in \cite[Lemma~4.2]{M2008}. For the boundary-weighted strata, the relevant transitions are
\[
z_i=z_j,\quad \gamma_i=0,\quad z_i\to\pa \Si .
\]
\begin{lem}\label{lem:B-finite-family}
Fix $D_0$ sufficiently large and an integer $n\ge0$. There are
$D_*\ge D_0$ and $\delta_*>0$ with the following property.
Let $(P,Q)$ be a finite polyhedral pair with $\dim P\le n$. Suppose that
$f_0,f_1:P\to X_{a_{\max},D_*}$ and
$\Ga :P\to B_l^\pa (\Si ,\pa \Si )$ are continuous,
$f_0=f_1$ on $Q$, and
\[
 d_{\rm BL}(\mathfrak m_{f_j(x)},\Ga (x))<\delta_*,
 \quad \forall\, x\in P,\ j=0,1.
\]
Then $f_0$ and $f_1$ are homotopic relative to $Q$ through maps into
$X_{1,D_0}$.
\end{lem}
\begin{proof}
Choose $1=a_0<a_1<\cdots<a_{n+1}=a_{\max}$ and set
$X_j=X_{a_j,D_j}$. Starting with any $\var _0>0$, apply
\eqref{eq:B-uniform-local-contraction} for
$j=0,\ldots,n$. At stage \(j\), obtain \(D_{j+1}\ge D_j\) and, after decreasing the
radius if necessary, choose \(0<\delta_j<\var _j/8\) such that
\[
 X_{j+1}\cap\{d_{\rm BL}(\mathfrak m_u,\si )<\delta_j\}
 \to
 X_j\cap\{d_{\rm BL}(\mathfrak m_u,\si )<\var _j\}
\]
is null-homotopic through maps with values in the set on the right for every $\si \in B_l^\pa $.
If $j<n$, next choose $0<\var _{j+1}<\delta_j/8$.
Set $D_*=D_{n+1}$ and
$0<\delta_*<\frac1{16}\min_{0\le j\le n}\delta_j$.
Subdivide the finite polyhedral pair $(P,Q)$ so that
$\operatorname{osc}_\Delta\Ga <\delta_*$ on each simplex.

Let $H(x,0)=f_0(x)$, $H(x,1)=f_1(x)$, and
$H(x,t)=f_0(x)$ for $x\in Q$. Choose \(x_\Delta\in\Delta\) for each simplex and set
\(\si _\Delta=\Ga (x_\Delta)\).
We extend by induction over $\Delta^k\times[0,1]$, $k=0,\ldots,n$,
keeping the following containment condition
\[
 H(\Delta^k\times[0,1])\subset
 X_{n-k}\cap
 \{d_{\rm BL}(\mathfrak m_u,\si _\Delta)<\var _{n-k}\}
 \quad\text{ for }\,\Delta^k\not\subset Q.
\]
At the \(k\)-th stage the faces \(t=0\) and \(t=1\), together with every
already prescribed piece over \(Q\), have distance less than
\(2\delta_*\) from
$\si _\Delta$ and take values in $X_{n+1}$.
For $k\ge1$, every other side comes from a $(k-1)$-face and takes values in
$X_{n+1-k}$; its distance from $\si _\Delta$ is less than
\[
 \var _{n+1-k}+\delta_*
 <\frac{\delta_{n-k}}8+\frac{\delta_{n-k}}{16}
 <\delta_{n-k}.
\]
Thus the boundary map takes values in the domain of the
null-homotopy at stage $n-k$. Composing the two maps defines a homotopy
of the boundary map to a constant. Since the prism is homeomorphic to
a closed ball, this homotopy defines a continuous extension over the
prism, with the prescribed boundary values and the stated containment condition.
For \(k=0\), the same construction joins the two prescribed values by a
path.
There are only finitely many prisms, so the pasting lemma proves
continuity. The resulting homotopy takes values in $X_0=X_{1,D_0}$
and is fixed on $Q$.
\end{proof}

\begin{prop}\label{prop:nonresonant-neumann-sublevels}
Let $l\ge1$ and $\rho\in(4\pi l,4\pi(l+1))$. For every sufficiently
large $D$, the maps constructed above, with $\lam _0$ sufficiently
large, are homotopy inverses:
\begin{equation}\label{eq:B-nonresonant-inverses}
 \psi_D\phi_{\lam _0}\simeq\operatorname{Id}_{B_l^\pa (\Si ,\pa \Si )},
 \quad
 \phi_{\lam _0}\psi_D\simeq\operatorname{Id}_{J_\rho^{-D}}.
\end{equation}
Both homotopies take values in the respective spaces in
\eqref{eq:B-nonresonant-inverses}. In particular,
$J_\rho^{-D}\simeq B_l^\pa (\Si ,\pa \Si )$. The upper sublevel $J_\rho^D$
is contractible for large $D$, and the lower homotopy type is stable
under further decreases of the level.
\end{prop}
\begin{proof}
Set
$X:=J_\rho^{-D}$, $B:=B_l^\pa (\Si ,\pa \Si )$.
By \eqref{eq:B-one-sided-homotopy},
\(\psi_D\phi_{\lam _0}\simeq\operatorname{Id}_{B},\)
and \eqref{eq:B-reinforced-test} shows that the maps
\(\phi_\lam \) remain in \(X\) as \(\lam \) increases.

Lemma~\ref{lem:B-reinforcement},
\eqref{eq:B-projection-closeness}, and
Lemma~\ref{lem:B-finite-family} with \(n=0\) yield
\[
(\phi_{\lam _0})_*:\pi_0(B)\xrightarrow{\;\cong\;}\pi_0(X).
\]
Fix $n\ge1$ and $\si _0\in B$. First choose $D_*$ and $\delta_*$
from Lemma~\ref{lem:B-finite-family}. Increase $D_*$, keeping $\delta_*$
fixed, until \eqref{eq:B-projection-closeness} yields
\[
\sup_{u\in X_{a_{\max},D_*}}
 d_{\rm BL}(\mathfrak m_u,\psi_D(u))<\delta_*/4.
\]
Increasing the depth $D_*$ does not affect the conclusion of Lemma~\ref{lem:B-finite-family}.
Choose an open star \(W\) of \(\si _0\) in a sufficiently fine
triangulation of \(B\), after subdividing so that \(\si _0\) is a
vertex, with
\(W\subset\{\si \in B:d_{\rm BL}(\si ,\si _0)<\delta_*/8\}.\)
Let \(K:W\times[0,1]\to W\) contract \(W\) to \(\si _0\), with
\(K(\si ,0)=\si \) and \(K(\si ,1)=\si _0\).
Then choose $\Lam \gg1$ satisfying
\[
\phi_\Lam (B)\subset X_{a_{\max},D_*},
\quad
d_{\rm BL}(
\mathfrak m_{\phi_\Lam (\si )},\si
)<\frac{\delta_*}{4},
 \quad\forall\,\si \in B,
\]
and, in addition,
\[
 \psi_D(\phi_\Lam (\si _0))\in W,
 \quad
 d_{\rm BL}(\mathfrak m_{\phi_\Lam (\si _0)},\si _0)
    <\delta_*/16.
\]
The additional conditions follow from the estimate for
\(\phi_\Lam \) and continuity of the neighborhood retraction at
\(\si _0\).
Set \(x_0:=\phi_\Lam (\si _0)\) as the base point. For a based map
\(f:(\Sn,*)\to(X,x_0)\), the strong retraction of
Lemma~\ref{lem:B-reinforcement} fixes \(x_0\) and yields
$f\simeq f'$, $f'(\Sn)\subset X_{a_{\max},D_*}$.
Set \(\Ga _0:=\psi_D\circ f'\). Since \(f'(*)=x_0\), continuity
yields a neighborhood \(N\) of \(*\) in \(\Sn\) such that
\[
 \Ga _0(N)\subset W,
 \quad
 d_{\rm BL}(\mathfrak m_{f'(x)},\si _0)<\delta_*/8,
 \quad \forall\,x\in N.
\]
Choose a continuous function \(\zeta:\Sn\to[0,1]\), with
\(\zeta(*)=1\) and support contained in \(N\), and set
\[
 \Ga (x):=
\left\{\begin{aligned}
& K(\Ga _0(x),\zeta(x)),&&x\in N,\\
 &\Ga _0(x),&&x\notin N.
 \end{aligned}\right.
\]
The two definitions agree wherever \(\zeta=0\), so \(\Ga \) is
continuous, \(\Ga (*)=\si _0\), and
\[
 d_{\rm BL}(\mathfrak m_{f'(x)},\Ga (x))<\delta_*/2,
 \quad\forall\,x\in \Sn.
\]
Lemma~\ref{lem:B-finite-family}, relative to \((\Sn,\{*\})\), yields
\(f' \simeq \phi_\Lam \circ\Ga \quad\text{ in }\,X.\)
Thus \((\phi_\Lam )_*\), hence \((\phi_{\lam _0})_*\), is surjective
on \(\pi_n\). Injectivity follows from
\(\psi_D\phi_{\lam _0}\simeq\operatorname{Id}_B\), after the standard
base-point change induced by the homotopy in
\eqref{eq:B-one-sided-homotopy}. Therefore \(\phi_{\lam _0}\)
is a weak homotopy equivalence.

The space \(B\) is a compact polyhedron by
Lemma~\ref{lem:B-concentration}. Increase the initial threshold for
\(D\), if necessary, so that \(D-1\) is admissible in
Lemma~\ref{lem:B-reinforcement}. Apply its stopped-flow construction
to \(J_\rho^{-D+1}\) and \(J_\rho^{-D}\). The functional \(J_\rho\)
is nonincreasing during the deformation. Consequently the deformation restricted
to the open set
\(\mathcal O:=\{u\in\oH:J_\rho(u)<-D+1\}\)
is a strong deformation retraction of \(\mathcal O\) onto \(X\).
Since an open subset of a separable Hilbert space has CW homotopy type
\cite{Palais1966}, the same is true of \(X\).
Whitehead's theorem \cite[Theorem~4.5]{Hatcher2002}, applied on each
component after passing to CW models, shows that the specific map
$\phi_{\lam _0}$ is a homotopy equivalence. The already constructed map $\psi_D$ is its homotopy
inverse. Indeed, if $g:X\to B$ is any homotopy inverse, then
\[
 \psi_D\simeq\psi_D\phi_{\lam _0}g\simeq g,
 \quad
 \phi_{\lam _0}\psi_D\simeq\phi_{\lam _0}g
                          \simeq\operatorname{Id}_X.
\]
Thus both identities in \eqref{eq:B-nonresonant-inverses} hold for the
maps specified there. Contractibility of $J_\rho^D$ and stability under
decreasing the lower level follow from Lemma~\ref{lem:B-reinforcement}.
\end{proof}

\begin{proof}[Proof of Proposition~\ref{prop:negative-sublevels-barycenters}]
Fix \(\rho_-\in(4\pi(m-1),\rho_*)\), set \(c_*=\rho_*/\rho_-\), and
write \(X=J_{\rho_*}^{-L}\), \(Y=J_{\rho_-}^{-L/c_*}\), and
\(B=B_{m-1}^\pa (\Si ,\pa \Si )\).
Proposition~\ref{prop:B-resonant-comparison} establishes a strong retraction
\(\mathcal R_L:X\to Y\) with inclusion \(\iota:Y\hookrightarrow X\).
Proposition~\ref{prop:nonresonant-neumann-sublevels} establishes homotopy
inverses \(\phi:B\to Y\) and \(\psi:Y\to B\). Therefore
\(\Phi_L:=\iota\phi,\quad \Psi_L:=\psi\mathcal R_L\)
are homotopy inverses, since
\[
\Psi_L\Phi_L=\psi\phi\simeq\operatorname{Id}_B,
\quad
\Phi_L\Psi_L=\iota\phi\psi\mathcal R_L
\simeq\iota\mathcal R_L\simeq\operatorname{Id}_X.
\]
Together with Proposition~\ref{prop:B-resonant-comparison}, this proves
\eqref{eq:B-main-topology} and stabilization. As \(J_{\rho_*}^L\) is
contractible and \(X\ne\varnothing\), the long exact sequence of the
pair yields
\[
H_k(J_{\rho_*}^{L},X;\mathbb F)
\cong\widetilde H_{k-1}(X;\mathbb F)\quad(k\ge1),
\quad H_0(J_{\rho_*}^{L},X;\mathbb F)=0.
\]
Substitution of \(X\simeq B\) proves \eqref{eq:B-main-relative}.
\end{proof}

The prescribed-index argument also uses the following nonresonant deformation result.
\begin{cor}\label{prop:B-nonresonant-Morse}
Let $\rho\in(4\pi l,4\pi(l+1))$, $l\ge1$, and suppose $J_{\rho,V}$
is Morse. For large $D$, the pair $(J_\rho^D,J_\rho^{-D})$ has a finite
sublevel filtration whose relative homology groups of consecutive pairs are the
direct sums of the local Morse groups of its critical points.
In particular, $J_{\rho,V}$ has a critical point of Morse index \(2l\).
\end{cor}
\begin{proof}
For every finite interval $[c,d]$, the set
\[
\{u=T_r(u):0\le r\le\rho,\ c\le J_\rho(u)\le d\}
\]
is compact in $\oH$. Indeed, a bounded sequence has a convergent subsequence by Lemma~\ref{lem:B-compact-gradient}. If an unbounded sequence exists,
Lemma~\ref{lem:B-exact-energy} yields a resonant limit
$\bar r\le4\pi l<\rho$ and hence $J_\rho(u)\to-\infty$, a
contradiction.

The set of genuine critical points is compact and nondegenerate,
so it is finite.

Choose disjoint small Morse neighborhoods of these critical points,
all inside a fixed ball. Outside them, the gradient is bounded away
from zero on every bounded part of this energy interval: otherwise
Lemma~\ref{lem:B-compact-gradient} yields another critical point.
Outside a sufficiently large ball, the convex hull of
$\{\nabla J_\rho(u),u\}$ is separated from zero on every bounded
annulus in this energy interval. Indeed, failure on such an annulus
yields $a_n\in[0,1]$ and
\(u_n-a_nT_\rho(u_n)\to 0.\)
Compactness on bounded sets and extraction of $a_n$ yield strong
convergence to $u=T_{a\rho}(u)$ in the same closed annulus and
energy interval, contradicting the compactness just proved. The construction in
Lemma~\ref{lem:B-bounded-descent} thus yields a locally Lipschitz descending
field outside the Morse neighborhoods, radially nonincreasing outside
that ball.

Inside the Morse neighborhoods, choose the standard local
negative-gradient vector fields, rescaled to be bounded. A locally
finite partition of unity preserves strict descent where several
local vector fields are combined. Outside the large ball the same partition
also preserves the nonincreasing norm: all local modifications have bounded support.

Every resulting trajectory in the energy interval stays bounded, and
the bounded speed excludes finite-time loss of existence. On intervals
containing no critical values, the lower bound for the gradient on
bounded subsets permits a deformation obtained by stopping the flow at the prescribed level. Across
a critical value, the relative group is the direct sum of the local groups in
the disjoint critical neighborhoods, by the local Morse lemma and excision.

This produces the stated finite filtration, also when several critical
points have the same value. No global Palais--Smale condition is deduced from compactness of
solutions of the equation.

For the index conclusion, the one-sided homotopy
\eqref{eq:B-one-sided-homotopy} in Lemma~\ref{lem:B-test-functions}
is sufficient: $(\phi_D)_*$ is injective on reduced homology.
Lemma~\ref{lem:B-reinforcement} implies $J_\rho^D\simeq\oH$.
The long exact sequence of the pair and \eqref{eq:topology-top-betti}
therefore imply
\[
 \dim_{\Z_2}H_{2l}(J_\rho^D,J_\rho^{-D};\Z_2)
 =\dim_{\Z_2}\widetilde H_{2l-1}(J_\rho^{-D};\Z_2)
 \ge\binom{b+l-1}{l}>0.
\]
The weak Morse inequality proves the existence of a critical point
of index $2l$, without the opposite homotopy in
\eqref{eq:B-nonresonant-inverses}.
\end{proof}

		\medskip

		\noindent\textbf{Acknowledgements.}~
Part of this work was carried out during the third-named author's visits to
Justus-Liebig-Universität Gießen  from August 2024 to  August 2025, supported
by the China Scholarship Council (CSC) and the National Natural Science
Foundation of China (Grant No.~202406040139). He gratefully acknowledges
the hospitality of the Mathematical Institute at Justus-Liebig-Universität
Gießen.
	M. Ahmedou is partially supported by the
	Deutsche Forschungsgemeinschaft (DFG) under grant AH~156/2-1. Z. Hu is partially supported by the
	National Natural Science Foundation of China
	(Grant No.~12601383) and the China Postdoctoral Science Foundation
	(Grant No.~2026M793392). H. Wang is  partially supported by China Postdoctoral Science Foundation 	(Grant No.~2026M793418). L. Zhang is partially supported by Simons Foundation Grant SFI-MPS-TSM-00013752.

		\medskip

		\noindent\textbf{Data availability statement.}~	 Data sharing is not applicable to this article as no datasets were generated or analysed during the current study.

		\medskip

		\noindent\textbf{Declarations.}~ \textbf{Conflict of interest.} ~ On behalf of all authors, the corresponding author states that there is no conflict of interest.

		\medskip

\noindent\textbf{AI Declaration.}
~
During the preparation of this manuscript, the authors used
ChatGPT (OpenAI, GPT-6.1 Sol) for language editing, consistency
checks, and assistance in developing parts of the arguments
in Appendix~B. The authors reviewed and revised the
AI-assisted material and take full responsibility for
the content of the manuscript.


{\small \begin{thebibliography}{10}

				\bibitem{ADN1959}
				S.~Agmon, A.~Douglis, and L.~Nirenberg.
				\newblock Estimates near the boundary for solutions of elliptic partial
				differential equations satisfying general boundary conditions. {I}.
				\newblock {\em Comm. Pure Appl. Math.}, 12:623--727, 1959.

\bibitem{ABF2023}
				M.~Ahmedou, T.~Bartsch, and T.~Fiernkranz.
				\newblock Equilibria of vortex type {H}amiltonians on closed surfaces.
				\newblock {\em Topol. Methods Nonlinear Anal.}, 61(1):239--256, 2023.

\bibitem{ABH20240829}
M.~Ahmedou, T.~Bartsch, and Z.~Hu.
\newblock Blow-up solutions for mean field equations on {R}iemann surfaces
with boundary {I}.
\newblock {\em Ann. Mat. Pura Appl. (4)}, 2026.

\bibitem{AB2017}
				M.~Ahmedou and M.~Ben~Ayed.
				\newblock Theory of ``critical points at infinity'' and a resonant singular
				{L}iouville-type equation.
				\newblock {\em Adv. Nonlinear Stud.}, 17(1):139--166, 2017.

\bibitem{AB2023}
				M.~Ahmedou and M.~Ben~Ayed.
				\newblock Morse inequalities at infinity for a resonant mean field equation.
				\newblock {\em Commun. Contemp. Math.}, 25(1):Paper No. 2150054, 37, 2023.
\newblock Preprint version: \href{https://arxiv.org/abs/2101.12611v2}{arXiv:2101.12611v2}.
\newblock Proposition and lemma locators used here refer to this preprint version.

\bibitem{ABL2017}
				M.~Ahmedou, M.~Ben~Ayed, and M.~Lucia.
				\newblock On a resonant mean field type equation: a ``critical point at
				infinity'' approach.
				\newblock {\em Discrete Contin. Dyn. Syst.}, 37(4):1789--1818, 2017.

\bibitem{AHW2025-1}
M.~Ahmedou, Z.~Hu, and H.~Wang.
\newblock A degree-counting formula for a {K}eller--{S}egel equation on a
surface with boundary.
\newblock {\em Commun. Contemp. Math.}, Paper No.~2650039, 2026.
\newblock \href{https://doi.org/10.1142/S0219199726500392}
{doi:10.1142/S0219199726500392}.
\newblock Preprint version: \href{https://arxiv.org/abs/2506.12783v2}{arXiv:2506.12783v2}.
\newblock Equation and proposition locators refer to this preprint version.

\bibitem{AK2026}
				M.~Ahmedou and S.~Kallel.
				\newblock Boundary-weighted barycenter spaces, 2026.
				\newblock \href{https://arxiv.org/abs/2608.22040}{arXiv:2608.22040}.

\bibitem{AKN2026}
M.~Ahmedou, S.~Kallel, and C.~B. Ndiaye.
\newblock The resonant boundary $Q$-curvature problem and
boundary-weighted barycenters.
\newblock Manuscript, version of 3 October 2026.

\bibitem{Bbook1989}
				A.~Bahri.
				\newblock {\em Critical points at infinity in some variational problems},
				volume 182 of {\em Pitman Research Notes in Mathematics Series}.
				\newblock Longman Scientific \& Technical, Harlow; copublished in the United
				States with John Wiley \& Sons, Inc., New York, 1989.

\bibitem{B1996}
				A.~Bahri.
				\newblock {An invariant for Yamabe-type flows with applications to
					scalar-curvature problems in high dimension}.
				\newblock {\em Duke Mathematical Journal}, 81(2):323--466, 1996.

\bibitem{BC1988}
				A.~Bahri and J.-M. Coron.
				\newblock On a nonlinear elliptic equation involving the critical {S}obolev
				exponent: the effect of the topology of the domain.
				\newblock {\em Comm. Pure Appl. Math.}, 41(3):253--294, 1988.

\bibitem{BC1991}
				A.~Bahri and J.-M. Coron.
				\newblock The scalar-curvature problem on the standard three-dimensional
				sphere.
				\newblock {\em J. Funct. Anal.}, 95(1):106--172, 1991.

\bibitem{BMP2019}
				T.~Bartsch, A.~M. Micheletti, and A.~Pistoia.
				\newblock The {M}orse property for functions of {K}irchhoff--{R}outh path type.
				\newblock {\em Discrete Contin. Dyn. Syst. Ser. S}, 12(7):1867--1877, 2019.

\bibitem{B2019}
				L.~Battaglia.
				\newblock A general existence result for stationary solutions to the
				{Keller}-{Segel} system.
				\newblock {\em Discrete \& Continuous Dynamical Systems - A}, 39(2):905--926,
				2019.

\bibitem{L1957}
L.~Bers.
\newblock {\em Riemann Surfaces}.
\newblock Lecture notes, Institute of Mathematical Sciences,
New York University, 1957--1958.

\bibitem{BCR1998}
J.~Bochnak, M.~Coste, and M.-F.~Roy.
\newblock {\em Real Algebraic Geometry}.
\newblock Ergebnisse der Mathematik und ihrer Grenzgebiete (3), Vol.~36.
Springer-Verlag, Berlin, 1998.

\bibitem{CL2003}
				C.-C. Chen and C.-S. Lin.
				\newblock Topological degree for a mean field equation on {R}iemann surfaces.
				\newblock {\em Comm. Pure Appl. Math.}, 56(12):1667--1727, 2003.

\bibitem{CL2015}
				C.-C. Chen and C.-S. Lin.
				\newblock Mean field equation of {L}iouville type with singular data:
				topological degree.
				\newblock {\em Comm. Pure Appl. Math.}, 68(6):887--947, 2015.

\bibitem{CL1991}
				W.~Chen and C.~Li.
				\newblock Classification of solutions of some nonlinear elliptic equations.
				\newblock {\em Duke Math. J.}, 63(3):615--622, 1991.

\bibitem{C1955}
				S.-S. Chern.
				\newblock An elementary proof of the existence of isothermal parameters on a
				surface.
				\newblock {\em Proc. Amer. Math. Soc.}, 6:771--782, 1955.

\bibitem{EF2014}
				P.~Esposito and P.~Figueroa.
				\newblock Singular mean field equations on compact {R}iemann surfaces.
				\newblock {\em Nonlinear Anal.}, 111:33--65, 2014.

\bibitem{F2024}
				P.~Figueroa.
				\newblock Bubbling solutions for mean field equations with variable intensities
				on compact {R}iemann surfaces.
				\newblock {\em J. Anal. Math.}, 152(2):507--555, 2024.

\bibitem{HW1955}
				P.~Hartman and A.~Wintner.
				\newblock On uniform {D}ini conditions in the theory of linear partial
				differential equations of elliptic type.
				\newblock {\em Amer. J. Math.}, 77:329--354, 1955.

\bibitem{Hatcher2002}
A.~Hatcher.
\newblock {\em Algebraic Topology}.
\newblock Cambridge University Press, Cambridge, 2002.

\bibitem{HB2024}
				Z.~Hu and T.~Bartsch.
				\newblock The {M}orse property of limit functions appearing in mean field
				equations on surfaces with boundary.
				\newblock {\em J. Geom. Anal.}, 34(7):Paper No. 220, 19, 2024.

\bibitem{LLWY2018}
				Y.~Lee, C.-S. Lin, J.~Wei, and W.~Yang.
				\newblock Degree counting and shadow system for {T}oda system of rank two: one
				bubbling.
				\newblock {\em J. Differential Equations}, 264(7):4343--4401, 2018.

\bibitem{LSY2023}
				J.~Li, L.~Sun, and Y.~Yang.
				\newblock The boundary value problem for the mean field equation on a compact
				{R}iemann surface.
				\newblock {\em Sci. China Math.}, 66(1):115--142, 2023.

\bibitem{L1999}
				Y.~Y. Li.
				\newblock Harnack type inequality: the method of moving planes.
				\newblock {\em Comm. Math. Phys.}, 200(2):421--444, 1999.

\bibitem{L2007}
				M.~Lucia.
				\newblock A deformation lemma with an application to a mean field equation.
				\newblock {\em Topol. Methods Nonlinear Anal.}, 30(1):113--138, 2007.

\bibitem{M2008}
				A.~Malchiodi.
				\newblock Morse theory and a scalar field equation on compact surfaces.
				\newblock {\em Adv. Differential Equations}, 13(11-12):1109--1129, 2008.

\bibitem{N2014}
				G.~Nardi.
				\newblock Schauder estimate for solutions of {P}oisson's equation with
				{N}eumann boundary condition.
				\newblock {\em Enseign. Math.}, 60(3-4):421--435, 2014.

\bibitem{N2008}
				C.~B. Ndiaye.
				\newblock Conformal metrics with constant {$Q$}-curvature for manifolds with
				boundary.
				\newblock {\em Commun. Anal. Geom.}, 16(5):1049--1124, 2008.

\bibitem{Palais1966}
R.~S. Palais.
\newblock Homotopy theory of infinite dimensional manifolds.
\newblock {\em Topology}, 5(1):1--16, 1966.

\bibitem{Smale1965}
S.~Smale.
\newblock An infinite dimensional version of {S}ard's theorem.
\newblock {\em Amer. J. Math.}, 87(4):861--866, 1965.

\bibitem{V1955}
				I.~N. Vekua.
				\newblock The problem of reduction to canonical form of differential forms of
				elliptic type and the generalized {C}auchy--{R}iemann system.
				\newblock {\em Dokl. Akad. Nauk SSSR (N.S.)}, 100:197--200, 1955.

\bibitem{YZ2021}
				Y.~Yang and J.~Zhou.
				\newblock Blow-up analysis involving isothermal coordinates on the boundary of
				compact {R}iemann surface.
				\newblock {\em J. Math. Anal. Appl.}, 504(2):Paper No. 125440, 39, 2021.

\end{thebibliography}}
	\end{document}